\documentclass[a4paper]{article}
\usepackage[T1]{fontenc}

\usepackage{graphicx} 

\usepackage{amssymb,amsmath,amsthm}
\usepackage{mathtools}
\usepackage[dvipsnames]{xcolor}
\usepackage{yhmath}
\usepackage{fullpage}
\usepackage{enumerate}
\usepackage{xspace}
\usepackage{amsthm,amsfonts,amsxtra,amscd}
\usepackage[all]{xy}
\usepackage{verbatim}
\usepackage{imakeidx}
\makeindex
\usepackage{hyperref}
\usepackage{tikz-cd}
\usepackage{csquotes}
\usepackage{ragged2e}
\usepackage{blindtext}

\newcommand{\twistAut}[2]{#2_{\Aut,#1}}

\newcommand{\gogl}{\check{\gog}}

\newcommand{\Tr}{T_r}
\newcommand{\Tl}{T_l}
\newcommand{\vac}{|0\rangle}

\newcommand{\noz}{\text{noz}\xspace}
\newcommand{\fsonoz}{fsonoz\xspace}
\newcommand{\ccg}{ccg}
\newcommand{\exhaustion}{exhaustion\xspace}
\newcommand{\cgmod}{semi-cg topological\xspace}
\newcommand{\ccgmod}{semi-\ccg\xspace topological\xspace}

\newcommand{\oSigma}{{ \overline \Sigma}}

\newcommand{\tsigma}{T_{\overline{\Sigma}}}
\newcommand{\VkSg}{\calV^\kappa_{\Sigma}(\gog)}

\newcommand{\mVkSgt}{\mV^\gog_{\Sigma,t}}
\newcommand{\calVkSgt}{\calV^\gog_{\Sigma,t}}

\newcommand{\Vkg}{V^\kappa(\gog)}
\newcommand{\calUgS}{\calU(\hgog_{\Sigma})}
\newcommand{\vbasic}{\calV_{\text{basic}}}
\newcommand{\vbasicpiu}{\calV_{\text{basic}+}}
\newcommand{\vcom}{\calV_{\text{com}}}
\newcommand{\vcompiu}{\calV_{\text{com}+}}

\newcommand{\mvbasic}{\mV_{\text{basic}}}
\newcommand{\mvbasicpiu}{\mV_{\text{basic}+}}
\newcommand{\mvcom}{\mV_{\text{com}}}
\newcommand{\mvcompiu}{\mV_{\text{com}+}}

\newcommand{\am}[1]{{\color{blue}\textsf{[[AM: #1]]}}}
\newcommand{\lc}[1] { {\color{orange} \textsf{[[LC: #1]]} } }

\DeclareMathOperator{\cont}{cont}

\newcommand{\qcc}[1]{{{\widetilde {#1}}^c}}
\newcommand{\QCC}{QCC\xspace}

\newcommand{\QCCF}{QCCF\xspace}

\newcommand{\tensor}[1]{{\stackrel{#1}{\otimes}}}
\newcommand{\exttensor}[1]{{\stackrel{#1}{\boxtimes}}}
\newcommand{\calHom}{\operatorname{\calH\! \it{om}}}

\newcommand{\calHomcont}{\operatorname{\calH\! \it{om}}^{\cont}}

\newcommand{\Homcont}{\operatorname{Hom}^{\cont}}

\newcommand{\finitosigma}{I}
\newcommand{\piuno}{\calR_r}
\newcommand{\barOX}{{\overline{\calO}_X}}

\newcommand{\pbarO}{{\calO_{\overline{ \Sigma}}}}
\newcommand{\pbarOpoli}{{\calO_{\overline{\Sigma}^*}}}
\newcommand{\pbarOK}[1]{{\calO_{\overline{\Sigma}_{#1}}}}
\newcommand{\pbarOpoliK}[1]{{\calO_{\overline{\Sigma}^*_{#1}}}}

\newcommand{\pbarF}[1]{{#1_{\overline{\Sigma}}}}
\newcommand{\pbarFpoli}[1]{{#1_{\overline{\Sigma}^*}}}
\newcommand{\pbarFI}[2]{{#1_{\overline{\Sigma}_{#2}}}}

\newcommand{\pbarOmega}{{\Omega^1_{\overline{\Sigma}}}}    
\newcommand{\pbarOmegapoli}{{\Omega^1_{\overline{\Sigma}^*}}}
\newcommand{\pbarD}{\calD_{\overline{\Sigma}}}
\newcommand{\pbarDpoli}{{\calD_{\overline{\Sigma}^*}}}

\newcommand{\barO}[1]{\overline{\calO}_{#1}}

\newcommand{\pbarOmegam}[1]{\Omega_{\overline{\Sigma}^{#1}}}

\newcommand{\pbarOm}[1]        {\calO^{#1}_{\overline{\Sigma}}}
\newcommand{\pbarOpolim}[1]    {\calO^{#1}_{\overline{\Sigma}^*}}
\newcommand{\pbarDm}[1]        {\calD^{#1}_{\overline{\Sigma}}}
\newcommand{\pbarDpolim}[1]   {\calD^{#1}_{\overline{\Sigma}^*}}

\newcommand{\unoU}{{\mathbf{1}}}

\newcommand{\Tan}{T}
\DeclareMathOperator{\Res} {Res}

\newcommand{\tRes}{\widetilde{\Res}}

\newcommand{\hgog}{{\hat{\gog}}}

\newcommand{\pbarOq}{{\calO^2_{\overline {\Sigma}}}}
\newcommand{\pbarOpoliq}{{\calO^2_{\overline{\Sigma}^*}}}
\newcommand{\pbarOpolil}{{\calO^{2,\ell}_{\overline{\Sigma}^*}}}
\newcommand{\pbarOpolir}{\calO^{2,r}_{\overline{\Sigma}^*}}
\newcommand{\pbarOqdelta}{{\calO^2_{\overline {\Sigma}}}(\infty\Delta)}

\newcommand{\otimesr}{\overrightarrow{\otimes}}
\newcommand{\otimesl}{\overleftarrow{\otimes}}
\newcommand{\otimesst}{\stackrel{*}{\otimes}}
\newcommand{\otimessh}{\stackrel{!}{\otimes}}
\newcommand{\pbarOn}{{\calO_{\overline {\Sigma}^n}}}
\newcommand{\pbarOpolin}{{\calO_{\left(\overline {\Sigma}^*\right)^n}}}

\newcommand{\pbarOpoliu}[1]{\calO^{#1}_{\overline{\Sigma}^*}}
\newcommand{\pbarOu}[1]{\calO_{\overline{\Sigma}^{#1}}}

\newcommand{\Divisorepi}{\nabla}

\newcommand{\surjmap}{\twoheadrightarrow}
\newcommand{\op}{\text{op}}

\newcommand{\gsig}[1]{\hat{\gog}_{\Sigma,#1}}
\newcommand{\ugsig}[1]{\calU_{#1}(\hat{\gog}_\Sigma)}

\DeclareMathOperator{\ffset}{ffSet}
\DeclareMathOperator{\Top}{Top}
\DeclareMathOperator{\tmod}{-mod}
\DeclareMathOperator{\Aut}{Aut}
\newcommand{\Autzero}[1]{\Aut^+_{#1}}
\newcommand{\Autpiu}[1]{\Aut^0_{#1}}
\DeclareMathOperator{\Aff}{Aff}
\DeclareMathOperator{\Fun}{Fun}
\DeclareMathOperator{\Sch}{Sch}

\theoremstyle{plain}
\newtheorem{lemma}{Lemma}[subsection]
\newtheorem{theorem}[lemma]{Theorem}
\newtheorem{proposition}[lemma]{Proposition}
\newtheorem{corollary}[lemma]{Corollary}

\theoremstyle{definition}
\newtheorem{definition}[lemma]{Definition}
\newtheorem{remark}[lemma]{Remark}
\newtheorem{example}{Example}

\theoremstyle{remark}

\theoremstyle{plain}

\theoremstyle{definition}

\theoremstyle{remark}
\newtheorem{ntz}[lemma]{Notation}

\newcommand{\mA}{\mathbb A}
 
\newcommand{\mC}{\mathbb C}

\newcommand{\mF}{\mathbb F} 
\newcommand{\mG}{\mathbb G}
\newcommand{\mH}{\mathbb H}

\newcommand{\mN}{\mathbb N}

\newcommand{\mV}{\mathbb V}

\newcommand{\mZ}{\mathbb Z}

\newcommand{\calA}{\mathcal A}
\newcommand{\calB}{\mathcal B} 
\newcommand{\calC}{\mathcal C} 
\newcommand{\calD}{\mathcal D}
\newcommand{\calE}{\mathcal E} 
\newcommand{\calF}{\mathcal F} 
\newcommand{\calG}{\mathcal G}
\newcommand{\calH}{\mathcal H} 
\newcommand{\calI}{\mathcal I} 
\newcommand{\calJ}{\mathcal J}
\newcommand{\calK}{\mathcal K} 
 
\newcommand{\calM}{\mathcal M}
\newcommand{\calN}{\mathcal N} 
\newcommand{\calO}{\mathcal O} 

\newcommand{\calQ}{\mathcal Q} 
\newcommand{\calR}{\mathcal R} 

\newcommand{\calT}{\mathcal T} 
\newcommand{\calU}{\mathcal U} 
\newcommand{\calV}{\mathcal V}
\newcommand{\calW}{\mathcal W} 
\newcommand{\calX}{\mathcal X} 
\newcommand{\calY}{\mathcal Y}

\newcommand{\goF}{\mathfrak F}

\newcommand{\gog}{\mathfrak g}

\newcommand{\gom}{\mathfrak m}

\newcommand{\gra}{\alpha} 
\newcommand{\grb}{\beta}       
\newcommand{\grg}{\gamma}
\newcommand{\grd}{\delta} 
\newcommand{\gre}{\varepsilon}

\newcommand{\grl}{\lambda}     
\newcommand{\grs}{\sigma}
\newcommand{\grf}{\varphi}

\newcommand{\ra}       {\rightarrow}
\newcommand{\la}       {\leftarrow}
\newcommand{\lra}      {\to}
\newcommand{\vuoto}    {\varnothing}

\newcommand{\isocan}   {\simeq}
\renewcommand{\geq}    {\geqslant}
\renewcommand{\leq}    {\leqslant}
\newcommand{\senza}    {\smallsetminus}

         \newcommand{\mand}     {\text{ and }}
        \newcommand{\mif}      {\text{ if }}
  
  \newcommand{\mforall}  {\text{ for all }}

\DeclareMathOperator{\Hom}  {Hom}
\DeclareMathOperator{\End}  {End}
\DeclareMathOperator{\Der}  {Der}

\DeclareMathOperator{\Spec} {Spec}

\DeclareMathOperator{\Sp}   {Sp}

\DeclareMathOperator        {\Lie}{Lie}

\newcommand{\limind}{\varinjlim}
\newcommand{\limpro}{\varprojlim}
\DeclareMathOperator{\Set}{Set}

\DeclareMathOperator{\id}{id}
\DeclareMathOperator{\univ}{univ}
\DeclareMathOperator{\triv}{triv}

\DeclareMathOperator{\ev}{ev}
\DeclareMathOperator{\loc}{-loc}
\DeclareMathOperator{\oloc}{loc}

\DeclareMathOperator{\Ran}{\mathbf{Ran}}
\DeclareMathOperator{\fact}{\mathbf{fact}}
\DeclareMathOperator{\Op}{Op}

\DeclareMathOperator{\fil}{fil}

\usepackage{xcolor}
\usepackage{soul}

\usepackage[toc]{appendix}

\makeindex
\title{Topological sheaves and spaces of distributions in the global case}
\author{Luca Casarin, Andrea Maffei}

\begin{document}

\maketitle

\noindent Luca Casarin: Dipartimento di Matematica, Sapienza Universitá di Rome, Italy \& INFN, P.le A. Moro 5, 00185 Rome, Italy. E-mail: luca.casarin@uniroma1.it \newline \noindent Andrea Maffei: Department of Mathematics, University of Pisa, Largo Bruno Pontecorvo, 5, 56127 Pisa, Italy. E-mail: andrea.maffei@unipi.it

\begin{center}
    \textbf{Abstract}
    
    \smallskip
    \justifying
    \noindent We extend the theory of fields/distributions developed in \cite{cas2023} to a general base scheme. In order to do so we introduce suitable notions of topological sheaves on schemes and study their basic properties. We then construct appropriate analogues of the spaces of fields, consider multiplication between them and rebuild the basic theory of vertex algebras in the setting of global distributions in place of formal power series, which takes the form of chiral algebras introduced by \cite{BDchirali}.
\end{center}

\tableofcontents

\section{Introduction}

The present paper is a generalization of \cite{cas2023}, whose goal was proving a version of the Feigin Frenkel theorem on the center of the enveloping algebra at the critical level in the multiple singularities setting. This problem has its roots in the work of Gaitsgory, Frenkel and Gaitsgory, Gaitsgory and Raskin on the geometric Langlands correspondence. In particular, the case of the affine algebra with multiple singularities was considered in Fortuna's thesis \cite{fortuna2013beilinson} and in Raskin's lectures \cite{Raskin}. In the case of $\mathfrak{sl}_2$ and in the $2$ singularities setting, the same problem was considered also in \cite{fortuna2022local}.

The point of view adopted in \cite{cas2023}, which makes it possible to use tools from vertex algebra theory also in this setting, was to reinterpret fields as distributions. Indeed, the formalism of formal power series would be extremely cumbersome in the multiple singularities case, while the formalism of distributions allows one to treat the objects in a uniform way, in direct analogy with the classical case of one singularity. Even if the point of view of distributions is well known, the theory of vertex algebras is usually developed with the formalism of formal power series, so that in \cite{cas2023} it was necessary to give several details on how to develop the theory in the former language. A very interesting consequence of the point of view of distributions is that it allows one to treat factorization properties in a very natural way, making it possible to use vertex algebra tools in order to prove factorizable statements. We will recall and refer to \cite{cas2023} from time to time, during this introduction.

The goal of this paper is to develop the language of fields/distributions from a global point of view (with respect to the base scheme) and in a coordinate free manner. Our real motivation to do so is that of proving a global, factorizable version of the Feigin-Frenkel isomorphism, which will heavily use the tools developed here and will be the topic of \cite{casmaffei2}. Our choice of working globally and coordinate freely forced us to revisit some foundational aspects and to reformulate some of the constructions of \cite{cas2023} in a more intrinsic fashion. Before giving more details, let us point out what we think are the main results of this work:

\begin{itemize}
	\item The generality of our geometric setting is that of an arbitrary smooth family of curves $X \to S$ together with a finite set of sections $\Sigma$. We will consider several sheaves which are supported on $\Sigma$, these are always defined over its formal neighborhood $\oSigma$, but for notational convenience we will consider them as sheaves on $S$ via pushforward. We assume $S$ to be topologically noetherian and in some cases integral for technical reasons;
	\item The language of distribution as developed in \cite{cas2023} is very well adapted to be generalized to a global context, which is the point of view we take in the present paper. In particular, given a sheaf of associative topological algebras whose topology is generated by left ideals the space $\mF^n_{\Sigma,\calU}$ of $n$ fields/distributions with values in $\calU$ is defined as a topological sheaf on the whole $S$. On the space of fields we define a \emph{chiral product} whose construction is global and does not depend on any local argument;
	\item In this context it is possible to give a good notion of mutually local fields and of an algebra generated by mutually local fields. The structure that arises from these notions is naturally a chiral algebra in the sense of Beilinson and Drinfeld \cite{BDchirali}, or rather a version of it which takes into account the topological nature of our objects;
	\item In particular, thanks to the description given in the previous point, given any finite dimensional Lie algebra $\gog$ with a non degenerate symmetric invariant form $\kappa$ it is possible to construct $\VkSg$, a chiral algebra over $\oSigma$ (which we think of as a formal curve over $S$) without mentioning any local construction. This chiral algebra is an analogue in our geometric setting of the chiral algebra constructed in \cite{frenkel2004vertex}, whose chiral product is constructed by making local considerations, while our product is canonically globally constructed;
	\item Across the entire paper we see how our global constructions reduce to more familiar ones in the local case. In particular, in the last chapter we study how our local descriptions behave under coordinate changes; this in particular allows us to confront our chiral algebra with the one constructed in \cite{frenkel2004vertex};
	\item The objects we introduce have all very good factorization properties. This aspect is developed in a systematic way along the whole paper at the cost of a bit of redundance. In \cite{casmaffei2} these factorization properties will be crucial to prove the factorizable version of the Feigin-Frenkel theorem on the center of the completed enveloping algebra.
\end{itemize}

Finally, let us address the elephant in the room: why is this paper so long? 

\noindent The main reason is that we constantly have to deal with the topological structures of our objects. There is no way to avoid this: without taking suitable completions our objects would lack of some of their crucial algebraic properties. These topological structures also allow for some more synthetic and clear arguments, but their setup requires some caution since there are many subtleties involved. These subtleties are not too hard too handle but we decided to treat them with particular attention. These aspects also appeared in some form in \cite{cas2023}, but for its specific modules the topological situation was way more clear. Thus, deciding to work in a more general setting forced us to develop some more technical tools to treat our topologies.

\subsection{Spaces of fields}

In the classical theory of vertex algebra some of the key players are the space of formal series $\End(V)[[z^{\pm 1}]]$ for $V$ a vector space, and the subspace of fields, which are formal series as above which satisfy some continuity property. It is a well known remark that the space of fields identifies with $\Homcont(\mC((z)),\End(V))$, where the topology of $\mC((z))$ is that given by the subspace $z^N\mC[[z]]$, $V$ has the discrete topology and $\End(V)$ as the point wise convergence topology. One of the main results of \cite{cas2023} is that the theory of fields can be extended by replacing $\mC((z))$ and $\End(V)$ with more general algebras, so that there is a notion of mutually local fields, $n$ product of fields and so on also in this more general setting.

In particular, in \emph{loc. cit.}, the topological ring $\mC((z))$ is replaced with $K_n = A_n[[t]][\prod (t-a_i)^{-1}]$, where $A_n = \mC[[a_1,\dots,a_n]]$. One can construct a version of the affine algebra $\hat{\gog}_{n,\kappa}$ and study fields in $\Homcont_{A_n}(K_n,U_\kappa(\hat{\gog}_n))$, where $U_\kappa(\hat{\gog}_n)$ the completed enveloping algebra of $\hat{\gog}_{n,\kappa}$. The theory is developed by defining mutually local fields and $n$-products between fields similarly to the case of vertex algebras, strongly using the coordinate $t \in K_n$. This language was particularly efficient to study the center of $U_\kappa(\hat{\gog}_n)$: in very few steps it allows, for instance, to retrieve the explicit formulas for the analogue of the Sugawara operators used to describe the center in \cite{fortuna2022local}. This formalism was also fruitfully applied in \cite{fortuna2023semi} in order to prove a weak form of an analogue of the Theorem of Frenkel and Gaitsgory \cite{FG6}, \cite{FG7} on spherical representations of the affine algebra in the setting of two singularities.

One of the main purposes of this paper is to develop a similar theory in a more general geometric setting and to do so in a coordinate free manner. 

We consider a smooth family of curves $X \to S$ over a topologically noetherian scheme $S$, equipped with a finite collection of sections $\Sigma = \{ \sigma_i : S \to X\}_{i \in I}$. In \cite{casmaffei2} we will use the tools developed here to study a global, factorizable version of the Feigin-Frenkel center, so that the case we are ultimately interested in is the one where $S = C^I$, for a fixed smooth complex curve $C$, $X = C \times C^I$ and $\Sigma = \{ \sigma_j\}_{j \in I}$ is composed by the canonical sections $\sigma_j((x_i)_{i \in I}) = (x_j,(x_i)_{i\in I})$.

Attached to the data of $S,X,\Sigma$ we consider a complete topological sheaf $\pbarO$ on $S$, which is the completion of $\calO_X$ along the subscheme determined by $\Sigma$. We also consider $\pbarOpoli$, which is constructed from $\pbarO$ by allowing singularities at the sections $\sigma_i$. The case of \cite{cas2023} is recovered by setting $S = \Spec A_n$, $X = \Spec A_n[t]$ and $\sigma^\sharp_i(t) = a_i$. We show that locally on $S$ there exists a coordinate $t \in \pbarO$, which allows us to describe the sections of $\pbarOpoli$ as a slight modification of the ring $K_n$. 

We are able to manipulate the topological sheaves $\pbarO$ and $\pbarOpoli$ as they were sheaves of functions on affine spaces $\oSigma, \oSigma^*$ over $S$. This interpretation makes mathematical sense only in the case of $\pbarO$, which may be interpreted as the sheaf of functions on an ind-affine scheme over $S$ (see Section \ref{sssec:spaces}) but it fails for $\pbarOpoli$, so we will use this interpretation only to justify our terminology. We are able to consider multiple tensor products $\pbarOpolim{n}$ of the sheaves $\pbarOpoli$. There is a natural multiplication map $\pbarOpoliq \to \pbarOpoli$, we call the kernel of the above map the diagonal ideal and denote it by $\calJ_\Delta$. We are also able to speak about the sheaf of differential operators $\pbarD$ in this setting as well. In the case we have a local coordinate $t \in \pbarO$ we see that $\calJ_\Delta$ is generated by $t\otimes 1 - 1\otimes t$, while $\pbarD = \pbarO[\partial_t]$. These facts hold only because we are considering suitable completions of our objects.

\smallskip

We fix a complete topological algebra with topology generated by left ideals $\calU$ and study fields in $\mF^1_{\Sigma,\calU} = \Homcont(\pbarOpoli,\calU)$ as well as multivariable fields in $\mF^n_{\Sigma,\calU} = \Homcont(\pbarOpolim{n},\calU)$. Given two fields $X,Y\in \mF^1_{\Sigma,\calU}$ their bracket is defined as an element $[X,Y] \in \mF^2_{\Sigma,\calU}$. It is already evident at this point of the discussion that our choice to work coordinate independently forces us to define the space of local $2$-fields as those $Z \in \mF^2_{\Sigma,\calU}$ which are killed by an high enough power of $\calJ_{\Delta}$, and to define two fields $X$ and $Y$ to be \emph{mutually local} if their bracket is a local $2$-field. 

It then follows by a version of Kashiwara's theorem that the space of local $2$-fields is isomorphic to the $\calD$-module pushforward $\Delta_!\mF^1_{\Sigma,\calU}$. This expression should be thought as the analogue in our geometric setting of the fact the space of local formal  power series $\End(V)[[z^{\pm 1},w^{\pm 1}]]^{\mathrm{loc}}$ is generated by the derivatives of $\delta(z-w)$. 

We are also able to construct analogues of negative $n$-products from vertex algebra theory in our geometric setting. Again, since we are working coordinate independently, we cannot directly consider $(t \otimes 1-1\otimes t)^{-n}$ as in \cite{cas2023} but we must rather consider a suitable localization at the ideal $\calJ_\Delta$. We see that the morphisms \begin{align*}
	m_r &: \mF^1_{\Sigma,\calU} \otimes \mF^1_{\Sigma,\calU} \to \mF^2_{\Sigma,\calU} \qquad X\otimes Y \mapsto (f \otimes g \mapsto X(f)Y(g)), \\
	m_\ell &: \mF^1_{\Sigma,\calU} \otimes \mF^1_{\Sigma,\calU} \to \mF^2_{\Sigma,\calU} \qquad X\otimes Y \mapsto (f \otimes g \mapsto Y(g)X(f))
\end{align*}
enhance to morphisms $\mF^1_{\Sigma,\calU} \otimes \mF^1_{\Sigma,\calU} (\infty\Delta) \to \mF^2_{\Sigma,\calU}$, where $\mF^1_{\Sigma,\calU} \otimes \mF^1_{\Sigma,\calU} (\infty\Delta)$ is the above mentioned localization. This enhancement is possible thanks to the fact that the topology on $\calU$ is generated by left ideals. We are then able to define the \emph{chiral bracket} on the space of fields as the difference
\[
	\mu = m_r - m_\ell : \mF^1_{\Sigma,\calU} \otimes \mF^1_{\Sigma,\calU} (\infty\Delta) \to \mF^2_{\Sigma,\calU}.
\]

\smallskip

In classical vertex algebra theory it is a well known fact that a subspace of mutually local fields which is closed under $n$ products and derivatives is naturally a (non-unital) vertex algebra. In our setting we say that a $\calD$-submodule $\calV \subset \mF^1_{\Sigma,\calU}$ consists of mutually local fields if the map $\mu$ restricted to $\calV \otimes \calV$ factors as a morphism $\calV \otimes \calV \to \Delta_!\mF^1_{\Sigma,\calU}$ to the space of local $2$-fields. The analogue of being closed under $n$-products becomes the condition for the morphism $\mu$ to factor as 
\[
	\mu : \calV \otimes \calV (\infty\Delta) \to \Delta_!\calV \subset \Delta_!\mF^1_{\Sigma,\calU}.
\]
One of the main results of this paper (Theorem \ref{thm:chiralalgebra}) is that a $\calD$ submodule of $\mF^1_{\Sigma,\calU}$ which satisfies the above condition is naturally a \emph{chiral algebra} in the sense of \cite{BDchirali}. Let us say that this, for us, provides a conceptual bridge between the theory of vertex algebras and the theory of chiral algebras: the discussion above shows how the formalism of chiral algebras is indeed perfect to describe spaces of fields in a coordinate free manner.

\subsection{Topological Sheaves}

It is clear by the above discussion that to develop a suitable theory of fields in our geometric setting we need a working theory of topological sheaves on schemes. We will develop this in Section \ref{app:top}, stating the main results and construction we will need.

Our definition of a topological sheaf is in some sense global: a topological sheaf is a sheaf $\calF$ with a filtered system of subsheaves $\calU \subset \calF$ which play the role of the neighborhoods of zero in $\calF$. We define complete topological sheaves as those topological sheaves for which the canonical map $\calF \to \varprojlim \calF/\calU$ is an isomorphism. We study limit and colimits of topological sheaves and of complete topological sheaves. Discuss $\Hom$ spaces and consider several tensor products. 

Let us say that the discussion on completed tensor product is the one that requires the most care. The lack of associativity or commutativity for some these, for instance, makes it not trivial to introduce a structure of algebra on certain completed tensor products of two topological algebras. Indeed, this is possible only under certain strict assumptions, which are satisfied in the cases we are interested in. Let us say that, even if we started with objects with a rather simple topology, such as the ring of Laurent series $\mC((t))$, our work naturally leads to treat more complicated ones and to single out the correct hypothesis that allowed us to treat everything in a uniform way is not always easy.

That said, the actual proofs of these results are not much complicated once the correct hypothesis are established. For this reason, in this paper we decided to only state the results we needed without many proofs, which will appear in a subsequent, more detailed, work. An exception of this is Lemma \ref{lem:BCMN} for which we actually gave a proof. This choice is due to the fact that its hypothesis are not so natural and are very much modeled on our objects.

\subsection{Pseudo-tensor structures}

In order to prove Theorem \ref{thm:chiralalgebra} we need to prove that $\mu$ satisfies some complicated commutation relations. In order to check those we use the formalism of pseudo-tensor structures as in \cite{BDchirali} and check that $\mu$ defines a Lie structure in the Beilinson Drinfeld chiral pseudo-tensor structure $P^{chBD}$. Let us recall that, in \cite{BDchirali}, $I$ operations for this pseudo-tensor structure are defined as
\[
	P^{chBD}_I(\{ M_i \},N) = \Hom_{\calD_C} \left((\boxtimes_{i \in I} M_i)(\infty\nabla(I)),\Delta(I)_!N\right),
\]
where $C$ is a smooth curve, $\{M_i\}_{i \in I},N$ are right $\calD_C$-modules, $\Delta(I) : C \to C^I$ is the diagonal and $\nabla(I) \subset C^I$ is the divisor determined by any two coordinates being equal. 

Although some topological issues force us to slightly change the definition (see Section \ref{ssec:filteredbeilinsondrinfeld}), we are able to easily translate this in our $\oSigma$ setting, but to prove that $\mu$ induces a Lie structure we need to make a detour and consider other pseudo-tensor structures. Let us say that after having carefully set up the pseudo-tensor structures and our operads the result on $\mu$ essentially follows from the associativity of $\calU$. This is all dealt with in Section \ref{sec:operads}; in what follows we explain a bit how all these pseud-tensor structures are related.

Since the $\mu$ originates from a morphism $\mu : \mF^1 \otimes \mF^1 (\infty \Delta) \to \mF^2$ it is reasonable to define our pseudo-tensor structures for graded modules. Thus, we will consider the category of graded (by finite sets) topological sheaves $M^\bullet$ such that for any finite set $A$ the module $M^A$ is also a $\pbarOpolim{A}$ (or $\pbarDpolim{A}$) module. The basic example of such a graded module is $\mF^\bullet$. 

We will consider two pseudo-tensor structures: $P^{\Delta}$ and $P^{\otimesr}$. Both are constructed considering certain tensor products $\exttensor{\Delta}_{i \in I} M^\bullet_i$ and $\exttensor{\ra} M^\bullet_i$ for graded modules $M^\bullet_i$ and then defined as
\begin{align*}
	P^{\Delta}_I ( \{M_i\},N) &= \Homcont \left( \exttensor{\Delta} M^\bullet_i, N^\bullet \right), \\
	P^{\otimesr}_I ( \{M_i\},N) &= \Homcont \left( \exttensor{\ra} M^\bullet_i, N^\bullet \right). \\
\end{align*}
Let us say that when defining these pseudo-tensor structures some care is needed. Indeed $\exttensor{\Delta}$ is obtained by considering a graded version of the $\otimes^{*}$ tensor product, which is not strictly associative, and then inverting suitable diagonals. On the other hand $\exttensor{\ra}$ is obtained by considering the $\otimesr$ tensor product, which is not commutative. In particular $\exttensor{\ra} M^\bullet_i$ must be defined as the sum of all graded $\otimesr$ tensor products over all possible orderings of $I$.

The $P^{\otimesr}$ pseudo-tensor structure is considered in \cite{beilinsonTopologicalAlgebras} as well. There, Lie algebras in $P^{\otimesr}$ are called chiral algebras and it is shown that these are exactly the same as associative algebras for the $\otimesr$ tensor product; this result applies in our graded setting as well. We are interested in this pseudo-tensor structure because $\mu$ can be constructed as follows. We show that since $\calU$ is an $\otimesr$ associative algebra, the morphisms
\begin{align*}
	m_r &: \mF^1_{\Sigma,\calU} \otimes \mF^1_{\Sigma,\calU} \to \mF^2_{\Sigma,\calU}, \\
	m_\ell &: \mF^1_{\Sigma,\calU} \otimes \mF^1_{\Sigma,\calU} \to \mF^2_{\Sigma,\calU}.
\end{align*}
enhance to morphisms
\begin{align*}
	m_r &: \mF^1_{\Sigma,\calU} \otimesr \mF^1_{\Sigma,\calU} \to \mF^2_{\Sigma,\calU}, \\
	m_\ell &: \mF^1_{\Sigma,\calU} \otimesl \mF^1_{\Sigma,\calU} \to \mF^2_{\Sigma,\calU}.
\end{align*}
The map $m_r$ is part of an associative graded algebra structure $\mF^\bullet \otimesr \mF^\bullet \to \mF^\bullet$ so that by Beilinson's result $\mu = m_r - m_\ell$ is a Lie structure in $P^{\otimesr}$. 

The morphism $\mu : \mF^1\otimes \mF^1(\infty\Delta) \to \mF^2$ is obtained by restricting $m_r$ and $m_\ell$ to $\mF^1\otimes \mF^1(\infty\Delta) \to \mF^1\otimesr \mF^1$ and $\mF^1\otimes \mF^1(\infty\Delta) \to \mF^1\otimesl \mF^1$ respectively. We show that this restriction is part of a natural morphism of pseudo-tensor structures $P^{\otimesr} \to P^{\Delta}$ so that $\mu : \mF^1\otimes \mF^1 (\infty\Delta) \to \mF^2$ is actually part of a Lie structure.

Finally we link our $P^{\Delta}$ with $P^{chBD}$ using the following construction: given a right $\calD$-module $M$ we consider the graded module $M_!^A = \Delta(A)_!M$, where $\Delta(A)_!$ is the $\calD$-module pushforward along the diagonal embedding $\oSigma \to \oSigma^A$. So we will consider $P^{\Delta}_I(\{M^\bullet_{!,i}\},N^\bullet_!)$. It will follow by construction that the latter is equipped with a natural map $P^{\Delta}_I(\{M^\bullet_{!,i}\},N^\bullet_!) \to P_I^{chBD}(\{M_i\},N)$. Thus, our pseudo-tensor structure $P^\Delta$ provides a bridge between the Beilinson-Drinfeld chiral pseudo-tensor structure $P^{chBD}$ in \cite{BDchirali} and the pseudo-tensor structure $P^{\otimesr}$ of \cite{beilinsonTopologicalAlgebras}.

\subsection{The chiral algebra \texorpdfstring{$\VkSg$}{V} and the map \texorpdfstring{$\calY_{\Sigma,t}$}{Y}}

We apply the general results on fields obtained so far to construct and study a chiral algebra on $\oSigma$ constructed starting from a finite dimensional Lie algebra $\gog$ over $\mC$.

We start by constructing $\hat{\gog}_{\Sigma,\kappa}$ an analogue of the affine algebra in our geometric setting. For this we need a residue morphism $\Res_{\Sigma} : \pbarOmega \to \calO_S$ which we define as the sum of the residues at each section $\sigma_i$, in Section \ref{sec:residuo}. We can then define $\hat{\gog}_{\Sigma,\kappa}$ by the usual commutation formulas:
\[
	\hat{\gog}_{\Sigma,\kappa} = \gog \otimes_{\mC} \pbarOpoli \oplus \calO_S \mathbf{1},
\]
where $\mathbf{1}$ is central, while for $X,Y \in \gog$ we have $[X\otimes f, Y \otimes g] = [X,Y]\otimes fg + \kappa(X,Y)\Res_{\Sigma}(gdf)\mathbf{1}$. From this sheaf of complete Lie algebras we construct a completed enveloping algebra $\calU_{\kappa}(\hat{\gog}_\Sigma)$ as in the usual setting and then consider the space of fields $\mF^1_{\Sigma,\gog} = \Homcont\left(\pbarOpoli,\calU_{\kappa}(\hat{\gog}_\Sigma) \right)$.

For any given $X \in \gog$ there is a naturally attached field $X : \pbarOpoli \to \ugsig{\kappa}$ given by $f \mapsto X \otimes f$. As in the classical theory of vertex algebra we may iterated the chiral product $\mu$, together with taking derivations, to generate a chiral algebra starting from the fields attached to $\gog$. We call this chiral algebra $\VkSg$. Let us notice that our construction has the benefit of being global: both $\VkSg$ and its chiral product are constructed directly on the whole $S$, so there is no gluing procedure going on. 

Nevertheless, we are able to describe $\VkSg$ locally in a very efficient way. After the choice of a local coordinate $t$ we see that chiral algebra structure induces a vertex algebra structure (which depends on the choice of $t$) and the assignment $\gog \to \VkSg$ induces a morphism of vertex algebras $V^\kappa(\gog) \to \VkSg$. We show that its $\pbarO$-linear extension
\[
	\calY_{\Sigma,t} : V^\kappa(\gog) \otimes \pbarO \to \VkSg
\]
is an isomorphism. This parallels \cite[Proposition 4.0.6]{cas2023} and should be thought as a version of the usual state/field correspondence.

A remarkable property of the isomorphism $\calY_{\Sigma,t}$ is that it \emph{factorizes}, just as in \cite[Corollary 7.2.3]{cas2023}. We expand on these ideas in \ref{sec:factorizationstructures} and in Theorem \ref{thm:factorizationofcaly}. This property is crucial to reduce the study of $\VkSg$ to the one section case where $\Sigma = \{\sigma \}$ and it will play a fundamental role in our proof of a factorizable version of the Feigin-Frenkel center \cite{casmaffei2}. Let us emphasize again that this is a crucial point for us: when we say that we are able to use vertex algebra tools in a factorizable way we do that via the factorization properties of $\calY_{\Sigma,t}$.

\subsection{Comparison with some classical constructions}

In the special case where $S = \Spec \mC$, $X = C$ is a smooth complex curve, and $\Sigma = \{ \sigma \}$ is a single section (that corresponds to a point $x \in C$) our constructions directly compare to the classical vertex algebra ones. 

Let $\calO_x$ be the completion of the local ring $\calO_{C,x}$ and let $\calK_x$ be its field of fractions. Our construction of the affine algebra boils down to
\[
	\hat{\gog}_{\Sigma,\kappa} = \gog \otimes \calK_x \oplus \mC\mathbf{1}.
\]
After the choice of a local coordinate $t$ around $x$ we obtain an isomorphism $\rho_t : \mC[[z]] \to \calO_x$ which maps $z \mapsto t$, and whose extension $\rho_t : \mC((z)) \to \calK_x$ induces an isomorphism of Lie algebras between $\hat{\gog}_{\Sigma,\kappa}$ and the usual affine algebra $\hat{\gog}_{\kappa}$. This induces an identification at the level of completed enveloping algebras $U_{\kappa}(\hat{\gog}_\Sigma) \simeq U_{\kappa}(\hat{\gog})$ as well. Under all these identifications induced by $\rho_t$ our map $\calY_{\Sigma,t}$ reads as a morphism
\[
	\calY_z : V^\kappa(\gog) \otimes \mC[[z]] \to \Homcont\left(\mC((z)),U_\kappa(\hat{\gog}) \right)
\]
which does not depend on the choice of $t$. Under the identification of fields/distributions with formal power series, we have that for any $v \in V^\kappa(\gog)$ the field $\calY_z(v\otimes 1)$ identifies with the formal series $\tilde{Y}(v,z)$ with coefficients in $U_\kappa(\hat{\gog})$ considered in \cite[Section 4.2.5]{frenkel2004vertex}; in particular under the projection 
\[
	\Homcont\left(\mC((z)),U_\kappa(\hat{\gog}) \right) \to \Homcont\left(\mC((z)),\End(V^\kappa(\gog)) \right)
\]
induced by the action of $U_{\kappa}(\hat{\gog}) \to \End(V^\kappa(\gog))$ the field $\calY_z(v\otimes 1)$ is mapped to the usual vertex operator $Y(v,z)$ under the identification of formal power series with fields/distributions.

\medskip

The fact that our constructions are coordinate independent leads to some more considerations and implies that the subspace determined by $\calY_z$ is canonical. Studying how $\calY$ behaves under coordinate changes $s = \tau(t)$, for $\tau \in \Aut O$ we see in Section \ref{ssec:autoactions} that $\calY$ induces and action $\tau \mapsto \calY_\tau$ of $\Aut O$ on $V^\kappa(\gog) \otimes \mC[[z]]$, and that this action is the restriction of the natural conjugation action of $\Aut O (\mC) = \Aut^{\cont}(\mC[[z]])$ on $\Homcont\left(\mC((z)),\End(V^\kappa(\gog)) \right)$ along the immersion $\calY_z$.  

\medskip

In this framework we are also able to recover a formula on how the usual vertex operators change under coordinate changes obtaining a result to the one of Y.-Z Huang \cite{Huang} (see also \cite[Lemma 6.5.6]{frenkel2004vertex}). That means studying how, on the vertex algebra $V^\kappa(\gog)$, the vertex operators $Y(v,z) \in \End(V^\kappa(\gog))[[z^{\pm 1}]]$ attached to $v \in V^\kappa(\gog)$ changes under an automorphism $\tau \in \Aut^{\cont} (\mC[[z]])$. 

Let us notice that it may be tricky to define $Y(v,\tau(z))$ as a formal power series, but if we consider $Y(v,z)$ as a field/distribution $Y(v) \in \Homcont\left(\mC((z)),\End(V^{\kappa}(\gog))\right)$, then we may simply define $Y(v,\tau(z)) = Y(v) \circ \tau^{-1}$ (the minus sign is chosen in order to get an actual action).
We will directly study our enhancement of $Y(v) = \calY_z(v \otimes 1)$ as an $U_\kappa(\hat{\gog})$ valued field. 

To describe $Y(v)\circ \tau^{-1}$ we consider $V^\kappa(\gog)$ as acted on by the group scheme $\Aut O(R) = \Aut^{\cont}_R(R[[w]])$, so that for any $R$ and any $r(w) \in \Aut^{\cont}_R(R[[w]])$ (here we identify an element $r \in \Aut O (R)$ with its value on $w$) there is a natural action of $r(w)$ on $V^\kappa(\gog) \otimes R$. In the particular case $R = \mC[[z]]$ we can consider the automorphism $r_\tau \in \Aut^{\cont}_{\mC[[z]]}(\mC[[z]][[w]])$ which satisfies $r_\tau(w)= \sum_{k\geq 1} \frac{1}{k!}(\partial_z^k \tau(z))w^k$ and for any $v \in V^\kappa(\gog)$ consider the element
\[
	r_\tau(w)\cdot v \in V^\kappa(\gog) \otimes \mC[[z]].
\]

A consequence of Section \ref{ssec:autoactions} is that
\[
	Y(v)\circ \tau^{-1} = \tau^{-1} \circ Y\left((r_\tau(w)\cdot v) \cdot \partial_z\tau(z)\right).
\]
where $\tau^{-1} \circ Y(\bullet)$ stands for the composition of $Y(\bullet) : \mC((z)) \to U_\kappa(\hat{\gog})$ with the isomorphism induced by $\tau^{-1}$ from $U_\kappa(\hat{\gog})$ to itself.

\subsection{Description of the work}

In the rest of the introduction we give some more details on what happens in each section and compare our constructions to the ones of \cite{cas2023}.

\medskip

Section \ref{app:top} is preliminary material, where we introduce our notion of topological sheaves and develop the basic terminology and tools which we will use in the following Sections. We decided to not include many proofs of our assertions, which will appear in a subsequent publication. Let us say that most of these proofs are not difficult but often require some caution.  We give the definition of quasi-coherent completed topological sheaf (QCC sheaf for short, c.f. Definition \ref{def:qccsheaves}) and study some properties of particular classes of QCC sheaves. We investigate various completed tensor products between such objects and compare them with the completed tensor products of \cite{cas2023} and \cite{beilinsonTopologicalAlgebras}. We introduce $\Hom$-spaces between topological sheaves, which are themselves topological sheaves when appropriate conditions are satisfied.

In Section \ref{sec:localstructure} we introduce the basic geometric setting in which we work. We consider a smooth family of curves $p : X \to S$ over a fixed quasi-compact and quasi-separated scheme $S$. In what follows we will consider sheaves supported on closed subschemes of $X$ which are finite over $S$; to easily manipulate them we consider their pushforward to $S$, so that it practice all our sheaves will actually be over $S$. Let us also mention that the factorization picture takes a slightly different form in this context, for which we refer to Notation \ref{ntz:factorization}.

\smallskip

Given $\Sigma$ as above we consider $\overline{\Sigma}$, the formal completion of $X$ at the union of the images of the embeddings $\sigma_i$. To this formal scheme we attach QCC  sheaves of $\calO_S$-commutative algebras: $\pbarO$, the sheaf of regular functions on $\overline{\Sigma}$, and $\pbarOpoli$, the sheaf of regular functions on $\overline{\Sigma}\setminus\Sigma$. These will be some of the key players of these pages. We study their local structure and show that locally on $S$ there exist functions $a_i \in \calO_S$, a surjection $\pi : I \twoheadrightarrow J$ and functions $t_j \in \pbarO$ such that they may be described as products of sheaves of the form:
\begin{align*}
	\pbarO &\simeq \prod_{j \in J} \varprojlim_n \frac{\calO_S[t_j]}{\prod_{i \in I_j} (t_j-a_i)^n}, \\ 
	\pbarOpoli &\simeq \prod_{j \in J} \varprojlim_n\frac{ \calO_S[t_j]}{\prod_{i \in J_i} (t_j-a_i)^n}\left[\prod_{i \in I_j} (t_j - a_i)^{-1}\right].
\end{align*}

A function $t = (t_j) \in \pbarO$ which induces this kind of description will be called a \emph{coordinate}. Open affine subschemes of $S$ which admit a coordinate will be called \emph{well covered}. These sheaves will play the role that the complete topological rings $R_n,K_n$ played in \cite{cas2023}, which are recovered in the case $S = \Spec \mC[[a_1,\dots,a_n]]$.

\medskip

In Section \ref{sec:residuo} we introduce the well known \emph{residue morphism}
\[ 
\Res_\Sigma : p_*\left( \frac{\Omega^1_{X/S}(\infty\Sigma)}{\Omega^1_{X/S}} \right) \to \calO_S,\]
we could not find a reference for the residue map adequate for our purposes so we propose a construction ourselves.
When $X = \mA^1_S$ and $\Sigma$ consists of a single section our $\Res$ coincides with the usual residue map. We show that the collection of morphisms $\Res_{\Sigma}$, as $X$ and $\Sigma$ vary, is unique under the following assumptions: it is compatible with the usual residue map in the $\mA^1$ case, it is invariant under pullback, behaves well under étale maps and satisfies some factorization properties.

\medskip

Section \ref{ssec:convenzionispazi} introduces the notion of $\calD$-modules over $\overline{\Sigma}$ and is basically devoted to introduce some notations, adapting some basics of the theory of $\calD$-modules to our topological setting. We prove an analogue of Kashiwara's Lemma, a version of the Cauchy formula and we introduce the de Rham residue map. In this Section we also describe a condition under which the ordered tensor product of two topological algebras acquires a structure of algebra. All these ingredients will be needed in the sequel.

\medskip

After everything is set up, in Section \ref{sec:distributionandfields} we finally start our study of fields and fix a new object, $\calU$, a complete associative QCC algebra with topology generated by left ideals (i.e. an $\otimesr$-algebra). One of the main players of this paper gets on stage: the space of \emph{fields} $\mF^1_{\Sigma,\calU} = \calHomcont_S(\pbarOpoli,\calU)$ and its multivariable versions $\mF^A_{\Sigma,\calU} = \calHomcont_S(\pbarOpolim{A},\calU)$, for $A$ a finite set. These are complete topological sheaves on $S$. Our $\mF^1_{\Sigma,\calU}$ plays the same role that the space of fields $\{\text{Fields on } V\} \subset \End_{\mC}(V)[[z^{\pm}]]$ plays in the usual theory of vertex algebras but in a coordinate free fashion. For instance, one can easily make sense of \emph{mutually local fields} using this language, as in \cite{cas2023}. We study the sections of $\mF_{\Sigma,\calU}$ on affine open subsets of $S$ and show that in the case where we consider sections over an affine well covered open subset we essentially get the space of fields considered in \cite{cas2023}. We introduce the multiplication maps 
\[
	m^r : \mF^1_{\Sigma,\calU}\otimesr\mF^1_{\Sigma,\calU} \to \mF^2_{\Sigma,\calU}, \qquad m^\ell : \mF^1_{\Sigma,\calU}\otimesl\mF^1_{\Sigma,\calU} \to \mF^2_{\Sigma,\calU}
\] 
and show associativity of their graded versions. We consider the restrictions $m^r,m^\ell : \mF^1_{\Sigma,\calU}\exttensor{*}\mF^1_{\Sigma,\calU}(\infty\Delta) \to \mF^2_{\Sigma,\calU}$ as well and we introduce what we call the \emph{chiral product} \[\mu = m^r - m^l : \mF^1_{\Sigma,\calU}\exttensor{*}\mF^1_{\Sigma,\calU}(\infty\Delta) \to \mF^2_{\Sigma,\calU}.\] We compare our construction to those of \cite{cas2023}.

\medskip 

We move forward showing that the morphism $\mu$ is a Lie bracket in a suitable sense, doing this in a purely operadic way. This leads to some heavy notation since we need to consider graded objects, but after having appropriately defined our operads, $P^\Delta,P^{\otimesr}$, the result easily follows from associativity of $m$. This is dealt with in Sections \ref{sec:operads} and \ref{sec:chiralprod}. We compare our operads with the Beilinson's chiral pseudo-tensor structure of \cite{beilinsonTopologicalAlgebras} and to the chiral pseudo-tensor structure of \cite{BDchirali}, thus establishing a conceptual bridge between these two different notions which bear the same name. Our main result is Theorem \ref{thm:chiralalgebra}, which should be thought as an analogue of \cite[Propositon 3.2]{kac1998vertex} and its consequences. Roughly, it states that a subsheaf of $\mF^1$ consisting of mutually local fields and closed under the chiral product is naturally a \emph{chiral algebra}\footnote{In these sections one finds that the chiral algebras we consider are actually filtered. This is just a technical device to properly formalize some results in our topological setting} (in place of a vertex algebra as in \cite{kac1998vertex}). Let us emphasize again that the work done so far, in the case of a single section and for $S=\mC$, shows that if one tries to develop the theory of vertex algebras in a coordinate free manner from the beginning then one is naturally led to the notion of a chiral algebra.

\smallskip

Finally, in Section \ref{sec:factorizationstructures} we adapt the notion of a factorization algebra to our $X,S,\Sigma$ setting. We give a definition of factorization algebra that is adapted to our purposes and prove some factorization properties of the space of fields with coefficients in a factorization algebra. Let us take the opportunity to establish our notation on finite sets and on diagonals.

\smallskip

The last two sections are devoted to the study of a particular chiral algebra $\VkSg$, which will play a crucial role in \cite{casmaffei2} where we prove the factorizable version of the Feigin-Frenkel theorem. In Section \ref{sez:chiralig} we construct $\VkSg$ as a chiral algebra of mutually local fields, but first we need to specify which associative algebra we are considering as coefficients for our fields. This will be a completed enveloping algebra of a version of the affine algebra. 

Starting from a finite Lie algebra $\gog$ over $\mC$ and a level $\kappa$ (by which we mean an invariant bilinear form on $\gog$) we introduce a complete topological sheaf of $\calO_S$-Lie algebras $\hat{\gog}_{\Sigma,\kappa}$. This is constructed following the construction of the classical affine algebra in our setting:
\begin{align*}
	\hat{\gog}_{\Sigma,\kappa} &= \gog\otimes_{\mC}\pbarOpoli \oplus \mathbf{1}\calO_S \\
	[\mathbf{1},\hat{\gog}_{\Sigma,\kappa}] &= 0 \\
	[X\otimes f, Y \otimes g] &= [X,Y] \otimes fg + \kappa(X,Y)\mathbf{1}\Res_{\Sigma}(gdf).
\end{align*}

in the case where $S = \mC$ and $\Sigma$ consists of a single section one recovers the usual affine Kac-Moody algebra, while for $S = \Spec \mC[[a_1,\dots,a_n]]$ one recovers the Lie algebra $\hat{\gog}_{\kappa,n}$ of \cite{cas2023}. We construct the associated sheaf of completed enveloping algebras $\ugsig{\kappa}$ and show some of its factorization properties. 

\smallskip

In Section \ref{sez:chiralig} we then apply the results on $\mF^1_{\Sigma,\calU}$ (which from now on is denoted $\mF^1_{\Sigma,\gog}$) obtained in the previous Chapter in the case $\calU = \ugsig{\kappa}$. There is a canonical map $\gog \to \mF^1_{\Sigma,\gog}$ whose image consists of mutually local fields; it then follows from the construction of Section  \ref{ssez:generatechiral} that we may iterate the chiral product to generate a canonically defined chiral algebra $\VkSg \subset \mF^1_{\Sigma,\gog}$ starting from $\gog$. We move forward studying it sections: locally on $S$, given a coordinate $t$, we construct a morphism $\calY_{\Sigma,t} : V^{\kappa}(\gog) \otimes \pbarO \to \VkSg$ (here $V^{\kappa}(\gog)$ is the usual vertex algebra attached to $\gog,\kappa$) and show that it is an isomorphism, analogously to Proposition 4.0.6 of \cite{cas2023}. The proof of $\calY_{\Sigma,t}$ being an isomorphism is slightly more involved than the one of \emph{loc. cit.} only because we must be more careful in our topological setting.

\smallskip

Finally, in Section \ref{sec:identificationopers1} we give coordinate independent descriptions of our chiral algebra $\calV^\kappa_{\Sigma}(\gog)$. We show that $\calV^\kappa_{\Sigma}(\gog)$ coincides with the chiral algebra constructed in \cite[6.5]{frenkel2004vertex} (adapting their construction to our geometric setting).

The strategy for proving this result is the following. To present a canonical identification we construct an isomorphism locally, using a coordinate $t$ for $\pbarO$, and then show that our construction is independent from the choice of such a coordinate. In the process of doing this we will study how the isomorphism $\calY_{\Sigma,t}$ behaves under change of coordinates. We treat first the case of a single section exploiting the fact that the set of relevant coordinates in this situation is an $\Autpiu{} O$ torsor. This remark implies that to check that our identifications are independent from the choice of a coordinate we may reduce ourselves to compare a priori different actions of $\Autpiu{} O$ and show that they are actually equal, so that the claim follows by direct computation of these actions. The case of multiple sections is then deduced by the case of a single section together with some factorization properties.

\subsection{Notation for finite sets}\label{ssez:artquestonotazionifiniteset}
	To avoid possible set theoretical issues we fix the following setting. We will denote by $\ffset$ the category of finite non-empty subsets of $\mN_{\geq 1}$ with surjective maps as morphisms. We will refer to elements $I \in \ffset$ simply as finite sets. The main properties of $\ffset$ that we are interested in are that it is itself a set and that any subset of an object in $\ffset$ is still in $\ffset$. Given a surjective map of finite sets $\pi : J \twoheadrightarrow I$ and a subset $I' \subset I$ we will write $J_{I'}$ for $\pi^{-1}(I')$, when $I' = \{i\}$ we will write simply $J_i$ \index{$J_i$}. Given two surjections $q : K \twoheadrightarrow J, p : J \twoheadrightarrow I$ and an element $i \in I$ we write $q_i: K_i \twoheadrightarrow J_i$ for the associated projection on the fibers.   
	
\subsection{Notation for diagonals}\label{sssez:art2notazionidiagonali}
	We introduce some notations to describe diagonals of products of a scheme $X$ over a base scheme $S$ (so products are to be understood relative to $S$), we also assume that the structural morphism $X \to S$ is separated. We denote by $\Delta:X\lra X^2$ the diagonal map. 
	More in general, given a finite set $I$, we define the small diagonal $\Delta(I):X\lra X^I$, and given $\pi :J\surjmap I$, a surjective map of finite sets, we define the diagonals $\Delta(\pi) 
 = \Delta_{J/I}:X^I\lra X^J$ \index{$\Delta(\pi),\Delta_{J/I}$} in the obvious way, so that, for instance, given surjections $K \twoheadrightarrow J \twoheadrightarrow I$ we have $ \Delta_{K/I} = \Delta_{K/J}\Delta_{J/I}$. If $a,b\in J$ we denote by $\pi_{a,b}$ or by $\pi^J_{a=b}$ a surjection from $J$ to a set with one element less than $J$ such that $\pi_{a=b}(a)=\pi_{a=b}(b)$, \index{$\pi_{a=b}$} and by $\Delta_{a=b}$ \index{$\Delta_{a=b}$} the corresponding diagonal. In many constructions the fact that $\pi_{a=b}$ is not fixed in a unique way will not be a problem, we will specify the map only when necessary.

	By a slight abuse of notation we write $\Delta_{J/I}$ also for the closed subscheme of $X^J$ determined by $\Delta_{J/I}$. Given a finite set $I$ we define the big diagonal $\nabla(I)$ of $X^I$ as the scheme theoretic union of all diagonals $\Delta_{a=b}$ and more in general given $\pi:J\surjmap I$, a surjective map of finite sets, we set
	$$
	\nabla(J/I)=\bigcup_{\pi(a)\neq \pi(b)}\Delta_{a=b}\subset X^J
	$$
	\index{$\nabla(J/I)$}so that $\nabla(I)=\nabla(I/I)$. We denote by $j_{J/I} : U_{J/I} = X^J \setminus \nabla(J/I) \hookrightarrow X^J$ the associated open immersion. Finally, given two disjoint subsets $A$ and $B$ of $J$ we define 
	$$
	\nabla_{A=B}=\bigcup_{a\in A, b\in B}\Delta_{a=b}
	$$
	and write $j_{A\neq B} : U_{A \neq B} = X^I \setminus \nabla_{A=B} \hookrightarrow X^I$ for the associated open immersion.


\section{Topological Sheaves and Algebras}\label{app:top}

In this Section we do some remarks on sheaves and modules equipped with a topology. We will start with a quick review of the theory of topological modules and then move on and define our notion of topological sheaves. An important class of topological sheaves is that of \QCC sheaves (see Section \ref{ssec:qccsheaves}) which substitute quasi-coherent sheaves in the topological setting. 

\subsection{Topological Modules}

In this Section we give an overview of the theory of topological modules which we will need in the rest of the paper, but let us say that the formalism that we are going to actually use is that of \emph{topological sheaves} which we are going to develop in the following Section. Let us emphasize that many of the constructions we are going to give for modules will translate without any effort to the setting of sheaves, in addition many of the results that hold for modules apply verbatim to the class of quasi-coherent completed sheaves.

We work over a fixed commutative ring $A$ which we will never consider with a topology (or, equivalently, regarded with the discrete topology), tensor products must be considered as defined over $A$ if not differently specified.

\begin{definition}[Topological Modules]
	A topology on an $A$-module $M$ is a family of submodules $U \subset M$, called the neighborhoods of zero (\noz for short), such that for every $U \subset U'$ if $U$ is a \noz then $U'$ is a \noz and such that for any two given \noz $U_1,U_2$, their intersection $U_1\cap U_2$ is again a \noz. A topology main be defined by a fundamental system of neighborhoods of zero (\fsonoz for short), this is just a filtered (with respect to inclusion) family of submodules $U_i \subset M$; a submodule of $M$ is then defined to be a \noz if it contains one of the $U_i$. Any topological module is an honest topological space: a subset $M' \subset M$ is open if for every $m \in M'$ there exists a \noz $U_m$ such that $m + U_m \subset M'$. For an $A$-linear morphism $f : M \to N$ to be continuous is equivalent to $f$ being continuous at $0$: for any \noz $V \subset N$ there exists a \noz $U \subset N$ such that $f(U) \subset V$. We call $A\mathrm{-mod}_{\Top}$ the category of topological modules with continuous morphisms.
\end{definition}

\begin{definition}[Topologically free]
	
\end{definition}
	
\begin{definition}[Subspaces, quotients]
	A submodule $N \subset M$ is a topological module with the induced topology, neighborhoods of zero are of the form $N\cap U$ for $U$ a \noz in $M$. The quotient $M/N$ with the quotient topology is a topological module, neighborhoods of zero in $M/N$ are given by $(U+N)/N$ for $U$ a \noz in $M$. The natural morphisms $N \to M \to M/N$ are continuous. If $N$ is open then $M/N$ is a discrete topological module.
\end{definition}

\begin{remark}
	The category of topological modules is complete and cocomplete, kernels and cokernels may be constructed using the subspace/quotients constructions above. The topology on the direct sum $\oplus_i M_i$ has \noz equal of the form $\oplus_i U_i$, where $U_i$ is a \noz of $M_i$, the topology on the product is the usual one: a \noz of $\prod_i M_i$ is of the form $\prod_i U_i$ where $U_i$ is a \noz in $M_i$ and $U_i = M_i$ for all but finitely many indices. It follows that colimits and limits are computed as in the category of ordinary modules and their topologies are described by combining the notion of the subspace/quotient topology and the topology on the direct sum/product.
\end{remark}

\begin{definition}[Separated and complete modules]
	A topological module is called \emph{separated} if the intersection of all \noz is $0$; it is called \emph{complete} if the natural map $M \to \varprojlim_{U \text{ \noz}} M/U$ is an isomorphism; a complete module is automatically separated. Asking the same requirements for a fixed \fsonoz gives equivalent definitions. A topological module is separated (or complete) if and only if it is so as a topological space.  Consider $(M)_{\mathrm{Sep}} = M/(\cap_{U \text{\noz}}U)$ and $\widehat{M} = \varprojlim_{U \text{ \noz}} M/U$; then the canonical morphism $M \to (M)_{\mathrm{Sep}}, M \to \widehat{M}$ are initial among morphisms $M \to N$ where $N$ is separated or complete respectively. We will call $\widehat{M}$ the \emph{completion} of $M$.
\end{definition}

\begin{remark}[Behaviour of completion]
	By general nonsense the category of complete topological modules admits arbitrary colimits which are computed by $\widehat{\varinjlim M_i}$. In the special case where we have a sequence of inclusions of complete modules $M_n \subset M_{n+1}$ with the subspace topology, then $M= \varinjlim M_n$ is already complete.
	
	One can show in addition that limits in the category of topological modules preserve completeness and so do direct sums. In addition
	$$
	\bigoplus \widehat{M}_i \simeq \widehat{\bigoplus M_i},\qquad
	\prod \widehat{M}_i \simeq \widehat{\prod M_i}\quad \mand\quad
	\widehat{\varinjlim \widehat{M}_i} \simeq \widehat{\limind M_i}
	$$
	for arbitrary colimits.
    If we assume that the modules are separated then we have also that the completion of the kernel of a morphism is isomorphic to the kernel of the morphism between the completions of the spaces. 
    More in general if $M_i$ are separated we have also $\limpro \widehat{M}_i \simeq \widehat{\limpro M_i}$.
\end{remark}

\begin{remark}[Closed submodules]
	A submodule $N \subset M$ is said to be \emph{closed} if it is a closed subspace. We list some remarks on closed submodules
	\begin{itemize}
		\item An open submodule $U \subset M$ is also closed, indeed $M = \sqcup_{m \in M/U} m + U$;
		\item Given any submodule $N \subset M$, its closure $\overline{N}$ is also an $A$-submodule since addition and multiplication by scalar are continuous;
		\item We have $\overline{N} = \cap_{N \subset V\text{\noz}} V$: the inclusion $\overline{N} \subset \cap_{N \subset V\text{\noz}} V$ follows by the fact that open subsets are closed. The other inclusion can be proven as follows: assume that $x \notin \overline{N}$ and let $V$ be a \noz such that $(x + V) \cap N = \emptyset$, then it is easy to check that $x \notin V + N$, which is a \noz which does not contain $N$;
		\item A closed submodule of a complete module is complete;
		\item If $M$ is a topological module and $N \subset M$ is a closed submodule, then $M/N$ is separated; even if $M$ is complete $M/N$ does not have to be complete, but it is so if there is a countable \fsonoz for $M$. 
	\end{itemize}
\end{remark}

\begin{definition}[Pullback and pushforward]
	Let $\varphi : \Spec B \to \Spec A$ be a morphism of affine schemes. For any topological module $M$ the $B$-module $\varphi^*M = M \otimes_A B$ has a natural topology where a \noz is a submodule which contains $U \otimes_A B$ for some $U$ a \noz of $M$, so that is $U_i$ is a \fsonoz of $M$ then $U_i \otimes B$ is a \fsonoz of $\varphi^*M$. We define the completed pullback $\hat{\varphi}^*M$ as the completion of $\varphi^*M$.
\end{definition}
	
\begin{remark}[Strongly separated]
	In general it is not true that if $M$ is a separated module then $M_f$ is separated for all $f$. We call a module with this property strongly separated. 
	Any sum of strongly separated modules is \emph{strongly separated}, however not any product of strongly separated module need to be strongly separated. 
	Finite limits also preserves this property, while if $M_n$ are strongly separated and $M_n\subset M_{n+1}$ are closed and complete and $M_n$ has the induced topology then $\limind M_n$ is strongly separated. \index{strongly separated}
\end{remark}

\subsubsection{Tensor product of topological modules}\label{ssec:tensorproductsmodules} Following Beilinson \cite{beilinsonTopologicalAlgebras} and \cite{cas2023} we introduce different completions of tensor products. 
Let $M$ and $N$ be two topological $A$-modules. We define seven different topologies on the tensor product $M\otimes N$. 
\begin{itemize}
	\item $ \hat{} \;$-topology: a \fsonoz is given by the products $U\otimes V$ where $U$ is a \noz of $M$ and $V$ is a \noz of $N$.  
	\item $r$-topology: a \fsonoz is given by the products $M\otimes V$ where $V$ is a \noz of $N$.  
	\item $\ell$-topology: is defined symmetrically to the $r$-topology.
	\item $\rightarrow$-topology: a \fsonoz is given by the sums
	$M\otimes V+\sum_{n\in N} U_n\otimes n$
	where, for $n\in N$, the modules $U_n$ are \noz of $M$ and $V$ is a \noz of $N$. 
	\item $\leftarrow$-topology: is defined symmetrically to the $\rightarrow$-topology.
	\item $*$-topology: a \fsonoz is given by the sums
	$U\otimes V+\sum_{n\in N} U_n\otimes n+\sum_{m\in M}m\otimes V_m$
	where $V$ and $V_m$, for $m\in M$ are \noz of $N$, and $U$ and $U_n$, for $n\in N$, are \noz of $M$. 
	\item $!$-topology: a \fsonoz is given by the sums $M\otimes V+U\otimes N$ where $U$ is a \noz of $M$ and $V$ is a \noz of $N$.  
\end{itemize}
We denote the completion of $M\otimes N$ with respect to the topology $\tau$ with $M\stackrel{\tau}{\otimes}_A N$. We also notice that the construction of these tensor products is functorial. Let us notice that once we fixed $N$ the functor $\_ \otimes^\tau N$ is additive, it preserves finite direct sums.\index{$\stackrel{\tau}{\otimes}$}\index{$\stackrel{\tau}{\otimes}$!$\hat{\otimes}$}\index{$\stackrel{\tau}{\otimes}$!$\otimes^r$}\index{$\stackrel{\tau}{\otimes}$!$\otimes^l$}\index{$\stackrel{\tau}{\otimes}$!$\otimesr$}\index{$\stackrel{\tau}{\otimes}$!$\otimesl$}\index{$\stackrel{\tau}{\otimes}$!$\otimesst$}\index{$\stackrel{\tau}{\otimes}$!$\otimessh$}

\begin{remark}[Adjunction for the $\otimesr$ tensor product]
	Let $M,P$ be two topological modules consider the module $\Homcont(M,P)$, we will give a topology on this in Section \ref{ssec:hommodules} but here we consider the topology given by the subspaces $\Homcont(M,W)$ with $W$ a \noz of $P$. We will denote this topological module by $\Homcont(M,P)_{\mathrm{abs}}$; notice that if $P$ is complete then $\Homcont(M,P)_{\mathrm{abs}}$ is complete. For any other complete topological module $N$ and under the assumption that $P$ is complete we have
	\[
		\Homcont(M\otimesr N,P) = \Homcont(N,\Homcont(M,P)_{\mathrm{abs}}).
	\]
	Thus, $M\otimesr\_$, as a functor to the category of complete topological modules to itself, is a left adjoint and hence commutes with all colimits.
\end{remark}

\begin{remark}\label{oss:prodottiregolarifreccetta} 
For the $\ra,!$-tensor products we have:
$$
M\tensor{\ra}\left(\bigoplus N_i\right) = \bigoplus M\tensor{\ra} N_i, \qquad  M\tensor{!}\left(\bigoplus N_i\right) = \bigoplus M\tensor{!} N_i
$$
for arbitrary direct sums. Notice that without assuming that the $N_i$ are all discrete but for a finite number this property does not hold for the $*$-product, indeed in general the topology on the right hand side is finer then the one on the left hand side. This property does not hold for the 
$\leftarrow$-tensor product.
\end{remark}

\begin{remark}\label{oss:tensoriisomorfi}
	There are natural maps among the completions of the tensor products which in some cases are isomorphisms of $A$-topological modules and more often as plain $A$-modules. We list some of these isomorphisms. If $N$ is discrete then $M\hat\otimes N\simeq M\otimes^{r}N\simeq M\otimes N$ with the discrete topology. Moreover $
	M\otimes^{*}N\simeq M\tensor{\rightarrow}N$ and 
$M\tensor{\leftarrow}N\simeq M\otimes^{!}N$ as topological modules. 
	If $N$ is topologically free then the natural map $
	M\tensor{\rightarrow}N \to M\otimes^{r}N
	$ is an isomorphism of plain $A$-modules; if both $M$ and $N$ are topologically free, then $M\otimes^{*}N\simeq M\hat \otimes N$ as plain $A$-modules.
\end{remark}

\subsubsection{Multiple tensor products} \label{sssec:dmtp}The tensor products $\rightarrow$, $\leftarrow$ and $!$ are associative. This is not true for the other tensor products. For this reason we introduce completion for the tensor products of more than one module. We explain it in the case of $\tau=*$ and three modules (the generalizations to the other cases are the most natural ones and in any case they are less important for us). 

Let $M,N,P$ be three topological modules. A \noz for the  $**$ topology is an $A$-submodule $Q\subset M\otimes N\otimes P$ such that 
\begin{enumerate}[\noindent i)]
\item There exist a \noz $U$ of $M$ a \noz $V$ of $N$ and a \noz $W$ of $P$ such that the image of $U\otimes V\otimes W$ is contained in $Q$.
\item For all $m\in M$ there exists a \noz $V_m$ of $N$ and a \noz $W_m$ of $P$ such that the image of $m\otimes V_m\otimes W_m$ is contained in $Q$. 
\item Conditions analogous to the previous one for $n\in N$ and $p\in P$. 
\item For all $m\in M$ and $n\in N$ there exists a \noz $W_{m,n}$ of $P$ such that the image of $m\otimes n \otimes W_{m,n}$ is contained in $Q$. 
\item Conditions analogous to the previous one for $(m,p)\in M\times P $, and for $(n,p)\in N\times P$.
\end{enumerate}
The completion with respect to this tensor product is denoted with $M\tensor{*}N\tensor{*}P$. \index{** topology}

\begin{remark}
In the case of $\tau=\rightarrow$ or $\tau =\leftarrow$ or $\tau=!$ the multiple tensor product coincide with the iterated tensor product: for example
$M\tensor{\rightarrow}N\tensor{\rightarrow}P\simeq (M\tensor{\rightarrow}N)\tensor{\rightarrow}P.$
\end{remark}

\begin{remark}\label{oss:multicontinuita}
A multilinear map $\grg:M\times N \times P\lra Q$ is continuous if and only if the associated linear map $M\otimes N \otimes P\lra Q$ is continuous for the $**$ topology. Similar remarks holds also for tensor products of $n\geq 2 $ modules. 
\end{remark}

\begin{remark}\label{oss:prodotticompletaprimaedopo}
	For $\tau=*,\rightarrow,\leftarrow,!$ we have $\overline M\tensor{\tau}\overline N\simeq M\tensor{\tau}N$. A similar remark holds for multiple tensor products. 
\end{remark}

\begin{remark}\label{oss:amtp}The $*$-tensor product is not associative, however we have natural maps such as
$M\otimes^{*}N\otimes^{*}P\to (M\otimes^{*}N)\otimes^{*}P$. Let us make an example: consider the case $M= (\mC[t])[[z]]$ with the topology $z^nM$ and $N=\mC((y)), P=\mC[[z]]$ with the usual topologies, then one can show that $(M\otimes^* N) \otimes^* P$ is not isomorphic (as a topological module) to $M \otimes^* (N \otimes^* P)$; actually there are no continuous map between the two which are the identity on $M \otimes N \otimes P$. This also implies that the canonical map $M\otimes^* N \otimes^* P \to (M \otimes^* N) \otimes^* P$ can fail to be an isomorphism.
\end{remark}

\begin{remark}\label{oss:!conservaqccopf}
	For $\tau=\,!$ tensor products and multiple tensor products of topologically free modules are topologically free. 
\end{remark}

\subsubsection{Compact modules}\label{ssec:compactmodules}
We say that a topological $A$-module $M$ is $w$-compact (respectively fp-compact and proj-compact)\index{$w$-compact module}\index{$w$-compact module!fp-compact module}\index{$w$-compact!proj-compact module} if it is complete and there exists a \fsonoz $U_i$ such that the quotient $M/U_i$ is finitely generated (respectively finitely presented and projective and finitely generated). If $M$ is not complete but $M/U_i$ is finitely generated (resp. finitely presented, projective and finitely generated) for a \fsonoz $U_i$, then the completion $\widehat{M}$ is $w$-compact (resp. fp-compact, proj-compact).

\begin{remark}\label{oss:prodottocompatti1}
If $M$ is a $w$-compact module then 
$$
M\tensor{*}N\simeq M\tensor{\rightarrow}N
\quad\mand\quad
M\tensor {\leftarrow} N \simeq M\tensor{!}N
$$
as topological modules. More in general if $M_1,\dots,M_h$ are $w$-compact then we have the following isomorphism of topological modules
$$
M_1\tensor{*}\cdots \tensor{*} M_h\tensor{*}N\simeq M_1\tensor{\rightarrow}\cdots\tensor{\rightarrow}M_h\tensor{\rightarrow}N
$$
\end{remark}

\begin{remark}\label{oss:prodottocompatti}
	If $M$ and $N$ are $w$-compact modules then 
	$$
	M\tensor{*}N\simeq M\tensor{\rightarrow}N\simeq M\tensor {\leftarrow} N \simeq M\tensor{!}N
	$$
	as topological modules and moreover this module is $w$-compact. Moreover if the modules are fp-compact (or proj-compact) then the same property holds for their completed tensor product. A similar result holds for multiple tensor products. In particular $*$-tensor product of $w$-compact modules is associative. 
\end{remark}

\begin{remark}\label{rmk:starassociativecompactdiscrete}
	The $*$-tensor product between compact and discrete modules is associative.
\end{remark}

\noindent For some future applications we record the following remark as well. 

\begin{remark}\label{oss:prodottiregolaristar}
	For $k=1,\dots,\ell$ let  $M(k)=\bigoplus_{i\geq 0} M_i(k)$ with  $M_i(k)$  discrete for $i\geq 1$. Let $M_{\leq n}(k)=\bigoplus _{0\leq i \leq n}M_i(k)$ then 
	$$
	M(1)\tensor{*}\cdots \tensor{*}M(\ell)=\bigoplus M_{i_1}(1)\tensor{*}\cdots\tensor{*}M_{i_\ell}(\ell)=\limind 
	M_{\leq n}(1)\tensor{*}\cdots\tensor{*}M_{\leq n}(\ell).
	$$
	In addition, if we assume $M_0(k)$ $w$-compact for all $k$ (c.f. Section \ref{ssec:compactmodules}), the tensor products $M_{i_1}(1)\otimes^*\cdots \otimes^* M_{i_l}(l)$ and $M_{\leq n}(1)\otimes^*\cdots\otimes^* M_{\leq n}(l)$ with the iterated tensor products. This follows from a version of Remark \ref{oss:prodottiregolarifreccetta} for the multiple $*$-tensor product together with Remark \ref{rmk:starassociativecompactdiscrete}.
\end{remark}

\begin{remark}\label{oss:compattezzaeomomorfismi}
	If $M$ is a $w$-compact module then it is not a compact object in the category of modules, however 
	$$\Homcont(M,\oplus X_i)=\oplus \Homcont(M,X_i)$$
	for arbitrary direct sums. Moreover if the modules $X_n$ are complete and $X_n\lra X_{n+1}$ are injective and $X_n$ has the subspace topology then, 
	$$
	\Homcont(M,\limind X_n)=\limind \Homcont(M,X_n)
	$$ 
\end{remark}

\subsubsection{Compactly generated modules, \ccg-modules and regular modules}
Given a topological module $M$, the set of $w$-compact submodules of $M$ is filtered, indeed given $H$ and $K$ $w$-compact submodules then $\overline{H+K}$ is also $w$-compact. We have a natural map $\limind K \lra M$ where the colimit is taken in the category of topological modules, over all $w$-compact submodules of $M$. We say that $M$ is compactly generated (for short cg-module)\index{cg-module} if the above map is an isomorphism. For many applications this hypothesis will be sufficient for us. Notice that we do not require $M$ to be complete. If there exists a sequence $K_n\subset K_{n+1}$ of $w$-compact submodules such that $M=\limind K_n$ then we say that $K_n$ is an \exhaustion  of $M$ and that $M$ is compact countably generated module (for short \ccg-module)\index{cg-module!\ccg-module}. Some important difference among the two concept is explained in the following remark. 

\begin{remark}\label{oss:proprietamoduliregolari}Every \ccg-module is also a cg-module, indeed if $M_n$ is an \exhaustion of the module $M$ and $K$ is a compact submodule then $K\subset M_n$ for some $n$ by Remark \ref{oss:compattezzaeomomorfismi}. Also \ccg-modules are always complete (recall that $w$-compact modules are assumed to be complete).
\end{remark}

These classes of module will be important for us. A even better class of modules should be that of regular modules. 

\begin{definition}\label{def:moduloregolare}
A topological module $M$ is $w$-regular\index{$w$-regular} if it is a \ccg-module and if every compact submodule has a countable \fsonoz. 
\end{definition}

That the topological module $M$ is a $w$-regular module is equivalent to the fact that there exists an \exhaustion of $M$ by $w$-compact submodules $M_n$ such that each $M_n$ has a countable \fsonoz. In particular the sheaf associated to a $w$-regular module has trivial cohomology on affine subsets. 

Let us remark the well behaviour of ccg-modules with respect of the $\ra$-tensor product.

\begin{remark}\label{oss:prodottiecolimiti}
	Assume $M$ is topologically free and that $N=\limind N_n$ is an \exhaustion of $N$. Then $\limind M\tensor{\ra} N_n$ is isomorphic to $M\tensor{\ra}N$. 
\end{remark}

\subsubsection{Homomorphisms of modules}\label{ssec:hommodules}Let $M$ and $N$ be two topological modules. On $\Homcont(M,N)$ we put the topology where neighborhoods of zero are the submodules 
\[
\mH_{K,V}=\Homcont(M,N)_{K,V}=\{\varphi:M\to N \,:\, \varphi(K)\subset V\}
\]
where $K$ is a $w$-compact submodule of $M$ and $V$ is a \noz of $N$.\index{$\mH_{K,V}$}

\begin{remark}\label{oss:Hometensorifasci}
	If $M$ is compactly generated and $N$ is complete then 
	$\Homcont(M,N)$ is complete.
	
	Moreover if $M$ is compactly generated, then the following natural isomorphisms of modules are also isomorphisms of topological modules:
	\begin{align*}
	\Homcont(M,N)
	&\simeq \limpro_{K}\Homcont(K,N)
	\simeq \limpro_{V}\Homcont\left(M,\frac N V\right)\\
	&\simeq \limpro_{K,V}\Homcont\left(K,\frac N V\right)
	\simeq \limpro_{K,V}\limind_{U}\Hom\left(\frac K U , \frac N V\right)
	\end{align*}
	where $K$ is a compact submodule of $M$, $V$ is a \noz of $V$ and $U$ is a \noz of $K$. 
\end{remark}

\subsection{Topological sheaves}

We move forward introducing our notion of a topological sheaf on a scheme, this will be more than a sheaf of topological modules: we ask a stronger condition that the topology is defined globally. We will then single out a peculiar class of topological sheaves which we call \emph{quasi-coherent completed} (\QCC for short), study their properties and then develop the theory paralleling the theory of modules exposed in the previous Section.

We will consider sheaves of $\calO_S$-modules on a base scheme $S$ which we do not require to be quasi-coherent. We always assume that $S$ is topologically noetherian and quasi separated, although these hypotheses is not always necessary.

\begin{definition}[Topological Sheaf] Let $\calF$ be a $\calO_S$-module. A topology on $\calF$ is a family of $\calO_S$-subsheaves $\calV \subset \calF$, called the neighborhoods of zero (\emph{\noz} for short),\index{\noz} such that if $\calV\subset\calV'\subset \calF$ and $\calV$ is a neighborhoods of zero, then $\calV'$ is a neighborhood of zero, and such that if $\calV,\calV'$ are neighborhood of zero, then $\calV\cap \calV'$ is a neighborhood of zero. Equivalently, a topology can be specified by a fundamental system of neighborhood of zero (\emph{\fsonoz} for short).\index{\fsonoz}

A continuous morphism of topological sheaves $\grf:\calF\lra\calG$ is a morphism of $\calO_S$-sheaves such that for all $\calW$ \noz of $\calG$ there exists a \noz $\calV$ of $\calF$ such that $\grf(\calV)\subset \calW$. The set of continuous morphism will be denoted by $\Homcont_{\calO_S}(\calF,\calG)$ and the associated sheaf by $\calHomcont_{\calO_S}(\calF,\calG)$. Since $S$ is topologically noetherian we can check continuity locally, in particular $\calHomcont$ can be indeed shown to be a sheaf.  \index{$\calHomcont$} A topological sheaf is said to be complete if the natural map $\calF \to \limpro_{\noz}\calF/\calU$ is an isomorphism and separated if it is injective. We will write $\hat{\calF} = \limpro \calF/\calU$ and call it the \emph{completion} of $\calF$.

These definitions make sense if we replace the word ``sheaf'' with ``presheaf'' so that there is an analogous notion of a topological presheaf on $S$ and a notion of continuous homomorphisms between them.
\end{definition}

\begin{remark}[Sheaves vs presheaves]\label{rmk:sheavesvspresheaves}
	Given a presheaf $\calF$ on $S$ we write $\calF^\sharp$ for its sheafification. If $\calF$ has a topology given by neighborhoods of zero $\calV_i$, then $\calV_i^\sharp \subset \calF^\sharp$ determines a topology on $\calF^\sharp$. It is easy to check that given a topological sheaf $\calG$ continuous morphisms $\calF \to \calG$ correspond to continuous morphisms $\calF^\sharp \to \calG$. It follows by adjunction that for a topological presheaf $\calF$ we have that $\widehat{(\calF^\sharp)}$ agrees with the completion of $(\widehat{\calF})^\sharp$.
\end{remark}

The category of topological sheaves is a quasi abelian category (while the category of complete topological sheaves is not), it has arbitrary limits and colimits just as in the case of modules. Limits are computed point-wise

\begin{remark}[Pushforward and pullback]
	If $\varphi :S' \to C$ and $\calF$ is a topological sheaf on $S'$ then the pushforward of the \noz of $\calF'$ define a topology on $\varphi_*(\calF')$ and similarly if 
$\calF$ is a topological sheaf on $S$ then the pullback of the \noz of $\calF$ define a topology on $\varphi^{-1}(\calF)$ (see Section \ref{ssec:pullbackcompletati} for the definition of the topology on $\varphi^*(\calF)$). Pushforward preserves completeness while, in general $\varphi^{-1}$ does not. Restriction to open subsets preserves completeness and moreover completeness is a local property. In addition, the restriction of the completion is the completion of the restriction.
\end{remark}

\begin{definition}[Subsheaves, quotients]
	Given a topological sheaf $\calF$ and a subsheaf $\calG \subset \calF$, $\calG$ inherits a topology by declaring a \fsonoz to be $\calG \cap \calU$, with $\calU$ a \noz of $\calF$. On the other hand, we define a topology on the quotient by declaring that a subsheaf $\calV \subset \calF/\calG$ is a \noz if and only if $q^{-1}(\calV)$ is a \noz in $\calF$, where $q : \calF \to \calF/\calG$ is the quotient map. Equivalently, neighborhoods of zero are of the form $(\calU + \calG)/\calG$ where $\calU \subset \calF$ is a \noz. These constructions allow us to construct kernels and cokernels in the category of topological sheaves.
\end{definition}

\begin{remark}\label{rmk:completionkernel}
	Let us notice that the kernel of a morphism $\calF \to \calG$ with $\calF$ complete and $\calG$ separated is complete. It follows that an arbitrary limit of complete sheaves is complete.
\end{remark}

\begin{definition}[Closure]
	Let $\calG \subset \calF$ be a subsheaf of a topological sheaf $\calF$. We define, in analogy to the case of modules, $\overline{\calG}$, the \emph{closure} of $\calG$ as the intersection $\overline{\calG} = \cap_{\calG \subset \calU \text{ \noz}}\, \calU$. If $\calF$ is a complete sheaf then $\overline{\calG}$ is complete (this follows from the previous remark and by the fact that open subsheaves are complete).
\end{definition}

\begin{remark}\label{lemma31}
	Let $\calF$ be a sheaf and let $\hat \calF$ be its completion. Let $\calG$ a subsheaf of $\calF$ and $\calH$ be a subsheaf of $\hat \calF$ and assume that the image of $\calG$ in $\hat \calF$ is contained and is dense in $\calH$.  
	Then the completion of $\calF/\calG$ is isomorphic to the completion of $\calF/\overline \calG$ and also to the completion of $\hat \calF/\calH$.
\end{remark}

\begin{remark}[Topology on sections]
	Given a topological sheaf $\calF$ and an open subset $U \subset S$ the module $\calF(U)$ has a natural topology, where a \fsonoz is given by $\calV(U) \subset \calF(U)$ for $\calV$ a \noz of $\calF$. If $\calF$ is complete (resp. separated) then $\calF(U)$ is complete (resp. separated) for any $U$. The converse is not true, if $\calF(U)$ is complete for every $U$ the sheaf $\calF$ does not need to be complete.
\end{remark}

\subsubsection{\QCC Sheaves}\label{ssec:qccsheaves}

Let $S = \Spec A$ be affine and $M$ be a topological $A$-module, denote by $\widetilde{M}$ the quasi-coherent module on $S$ attached to $M$. This is naturally a topological sheaf where a fundamental system of neighborhoods of zero in $\widetilde{M}$ is given by $\widetilde{U}$ for $U$ a \noz in $M$.

\begin{definition}\label{def:qccsheaves}
	We write $\widetilde{M}^c$ for the completion of the sheaf $\widetilde{M}$ with the above topology. Notice that for any $f \in A$ we have
	\[
		\widetilde{M}^c(\Spec A_f) = \varprojlim_{U \text{ \noz}} \frac{M_f}{U_f}.
	\]
	A topological sheaf on an a scheme $S$ is called quasi-coherent completed (\QCC for short) \index{\QCC} if its restriction to any affine open subset $U \subset \calF$ is isomorphic to $\qcc{\calF(U)}$. 
\end{definition}

\begin{definition}
	A \QCC sheaf $\calF$ is called \QCCF if for every open affine subset $U\subset S$ the module $\qcc{\calF(U)}$ is topologically locally free (see Definition)
\end{definition}

\begin{remark}\label{oss:sqccclocalita}
	Being a \QCC sheaf is a local property meaning the following: if $\calF$ is a topological sheaf on a scheme $S$ and $S$ can be covered by affine open subsets $S_i$ such that the restriction of $\calF$ to $S_i$ is a \QCC sheaf, then $\calF$ is a \QCC sheaf.
\end{remark}

\begin{remark}\label{oss:sqccccomologia}
Let $\calF$ be a complete topological sheaf that admits a countable \fsonoz $\{\calU_n\}_{\mN}$ such that for every $n \in \mN$, $\calF/\calU_n$ is a quasi-coherent sheaf. Then 
\begin{itemize}
	\item $\calF$ is a \QCC sheaf 
	\item $H^p(U,\calF)=0$ for any $p>0$ and for any open affine subset $U$.  
\end{itemize} 
\end{remark}

\begin{remark}\label{oss:qcclimiti}
	Let $\calF$ be a topological sheaf, assume that $\calF_n \subset \calF_{n+1}$ is a countable system of immersions and that $\calF = \varinjlim \calF_n$. Then, if all $\calF_n$ are \QCC sheaves, $\calF$ is a \QCC sheaf as well.
\end{remark}

\subsubsection{Tensor products of sheaves}\label{ssec:topologytensorprodsheaves}
Let $\calF$ and $\calG$ be two topological $\calO_S$ modules; here we define several different topologies on the tensor product $\calF\otimes_{\calO_S}\calG$ (we usually omit the subscript $\calO_S$ when the tensor product is done with respect the structure sheaf of our base scheme). Let us emphasize that we do not need the sheaves in question to be \QCC for the following definitions to work. We define the following topologies at the level of presheaves; in particular we may assume that $\calF,\calG$ are topological presheaves. In the case where $\calF,\calG$ are sheaves the following induce topologies on the sheaf theoretic tensor product by the discussion of Remark \ref{rmk:sheavesvspresheaves}.

\begin{itemize}
	\item $ \hat{} \;$-topology: a \fsonoz is given by the products $\calU\otimes \calV$ where $\calU$ is a \noz of $\calF$ and $\calV$ is a \noz of $\calG$.  
	\item $r$-topology: a \fsonoz is given by the products $\calF\otimes \calV$ where $\calV$ is a \noz of $\calG$.  
	\item $\ell$-topology is defined symmetrically to the $r$-topology.
    \item $\rightarrow$-topology: we consider the presheaf $\calH$ given by $U\mapsto \calF(U)\otimes \calG(U)$. A \noz $\calQ$ of $\calH$ is a subpresheaf such that 
    \begin{itemize}
       	\item there exists a \noz $\calV$ of $\calG$ such that $\calQ$ contains the image of the presheaf $U\mapsto\calF(U)\otimes \calV(U)$ in $\calH$.
       	\item for each open set $U$ of $S$ and for each $\grs\in \calG(U)$ there exists a \noz $\calU_\grs$ of $\calF$ such that $Q|_U$ contains the image of the $\calU_\grs\otimes \grs$ in $\calH$. 
    \end{itemize} 
    The $\rightarrow$ topology on $\calF\otimes_{\calO_S}\calG$ 
    is the sheafification of this topology. 
    \item $\leftarrow$-topology: is defined symmetrically to the $\rightarrow$-topology.
    \item $*$-topology: we consider the presheaf $\calH$ given by $U\mapsto \calF(U)\otimes \calG(U)$. A \noz $\calQ$ of $\calH$ is a subpresheaf such that 
    \begin{itemize}
    	\item there exists a \noz $\calU$ of $\calF$ and a \noz $\calV$ of $\calG$ such that $\calQ$ contains the image of the presheaf $U\mapsto\calF(U)\otimes \calV(U)$ in $\calH$.
    	\item for each open set $U$ of $S$ and for each $\grs\in \calG(U)$ there exists a \noz $\calU_\grs$ of $\calF$ such that $Q|_U$ contains the image of the $\calU_\grs\otimes \grs$ in $\calH$ and symmetrically for each $\grs'\in \calF(U)$. 
    \end{itemize} 
    The $*$ topology on $\calF\otimes_{\calO_S}\calG$ 
    is the sheafification of this topology. 
    \item $!$-topology: a \fsonoz is given by the sums $\calU\otimes \calG+\calF\otimes \calV$ where $\calU$ is a \noz of $\calF$ and $\calV$ is a \noz of $\calG$.  
\end{itemize}
We denote the completion of $\calF\otimes \calG$ with respect to the topology $\tau$ with $\calF\stackrel{\tau}{\otimes} \calG$. We notice also that the construction of these tensor products is functorial.

The topology $*\cdots*$ on multiple tensor products is defined in a completely analogous way to the case of modules. 

\begin{remark}\label{oss:prodottolocalizzabene}
	The construction of these topologies is local, meaning that if 
	$S'$ is an open subset of $S$ then the $\tau$-topology on the restriction of the product is equal to the restriction of the topology. Moreover, if $M$ and $N$ are complete modules, then on an affine scheme $S = \Spec A$, we have 
	$$\qcc{M}\tensor{\tau}\qcc{N} \simeq \qcc{M\tensor{\tau}N}$$ and on an open subset $\Spec A_f$ we have $$\big(\qcc{M}\tensor{\tau}\qcc{N}\big)(S_f)=\overline {M_f}\tensor{\tau}\overline{N_f}$$ for all the topologies above. In particular the completed tensor products of two \QCC sheaves is \QCC. Similar results hold for multiple tensor products.
\end{remark}

\begin{remark}\label{oss:tensoriisomorfifasci}
	As in the case of modules, there are natural maps among the completions of the tensor products. Let us notice here that if $\calN$ is discrete then $\calM\hat\otimes \calN\simeq \calM\otimes^{r}\calN\simeq \calM\otimes \calN$ with the discrete topology. Moreover $
	\calM\otimes^{*}\calN\simeq \calM\tensor{\rightarrow}\calN$ and 
$\calM\tensor{\leftarrow}\calN\simeq \calM\otimes^{!}\calN$ as topological sheaves. 
	If $\calN$ is \QCCF then the natural map $
	\calM\tensor{\rightarrow}\calN \to \calM\otimes^{r}\calN
	$ is an isomorphism of plain sheaves; if both $\calM$ and $\calN$ are \QCCF, then $\calM\otimes^{*}\calN\simeq \calM\hat \otimes \calN$ as plain sheaves.
\end{remark}

\begin{remark}\label{prop:otimesshcontrootimesr}
    Let $\calF_1,\calF_2,\calG_1,\calG_2$ be topological sheaves. Then there exists a canonical map
    \[
        (\calF_1\otimessh \calG_1) \otimesr (\calF_2\otimessh \calG_2) \to (\calF_1\otimesr \calF_2)\otimessh (\calG_1\otimesr \calG_2),
    \]
    which satisfies natural associativity properties. In particular, given $\otimesr$-algebras $\calA$ and $\calB$, the tensor product $\calA\otimes^!\calB$ is naturally a $\otimesr$ algebra, whose product comes from the above morphism taking $\calF_1 = \calF_2 = \calA$, $\calG_1=\calG_2 = \calB$ together composed with the products of $\calA$ and $\calB$.
\end{remark}

We state the following Lemma that explains how completion is related to the various tensor products.

\begin{lemma}\label{lemma:pullbackprodottitensori}
Assume $S$ is topologically noetherian.
Let $\calF$ and $\calG$ be two topological sheaves on $S$ and let $\overline\calF$ and $\overline \calG$ be their completion. Then 
$$
\calF\tensor {\ra} \calG \simeq \calF\tensor {\ra} \overline \calG,\quad
\calF\tensor {\ra} \calG \simeq \overline \calF\tensor {\ra}  \calG,\quad
\calF\tensor {*} \calG \simeq \calF\tensor {*} \overline \calG,\quad
\calF\tensor {!} \calG \simeq \calF\tensor {!} \overline \calG.
$$
\end{lemma}
\begin{proof}
We prove the first statement. The proof of the other statement is similar. 
By functoriality there is a map from the left hand side to the right hand side. 
We construct a map on the opposite direction. Fix an open subset $U$ and a section $\grs$ of $\calF(U)$. The map
$$
\calG|_U \lra (\calF\tensor {\ra}  \calG)|_U
\quad \text{ given by }\quad 
\tau\longmapsto \grs\otimes \tau 
$$
defined a continuous map from the two sheaves (this fact would be false for the tensor products $\tensor r$ or $\tensor \ell$ or $\hat\otimes $). Hence it determines a map 
$$\overline \calG|_U \lra (\calF\tensor {\ra}  \calG)|_U$$
In particular, since $\overline{\calG}|_U = \overline{\calG|_U}$, we have constructed for every $U$ a map
$$
\gra:\calF(U)\otimes (\overline\calG)(U)\lra (\calF\tensor {\ra}  \calG)(U)
$$
This determines a map among the two presheaves. We prove that this is a continuous map of topological presheaves if we put on the left hand side the $\tensor{\ra}$-topology. 

Fix a noz $\calQ$ of $\calF\tensor{\ra}\calG$. Recall that such a noz  contains the image of $\calF\tensor{\ra} \calV$ for a noz $\calV$ of $\calG$ and for every open subset $W$ and every $\tau\in \calG(W)$ there exists a noz $\calU$ of $\calF$ such that $\calQ|_W$ contains the image of $\calU|_W\otimes \tau$. 

It is clear that the map $\gra$ restricted to $\calF \otimes \overline \calV$ factors through $\calF\tensor{\ra} \calV$ hence its image is contained in $\calQ$. 

Now given an open subset $W$ and a section $\tau'$ of $\overline\calG(W)$ we want to prove that there exists a noz $\calU'$ of $\calF$ such that the  image of $\calU'|_W\otimes \tau'$ is contained in $\calQ|_W$.

The section $\tau'$ determines a section of $\overline \calG/\overline \calV\simeq \calG/\calV$ on $W$. By quasi compactness of $W$ there exists a finite covering $W_1,\dots,W_m$ of $W$ and sections $\tau_i\in \calG(W_i)$ such that 
$\tau'|_{W_i}$ is represented in this quotient by $\tau_i$. Let $\calU_i$ be a noz of $\calF$ such that $\calQ|_{W_i}$ contains $\calU_i|_{W_i}\otimes \tau_i$. Set $\calU=\cap\,\calU_i$. 
We want to prove that the image of $\calU|_W\otimes \tau'$ is contained in $\calQ|_{W}$. It is enough to prove that the image of the map $\gra$ restricted to $\calU|_{W_i}\otimes \tau'|_{W_i}$ is contained in $\calQ|_{W_i}$ for all $i$. Indeed 
$$
\gra\Big(\calU|_{W_i}\otimes \tau'|_{W_i} \Big)\subset \gra\Big(\calU_i|_{W_i}\otimes (\tau_i+\overline \calV|_{W_i})\Big)\subset 
\gra\Big(\calU_i|_{W_i}\otimes \tau_i +\calF|_{W_i}\otimes \calV|_{W_i}\Big)
\subset \calQ|_{W_i}.$$
This proves that the map $\gra$ is continuos. Since the target of $\gra$ is a complete topological sheaf it determines a continuos map of topological sheaves  
$$\gra_1:\calF\tensor\ra \overline \calG\lra \calF\tensor\ra\calG$$ where on the domain of $\gra_1$ we put the $\tensor\ra$ topology. 

Finally the two maps are inverse to each other. 
\end{proof}

\subsubsection{Compact sheaves}\label{ssec:compactsheaves} 
We say that a \QCC sheaf $\calF$ on $S$ is $w$-compact (respectively fp-compact, proj-compact)\index{$w$-compact sheaf}\index{$w$-compact sheaf!fp-compact sheaf}\index{$w$-compact sheaf!proj-compact sheaf} if it is complete and there exists a \fsonoz $\calU_i$ such that $\calF/\calU_i$ is a quasi coherent Zariski-locally finitely generated $\calO_S$-sheaf (respectively quasi coherent and  finitely-presented, locally free $\calO_S$-sheaf).

To our understanding these are not local properties since the sheaves $\calU_i$ are fixed globally. However we have the following remark.

\begin{remark}\label{oss:compattiaffini}
	If $S=\Spec A$ and $M$ is a complete topological module then $\qcc{M}$ is $w$-compact (respectively fp-compact, proj-compact) if and only if the module $M$ is $w$-compact (respectively fp-compact and proj-compact). 

    This implies that a \QCC sheaf is $w$-compact (respectively fp-compact, proj-compact ) if there exists a \fsonoz $\calU_i$ such that the restriction of $\calF/\calU_i$ to any affine subset $S'=\Spec A'$ of $S$ is of the form $\qcc{Q}$ with $Q$ respectively finitely generated, finitely presented and finitely generated and projective $A'$ module. 
\end{remark}

\begin{remark}\label{oss:fascicompattimorfismi}
Let  $\calF$ be a $w$-compact sheaf and recall that we are assuming $S$ to be topologically noetherian. 
Let $\calG=\limind\calG_n$ with $\calG_n\subset \calG_{n+1}$  \QCC sheaves (in this case also $\calG$ is a \QCC sheaf), then 
$$ \Homcont(\calF,\calG)=\limind \Homcont(\calF,\calG_n). $$
As in the case of modules, if $\calF$ is $w$-compact then $\Homcont(\calF,\cdot)$ commutes with arbitrary direct sums. 
Similar considerations hold for the sheaf $\calHomcont$.
\end{remark}

\begin{remark}\label{oss:prodottotensorfascicompatti}We have analogues of Remark \ref{oss:prodottocompatti1} and Remark \ref{oss:prodottocompatti} for $w$-compact sheaves. In particular  if $\tau=*,\rightarrow,\leftarrow$ and $!$ the multiple tensor products of $w$-compact sheaves coincide. Moreover this tensor product preserves $w$-compactness, fp-compactness and proj-compactness.
\end{remark}

\subsubsection{Compactly generated sheaves and regular sheaves}\label{ssec:ccgeregolari}
In the definition of compactly generated sheaves we need to make some slight changes since it is not clear that the closure of the sum of two compact subsheaves is still compact. Thus, given a topological sheaf $\calF$, we define an \exhaustion of $\calF$ to be a filtered system $\calK_i \subset \calF$ consisting of $w$-compact subsheaves (in particular \QCC and complete) such that $\varinjlim_i \calK_i = \calF$ in the category of topological sheaves. A compactly generated sheaf is then a topological sheaf which admits an exhaustion. Compact countably generated sheaves (for short \ccg) sheaves, and regular sheaves are defined in a completely analogous way to the case of modules. \index{\QCC!cg-sheaf}\index{\QCC!cg-sheaf!\ccg-sheaf} For example an \exhaustion of a sheaf $\calF$ is an increasing  sequence $\calF_n$ of $w$-compact \QCC subsheaves of $\calF$ such that $\limind \calF_n=\calF$. 

\begin{remark}\label{oss:fascicgccgregulari}
Thanks to remark \ref{oss:compattiaffini} if $S=\Spec A$ and $M$ is a complete topological module then $\qcc{M}$ is a compactly generated \QCC sheaf if and only if $M$ is a compactly generated module. Similarly for 
\ccg, \exhaustion and regular sheaves. In addition, if we have an exhaustion by $w$-compacts $M_n \subset M_{n+1}$, $M = \varinjlim M_n$ we have $\qcc{M} = \varinjlim \qcc{M_n}$. 
Assuming that $S$ is topologically noetherian, as we always do, if $\calF_n$ is an \exhaustion of the sheaf $\calF$ and $\calK$ is a $w$-compact sheaf of $\calF$ then $\calK\subset \calF_n$ for some $n$. 

\end{remark}
	
\begin{remark}\label{oss:fasciregolari1}
Regular sheaves have trivial non zero cohomology on affine subsets. We already knew this result for a particular class of \QCC sheaves, by Remark \ref{oss:sqccccomologia}\end{remark}

\subsubsection{Homomorphisms of sheaves}\label{ssec:omomorfismi}
In the following discussion our assumption that $S$ is topologically noetherian becomes crucial. We will also assume that $\calF$ is compactly generated, meaning $\calF=\limind \calF_i$ over some filtered system of $w$-compact subsheaves of $\calF$. 
We define a topology on $\Homcont(\calF,\calG)$ as follows. A \fsonoz of is given by the subsheaves
$$
\mH_{\calK,\calV}=\Homcont(\calF,\calG)_{\calK,\calV}=\{\grf:\calF\to \calG\,:\, \grf(\calK)\subset \calV\},
$$
\index{$\mH_{\calK,\calV}$}where $\calK$ is a $w$-compact subsheaf of $\calF$ and $\calV$ is a \noz of $\calG$. Similarly, we define a topology on the sheaf $\calHomcont(\calF,\calG)$: given $\calK$ and $\calG$ as above the subsheaf $\calHomcont(\calF,\calG)_{\calK,\calV}$
on an open subset $U \subset S$ is given by $\Homcont(\calF|_{U},\calG|_{U})_{\calK|_{U},\calV|_{U}}$.

\begin{remark}\label{oss:compatibilitaHomcalHom}
The two definitions are compatible: meaning that the topological module $\calHomcont(\calF,\calG)(S)$ is isomorphic to $\Homcont(\calF,\calG)$.
\end{remark}

\begin{remark}\label{rmk:completezzaHom}
If $\calF$ is compactly	generated and  $\calG$ is complete then the sheaf $\calHomcont(\calF,\calG)$ is complete and moreover we have the following isomorphism of topological sheaves 
\begin{align*}
\calHomcont(\calF,\calG)
&\simeq \limpro_{\calK}\calHomcont(\calK,\calG)
\simeq \limpro_{\calV}\calHomcont\left(\calF,\frac \calG \calV\right)\\
&\simeq \limpro_{\calK,\calV}\calHomcont\left(\calK,\frac \calG \calV\right)
\simeq \limpro_{\calK,\calV}\limind_{\calU}\calHom\left(\frac {\calK} \calU , \frac \calG \calV\right)
\end{align*}
where $\calK$ is a $w$-compact subsheaf of $\calF$, $\calV$ is a \noz of $\calG$ and $\calU$ is \noz of $\calK$. The analogous properties for the modules $\Homcont$ follows by Remark \ref{oss:compatibilitaHomcalHom}.
\end{remark}

\begin{remark}\label{rmk:sectionshom}
	Assuming that $\calF$ is topologically free, meaning there exists a \fsonoz $\calF_i$ with $\calF/\calF_i$ free, we have that 
	\[
		\Homcont_{\calO_S}(\calF,\calG)(S) = \Homcont_{\calO_S}(\calF(S),\calG(S)).
	\]
\end{remark}

\noindent We will be mainly interested in the case of \QCC sheaves, where sections of $\Hom$ sheaves are very well behaved. 

\begin{remark}\label{oss:morfismidifasci1}
If $S=\Spec(A)$, $N$ is a complete module and $M=\limind M_n$ is an \exhaustion of $M$ such that each $M_n$ is proj-compact and $M_n$ has a complement in $M$ for all $n$, then 
$$ \calHomcont(\qcc{M},\qcc{N})=\qcc{\Homcont(M,N)}. $$
It follows that in the case $\calF$ is \QCC, compactly generated by proj-compact modules and $\calG$ is \QCC, we have that $\calHomcont(\calF,\calG)$ is \QCC and that on any affine open subset $U \subset S$ we have $\calHomcont(\calF,\calG)(U) = \Homcont(\calF(U),\calG(U))$.
\end{remark}

\subsection{Topological algebras and their modules}
\label{ssec:topologicalalgebrasandmodule}

In this section we develop a bit of the theory of topological algebras in our setting of topological sheaves. We deal directly with the case of sheave, since it is the case we will be working with in the rest of the paper. There are several notions of a topological algebra depending on the completed tensor product one may consider. It is also possible to relax the continuity condition in several ways, we will explain some constructions here.

Let us first introduce the notion of separately continuous morphism. If $\calN_1,\calN_2,\calM$ are topological sheaves of $\calO_S$-modules and we are give a bilinear morphism $b: \calN_1 \otimes \calN_2 \to \calM$, we say that $b$ is \emph{continuous in the first variable} if for every open subset $U \subset S$ and any local section $n \in \calN_1(U)$ the induced morphism $b(n\otimes\_) : \calN_2|_U \to \calM|_U$ is continuous. To be continuous on the second variable is defined analogously. We call $b$ continuous if it is continuous in both variables and for every \noz $\calV \subset \calM$ there exists \noz $\calU_1 \subset \calN_2,\calU_2\subset\calN_2$ such that $b(\calU_1\otimes\calU_2) \subset \calV$. This is equivalent to the fact that $b$ is continuous for the $*$-topology. 

\begin{definition}[Topological Algebra]
	A topological algebra is a topological sheaf $\calR$, which is also an algebra whose product is $*$-continuous.\index{topological algebra} A  $\calR$-topological module $\calF$ is a $\calR$-module which is also a $\calO_S$ topological module such that the product $\calR\times\calF\to \calF$ is continuous.\index{topological algebra!topological module}. We say that $\calF$ is a linear $\calR$-topological module if the topology is defined by a \fsonoz that are $\calR$- submodules.\index{topological algebra!topological module!linear topological module} Occasionally, we will consider $\calO_S$-algebra (or modules) that are topological $\calO_S$-sheaf but such that the multiplication map is continuous for a topology $\tau$ on the tensor product different from the $*$-topology. We call these $\tau$-topological algebra (or modules).
\end{definition}

It is natural to ask if the tensor product of topological algebras is again a topological algebra, and if so if the tensor product of modules is a module for the latter. This is certainly true if we consider complete $!$-topological algebras and modules, since the $\otimes^!$ is associative and commutative. The lack of associativity of the $\otimes^*$ and the lack of commutativity of the $\otimesr$ product is the source of some technical problems which we will treat in the following sections. Let us remark that, when using the more naive tensor products $\hat{\otimes},\otimes^r$ these issues disappear.

\begin{remark}\label{oss:prodottodialgebrehat}
	If $\calR_1$ and $\calR_2$ are topological algebras then the tensor products
	$\calR_1\tensor{r}\calR_2$ and $\calR_1\hat\otimes\calR_2$ are topological algebras. 
\end{remark}

Let us state here the following remark, which deals with the simplest case of a module locally homeomorphic to $\calR$.

\begin{remark}\label{oss:continuitamorfismimodulitopologici}If $S$ is quasi compact, $\calR$ is a topological algebra on $S$, and $\calF$ and $\calG$ are $\calR$-topological sheaves such that $\calF$ is locally homeomorphic to $\calR$ then every $\calR$-linear morphism from $\calF$ to $\calG$ is continuous. 
\end{remark}

\subsubsection{The case of \texorpdfstring{$w$}{w}-compact algebras}
Let $\calR_i$ be $w$-compact $\calO_S$-topological algebras. We already noticed that the $*$ tensor product coincide with the $\rightarrow$ tensor product and with the $!$ tensor product and that multiple tensor products coincide with iterated tensor product. In particular from the associativity and commutativity it follows that the tensor product $\calR_1\tensor{*} \calR_2$ is a topological algebra and it is $w$-compact. The results for multiple tensor products follows. 

It is clear that the $!$-product of modules is still a $!$-topological module for the $!$-product of algebras. Let us investigate what happens for the other tensor products. Notice first that given topological modules $\calK$ and $\calM$, with $\calK$ is $w$-compact, and a bicontinuous map $b:\calK\otimes \calM\lra \calM$, then from $\calK\tensor{*}\calM\simeq \calK\tensor\ra \calM$ it follows that for every $\calU$ \noz of $\calM$ there exists $\calV$ \noz of $\calM$ such that $b(\calK\otimes \calV)\subset \calU$.

\begin{remark}\label{oss:prodottoalgebremoduli}
Assume that $\calR_1, \dots, \calR_h$ are $w$-compact topological algebras and that 
$\calM_i$ is a $\calR_i$ topological module. Set $\calR=\calR_1\tensor{*}\cdots\tensor{*}\calR_h$
Then the action of $\calR_1\otimes\cdots \otimes \calR_h$ on the tensor product of the modules $\calM_i$ equipped with $**$ or $\rightarrow$ topology is continuous, in particular
$$
\calM_1\tensor{\rightarrow}\cdots\tensor{\rightarrow}\calM_h
\quad\mand\quad 
\calM_1\tensor{*}\cdots\tensor{*}\calM_h
$$
are topological $\calR$-modules. 
\end{remark}

\subsubsection{\ccg-topological algebras and cg-tensor product}\label{ssec:ccgecg}
The only general result that we know on the $*$ and $\rightarrow$ tensor product is the one above about $w$-compact algebras. To deal with more general situations we need to relax the condition of continuity and bootstrap the constructions from the $w$-compact case. 

\begin{definition}[\ccg-topological algebras and modules]\label{def:ccgalgmod}
Let $\calR$ be a \ccg-topological algebra and $\calR=\limind \calR_n$ be an \exhaustion by compact subsheaves $\calR_n$. Assume that $\calR$ is also a $\calO_S$-algebra. We say that $\calR$ is a \ccg-topological algebra if the restriction of the product to $\calR_n\times \calR_n$ is continuous for all $n$. Notice that by Remark \ref{oss:fascicompattimorfismi} this condition does not depend on the choice of the \exhaustion $\calR_n$. If $\calM$ is an $\calR$-module and a topological $\calO_S$-sheaf, we say that $\calM$ is an $\calR$ \cgmod module if the action restricted to $\calR_n\times \calM$ is continuous for all $n$.
\end{definition}

\begin{remark}
Let $\calR$ be a \ccg-algebra, and $\calM$ a cg-module over $\calR$. Then the multiplications maps
$$
\calR\times \calR\lra \calR \quad\mand\quad \calR\times \calM\lra \calM
$$
are continuous in both variables.
\end{remark}

We also consider a suitable tensor product of \ccg-algebras.

\begin{definition}(cg-tensor product)
Let $\calA=\limind \calA_n$, $\calB = \limind\calB_n$ be \ccg-topological $\calO_S$-modules. Then define 
$$
\calA\tensor{cg}\calB=\limind \calA_n\tensor{*}\calB_n.
$$
The colimit is independent on the choice of the \exhaustion. And this tensor product is commutative and associative in the category of \ccg-topological modules. A \ccg-topological algebra is exactly an algebra with respect to the $\otimes^{cg}$-tensor product. In particular, if $\calA,\calB$ are ccg-topological algebras then $\calA\otimes^{cg}\calB$ is a cg-topological algebra.
\end{definition}

\begin{remark}\label{oss:algebreregolariemoduli}If $\calM_i$ is a \cgmod module of the \ccg-topological-algebra $\calR_i$ then 
	$$
	\calM_1\tensor{\rightarrow}\cdots\tensor{\rightarrow}\calM_h
	\quad\mand\quad 
	\calM_1\tensor{*}\cdots\tensor{*}\calM_h
	$$
are \cgmod modules of the \ccg-topological algebra $\calR_1\tensor{cg}\cdots\tensor{cg}\calR_h$. 
\end{remark}

\subsubsection{Ordered tensor products}\label{sssec:prodottiordinatigenerale}

In this section we consider the following questions: if $\calA,\calB$ are topological algebras is it true that $\calA\otimesr\calB$ is a topological algebra? If $\calM,\calN$ are modules for $\calA,\calB$ respectively, $\calM\otimesr\calN$ is a $\calA\otimesr\calB$ module? 
The answer is generally no to both question, we will introduce some hypothesis that may look strange to make this machinery work in some cases. The assertion we make here are rather easy to check, let us say that in Section \ref{ssec:prodottiordinatiOeD} we will give another tool, which requires a more involved proof, to deal with these problems with hypothesis that suit our case of interest.
\smallskip

We will consider ordered tensor products $\tensor{\ra}$ and $\tensor{r}$ at the same time. The second one behaves very well: if $\calA$ and $\calB$ are topological algebras, $\calM$ is a topological $\calA$ module and $\calN$ is a topological $\calB$-module, then $\calA\tensor{r}\calB$ is a topological algebra and 
$\calM\tensor{r}\calN$ is a topological $\calA\tensor{r}\calB$ module (indeed $\calA$ and $\calM$ do not even need to have any topological structure here).
We do an example which establishes some limitations 
to what we could expect to be true for the $\ra$-tensor product. 

\begin{remark} Take $S=\Spec \mC$, $\calA = \calM =\mC((t))$,  $\calB = \mC[[s]]$, and $\calN = \mC((s))$. In this case we have 
	$\calA\tensor{\ra}\calB=\mC((t))[[s]]$ and $\calM\tensor{\ra}\calN=\mC((t))((s))$. Then the natural action of $\calA\tensor{\ra}\calB$ on $\calM\tensor{\ra}\calN$ is not continuous. 
\end{remark}

There are two possible escamotages to introduce a natural product on the $\ra$-tensor product of two topological algebras. The first one is based on the fact that in many cases we will be interested in, there are identifications at the level of sheaves (forgetting about the topology) between the $\ra$ tensor product and the $r$-tensor product.

\begin{remark}\label{rmk:otimesrproductalgebras}
Let $\calA$ and $\calB$ two topological algebras and let $\calM$ a $\calA$-topological modules and $\calN$ a $\calB$-topological module. 
Assume that $\calA\tensor{\ra}\calB=\calA\tensor{r}\calB$ as plain sheaves (by Remark \ref{oss:tensoriisomorfi} this condition is satisfied, for example, if $\calB$ is \QCCF).
Then 
\begin{enumerate}[\indent a)]
	\item $\calA\tensor{\ra}\calB$ has a natural structure of algebra 
	\item for all $f\in \calA\tensor{\ra}\calB$ it is defined the multiplication by $f$ from $\calM\tensor{\ra}\calN$ to $\calM\tensor{\ra}\calN$ is continuous,
	\item the action of $\calA\tensor{\ra}\calB$ on $\calM\tensor{\ra}\calN$ described in the previous point makes $\calM\tensor{\ra}\calN$ a $\calA\tensor{\ra}\calB$ module. 
    \item If $\calM\tensor{\ra}\calN=\calM\tensor{r}\calN$ (by Remark \ref{oss:tensoriisomorfi} this condition is satisfied, for example, if $\calN$ is \QCCF) then 
    for all local sections $p\in \calM\tensor{\ra}\calN$ the multiplication with $p$ from $\calA\tensor{\ra}\calB$ to $\calM\tensor{\ra}\calN$ is continuous,
\end{enumerate}
\end{remark}

We will not use this result or the next one. However, in Section \ref{ssec:prodottiordinatiOeD} we will prove a similar 
results with some hypothesis that now could look strange, but that are satisfied in the cases we are interested in. The second escamotage can be applied when $\calB$ is a \ccg-topological sheaf, so that $\calB=\limind \calB_n$ and $\calA\tensor{\ra}\calB\simeq \limind \calA\tensor{\ra}\calB_n$ (and similarly for $\calM$ and $\calN$). These hypothesis are fulfilled, for example, if we are in one of the two situations described in Remark \ref{oss:prodottiregolarifreccetta} or Remark \ref{oss:prodottiecolimiti}. 

\begin{remark}\label{lem:prodottocompattofreccetta}
	Let $\calA$ and $\calB$ two topological algebras, let $\calM$ a $\calA$-topological module and let $\calN$ a $\calB$-topological module. 
	Let $\calK$ be a $w$-compact $\calO_S$-submodule of $\calB$ and
	$\calH$ be a $w$-compact $\calO_S$-submodule of $\calN$. Then there is a natural product
	$$
	(\calA\tensor{*} \calK)\otimes (\calM\tensor{\ra} \calH)\lra \calM\tensor{\ra}\calN
	$$
	which is continuous for the $*$-topology.
\end{remark}

The lemma in particular applies to the $\calM=\calA$ and $\calN=\calB$. The hypothesis of the two statements can be weakened. In section \ref{ssec:prodottiordinatiOeD} we will prove a similar result for \ccg-algebras whose proof is completely analogous. Finally we record a positive result, which follows easily for formal reasons.

\begin{remark}
	Assume that $\calA$ is a $\ra$ topological algebra and $\calB$ is a $w$-compact topological algebra, then $\calA\tensor{\ra}\calB$ is a $\ra$-topological algebra. 
\end{remark}

\subsubsection{Localization of topological algebras} \label{ssec:localization}
Let $\calR$ be a $\calO_S$-topological commutative algebra which is also a \QCC sheaf. Let $\calI$ be a closed ideal of $\calR$ and assume that $\calI$ is locally free of rank one, meaning that for small enough open subsets $U$ there exists a section $\grf\in\calI(U)$ such that the product $r\mapsto r\grf$ is an isomorphism of topological modules. This assumption implies that we have isomorphisms between the tensor product over $\calR$ of $m$ copies of $\calI$ and  $\calI^m\subset \calR$ and that this module is locally isomorphic to $\calR$ as a topological module. 

In this setting define $\calI^{-m}=\calHom_{\calR}(\calI^m,\calR)$ where a \fsonoz is given by the subsheaves $\calHom_{\calR}(\calI^m,\calI^n)$ for positive $n$. We will use the notation $\calR(-m\calI)=\calI^m$. These are all $\calR$-topological modules locally isomorphic to $\calR$. Finally we notice that we have natural closed immersions $\calR(m\calI)\subset \calR(n\calI)$ for $m<n$ and we define 
$$
\calR(\infty \calI)=\limind \calR(n\calI)
$$ 
with the colimit topology.\index{$\calR(\infty\calI)$}\index{$\calR(\infty\calI)$!$\calR(n\calI)$} By Remark \ref{oss:qcclimiti} this is also \QCC sheaf and in particular it is complete. 
This space has a natural structure of algebra and it follows from the local description below that it is a topological algebra. 

If $\calF$ is a linear $\calR$-topological module we define 
$$
\calF(n\calI)=\calF\otimes _\calR \calR(n\calI)
$$
where a \fsonoz is given by the subsheaves $\calU \otimes_\calR \calR(n\calI)$ where $\calU$ is a $\calR$-linear \noz of $\calF$. 
We notice that $\calF(n\calI)$ is locally isomorphic to $\calF$ as a topological module. Finally we define
$$
\calF(\infty\calI)=\limind \calF(n\calI)
$$ 
with the colimit topology.\index{$\calF(\infty\calI)$}\index{$\calF(\infty\calI)$!$\calF(n\calI)$} We notice also that as a $\calR$-module we have that $\calF(\infty\calI)=\calF\otimes_\calR\calR(\infty\calI)$.

The localization $\calF(\infty\calI)$ has the following universal property: if $\grf:\calF\lra \calG$ is a continuous morphism of $\calR$-topological modules and that the multiplication map $\calI\otimes_\calR\calG \to \calG$ is an isomorphism, then there exists a unique continuous extension $\tilde \grf:\calF(\infty\calI)\lra \calG$ of $\grf$.

\begin{remark}[Local description]\label{sssec:localdescription} Assume $S=\Spec A$, let $R=\calR(S)$ and let $\grf\in R$ be a generator of $\calI$ on $S$. 
Then $\grf$ is not a zero divisor in $R$. Then the sections of $\calR(\infty \calI)$ on $S$ are equal to $R_\grf$ and a \fsonoz of $R_\grf$ is given by 
$\sum U_n\grf^{-n}$ where $U_n\subset R$ is a \noz of $R$. Similarly if $F=\calF(S)$ then the section of $\calF(\infty \calI)$ on $S$ are equal to $F_\grf$ and a \fsonoz of $F_\grf$ is given by $\sum V_n\grf^{-n}$ where $V_n\subset F$ is a \noz of $F$. From this description it is easy to check that the multiplication of $R_\grf$ and the action of $R_\grf$ on $F_\grf$ are continuous. 
\end{remark}

\subsection{Pullbacks}\label{ssec:pullbackcompletati}
	
	In this section we define the topology on the pullback of a sheaf and study how completed pullback behaves with respect to limits, colimits and \QCC sheaves.

	\begin{definition}
		Let $\varphi : S' \to S$ be a morphism of schemes and $\calF$ a topological sheaf of $\calO_S$-modules. Then the pullback
		\[
			\varphi^*\calF = \calO_{S'} \otimes_{\varphi^{-1}\calO_S} \varphi^{-1}\calF 
		\]
		has a natural topology, where a \fsonoz is given by the images of the morphisms $\varphi^*\calU \to \varphi^*\calF$, where $\calU$ is a \noz in $\calF$. We define the completed pullback $\hat{\varphi}^*\calF$ as the completion of $\varphi^*\calF$ along this topology.
	\end{definition}

	\begin{remark}
		Let $\calF$ be a topological sheaf and $\calF_i$ a \fsonoz of $\calF$. Since tensor product commutes with colimits and $\hat{\varphi^{-1}}$ is exact, we have
		\[
			\hat{\varphi}^*\calF = \varprojlim_{i} \varphi^*\left( \calF/\calF_i \right).
		\]
		It follows that for any topological sheaf $\calF$ we have $\hat{\varphi}^*\calF = \hat{\varphi}^*\overline{\calF}$.
	\end{remark}

Notice that the operation of pullback is well behaved with respect to restriction to open subsets. Indeed, for open subsets $U' \subset S'$, $U \subset S$ such that $\varphi(U') \subset U$, we have an isomorphism of topological sheaves $(\varphi^*\calF)|_{U'} = (\varphi|_{U'})^*(\calF|_U)$. The same holds for completed pullbacks since completion commutes with restriction.

\begin{remark}[\QCC sheaves and pullback]\label{rmk:qccpullback}
	Let $\varphi : \Spec A' \to \Spec A$ be a morphism of affine schemes. Let $M$ be a complete topological $A$ module, let $\tilde{M}$ be its attached quasi-coherent sheaf and let $\qcc{M}$ be its attached quasi-coherent completed sheaf. Then the topology of $\varphi^*\tilde{M} = \widetilde{M \otimes_A A'}$ agrees with the one of the topological module $M \otimes A'$. It follows that $\hat{\varphi}^*\qcc{M}$ identifies with $\qcc{M \otimes_A A'}$ and by the discussion above it follows that the completed pullback of a \QCC sheaf is a \QCC sheaf.
\end{remark}

\begin{lemma}[Pullback vs limits and colimits]\label{lem:pullbacklimits}
	The following hold
	\begin{itemize}
		\item Let $\calF_i$ be a system of topological sheaves, then $\hat{\varphi}^*(\varinjlim \calF_i)$ is isomorphic to the completion of $\varinjlim \varphi^*\calF_i$ as a topological sheaf;
		\item Given a collection of complete topological sheaves $\calF_i$, then $\hat{\varphi}^*(\oplus_i \calF_i) = \oplus_i \hat{\varphi}^*\calF_i$;
		\item  We have $\hat\grf^*(\calF)\simeq \hat\grf^*(\hat \calF)$ and 
	$\hat\grf^*(\calF/\calG)$ is isomorphic to the completion of $\hat\grf^*(\calF)/\hat\grf^*(\cal G)$.
	\end{itemize}
\end{lemma}

\begin{proposition}\label{prop:pullbacktensorproduct}
	Let $\calF$ and $\calG$ be two topological sheaves. Then 
	$$
	\hat \grf^*\big(\calF\tensor {\ra} \calG \big)\simeq \hat \grf^*\big(\calF\big)\tensor {\ra} \hat \grf^*\big(\calG\big),\quad
	\hat \grf^*\big(\calF\tensor {*} \calG \big) \simeq \hat \grf^*\big(\calF\big)\tensor {*} \hat \grf^*\big(\calG\big),\quad
	\hat \grf^*\big(\calF\tensor {!} \calG \big) \simeq \hat \grf^*\big(\calF\big)\tensor {!} \hat \grf^*\big( \calG\big).
	$$
\end{proposition}


\section{Preliminary remarks and local structure}\label{sec:localstructure}
In this Section we fix some notation and we do some preliminary remarks on our basic geometric objects. All schemes we consider will be over the complex numbers. Let $S$ be a topologically noetherian and quasi separated scheme and let $p:X\lra S$ be a smooth map (in particularly locally of finite presentation) of relative dimension 1. Consider $\grs:S \to X$, a section of $p$\index{$\sigma$}, which is in particular a closed immersion. We denote by $j:X\setminus \grs(S)\lra X$ the open immersion of the complement.  
We denote by $\calI_\grs\subset \calO_X$\index{$\calI_\sigma$} the ideal defining $\grs(S)$ in $X$. Under these assumptions $\grs(S)$ is a Cartier divisor of $X$. More precisely we have the following Lemma \cite[Theorem 2.5.8]{fu2011etale}
\begin{lemma}\label{lem:intornoregolare}
	For all $x\in \grs(S)$ there exists an affine neighborhood $X'=\Spec B$ of $x$, an affine neighborhood $S'=\Spec A$ of 
	$p(x)$ and an \'etale map $f:X\lra \mA^1_{S'}=\Spec A[t]$ such that 
	\begin{enumerate}[\indent i)]
		\item $p(X')=S'$ and $\grs(S')\subset X'$.
		\item $p=\pi_{S'}\circ f$, where $\pi_{S'}:\mA^1_{S'}\lra S'$ is the projection. 
		\item the ideal $I_\grs$ defining $\grs(S')$ in $X'$ is generated by the image of $t$ in $B$ under the map $f^\sharp$.
	\end{enumerate}
\end{lemma}

Notice that, in particular, the last claim says that $\grs$ is a pull back of the zero section $S'\lra\mA^1_{S'}$ through the map $f$. We want also to stress the following consequence of the previous Lemma. 

\begin{corollary}\label{coro:affineminussigma}
	The morphism $j:X \setminus \sigma(S) \to X$ is affine.
\end{corollary}

We will consider also the case of more than one section. Given a finite set $\finitosigma$ we consider a section $\Sigma : S \to X^I$ which we consider as an $I$-tuple of sections $\Sigma = \{ \sigma_i : S \to X \}_{i \in I}$\index{$\Sigma$}. When we will need to emphasize the dependence on the set $I$ we will write $\Sigma = \Sigma_I$\index{$\Sigma_I$}, so that, for instance, given $I' \subset I$ the symbol $\Sigma_{I'}$ denotes the sub-collection of sections of $\Sigma$ indexed by $I'$. 

We denote by $\calI_{\Sigma}=\prod_{i\in \finitosigma} \calI_{\sigma_i}$\index{$\calI_\Sigma$} and by a slight abuse of notation we will denote by $\Sigma$ the closed subscheme of $X$ defined by this ideal and by $j:X\setminus\Sigma \lra X$ the immersion of the complement. By the results above, $j$ is affine and $\Sigma$ is a Cartier divisor. 

\begin{ntz}\label{sigmacalI}
We denote by $\calO_X(\Sigma)$ the invertible sheaf on $X$ associated to $\Sigma$ and for $n\in \mZ$ we denote by  $\calO_X(n\Sigma)$ the sheaf $\calO_X(\Sigma)^{\otimes n}$. In particular $\calO_X(-\Sigma)=\calI_{\Sigma}$. Finally, we denote 
by $\calO_X(\infty \Sigma)=j_*j^*\calO_X$. For all $n$ we have $\calO_X (n\Sigma)\subset \calO_X(\infty\Sigma)$ and 
$\calO_X(\infty\Sigma)=\bigcup_n\calO_X(n\Sigma)$. The sheaves $\calO_X$ and $\calO_X(\infty\Sigma)$ have $\calO_X$-linear topologies defined by the ideals $\calO(-n\Sigma)$ for $n>0$ and we denote by 
$\overline{\calO_X}$ and $\overline{\calO_X}(\infty\Sigma)$ their completions so that 
$$
\barOX=\limpro_n \frac{\calO_X}{\calI_{\Sigma}^n},\qquad 
\barOX(\infty\Sigma)=\barOX\otimes _{\calO_X}\calO_X(\infty\Sigma).
$$
The sheaf $\barOX(\infty\Sigma)$ is exactly the topological sheaf introduced in Section \ref{ssec:localization}.
\end{ntz}

As a consequence of Lemma \ref{lem:intornoregolare} we have the following local description of $\Sigma$. 

\begin{lemma}\label{lem:intornoregolareI1} The following hold.

a) Let $s\in S$, let $\{x_j\}_{j \in J} \in X$ be distinct points and let $\pi : \finitosigma \twoheadrightarrow J$ a surjective map of finite sets such that  $\grs_i(s)=x_{\pi(i)}$. Then there exists an affine neighborhood $S'$ of $s$ and an open subset $X'$ of $X$ such that $p(X')\subset S'$, $\grs_i(S')\subset X'$ and for all $i_1,i_2 \in I$ with $\pi(i_1) \neq \pi(i_2)$, we have $\grs_{i_1}(S')\cap \grs_{i_2}(S') =\vuoto$. 

 b) Let $s\in S$, $x\in X$ and assume that for all $i\in \finitosigma$ we have $\grs_i(s)=x$. Then there exists 
an affine neighborhood $S'$ of $s$, an affine neighborhood $X'$ of $x$ and an \'etale map $f:X'\lra \mA^1_{S'}$ such that 
	\begin{enumerate}[\indent\indent  i)]
	\item $p(X')=S'$ and $\grs_i(S')\subset X'$ for all $i\in \finitosigma$.
	\item $p=\pi_{S'}\circ f$, where $\pi_{S'}:\mA^1_{S'}\lra S'$ is the projection. 
    \item for all $i$, the following diagram, 
    where $\tau_i=f \circ \grs_i|_{S'}$, is a pull back 
    $$ \xymatrix{ S'\ar[r]^{\grs_i} \ar@{=}[d] & X'\ar[d]^{f} \\ S'\ar[r]^{\tau_i}& \mA^1_{S'}} $$ 
\end{enumerate}
\end{lemma}

\noindent In what follows we will need the following Lemma as well. 

\begin{lemma}\label{lem:intornoregolareI2} 
Assume $X=\Spec B$, $Y=\Spec C$ and $S=\Spec A$ are affine and that 
$p:X\lra S$, $q:Y\lra S$ are smooth maps of relative dimension $1$. Let  $f:X\lra Y$ be an \'etale map over $S$. 
For $i\in \finitosigma$ let $\grs_i:S\lra X$ be a section of $p$ and set  $\tau_i=f\circ \grs_i$. Assume that $\grs_i$ is a pull back of $\tau_i$  through $f$ and that $\tau_i(S)$ is defined by a function $\grf_i$ which is  not a zero divisor in $C$. Set $\psi_i=f^\sharp(\grf_i)$ for all $i$.

Then for all $k\geq 0$ and all $k$-tuple of distinct indices $i_1,\dots,i_k\in \finitosigma$, the map $f^\sharp$ induces an isomorphism 
$$ \frac{C}{( \varphi_{i_1} \cdots \varphi_{i_{k}} )}\simeq 
\frac{B}{( \psi_{i_1} \cdots \psi_{i_{k}} )}. $$ 
and these are $A$ free modules of rank $k$. 
\end{lemma}	
\begin{proof} 
We prove the statement by induction on $k$. For $k=1$ it follows from the  assumptions about the pull back of $\tau_i$  along $f$. In this case the first quotient is the coordinate ring of $\tau_i(S)$ and in particular is isomorphic to $A$. 

Let now $k>1$. By assumptions   $\varphi_i$ are not zero divisors in $C$ and, by flatness of $f$, the elements $\psi_i$ are non zero divisor in $B$. Hence we have the following commutative diagram with exact rows
\[ \xymatrix{ 
	0 \ar[r] &  
	\frac{\varphi_{i_k}C}{(\varphi_{i_1}\cdots \varphi_{i_{k}} )} \ar[r]\ar[d] & 
 	\frac{C}{( \varphi_{i_1} \cdots \varphi_{i_{k}} )}      \ar[r]\ar[d] & 
	\frac{C}{( \varphi_{i_{k}} )}                     \ar[r]\ar[d] & 
	0 \\ 
0 \ar[r] & \frac{\psi_{i_k}B}{(\psi_{i_1}\cdots \psi_{i_{k}} )} \ar[r] & 
\frac{B}{( \psi_{i_1} \cdots \psi_{i_{k}} )}      \ar[r] & 
\frac{B}{( \psi_{i_{k}} )}    \ar[r] & 0} \] 
The map on the right column is an isomorphism, by what noticed in the case $k=1$. Modules on the left are isomorphic to 
$\frac{C}{( \varphi_{i_1} \cdots \varphi_{i_{k-1}} )}$, and $\frac{B}{( \psi_{i_1} \cdots \psi_{i_{k-1}} )}$, hence the map on the left column is an isomorphism by induction. The thesis follows. 
\end{proof}

Notice that the assumptions that $\varphi_i$ is not a zero divisor, is always satisfied in case $Y=\mA^1_S$, since $\varphi_i=t-\grs_i^\sharp(t)$. In particular it is always satisfied if we take a neighborhood which is small enough as in Lemma \ref{lem:intornoregolareI1}. This remarks and the Lemma above imply the following result.

\begin{lemma}\label{lem:mappeetale1}
Let $p:X\lra S$ and $q:Y\lra S$ be smooth schemes of relative dimension $1$ over $S$ and let $\Sigma = \{ \grs_i \}_{i \in I}$ be a collection of sections of $p$. Let $f:X\lra Y$ be an \'etale map over $S$ and set $\tau_i=f\circ\grs_i$, $\calT = \{ \tau_i \}_{i \in I}$. Assume that $\grs_i$ is the pull back of $\tau_i$ through $f$ for all $i$. 
Then 

\begin{enumerate}[\indent a)]
	\item the pull back of functions induces isomorphisms $\displaystyle f^*(\calO_Y(n \calT))\simeq \calO_X(n \Sigma)$ which, taking the colimit along $n$ induce $\displaystyle {f^*(\calO_Y(\infty \calT))\simeq \calO_X(\infty \Sigma)}$;
    \item $f$ induces an isomorphism $\displaystyle \Sigma \simeq \calT$;
    \item if $\calF$ is a quasi-coherent sheaf supported on $\calT$ (i.e. $\calF|_{X\setminus\calT = 0}$), then $p_*f^*(\calF)\simeq q_*(\calF)$;
    \item $\displaystyle f_*(\calO_X/\calI_{\Sigma}^n)\simeq \calO_Y/\calI_{\calT}^n$ for all $n\geq 0$;
    \item $\displaystyle f_*(\calO_X(\infty \Sigma)/\calO_X)\simeq \calO_Y(\infty \calT)/\calO_Y$;
    \item $\displaystyle p_*(\calO_X(n \Sigma)/\calO_X))\simeq q_*(\calO_Y(n \calT)/\calO_Y)$ and $\displaystyle p_*(\calO_X(\infty \Sigma)/\calO_X))\simeq q_*(\calO_Y(\infty \calT)/\calO_Y)$ and they are $\calO_S$-locally free module. 
\end{enumerate}
\end{lemma}
\begin{proof}
We prove the first claim in point a) and we sketch the proof of the other claims. For $n=0$ the claim is trivial. Since $f^*$ commutes with tensor products it is enough to prove the case $n=-1$, that is $f^*(\calI_\calT)\simeq \calI_\Sigma$. Since $f$ is flat we have that $f^*(\calI_\calT)$ is a subsheaf of $f^*(\calO_Y)\simeq\calO_X$. Hence we need to prove that the image of $f^*(\calI_\calT)$ under the isomorphism $f^\sharp: f^*(\calO_Y)\lra \calO_\Sigma$
is equal to $\calI_\Sigma$. We can check this claim locally. Hence we can assume that all schemes are affine and that $\calI_{\tau_i}$ is generated by a element $\grf_i$. The condition that $\grs_i$ is the pull back of $\tau_i$ and the flatness of $f$ implies that $\calI_{\sigma_i}$ is generated by 
$\psi_i=f^\sharp(\grf_i)$. By  definition $\calI_\calT$ is generated by the products of the elements $\psi_i$ and $\calI_\Sigma$ is generated by the products of the elements $\grf_i$. This implies the second claim in point a). 
Notice that we can further assume that $\grf_i$ are non zero divisors, hence point b) and d) and the first claim in point f) follows from \ref{lem:intornoregolareI2}. For point c), notice that in the case $\calF$ is annihilated by $\calI_{\calT}$ this  follows from b); and that the case where it is annihilated by $\calI_{\calT}^n$ follows by induction on $n$. The general case follows by filtering the quasi-coherent sheaves by subsheaves annihilated by powers of $\calI_{\calT}$. 
The claim about the sheaf $\calO_X(\infty\Sigma)$ can be analyzed similarly.
\end{proof}

The sheaves $\barOX$ and $\barOX(\infty\Sigma)$  are supported on $\Sigma$. It will be convenient for us to work on them as sheaves over $S$. Since $\Sigma$ is finite over $S$ this will not be a big difference. 

\begin{definition}\label{def:pbarOpbarOpoli}
	We introduce the following notations:
\[
\pbarO = p_* \barOX, \qquad \pbarO(n) = p_* \big(\barOX(n\Sigma)\big), \qquad \pbarOpoli = p_* \big(\barOX(\infty\Sigma)\big). 
\]
\index{$\pbarO$}\index{$\pbarO(n)$}\index{$\pbarOpoli$}When we want to make explicit the dependence of these sheaves on the set of sections we will write $\pbarO = \pbarOK{I}$. Sheaves on $X$ will not appear anymore, in particular we will adopt the notation $\calI_{\Sigma} =\pbarO(-1)$.

There are obvious inclusions among these sheaves and  $\pbarOpoli=\bigcup \pbarO(n) = \pbarO(\infty\Sigma)$. Moreover they are $\calO_S$-topological algebras where a fundamental system of neighborhoods or zero is given by $\pbarO(-n)$.
\end{definition}

\subsection{Local description and bases}\label{ssec:descrizionelocale1} The lemmas above imply the following local description of these sheaves.
Let $s \in S$ and let $x_j$ for $j\in J$ be the distinct images of the point $s$ under the sections $\grs_i$. Let $\pi:\finitosigma\twoheadrightarrow J$ the surjection defined by $\grs_i(s)=x_{\pi(i)}$. 
As a consequence of Lemma \ref{lem:intornoregolareI1} point a) we have that there exists an affine neighborhood $S'$ of $s$ such that on $S'$ we have 
\begin{equation}\label{eq:Oprodotto1}
\pbarOK{I}\simeq \prod_{j\in J} \pbarOK{\finitosigma_j}
\end{equation}
where $\finitosigma_j=\pi^{-1}(j)$. The sheaves $\pbarOK{\finitosigma}(n)$ and $\pbarOpoliK{\finitosigma}$ decompose similarly. 

By point b) of Lemma \ref{lem:intornoregolareI1} there exist an affine neighborhood $S''=\Spec A\subset S'$ of $s_0$ a neighborhood $X'_j$ of $x_j$ such that $p(X_j')\subset S''$ and \'etale maps $f_j:X_j'\lra \mA^1_{S''}=\Spec A[t]$ such that $\grs_i$ is the pull back of a section $\tau_i:S''\lra \mA^1_{S''}$  for $i\in \finitosigma_j$. 
We call such such an open affine $S''$ a \emph{well covered neighborhood} and $t$ a \emph{local coordinate}. Using Lemmas 
\ref{lem:intornoregolareI2} and \ref{lem:mappeetale1} over a well covered neighborhood we have 
\begin{equation*}\label{eq:Oprodotto2}
\pbarOK{\finitosigma_j} \simeq \calO_{\overline{\calT}_{\finitosigma_j}}
\end{equation*}
for all $j\in J$ and similarly for the sheaves $\pbarOK{\finitosigma}(n)$ and $\pbarOpoliK{\finitosigma}$. These reduces the local description of these sheaves to the case where $p:\mA^1_S\lra S$ is the standard projection and $S$ affine. In particular if
$\tau_i^\sharp (t)=a_{ji}$ and \index{$\varphi$}
$$\varphi_j=\prod_{i\in \finitosigma_j} (t-a_{ji})
$$
then 
$$
\pbarOK{\finitosigma_j}\simeq  \varprojlim_n \frac{\calO_{S}[t]}{\left(\varphi_j^n\right)} 
\quad \mand \quad
\pbarOpoliK{\finitosigma_j}\simeq\pbarOK{\finitosigma_j}[\varphi_j^{-1}].
$$

\begin{remark}\label{oss:unsolof}
Above we have described $\pbarO$ as product of rings of the form $\limpro \calO[t]/(\varphi_j^n)$. We notice, to simplify some notations in the sequel, that up to restrict further the neighborhood of $s_0$ we can describe $\pbarO$ as a unique projective limit of the same form. Indeed let $\pbarO=\prod \limpro \calO[t]/(\varphi^n_j)$ as in the discussion above, with $\varphi_j=\prod (t-a_{ji})$. We can assume, up to a translation along $\mA^1$ that $a_{ji}(s_0)=0$ in the residue field of $s_0$. Choose distinct complex numbers $\grl_j$ and define $g_j(t)=\varphi_j(t+\grl_j)$. Then the projective limits $\limpro A[t]/{\varphi_j^n}$
and  $\limpro A[t]/{g_j^n}$ are isomorphic. Moreover, by construction, over the residue field of $s_0$ the polynomials $g_j$ and $g_h$ are coprime for $j\neq h$. Hence they are coprime over the local ring of $s_0$ and up restrict $S$ we can assume they are coprime in $A[t]$. Hence, if we set $g=\prod g_j= \prod (t-a_{ji}-\grl_j)$, we have
$$
\pbarO\simeq \limpro \frac{A[t]}{(g^n)}.
$$ 
\end{remark}

\subsubsection{Bases}\label{sssec:basespbarO}
The description above implies that the sheaves $\pbarO$ and $\pbarOpoli$  are locally \QCCF (quasi coherent completed, topologically free, c.f. Definition \ref{def:qccsheaves}) sheaves. Indeed we can easily describe a complement of a \fsonoz (fundamental system of neighborhoods of zero, c.f. Section \ref{app:top}) of $\pbarO$ or 
$\pbarOpoli$. Assume $S=\Spec A$ is a well covered neighborhood and let $a_i$, $\varphi_j$ be as as before. Then $A[t]/\varphi_j$ is free over $A$ and let $u_{1j},\dots,u_{n_jj}$ be a basis. Then the set
$$
\varphi_j^hu_{\ell j}\; \text{ such that }\; j\in J, \quad \ell=1,\dots,n_j,\quad
h<n
$$
is a basis of $\pbarOpoli/\pbarO(n)$. If we restrict the possible values of $h$ to $0\leq h<n$ then we get a basis of  $\pbarO/\pbarO(n)$. It is also clear that the sheaves $\pbarO(n)$ are $w$-compact.

In the setting of Remark \ref{oss:unsolof} above, we introduce the following bases. We will eventually use them to describe the local behaviour of some of our constructions. Let 
$\varphi=\prod_{i=1}^n(t-a_i)$ and define
\begin{equation}\label{eq:fni}
\varphi^+_{m,i}(t)=\varphi(t)^m \cdot (t-a_1)\cdots (t-a_{i})\quad\mand\quad
\varphi^-_{m,i}(t)=\varphi(t)^m \cdot (t-a_{n-i+1}) \cdots (t-a_n)
\end{equation}
for $m\in \mZ$ and $i=0,\dots,n-1$, by that we mean $\varphi^+_{m,0} =\varphi(t)^m,\varphi^-_{m,0} = \varphi(t)^m$. Then the elements $\varphi^+_{m,i}(t)$ for $m<k$ (as well as $\varphi^-_{m,i}(t)$) constitute a basis of $\pbarOpoli/\pbarO(k)$ which is compatible with the subsheaves $\pbarO(h)/\pbarO(k)$. \index{$\varphi^+_{m,i}$}\index{$\varphi^-_{m,i}$} 

\medskip

Finally, notice that in the case of a single section $\grs$ we have 
	\[
	\frac{\calO_S[z]}{z^n} \isocan  p_*\left(\frac{\calO_X}{\calI_\grs^n}\right)
	\]
	where $z$ is an indeterminate and the isomorphism is induced by $p^\sharp$ on $\calO_S$ and sends $z$ to $t$. As a consequence $p_*(\overline{\calO_X})\simeq \calO_S[[z]]$ and $p_*(\overline{\calO_X}(\infty\grs))\simeq \calO_S((z))$.

\subsection{The case of a general sheaf}\label{ssec:definizionefascibarF} \index{$\calF_{\overline{\Sigma}}$}\index{$\calF_{\overline{\Sigma}}(n)$}\index{$\calF_{\overline{\Sigma}^*}$}The discussion above easily generalizes to the following setting. 
Let $\calF$ be a quasi coherent $\calO_X$-sheaf on $X$ and set
$\calF(\infty\Sigma)=\calF \otimes _{\calO_X}\calO_X(\infty\Sigma)$. On  $\calF$ consider the discrete topology and define  
$\overline\calF=\calF \stackrel{\rightarrow}{\otimes}_{\calO_X}\barOX$,
$\overline\calF(n\Sigma)=\calF \stackrel{\rightarrow}{\otimes}_{\calO_X}\barOX(n\Sigma)$
and $\overline\calF(\infty\Sigma)=\calF \stackrel{\rightarrow}{\otimes}_{\calO_X}\barOX(\infty\Sigma)$. These are complete sheaves supported on $\Sigma$. Define also
$$
\pbarF{\calF}=p_*(\overline \calF)\qquad 
\calF_{\bar\Sigma}(n)=p_*\left(\overline \calF(n\Sigma)\right)
\qquad
\pbarFpoli{\calF}=p_*(\overline \calF(\infty\Sigma)) 
$$
We write $\pbarFI{\calF}{I}$, etc, if we want to emphasize the set of sections $\Sigma_I$. These sheaves are complete and they are respectively topological $\pbarO$ and $\pbarOpoli$-topological modules. The description given in formula \eqref{eq:Oprodotto1}, generalizes to this setting. Assume $S$ is well covered. Then, using the notation of Section \ref{ssec:descrizionelocale1} we will have 
$$
\pbarFI {\calF}{\finitosigma} =\bigoplus _{j\in J}\pbarFI {\calF}{\finitosigma_j}
$$
and similarly for the the other sheaves introduced above.

\subsection{Pullbacks}\label{ssec:pullback1}

We state the behaviour of the sheaves introduced so far under pullback on the base.

Let $p : X\to S$ be a smooth family of curves as always, $S'$ be a topologically noetherian and quasi separated scheme and $f : S' \to S$ be a morphism and $p' : X'\to S'$ be the pullback of $X$ along $f$, and let  $f_X : X' \to X$ be the projection map. Let $\Sigma = \{\sigma_i\}_{i \in I}$ be a collection of sections of $p$ and $\Sigma'$ the pullback of $\Sigma$. With an abuse of notation we denote by $\Sigma$ also the closed subscheme of $X$ given by the union of the sections of $\sigma_i$ and analogously we denote by $\Sigma'$ the closed subscheme of $X'$ given by the union of the sections of $\sigma'_i$. Recall that $\Sigma/S$ is affine and that $\Sigma'$ is the pullback of $\Sigma$ along $\varphi$.

\begin{remark}\label{rmk:basechangecoerente}
Exactly as in the proof of Lemma \ref{lem:mappeetale1}, point a), we have $f_X^*\calO_{X/S}(n\Sigma)\simeq \calO_{X'/S'}(n\Sigma')$ for all $n$ (included $n=\infty$). More in general since $f_X^*$ preserves tensor product we have $f_X^*(\calF(n\Sigma))\simeq (f_X^*\calF)(n\Sigma)$. We notice also that if $\calF$ is a $\calO_X$ quasi coherent sheaf supported on $\Sigma$ then 
$ f^*p_*\calF \simeq p'_*f_X^*\calF$. Indeed, by adjunction we have a map from the left hand side to the right hand side, we can check it is an isomorphism locally. Since $\calF$ is supported on $\Sigma$ we can assume $S$ is well covered and that $X$ is affine. Then the claim follows by base change for affine maps. 
\end{remark}

\begin{remark}\label{rmk:pullbackcompletion}
	We notice that $\hat{f}^*\pbarO=\calO_{\overline{\Sigma}'}$ and 
$\hat{f}^*\pbarOpoli=\calO_{\overline{\Sigma'}^*}$. More in general, with the notation above if $\calF$ is a sheaf of quasi-coherent $\calO_X$ modules, then for every positive integer $n$ we have $$(f_X)^*\left(\frac{\calF(\infty\Sigma)}{\calF(-n\Sigma)} \right) \simeq \frac{((f_X)^*\calF)(\infty\Sigma')}{((f_X)^*\calF)(-n\Sigma')}$$
Hence $\hat{f}^* \calF_{\overline{\Sigma} }=(f^* \calF)_{\overline{ \Sigma}' }$. This remark will apply to many canonical sheaves (for example one forms), in the sequel. 
\end{remark}

\section{Residues}\label{sec:residuo}
This Section is devoted to defining the residue morphism in our $X,S,\Sigma$ geometric setting. The definition we give is guided by the fact that we want our residue morphism to satisfy certain \emph{factorization properties}. So, let $I$ be a finite set and let $ \Sigma = ( \grs_i:S\lra X )_{i \in I}$ be a collection of sections of $p$ as above. We are going to define an $\calO_S$-linear morphism 
$$
\Res_{\Sigma}:p_*\left(\frac{\Omega^1_{X/S}(\infty\Sigma)}{\Omega^1_{X/S}}\right)\lra \calO_S
$$
\index{$\Res_{\Sigma}$}\noindent with the following additional properties:
\begin{enumerate}[\indent Res1:]
	\item it is invariant under pull back;
	\item it is factorizable;
	\item it behaves well under \'etale maps;
	\item In the case where $S=\Spec \mC$, $X=\Spec \mC[t]$ it is the usual residue map.
\end{enumerate}
Let us start by making explicit the meaning of points a), b) and c).

\subsection{Invariance under pull backs} \label{ssec:PBresiduo} Consider a pullback diagram 
\begin{equation}\label{diag:PB1}
\xymatrix{
X'\ar[r]^{f_X}\ar[d]_{p'} & X \ar[d]^p\\
S'\ar[r]^{f}              & S
}
\end{equation}
Let $\grs'_i$ be the section of $p'$ obtained by pullback from $\sigma_i$. With these notation we have the following Lemma.

\begin{lemma}\label{lem:residuoPB}With the notation above we have the following natural isomorphisms:
	\begin{enumerate}[\indent a)]
\item $\displaystyle {(f_X)^* \left(\frac{\Omega^1_{X/S}(\infty\Sigma)}{\Omega^1_{X/S}}\right)\simeq
\frac{\Omega^1_{X'/S'}(\infty\Sigma')}{\Omega^1_{X'/S'}}}$
\item $\displaystyle{p'_*\left(\frac{\Omega^1_{X'/S'}(\infty\Sigma')}{\Omega^1_{X'/S'}}\right)\simeq f^*p_*\left(\frac{\Omega^1_{X/S}(\infty\Sigma)}{\Omega^1_{X/S}}\right)}$
\end{enumerate}
\end{lemma}
\begin{proof}
	Both claims follow from $f^*\Omega^1_{X/S} \simeq \Omega^1_{X'/S'} $ the discussion in Remark \ref{rmk:basechangecoerente}
\end{proof}

To be invariant under pull back means that $f^*\Res_{\Sigma}=\Res_{\Sigma'}$ using the isomorphism given by point b) of the above Lemma.

\subsection{Factorizability}

\begin{lemma}\label{lem:residuofattorizzazione}

a) Let $I=I_1\sqcup I_2$ and assume that $\grs_{i_1}(S)\cap	\grs_{i_2}(S)=\vuoto$ for all  $i_1\in I_1$ and $i_2\in I_2$. Then 
$$
\frac{\Omega^1_{X/S}(\infty\Sigma_I)}{\Omega^1_{X/S}}=
\frac{\Omega^1_{X/S}(\infty\Sigma_{I_1})}{\Omega^1_{X/S}}\oplus
\frac{\Omega^1_{X/S}(\infty\Sigma_{I_2})}{\Omega^1_{X/S}}
$$

\indent b) Let $\pi:I\twoheadrightarrow J$ be a surjective map of finite sets and assume that if $\pi(i_1)=\pi(i_2)$ then $\grs_{i_1}=\grs_{i_2}$ let $\Sigma_J$ be the induced collections indexed by $J$. Then 
$$\Omega^1_{X/S}(\infty\Sigma_J) = \Omega^1_{X/S}(\infty\Sigma_I).$$

\end{lemma}

\begin{proof}
The sheaf ${\Omega^1_{X/S}(\infty\Sigma_I)}/{\Omega^1_{X/S}}$ is supported on $\Sigma_I=\Sigma_{I_1}\sqcup\Sigma_{I_2}$, hence it is the sum of its restriction to $\Sigma_{I_1}$, which is equal to ${\Omega^1_{X/S}(\infty\Sigma_{I_1})}/{\Omega^1_{X/S}}$ and to its restriction 
to $\Sigma_{I_2}$ which is equal to ${\Omega^1_{X/S}(\infty\Sigma_{I_1})}/{\Omega^1_{X/S}}$. The second claim is trivial
\end{proof}

Then to be factorizable means than in the hypothesis of point a) of the Lemma above we have
\begin{equation}\label{eq:residuofattorizzabilitaA}
\Res_{\Sigma_I}=\Res_{\Sigma_{I_1}}+\Res_{\Sigma_{I_2}}.
\end{equation}
and that under the hypothesis of point b) of the Lemma above we have
\begin{equation}\label{eq:residuofattorizzabilitaB}
\Res_{\Sigma_I}=\Res_{\Sigma_J}.
\end{equation}

This last condition seems trivial, but together with compatibility under pullback it becomes the kind of factorization property we are looking for.

\subsection{Well behaviour under \'etale maps}\label{ssec:etaleresiduo}
Let $p:X\lra S$ and $q:Y\lra S$ be smooth schemes of relative dimension $1$ over $S$ and let $\Sigma = \{ \grs_i \}$ be a collection of sections of $p$. Let $f:X\lra Y$ be an \'etale map over $S$ and set $\tau_i=\grf\circ\grs_i$. Assume that $\grs_i$ is a pull back of $\tau_i$ through $\grf$ for all $i$. Then, by Lemma \ref{lem:mappeetale1}, we have a natural isomorphism
$$
p_{*}\left(\frac{\Omega^1_{X/S}(\infty \Sigma)}{\Omega^1_{X/S}}\right)\simeq
q_{*}\left(\frac{\Omega^1_{Y/S}(\infty \calT)}{\Omega^1_{Y/S}}\right).
$$
To behave well under étale maps means that $ \Res_{\Sigma}=\Res_{\calT} $ under the above isomorphism. 

\smallskip

\noindent We can now state the main Theorem of this section. 

\begin{theorem}\label{teo:residuo}
	The exists a unique collection of maps $\Res_{\Sigma}$ which satisfy  hypothesis $\Res1$, $\Res2$, $\Res3$ and $\Res4$.
\end{theorem}

\noindent We will prove this Theorem by analyzing different cases. 

\subsection{The case of one section}\label{sssec:residuounasezione} We first construct the residue map in the case of a single section $\grs$. Notice that the sheaf $\calF=\Omega_{X/S}^1(\infty\sigma)/\Omega_{X/S}^1$ is supported over 
$\grs(S)$, this implies that $p_*(\calF)\simeq \grs^{-1}(\calF)$. Hence, in this case the existence and uniqueness of a map $\Res_\grs$ satisfying the properties Res1, Res 3 and Res4 (Res2 is empty for one single section) is equivalent to the existence and uniqueness of a map
$$
\tRes_{\grs}:\frac{\Omega_{X/S}^1(\infty\sigma)}{\Omega_{X/S}^1}\lra \grs_*\calO_S
$$
satisfying
\begin{itemize}
	\item $\tRes$ 1: if $p,p',f,f_X,\grs'$ are as in Section \ref{ssec:PBresiduo} then 
	$\tRes_{\grs'}=f^*\Res$. (by Lemma \ref{lem:residuoPB} (a) and the fact that $f^*\grs_*\calO_S\simeq\grs'_*\calO_S'$ this claim makes sense)
	\item $\tRes$ 3: if $p,q,f,\tau$ are as in Section \ref{ssec:etaleresiduo} then $\tRes_\tau=f_*\tRes_{\grs}$.
	\item $\tRes$ 4: if $S=\Spec \mC$, $X=\Spec \mC[t]$ is the usual residue map.
\end{itemize}

The uniqueness of such a map is clear: indeed the properties about pull back implies that it is enough to check uniqueness locally. 
So we can assume to be in the situation of Lemma \ref{lem:intornoregolare}, and by property $\tRes$3 it is enough to check uniqueness 
in the case of $\mA^1_S$ and $\grs$ is the zero section. However this is determined by the last property and the property about pull back. 

This remark gives also a local description of the map $\tRes$, which is the one expected, once a local coordinate is chosen as in Lemma \ref{lem:intornoregolare}
then $\tRes(gdt)$ with $g\in \calO_X(\infty\grs)/\calO_X$ is equal to the coefficient of $t^{-1}$. However we need to check that this construction does not depend on the choice of the local coordinate $t$. To check  this claim we give a different construction.

Let $\calD_{X/S}$ be the sheaf of relative differential operators, the sheaf generated by $\calO_X$ and the vertical vector fields $T_{X/S}$. Recall
that the zero de Rham cohomology of a right $\calD_{X/S}$-module $\calF$ is defined as $h^0(\calF)=\calF/(\calF\cdot T_{X/S})$\index{$h^0$}. 
Finally we denote by $\grs_!$ the $\calD_{X/S}$-module pushforward. 

\begin{lemma}\label{lem:sigmaD}
Let $X,S,p,\grs$ be as above. Then 
\begin{enumerate}[\indent a)]
   \item $\displaystyle \frac{\Omega^1_{X/S}(\grs)}{\Omega^1_{X/S}}\simeq \grs_*\calO_S$.
   \item The right action of $\calD_{X/S}$ on $\Omega^1_{X/S}$ induces a right action on $\Omega^1_{X/S}(\infty \grs)$ and we have an isomorphism
   $\displaystyle \frac{\Omega^1_{X/S}(\infty \grs)}{\Omega^1_{X/S}}\simeq \grs_!\calO_S$.
   \item We have an isomorphism of $\calO_X$-module $h^0(\grs_!\calO_S)\simeq\grs_*\calO_S$.
\end{enumerate}
\end{lemma}
\begin{proof}
	a) Recall that $\Omega^1_{X/S}\simeq \calI_\grs/\calI_\grs^2$, hence we can define a map 
	$$
	\Omega^1_{X/S}(\grs)\simeq \frac{\calI_\grs}{\calI_\grs^2}\otimes_{\calO_X}\calHom_{\calO_X}(\calI_\grs,\calO_X)\lra \grs_*\calO_S
	$$
	by $[i]\otimes g \longmapsto p^\sharp(g(i))$. It is easy to see that this map is well defined and induces the required isomorphism. 
	
	b) We have a left action of $\calD_{X/S}$ on $\calO_{X}(\infty \grs)$ which induces a right action on $\Omega^1_{X/S}(\infty \grs)$. The map in point a)  induces an homomorphism 
	of right $\calD_{X/S}$-modules from the pushforward $\grs_!\calO_S$ to $\Omega^1_{X/S}(\infty \grs)/\Omega^1_{X/S}$. It is easy to check locally that this is an isomorphism. 
	
	c) We have a morphism $\grs_*\calO_S$ to $h^0(\grs_!\calO_S)$ given by the natural immersion composed with the projection. Locally it is easy to check that it is an isomorphism.  
\end{proof}

The Lemma above implies that we have a natural map $$\Omega^1_{X/S}(\infty\grs)/\Omega^1_{X/S}\rightarrow \grs_*\calO_S,$$ it is easy to check that it locally coincides with the previous construction. Hence $\tRes$ is well defined and has the required properties. 

\subsection{The case of \texorpdfstring{$S=C^n$}{S=Cn}, \texorpdfstring{$X=C^{n+1}$}{X=Cn+1}}\label{sssec:SCn} In this Section we consider a smooth curve $C$ over a base scheme $T$ and focus on the case where $S=C^n$, $X=C\times C^{n}$, (all products are over $T$), $p$ is the projection on the last $n$ coordinates and the sections are given by $\sigma^{\univ}_i(x_1,\dots,x_n)=(x_i,x_1,\dots,x_n)$, we write $\Sigma^{\univ}_{C,I} = \Sigma^{\univ}_{C} = \{\sigma^{\univ}_i\}_{i \in I}$\index{$\sigma^{\univ}_i$}\index{$\Sigma^{\univ}_{C,I}$}. Set
$S'=C^n\senza\Delta$, where $\Delta$ is the big diagonal (the $n$-tuples of points of $C$ which are all distinct), $X'=p^{-1}(S')$, $\sigma_i'$ and $p$ are respectively the restriction of $\sigma_i$ and $p$ to $S'$ and $X'$ and $j_S$ is the open immersion of $S'$ in $S$. 

The map $\Res_{\Sigma'}$ is uniquely determined by the factorization property of formula \eqref{eq:residuofattorizzabilitaA} and the definition of $\Res_{\sigma'_i}$ in the case of one single section. Hence 
\begin{equation}\label{eq:resdistinti}
\Res_{\Sigma'}=\sum_{i\in I}\Res_{\sigma_i'}.
\end{equation}
This formula determines also $\Res_{\Sigma}$. Indeed, by property $\Res$1, we must have that 
$\Res_{\Sigma'}$ is the localization of $\Res_{\Sigma}$. So, to prove that $\Res_{\Sigma}$ is well defined, we need to prove that 
\begin{equation}\label{eq:res1}
\Res_{\Sigma'}\left( p_*\left(\frac{\Omega^1_{X/S}(\infty \Sigma)}{\Omega^1_{X/S}}\right) \right)\subset \calO_S \subset j_{S,*}j_S^*\calO_S
\end{equation}
where we consider $$p_*(\Omega^1_{X/S}(\infty \Sigma)/\Omega^1_{X/S}) \subset p'_*(\Omega^1_{X'/S'}(\infty \Sigma') / \Omega^1_{X'/S'})$$ 
and $\calO_S \subset \calO_{S'}$. Moreover we check that if $\pi:I\twoheadrightarrow J$ is a surjective map and $\Delta_\pi:X^J\subset X^I$ is the associated diagonal then 
\begin{equation}\label{eq:res2}
\Delta_\pi^*\Res_{\Sigma_I}=\Res_{\Sigma_J}
\end{equation}
as prescribed by property $\Res$1 (pullback properties) and formula \eqref{eq:residuofattorizzabilitaB}.

Since uniqueness has been established we can check these properties locally and by Lemma 
\ref{lem:mappeetale1} we reduce this verification to the case $C=\mA^1_T$. 
Moreover the case of $\mA^1_T$ for a general $T$ is deduced from the case of  the complex affine line $C=\mA^1$. Hence $S=\Spec \mC[x_1,\dots, x_n]$, $X=\Spec \mC[x_1,\dots, x_n,t]$ and $S'=\Spec\mC[x_1,\dots,x_n]_\grd$ with 
$\grd=\prod_{i<j}(x_i-x_j)$. In this case we can compute the sum of residues of formula \eqref{eq:resdistinti} as a single residue at infinity. More explicitly
\begin{equation*}
\Res_{s'_I}g(x,t)dt=-\Res_{0}g(x,\tau^{-1})\tau^{-2}d\tau
\end{equation*}
where $\Res_0$ is the coefficient of $\tau^{-1}$. Using this formula the properties expressed by formula \eqref{eq:res1} and \eqref{eq:res2} are clear.

\subsection{The general case}Consider the diagram 
\begin{equation*}\label{diag:PB2}
\xymatrix{
	X\ar[r]^{F\quad}\ar[d]_{p}      & X\times _S X^I \ar[d]^{\pi}\\
	S\ar[r]^{\Sigma}              & X^I
}
\end{equation*}
where $F=(\id,\Sigma)$. The  right column is the case analyzed in the previous section with $T=S$ and $C=X$ and $\grs_i$ are the pull back of the sections $\grs^{univ}_i$. Hence $\Res_{\Sigma}$ must be  given by   
 $$\Res_{\Sigma}=\Sigma^*(\Res_{\Sigma^{\univ}_X}),$$
where $\Sigma^{\univ}_X$ is the universal $I$-collection of sections $X^I \to X\times X^I $ introduced in \ref{sssec:SCn}.
The compatibility of this definition with the previous one is a consequence of formula \eqref{eq:res2} and it easy to check that the map defined in this way has properties $\Res1$, $\Res2$, $\Res3$ and $\Res$4. 

This completes the construction of the residue map. Later on we will need the following property.

\begin{lemma}
The residue map $\Res_{\Sigma}$ is surjective. 
\end{lemma}
\begin{proof}By the second property of factorizability it is enough to prove the statement in the case of one section. The case of one section was discussed in Section \ref{sssec:residuounasezione} and the residue map is clearly surjective. 	
\end{proof}

\subsection{The case of \texorpdfstring{$\pbarOmegapoli$}{} and non degeneracy}
\label{sssez:basiduali}
The residue map is zero on $\Omega^{1}_{X/S}$, and in particular is a continuous 
as map from $\Omega^{1}_{X/S}(\infty \Sigma)$ to $\calO_S$. Hence it induces a map
from $\pbarOmegapoli$ (this is the sheaf constructed ad in Section \ref{ssec:definizionefascibarF} in the case $\calF=\Omega^1_{X/S}$) to $\calO_S$ that we will denote with the same symbol $\Res_\Sigma$ and which will play an important role in our constructions. 

We end this Section by a local computation. Assume we are in the situation described in Remark \ref{oss:unsolof} and we will use the basis $\grf^+_{a,i}$ and $\grf^-_{a,i}$ introduced in equation \eqref{eq:fni}. Then
$$
\Res_{\Sigma}\Big(\grf^+_{a,i} \cdot \grf^-_{b,j} dt \Big) =\delta _{a+b,-1}\delta_{i+j,n-1}
\quad \mand \quad 
\Res_{\Sigma}\Big(\grf^-_{c,i} \cdot \grf^-_{d,j} dt \Big) =0
$$
for all $c,d<0$. This immediately implies the following lemma

\begin{lemma}\label{lem:resnondegenere}
The pairing from $\pbarOpoli \otimes \pbarOmegapoli  $ to $\calO_S$ defined by 
$$
f\cdot \omega \lra \Res_\Sigma (f\omega)
$$
is perfect (meaning that on small enough affine opens subsets induces a perfect topological pairing among the sections).

Moreover, in the local situation described above if we identify if we define $\calO_{\overline{\Sigma}^*}^-$ as the $\calO_S$ subsheaf generated by $\varphi^-_{a,i}$ for $a<0$ then 
$\pbarOpoli = \calO_{\overline{\Sigma}^*}^-\oplus \pbarO$ and the summands
$\calO_{\overline{\Sigma}^*}^-$ and $\pbarO$ are maximal isotropic submodules with respect to the pairing 
$(f,g)\mapsto \Res_{\Sigma}(fgdt)$.
\end{lemma}

\section{Conventions on spaces and \texorpdfstring{$\calD$}{D}-modules}\label{ssec:convenzionispazi}

As we have done in the previous chapters, we will need to consider some constructions about sheaves on $X$ and on $X^n$, but we will  actually be interested in the case of sheaves supported on $\Sigma$ or on $\Sigma^n$. For technical reasons it will be convenient for us to consider these  as sheaves over $S$. Since the projection from $\Sigma$ to $S$ is a finite map this does not make any serious difference. 

We introduce some notations to treat uniformly these constructions. 
Let $p:X\lra S$ be a smooth family of curves over $S$, $\finitosigma$ a finite set, and let $\Sigma = \{ \grs_i\}_{i \in I}$ a collection of sections of $p$. We already introduced in Section \ref{sec:localstructure} the sheaves of topological algebras $\pbarO =p_*(\barO{X})$, $\pbarO(n)=p_*(\barOX(n\Sigma))$ and $\pbarOpoli =p_*(\barOX(\infty\Sigma))$.
We will consider, for a finite set $A$, the completed tensor products
\[
    \pbarOm{A} = \otimes^*_A \pbarO \qquad \pbarOpolim{A} = \otimes^*_A \pbarOpoli
\]

\index{$\pbarOm{A}$}\index{$\pbarOpolim{A}$}The first one is a $w$-compact (c.f. Section \ref{ssec:compactsheaves}) sheaf of topological algebras and on any well covered affine subset is a \QCCF sheaf. By remark \ref{oss:prodottiregolaristar}
the second sheaf is a regular sheaf (see Section \ref{ssec:ccgeregolari}) and an \exhaustion  of this sheaf is given by the subsheaves $\pbarOm{A}(n) = \otimes^*_A \pbarO(n)$\index{$\pbarOm{A}(n)$} which are locally isomorphic to $\pbarOm A$. 
Moreover by Remark \ref{oss:prodottiregolaristar} this product coincide with the $cg$-product and this sheaf has a natural structure of \ccg-topological algebra (see Section \ref{ssec:ccgecg}). In this case this product turns out to be also continuous as can be easily checked locally using the description of the next Section. 

\subsection{Local description}\label{ssec:descrizionelocalepbarO}
Locally, on a well covered neighborhood, using the notation of section \ref{ssec:descrizionelocale1}, we have that $\pbarOm{A}$ and $\pbarOpolim{A}$ are finite direct products of sheaves of the form 
$$
\pbarOK{I_{j_1}}\tensor{*}\cdots\tensor{*}\,\pbarOK{I_{j_\ell}} 
\qquad\mand\qquad
\pbarOpoliK{I_{j_1}}\tensor{*}\cdots\tensor{*}\,\pbarOpoliK{I_{j_\ell}}$$where $\ell$ is the cardinality of $A$.
The first tensor product is isomorphic to 
$$
\pbarOK{I_{j_1}}\tensor{*}\cdots\tensor{*}\,\pbarOK{I_{j_\ell}} 
=\limpro_n\frac{\calO_S[t_1,\dots,t_\ell]}
{\big(\grf_{j_1}(t_{1}),\dots,\grf_{j_\ell}(t_{\ell}) \big)^n}.
$$
The second tensor product, as a sheaf, is isomorphic to 
$$
\pbarOpoliK{I_{j_1}}\tensor{*}\cdots\tensor{*}\,\pbarOpoliK{I_{j_\ell}}=\pbarOK{I_{j_1}}\tensor{*}\cdots\tensor{*}\,\pbarOK{I_{j_\ell}} \big[\grf^{-1}\big]
$$
where $\grf=\grf_{j_1}(t_{1})\cdots \grf_{j_\ell}(t_{\ell})$. 
Notice that the topology on the last tensor product is not the one given by the \noz (neighborhoods of zero, c.f. Section \ref{app:top}) of $\pbarOK{I_{j_1}}\tensor{*}\cdots\tensor{*}\,\pbarOK{I_{j_\ell}}$ but by sums $\sum _{h\geq 0}U_h \grf^{-h}$ where $U_h$ is a \noz of $\pbarOK{I_{j_1}}\tensor{*}\cdots\tensor{*}\,\pbarOK{I_{j_\ell}}$. It follows easily from this description that also the product on $\pbarOpolim{A}$ is continuous and that $\pbarOpolim{A}$ is the localization of $\pbarOm{A}$ at the ideal generated by
$\calI_{\Sigma}\tensor{*}\cdots\tensor{*}\calI_{\Sigma}$ (see Section \ref{ssec:localization}).

In the affine and well covered case there is a natural basis basis of $\pbarOm{A}$ and $\pbarOpolim{A}$ which be described in a similar way to what we have done in the case of $\pbarO$ in Section \ref{sssec:basespbarO}.  Indeed, denote by $w_\ell$ with $\ell \in \mZ$ the topological basis constructed in that Section and let $N_h$ be such that $w_\ell\in \pbarO(h)$ for $\ell\geq N_h$ so that $w_\ell$ for $\ell<N_h$ is a basis of $\pbarOpoli/\pbarO(h)$. Define the tensor products
$$ w_{\underline \ell}=\otimes_{a\in A} w_{\ell_a} $$
where $\underline \ell=(\ell_a)_{a\in A}\in \mZ^A$. Fixed $h$, the elements $w_{\underline \ell}$ with $N_h<\ell_a$ for all $a\in A$, are a topological basis of $\pbarOm{A}(h)$. Indeed a \fsonoz of $\pbarO(h)$ is given by the subspaces $U_{h,k}$, defined as
$$\begin{array}{cl} \pbarO(k)\tensor{*}\pbarO(h)\cdots \tensor{*} \pbarO(h)+\pbarO(h)\tensor{*}\pbarO(k)\cdots \tensor{*} \pbarO(h)+\cdots +\pbarO(h)\tensor{*}\cdots \pbarO(h) \tensor{*} \pbarO(k)\end{array} $$
where $k<h$. The quotients $\pbarOm{A}(h)/U_{h,k}$ is free and the element $w_{\underline \ell}$ with $N_h<\ell_a\leq N_k$ are a basis of this quotient. In particular notice that the modules $\pbarOm{A}(h)$ are a product of countable many copies of $\calO_S$. 

\subsection{Forms and differential operators}\label{ssec:operatoridifferenziali} We denote by $\pbarOmega$ the sheaf constructed constructed as in Section \ref{ssec:definizionefascibarF}  from the sheaf of relative $1$-forms $\Omega_{X/S}$. The sheaf $\pbarOmega$ \index{$\pbarOmega$} co-represents the functor of continuous derivations: more precisely there is a universal continuous derivation $d : \pbarO \to \Omega^1_{\overline{\Sigma}}$ which induces an isomorphism
\[
    \Hom^{\text{cont}}_{\pbarO} ( \pbarOmega, \calM ) = \Der^{\text{cont}}_{\calO_S}(\pbarO,\calM)
\]
for any $\pbarO$-topological sheaf $\calM$ (see Section \ref{ssec:topologicalalgebrasandmodule} for the definition of topological sheaf over a topological algebra). 
When we restrict $\pbarOmega$ to a well covered open subset with a coordinate $t$ we have
\[
    \pbarOmega \simeq \pbarO dt.
\]
In particular it is a $w$-compact, \QCCF sheaf. Similarly, we consider the sheaves of 
vertical vector fields $\pbarF{T} = \calHom_{\pbarO}(\pbarOmega,\pbarO)$\index{$\pbarF{T}$} and the algebra of differential operator $\pbarD$\index{$\pbarD$} defined as the sheaf of subalgebras $\pbarD\subset \calHomcont_{\calO_S}(\pbarO,\pbarO)$ generated by $\pbarO$ and $\pbarF{T}$. This is the colimit of the subsheaves $\pbarD^{\leq n}$ \index{$\pbarD$:$\pbarD^{\leq n}$}
of differential operators of degree less or equal to $n$. On the sheaves $\pbarD^{\leq n}$ we put the topology induced by the inclusion in 
 $\calHomcont_{\calO_S}(\pbarO,\pbarO)$ and on $\pbarD$ we put the colimit topology, which is finer that the topology given by the inclusion  in 
 $\calHomcont_{\calO_S}(\pbarO,\pbarO)$. 
The sheaves $\pbarD^{\leq n}$ are $w$-compact and $\pbarD$ is a regular sheaf and a \ccg-topological algebra. 

Over a well covered open subset $U \subset S$ with a coordinate $t$ there is an isomorphisms of topological sheaves $$\pbarF{T}\simeq \pbarO\,\partial_t \quad\mand\quad\pbarD\simeq \bigoplus_{n \geq 0} \pbarO\,\partial_t^n.$$ 
By construction $\pbarO$ is a left $\pbarD$-module and, as in the classical case, $\pbarOmega$ is a right-$\pbarD$-module where locally the action is given by 
$f(t)dt\cdot \partial_t=-\partial_tf(t) dt$. In both cases these are $\pbarD$ \cgmod modules.

Similar definition can be given in the case of the sheaf $\pbarOm{A}$. 
In particular we define 
$$
\pbarDm{A}=\tensor{cg}_A\pbarD
$$
which is a \ccg-topological algebra and $\pbarOm{A}$ is a \ccgmod module over this algebra. On a well covered neighborhood, using the notation of Section \ref{ssec:descrizionelocalepbarO}, we have 
\begin{equation}\label{eq:DAlocale}
\pbarDm{A}=\bigoplus_\gra \pbarOm{A}\partial^{\gra}
\end{equation}\index{$\pbarD$:$\pbarDm{A}$}
where $\partial^\gra=\partial_{t_1}^{\gra_1}\cdots\partial_{t_n}^{\gra_n}$
if $A=\{1,2,\dots,n\}$. We define $(\pbarDm{A})^{\leq m}$ as subspace of $\pbarDm{A}$ of operator of total degree less or equal to $m$, that is $\bigoplus_\gra \pbarOm{A}\partial^{\gra}$
with $\sum \gra_i\leq m$. These are w-compact subsheaves of $\pbarDm{A}$ and their colimit is equal to $\pbarDm{A}$. 

Completely similar considerations can be given for $\pbarDpoli$. The only difference is that in this case $\pbarDpoli^{\leq n}$ \index{$\pbarD$:$\pbarDpoli$} is not anymore $w$-compact but it is a regular sheaf.  However, also in this case, $\pbarDpoli$ has a natural \exhaustion and it is a \ccg-algebra.

\begin{remark}[Behaviour under pullback]\label{rmk:pullbacktangentdifferential}
	Consider a morphism $\varphi : S' \to S$ and let $\Sigma'$ be the collection of sections of $X'=X\times_S S' \to S'$ induced by $\Sigma$. It follows from Remark \ref{rmk:pullbackcompletion} that we have $\hat{\varphi}^*T_{\oSigma} = T_{\overline{\Sigma}'}$ which implies $\hat{\varphi}^*\pbarD = \calD_{\overline{\Sigma}'}$ as well. 
\end{remark}

\subsection{Ordered tensor products}\label{ssec:prodottiordinatiOeD}
In this Section we make some remarks on the ordered tensor products of the rings introduced so far. 
Given $\calO_S$-algebras $\calB$ and $\calC$ the tensor product $\calB\otimes \calC$ is an $\calO_S$-algebra.
If $\calB$ and $\calC$ are topological algebras one would like to extend multiplication to the completion of this tensor product. 
In general there are technical problems in doing this, mainly due to the fact these topological tensor product are not associative.
In Section \ref{ssec:topologicalalgebrasandmodule} we have given some results that allow us to construct a multiplication on the tensor product; in this Section we give another result in this direction regarding ordered tensor products. The hypothesis we will use may seem a bit awkward, but they are satisfied by 
the objects we are interested in. One should think about the following proposition keeping in mind that we will apply these results to the topological algebras $\pbarOpoli,\pbarDpoli$. 

Let $\calB$ and $\calC$ two $\calO_S$-algebras equipped with a topology (we are not assuming they are topological algebras in the sense of Section \ref{ssec:topologicalalgebrasandmodule}), assume that
\begin{itemize}
	\item for all open subset $W\subset S$ and all section $b\in \calB(W)$ the left and right product by $b:\calB|_W\lra \calB|_W$ are continuous (in this case we say the product is continuous in both variables);
	\item $\calC$ is a \ccg-topological algebra (see Definition \ref{def:ccgalgmod})
	so that we have an \exhaustion $\calC=\limind_{l \in \mN} \calC_\ell$ with $\calC_\ell$ compact;
	\item $\calB\tensor{\ra} \calC_\ell\simeq \calB\tensor{r}\calC_\ell$ as plain sheaves for all $\ell$; 
	\item $\calB\tensor{\ra} \calC=\limind \calB\tensor{\ra}\calC_\ell$ as a topological sheaves.
\end{itemize}

These hypothesis, by Remarks \ref{oss:tensoriisomorfifasci},
\ref{oss:prodottiregolarifreccetta} are satisfied, for example, if $\calB$ or $\calC$ are equal to $\pbarOpolim{A}$ or $\pbarDpolim{A}$.
Under the assumptions above the tensor product $\calB\tensor{\ra}\calC$ has a natural structure of algebra whose multiplication is continuous in both variables 
(the proof of this claim is completely analogous to the proof of point b) in the Lemma \ref{lem:BCMN} below which treats the more general case of modules).
In particular the topological modules
$$
\pbarOpolim{A_1}\tensor{\ra}\cdots \tensor{\ra}\pbarOpolim{A_k} 
\quad \mand \quad
\pbarDpolim{A_1}\tensor{\ra}\cdots \tensor{\ra}\pbarDpolim{A_k} 
$$
have a natural structure of algebras, and the product is continuous in both variables.  

\begin{lemma}\label{lem:BCMN} Assume that $\calB$ and $\calC$ satisfy the assumptions above. 
	Let $\calM$ be a $\calB$ module, $\calN$ be a $\calC$ module and assume that 
	\begin{itemize}
		\item $\calN$ is a \ccgmod module (see Definition \ref{def:ccgalgmod})
		\item the action of $\calB$ on $\calM$ is continuous in both variables (see Section \ref{ssec:topologicalalgebrasandmodule});
	\end{itemize}
	Then $\calM\tensor\ra\calN$ has a natural structure of  $\calB\tensor\ra\calC$-module such that
	\begin{enumerate}[\indent a)]
		\item for all open subsets $W\subset S$ and for all $m\in\calM(W)$ and $n\in \calN(W)$ the map $f\mapsto f\cdot (m \otimes n)$ from $\calB|_W\tensor\ra \calC|_W$ to $\calM|_W\tensor\ra \calN|_W$ is continuous;
		\item for all open subset $W\subset S$ and all    $f\in (\calB\tensor{\ra}\calC)(W)$ it is defined the multiplication by $f$ from $\calM|_W\tensor{\ra}\calN|_W$ to $\calM|_W\tensor{\ra}\calN|_W$ and it is continuous. The same holds by replacing $\otimesr$ with $\otimes^r$;
		\item the action of $\calB\tensor{\ra}\calC$ on $\calM\tensor{\ra}\calN$ described in the previous point makes $\calM\tensor{\ra}\calN$ a $\calB\tensor{\ra}\calC$-module;
		\item if $\calM\tensor{\ra}\calN=\calM\tensor{r}\calN$ as plain sheaves (by Remark \ref{oss:tensoriisomorfifasci} this condition is satisfied, for example, if $\calN$ is \QCCF) then 
		for all $W$ and all sections $p\in (\calM\tensor{\ra}\calN)(W)$ the multiplication by $p$ from $\calB|_W\tensor{\ra}\calC|_W$ to $\calM|_W\tensor{\ra}\calN|_W$ is continuous.
	\end{enumerate}
\end{lemma}
\begin{proof}
	Recall the topologies of Section \ref{ssec:topologytensorprodsheaves} and that our topologies are defined at the level of the tensor product as presheaves. We will denote by $\otimes^{\mathrm{pre}}$ the tensor product as presheaves. We start by noticing that the fact that $\calN$ is a $\calC$ \ccgmod module implies that the action of $\calC$ on $\calN$ is continuous in both variables. 
	
	Fix an open subset $W \subset S$ and sections 
	$m\in \calM(W)$ and $n\in \calN(W)$, then acting on  $m$ and $n$ defines continuous morphisms $\calB|_W \to \calM|_W$ and $\calC|_W\to \calN|_W$ (here we use the continuity in the first variable). By functoriality of the topological tensor products these morphism induce a morphism $\calB|_W \otimes \calC|_W \to \calM|_W\otimes\calN|_W$ which is continuous for both the $r$ or the $\ra$ topology, in addition, the diagram
	\begin{equation}\label{eq:prodottofreccetta1}
	\xymatrix{
		\calB|_W \tensor{r}   \calC|_W \ar[rr]^{\cdot (m\otimes n)}\ar[d]   &  &  \calM|_W \tensor{r}   \calN|_W \ar[d] \\
		\calB|_W \tensor{\ra} \calC|_W \ar[rr]^{\cdot (m\otimes n)}       &  &  \calM|_W \tensor{\ra} \calN_W
	}
	\end{equation}
	is commutative. Indeed all maps are continuos and the restriction of the two composition maps to $\calB|_W\otimes \calC|_W$, which is dense, is the same. 
	In particular for every $f\in (\calB\tensor{\ra}\calC)(W)$ the product $f\cdot (m\otimes n)$ is defined. Hence for all $W'\subset W$ open subsets we have a map
	$f\cdot : \calM(W')\otimes \calN(W') \lra  (\calM\tensor{\ra}\calN)(W')$.  
	These maps are compatible with restrictions, hence they define a morphism of presheaves 
	$$
   f\cdot_{\mathrm{pre}} : \calM|_W\tensor{\mathrm{pre}} \calN|_W \lra  (\calM\tensor{\ra}\calN)|_W.  
	$$
	where the tensor product on the left hand side is the presheaf $W'\mapsto \calM(W')\otimes \calN(W')$ without any completion. 
	Claim (b) is the statement that this map is continuous with respect to the $\ra$-topology. We begin by proving the analogous statement for the $r$-topology. Indeed we have that
	$f\in (\calB\tensor{\ra} \calC_\ell)(W) =(\calB\tensor{r}\calC_\ell)(W)$ for some $\ell$, by our assumption on $\calB$ and $\calC$. Being $\calN$ a $\calC$ \ccgmod module we have that the product $\calC_\ell\otimes \calN\lra \calN$ is continuous, hence it induces a continuous morphism $\calB\otimes \calC_\ell \times \calM\otimes \calN\lra \calM\otimes \calN$ where we consider all the three tensor product with the $r$-topology. In particular the product by $f$ from $\calM|_W\otimes \calN|_W$ to $(\calM\tensor{r}\calN)|_W$ is continuous if we put on the domain the $r$-topology and a fortiori, the multiplication by $f$ from $\calM|_W\otimes \calN|_W$ to $(\calM\tensor{\ra}\calN)|_W$ is continuous   if we put on the domain the $r$-topology.
	
	We now prove point (b) for the $\ra$-topology. Let  $\calQ$ be a \noz of $\calM\tensor{\ra}\calN$ such that  $\calQ\supset \calM\otimes \calV$ for $\calV$ a \noz of $\calN$ and, for all open subset $W'$, and all $n$ in $\calN(W')$, there exist a noz $\calU_n$ of $\calM$ such that 
	$\calQ\supset \calU_n|_{W'}\otimes n$. 
	We have the following diagram
	$$
	\xymatrix{
		\calM|_W\tensor{\mathrm{pre}}     \calN|_W\ar[rr]^{f\cdot_{\mathrm{pre}} }\ar@{=}[d]  & &  
		(\calM\tensor{r}  \calN)|_W\ar[d]\\
		\calM|_W\tensor{\mathrm{pre}}     \calN|_W\ar[rr]^{f\cdot_{\mathrm{pre}}}             & &  (\calM\tensor{\ra}\calN)|_W.
	}
	$$
	The commutativity of this diagram follows from the commutativity of the diagram \eqref{eq:prodottofreccetta1}. 
	Since the top horizontal map is continuous when we put on the domain the $r$-topology, there exists a \noz $\calV'$ of $\calN$ such that $f\cdot_{\mathrm{pre}}(\calB\otimes \calV')\subset \calQ$. 	
	Now fix an open subset $W'$ of $W$ and $n\in \calN(W')$ and we prove there exists  a \noz $\calU'_n$ of $\calM$  such that $f\cdot_{\mathrm{pre}} (\calU'_n|_{W'}\otimes n)\subset \calQ|_{W'}$. Recall that we can assume $f\in (\calB\tensor{\ra} \calC_\ell)(W)$ for some $\ell$. 
	Since the product $\calC_\ell\times \calN\lra \calN$ is continuous there exists a \noz $\calV_n$ of $\calC_\ell$ such that $\calV_n|_{W'}\cdot n\subset \calV|_{W'}$.
   Recall also that $\calB\tensor{\ra} \calC_\ell=\calB\tensor{r} \calC_\ell$, hence $f\in (\calB\tensor{r}   \calC_\ell)(W)$. Let $\overline f$ be the image of $f$ in $\big(  \calB\tensor \ra (\calC_\ell/\calV_n)\big)(W')$. By the definition of $r$-topology we have 
   $$
   \frac{\calB\tensor r \calC_\ell}{\calB\tensor r \calV_n}\simeq \calB\tensor r \frac {\calC_\ell}\calV_n\simeq \calB\otimes \frac {\calC_\ell}\calV_n
   $$
   (we used that we are specifically working with the $r$ topology for both isomorphisms, and we used the fact that $\calC_\ell/\calV_n$ is discrete in the second isomorphism). Since $S$ is topologically noetherian we may check continuity locally and assume that $\overline{f} \in \calB \otimes^{\mathrm{pre}} (\calC_\ell/\calV_n)$ and write
   \[
		f = x + \sum_{j} b_j\otimes c_j
   \]
   with $x \in \calB \tensor{r}\calV_n(W)$, $b_j \in \calB(W)$ and $c_j \in \calC_\ell(W)$. Since multiplication by $b_i$ as a morphism $\calM \to \calM$ is continuous we can find open subsheaves $\calU'_j \subset \calN$ such that $b_j\calU'_j|_{W'} \subset \calU_{c_jn}|_{W'}$. Considering $\calU'_n = \cap\, \calU'_j$ we get that $(x + \sum_j b_j \otimes c_j)(\calU'_n|_{W'}\otimes n) \subset \calM \otimesr \calV|_{W'}+ \sum_j \calU_{c_jn} \otimes (c_jn) \subset \calQ$.
    This proves that the map $f\cdot_{\mathrm{pre}}$ is continuous. Being the right hand side a complete sheaf, we finally get a continuous map 
    	$$
    f\cdot : (\calM\tensor{\ra} \calN)|_W \lra  (\calM\tensor{\ra}\calN)|_W.  
    $$
    In particular this proves that there is an action of $\calB\tensor\ra\calC$ on $\calM\tensor\ra \calN$ and that this action has properties a) and b).  
    	
	c) By continuity of the two maps it is enough to prove that 
	$$
	(fg)\cdot (m\otimes n)=f\cdot(g\cdot(m\otimes n))
	$$
	for $m \in M$ and $n\in N$. By the commutativity of diagram \eqref{eq:prodottofreccetta1} this product are the same as the one defined using the $r$-tensor product. 
	We can assume that $f,g,fg\in B\tensor{r}C_\ell$ for a big enough $\ell$. Hence the claim follows from the continuity of the action of $C_\ell$ on $N$. 
	
	d) The proof is completely analogous to that of point b)
\end{proof}

We want to extend the above result to the case of multiple products. This is not really strictly necessary for our purposes but it makes some arguments more uniform. 

We assume $\calB_1,\dots\calB_k$ are ccg-topological algebras and let $\calB_i=\limind \calB_{i,\ell}$ be an exhaustion of $\calB_i$. Assume furthermore that the sheaves $\calB_{i,\ell}$ are \QCCF, so that the iterative version of the condition $\calB \otimesr \calC_l \simeq \calB \tensor{r}\calC_l$ above holds. Then 
$$
\calB=\calB_1\tensor \ra \cdots \tensor \ra \calB_k
$$
has a structure of $\calO_S$ algebra and the product is continuous in both variables. This is easily proved by induction using the previous Lemma. This applies for example to the case of the sheaves of algebras like $\pbarOpolim A$ and $\pbarDpolim A$. 

Now assume $\calM_i$ is \ccgmod $\calB_i$ module. Then if $\calM_i$ is \QCCF, we can argue by induction and prove that the module
$$
\calM=\calM_1\tensor \ra \cdots \tensor \ra \calM_k
$$
has a natural structure of $\calB$-module and the product is continuous in both variables. The hypothesis that $\calM_i$ is \QCCF will be satisfied in all our examples, however it would be preferable not to use this hypothesis. 

Without assuming that the modules are \QCCF we get the following slightly weaker result.

\begin{lemma}\label{lem:BCMN2}
Assume that $\calB_i$ are $\calO_S$ ccg-topological algebras satisfying all hypotheses above. Assume also that $\calM_i$ is a \ccgmod $\calB_i$ module. Then $\calM$ has a natural structure of $\calB$-module such that the product is continuous in the second variable. 
Moreover for all open subset $W$ and all section $m_i\in \calM_i(W)$ the product by $m_1\otimes\dots\otimes m_k$  from $\calB|_W$ to $\calM|_W$ is continuous.   
\end{lemma}
\begin{proof}
The proof is by induction on $k$ and will follow the same lines of the proof of the previous Lemma. We sketch only the main steps.

Let $\calB'$ the $\ra$ tensor product of the first $k-1$ algebras and $\calM'$ the product of the $k-1$ modules. 

As in the previous proof we fix an opens subset $W$ and sections $m_i\in\calM_i(W)$. Then by functoriality the product by $m=m_1\otimes\dots\otimes m_k$ is a well defined continuous map
$$
\cdot m:\calB|_W\lra \calM|_W.
$$
Hence for every $f\in \calB(W)$ we have a well defined map 
$$
\cdot f:\calM_1\tensor {\mathrm{pre}}\dots\tensor{\mathrm{pre}}\calM_k\lra\calM.
$$
We need to check that this map is continuous and we proceed exactly as in the previous lemma, using the fact that by induction the product $\calB' \times \calM'\lra \calM'$ is continuous in the second variable. 
\end{proof}

\subsection{Poles along the diagonal}\label{ssec:polesalongdiagonal} In the following discussion the assumption that  $S$ is quasi compact is important since allows to check continuity of a morphism locally. We already noticed that $\pbarOq = \pbarO\otimes^*\pbarO$ is a $w$-compact topological algebra and its topology is defined by ideals. We define $\pbarOq(-\Delta)=\calI_\Delta$\index{$\calI_\Delta$} to be the kernel of the multiplication map  $\pbarOq \to \pbarO$. 

Using the convention of Section \ref{ssec:localization}, given a $\pbarOq$-topological module $\calF$ who's topology is $\pbarOq$-linear we define  $\calF(n\Delta) = \calF(n\calI_\Delta)$ and $\calF(\infty\Delta) = \calF(\infty\calI_\Delta)$\index{$\calF(n\Delta),\calF(\infty\Delta)$}.

By what we noticed in Section \ref{ssec:localization}, $\pbarOq(\infty\Delta)$ is a topological algebra and $\calF(\infty\Delta)$ is $\pbarOq(\infty \Delta)$-module.

\begin{remark}
From the local triviality of $\pbarOq(n\Delta)$ we see that the sheaves $\calF(n\Delta)$
are locally homeomorphic to $\calF$, in particular if $\calF$ is a \QCC sheaf, they are also \QCC sheaves. Moreover local triviality implies that the natural map
$$
	\calF \otimes_\pbarOq \pbarOq(n\Delta)\lra \calHom_{\pbarOq}(\pbarOq(-n\Delta),\calF)
$$
is an isomorphism of sheaves. Finally notice that by remark \ref{oss:continuitamorfismimodulitopologici} every morphism 
from $\pbarOq(n\Delta)$ to $\calF$ is continuous. Hence we have 
$$
\calF(n\Delta)=
\calHomcont_{\pbarOq}(\pbarOq(-n\Delta),\calF)\subset
\calHomcont_{\calO_S}(\pbarOq(-n\Delta),\calF).
$$
In particular $\calF(n\Delta)$ it is defined as a subsheaf of the RHS sheaf by the continuous equation imposing $\pbarOq$-linearity. In particular it is a closed subsheaf, and the topology coincide with the subsheaf topology induced by the right hand side.
\end{remark}

\begin{remark}[Behaviour under pullback]\label{rmk:pullbackpolidiagonal}
	Consider a morphism $\varphi : S' \to S$ and let $\Sigma'$ be the collection of sections of $X'=X\times_S S' \to S'$ induced by $\Sigma$. It follows by local inspection that $\hat{\varphi}^*\calI_\Delta = \calI_{\Delta'}$, by this the fact that $\hat{\varphi}^*\pbarO = \calO_{\overline{\Sigma}'}$, and by the fact that $\hat{\varphi}^*$ commutes with $*$-tensor products (see Proposition \ref{prop:pullbacktensorproduct}) it follows that $\hat{\varphi}^*(\pbarOq(n\Delta)) = \calO^2_{\overline{\Sigma}'}(n\Delta')$,  $\hat{\varphi}^*(\pbarOq(\infty\Delta)) = \calO^2_{\overline{\Sigma}'}(\infty\Delta')$ as well as  $\hat{\varphi}^*(\calF(n\Delta)) = \calF_{\overline{\Sigma}'}(n\Delta')$,  $\hat{\varphi}^*(\calF(\infty\Delta)) = \calF_{\overline{\Sigma}'}(\infty\Delta')$ for any $\pbarOq$-linear topological module $\calF$.
\end{remark}

\subsection{Poles along diagonals in the multivariable case}\label{ssec:polimolti}
Let $a\neq b$ be indices in $A$ we will consider ideals $\calI_{\Delta_{a=b}} \subset \pbarOm{A}$ which analogous to $\calI_{\Delta}$: let $A\surjmap B$, where $B$ has one element less than $A$, be any surjective map which collapses $a$ and $b$ and let
$\calI_{\Delta_{a=b}}$ be the kernel of the corresponding surjection $\pbarOm{A} \to \pbarOm{B}$. This is a locally free ideal of rank $1$ which of course does not depend on the choice of $B$. We can define the topological algebras and modules
\[
\pbarOm{A}(n\Delta_{a=b}),\quad  \calF(n\Delta_{a=b}), \quad \pbarOm{A}(\infty\Delta_{a=b}),\quad \calF(\infty\Delta_{a=b})
\]
as we did above in the case of $\text{card}\, A=2$, by localizing this ideal. Notice that if $\calF$ is a left (resp.\ right) $\pbarDm{A}$-module then also these will have a natural structure of left (resp.\ right) $\pbarDm{A}$-module. 

We can iterate this construction and define 
$$
\calF\left(\infty(\Delta_{a_1=b_1}\cup\dots\cup \Delta_{a_k=b_k})\right) = \calF(\infty\Delta_{a_1=b_1})\dots(\infty\Delta_{a_k=b_k}).
$$ 
Different choices for the ordering $(a_1,b_1),\dots,(a_k,b_k)$ lead to canonically isomorphic algebras and modules.

As in Section \ref{sssez:art2notazionidiagonali}, more in general, given a surjective map $\pi:A\surjmap B$, we can consider the divisor 
$$
\Divisorepi (\pi)=\bigcup_{\pi(a)\neq \pi(b)}\Delta_{a=b}
$$\index{$\nabla(\pi)$}
of $X^A$ and the ideal $\calI_{\Divisorepi(\pi)}$ defined as the product of the ideals $\calI_{\Delta_{a=b}}$. Define  $$
\calF(n\Divisorepi(\pi)), \quad \calF(\infty\Divisorepi(\pi))
$$
by localizing this ideal. We notice that 
$
\calF(\infty \Divisorepi (\pi))= \calF(\infty\Delta_{a_1=b_1})\cdots(\infty\Delta_{a_k=b_k})
$
\index{$\calF(\infty\nabla(\pi)),\calF(\infty\nabla(J/I))$}where on the right hand side we considered all couples $a_\ell,b_\ell$ such that $\pi(a_\ell)\neq\pi(b_\ell)$.

\begin{remark}[Taylor expansion]\label{ssez:Taylor}
Finally we recall the form of the Taylor expansion we will use very often in the paper.
Assume $S$ is well covered and let $t$ be a local coordinate. Set $u=t\otimes 1$ and $v=1\otimes t$ and  for $f\in \pbarO$ set also $f(u)=f\otimes 1$ and $f(v)=1\otimes f$. Then we have the following Taylor expansion 
$$
f(u)-f(v)=\sum_{n=1}^{N-1} \frac 1{n!}(u-v)^n \, (\partial_t^n f)(v) \in \pbarOq/\calI_\Delta^N.
$$
Its proof and a more general formula can be found in \cite[2.2.3]{cas2023}.
\end{remark}

\subsection{Right and left expansion of poles along the diagonal}
The following Lemma is essentially a version of Lemma 2.3.13 in \cite{cas2023} and will be crucial for our constructions.

In the discussion below we are going to consider on $\pbarOpoli\otimes\pbarOpoli$ also the $r,\ra$ topologies. By Remark \ref{oss:prodottodialgebrehat} since these are \QCCF sheaves, we have the following isomorphism of sheaves (as opposed to topological sheaves):
$$
\pbarOpolir\stackrel{\text{def}}{=}\pbarOpoli \tensor {r}\pbarOpoli\simeq\pbarOpoli \tensor {\rightarrow} \pbarOpoli.
$$
A similar remark and notation holds for left tensor products. Moreover $\pbarOpolir$ and $\pbarOpolil$, by \ref{oss:prodottodialgebrehat} are topological algebras. Finally we have morphism of algebras
$$
\pbarOpoliq\subset \pbarOpoli\tensor {\rightarrow} \pbarOpoli.
$$

\begin{lemma}\label{lem:extensiondiagonal}
	There are unique morphisms of algebras
	\[
	Exp^r:\pbarOpoliq(\infty\Delta) \to \pbarOpoli\tensor{\ra} \pbarOpoli \qquad 
	Exp^\ell:\pbarOpoliq(\infty\Delta) \to \pbarOpoli\tensor{\leftarrow} \pbarOpoli
	\]
	\index{$Exp^r,Exp^\ell$}which extend the natural maps
	\(
 \pbarOpoliq \subset \pbarOpoli\otimesr\pbarOpoli, \pbarOpoliq \subset \pbarOpoli\otimesl\pbarOpoli
	\)
	Moreover these morphisms are continuous.
\end{lemma}

\begin{proof} 
		We prove the claim for the $r$-tensor product since $\pbarOpolir = \pbarOpoli \otimesr \pbarOpoli$ as sheaves of algebras. Since locally on $S$, the sheaf $\pbarOpoliq(\infty\Delta)$ is a localization of $\pbarOpoliq$, uniqueness is clear. To check that this possible unique morphism is well defined we prove that there exists a basis of open subsets $U$ such that the ideal generated by $\calI_\Delta(U)$ in $\pbarOpolir(U)$ is equal to $\pbarOpolir(U)$. This implies the thesis. 
	
	Let $s_0 \in S$. We can assume that $U$ is well covered  \noz of $s_0$ and 
	we can use the local description given in section \ref{ssec:descrizionelocale1}. To simplify the notation we also assume that 
	$\grs_i(s_0)=x_0$ for all $i$, since in the general case is a direct product of algebras obtained in this way. 
	Let  $t$ be a local coordinate. Notice that under this assumption the ideal $\calI_\Delta$ is generated by $\delta=t\otimes 1-1\otimes t$. We need to prove that this element is invertible in $\pbarOpolir(U)$. We argue as in Lemma 2.3.13 of \cite{cas2023}. 
	To prove the claim, it is enough to prove that $\delta$ is invertible in $\pbarOpoli(U) \otimes^r \pbarO(U)$. The right topology is defined by the ideal $\pbarOpoli(U)\otimes \calI_\Sigma(U)$ in $\pbarOpoli(U) \otimes \pbarO(U)$. Hence, it is enough to show that $\delta$ is invertible in 
	$\pbarOpoli(U) \otimes^r\frac{\pbarO(U)}{\calI_\Sigma(U)} $, or even in the ring
	$$B=\calO_X(U)_\grf \otimes \frac{\calO_X(U)}{(\grf)}$$
	where $\grf$ is a generator of $\calI_\Sigma(U)$. Now notice that $\delta$ is also a local generator of the kernel of the multiplication map $\calO_X\otimes\calO_X\to\calO_X$. In particular $\grf\otimes 1-1\otimes \grf=\delta \cdot h $ for some $h\in \calO_X(U)\otimes\calO_X(U)$ and that the left hand side
	is clearly invertible in the ring $B$.
	
     Finally we prove that this morphism is continuous. Since $\pbarOpoliq(\infty\Delta)$ is the colimit of 
     $\pbarOpoliq(n\Delta)$ it is enough to prove that the restriction to  
     $\pbarOpoliq(n\Delta)$ is continuous. Again this is a local statement hence we can assume that the $\pbarOpoliq(n\Delta)=\frac1{\delta^n}\pbarOpoliq$. Since we proved that locally $\delta$ is invertible in $\pbarOpoli\tensor{\ra} \pbarOpoli$ and the product is continuos in both variables in this algebra then the continuity of the morphism follows from the continuity of the map from $\pbarOpoliq$ to $\pbarOpoli\tensor{\ra} \pbarOpoli$.
\end{proof}

\begin{remark}\label{oss:inversadidelta}
We want to describe the inverse of $\delta$ in $\pbarOpoli\tensor\ra \pbarOpoli$
in the case we are in the situation described in \ref{oss:unsolof}. Recall the definition of the functions $\varphi^+_{m,i}(t)$ and $\varphi^-_{m,i}(t)$ from formula \eqref{eq:fni}. Set $u=t\otimes 1$ and $v=1\otimes t$ so that $\delta=u-v$. Since $u-v$ generates the ideal of the diagonal there exists $h(u,v)$ such that 
$$
\varphi(u)-\varphi(v)=h(u,v)\cdot (u-v)
$$
where $\grf(t)=\prod(t-a_i)$. We write an explicit formula for $h(u,v)$. Indeed it is easy to prove by induction that we have the following expression for $h$:
$$ h(u,v)=\grf^-_{0,n-1}(u)\grf^+_{0,0}(v) +\grf^-_{0,n-2}(u)\grf^+_{0,1}(v)+\cdots 
+\grf^-_{0,0}(u)\grf^+_{0,n-1}(v)
$$
Hence in $\pbarOpoli\tensor\ra \pbarOpoli$
 we have 
\begin{align*}
Exp^r\left(\frac 1 {u-v}\right)&= \frac{h(u,v)}{\grf(u)-\grf(v)} =\sum_{m\geq 0} \frac {\grf(v)^{m}h(u,v)}{\grf(u)^{m+1}} \\
&=\sum_{m\geq 0,i=0..n-1} \grf^-_{-m-1,n-i-1}(u) \cdot \grf^+_{m,i}(v)\\
&= \frac{1}{u-a_1}+\frac{v-a_1}{(u-a_1)(u-a_2)}+\frac{(v-a_1)(v-a_2)}{(u-a_1)(u-a_2)(u-a_3)}+\cdots 
\end{align*}
Notice that the series is well defined in $\pbarOpoli\tensor\ra \pbarOpoli$ and that the monomial $\grf^-_{-m-1,n-i-1}(u) \cdot \grf^+_{m,i}(v)$ involves exactly 
the basis  $\grf^+_{m,i}$ with its dual $\grf^-_{-m-1,n-i-1}$ with respect to the residue form, see Section \ref{sssez:basiduali}. Of course we have a similar formula if we 
exchange the $\grf^+_{m,i}$ with the $\grf^-_{m,i}$ (it corresponds to a different ordering of the roots).

Similarly for the left expansion we have 
\begin{align*}Exp^\ell \left(\frac 1 {u-v}\right)
&=-\sum_{m\geq 0,i=0..n-1}  \grf^-_{m,i}(u)  \grf^+_{-m-1,n-i-1}(v) \\ &=-\frac{1}{v-a_n}-\frac{u-a_n}{(v-a_n)(v-a_{n-2})}\cdots
\end{align*}
\end{remark}

We extend the Lemma \ref{lem:extensiondiagonal}  to the multivariable case. First we introduce a notation that will be important through the rest of the paper. If $\pi:A\surjmap B$ is a surjective map we denote by 
$$
\Delta(\pi)^\sharp :\pbarOpolim{A}\lra \pbarOpolim{B}
$$
the continuous morphism of $\calO_S$ algebras defined by $\otimes_{a\in A}f_a \mapsto \otimes_{b\in B}\prod_{a\in A_b} f_a$. We notice that this map extends to a map that we denote by the same symbol
$$
\Delta(\pi)^\sharp :\pbarOpolim{B}(\infty\nabla (\pi))\lra \pbarOpolim A(\infty \nabla (A))
$$
where $\nabla(A)$ is the big diagonal, the union of all possible diagonals (see  Section  \ref{sssez:art2notazionidiagonali} for the convention about diagonals). We denote by the same symbol also the analogous morphism in the case of the rings $\pbarOm A$. 

\begin{lemma}\label{lem:espansionediagonale}
	Let $q:A\lra\{1,\dots,k\}$ be a surjective map. 
Then \begin{enumerate}[\indent a)]
\item there exists a unique morphism of sheaves of algebras 
$$
Exp^r_q:\pbarOpolim{A} (\infty \nabla(q)) \lra 
\pbarOpolim{A_1}\tensor{\ra}\pbarOpolim{A_2}\tensor{\ra}\cdots
\tensor{\ra}\pbarOpolim{A_k}
$$\index{$Exp^r,Exp^\ell$!$Exp^r_q$}
where $A_j=q^{-1}(j)$ which extends the natural morphism from $\pbarOpolim{A} $ to 
$$\pbarOpolim{A_1}\tensor{\ra}\pbarOpolim{A_2}\tensor{\ra}\cdots
\tensor{\ra}\pbarOpolim{A_k}$$. Moreover this morphism is continuous; It is possible to define analogously $Exp^\ell_q$;\index{$Exp^r,Exp^\ell$!$Exp^\ell_q$}
\item if $\pi:B\surjmap A$ and for all $1\leq j \leq k$ 
we denote by $\pi_j:B_j\lra A_j$ the restriction of $\pi$, then
$$
Exp^r_q\Big(\Delta(\pi)^\sharp (f)\Big)=\Delta(\pi)^r \Big(Exp^r_{q\circ\pi}(f)\Big)
$$
where 
$\Delta^r(\pi)=\tensor\ra \Delta(\pi_j)^\sharp:
\pbarOpolim{B_1}\tensor{\ra}\pbarOpolim{B_2}\tensor{\ra}\cdots
\tensor{\ra}\pbarOpolim{B_k}
\lra \pbarOpolim{A_1}\tensor{\ra}\pbarOpolim{A_2}\tensor{\ra}\cdots
\tensor{\ra}\pbarOpolim{A_k}
$.
\end{enumerate}
Similar claims hold for left tensor products. 

\end{lemma}

\begin{proof}
As in the previous lemma we are reduced to prove a local statement and we need to prove that the generator of $\calJ_{\Delta_{a=b}}$:
$$\delta_{ab}=1\otimes \cdots \otimes \stackrel{a}t\otimes \cdots \otimes 1-1\otimes \cdots \otimes \stackrel{b}t\otimes \cdots \otimes 1$$  
(the $t$ is position $a$ in the first summand and in position $b$ in the second)
is invertible in the right hand side if $q(a)\neq q(b)$. In the previous lemma we proved that $\delta =t\otimes 1-1\otimes t$ is invertible in $\pbarOpoli\tensor{\ra}\pbarOpoli$. We can embed this 
sheaf in $\pbarOpolim{A_{q(a)}}\tensor{\ra}\pbarOpolim{A_{q(b)}}$ and finally on product on the right hand side. Since the image of $\delta$ is equal to $\delta_{ab}$ we deduce that $\delta_{ab}$ is invertible. This proves a).

To prove b) we notice that the remark is trivial if $f$ has no poles. It then follows by noticing that they are morphism of algebras. 
\end{proof}

\subsubsection{The Cauchy formula}\label{ssez:Cauchyformula}\index{Cauchy formula} We give now a formulation of Cauchy formula in this setting. This will be a local computation. Assume $S=\Spec A$ is well covered and that $t$ is a local coordinate, 
and set, as above, $\delta=t\otimes 1-1\otimes t$. We denote $u=t\otimes 1$ and $v=1\otimes t$ so that $\delta=u-v$. With this notations we have the following formulation of Cauchy formula (the actual Cauchy formula is covered by the case $f\in \pbarO$). We define
$T_r:\pbarOpoli\tensor{\ra} \pbarOpoli \lra \pbarOpoli$ and
$T_\ell:\pbarOpoli\tensor{\la} \pbarOpoli \lra \pbarOpoli$ by 
$$
\Tr(f\otimes g)=\big(\Res_{\Sigma} (fdt) \big) g 
\quad\mand\quad
\Tl(f\otimes g)=\big(\Res_{\Sigma} (fdt) \big) g.
$$\index{$\Tr,\Tl$}
Notice that the these two functions are continuous for the $\ra$ and $\la$ topologies, hence are well defined. Notice that the map $f\otimes g \mapsto \left( \Res_\Sigma(fdt)\right) g$ is defined on $\pbarOpoli\otimes^!\pbarOpoli$, and that our maps $\Tr,\Tl$ are the restrictions of this map to $\pbarOpoli\otimesr\pbarOpoli$ and $\pbarOpoli\otimesl\pbarOpoli$, respectively.

\begin{proposition}[Cauchy Formula]\label{prop:Cauchy}Let $f\in \pbarOpolim 2$, then
 $$
 T_r(f\cdot Exp^r(\delta^{-1}))-  T_\ell(f\cdot Exp^\ell(\delta^{-1}))=\Delta^\sharp(f).
 $$   
\end{proposition}
\begin{proof}
We use the notation and the expansion of $\delta^{-1}$ given in Remark \ref{oss:inversadidelta}, so that $\delta=u-v$. By continuity it is enough to prove the formula for $f(u,v)=g(u)h(v)$. Write $g=\sum _{m,i} \gra_{m,i}\grf^+_{m,i}$ with $\gra_{m,i}\in A$. Then, using the duality among the $\grf^-_{m,i}$ and the $\grf^+_{m',i'}$
we get 
\begin{align*}
	\Tr(f \cdot Exp^r(\delta^{-1})) 
&= \sum_{m\geq 0,i} \Res_{\Sigma }\Big( \grf^-_{-m-1, n-i-1}(t)g(t)dt \Big)  h(v) \grf^+_{m,i}(v) \\ &= h(v) \sum_{m\geq 0,i} \gra_{m,i} \grf^+_{m,i}(v)
\end{align*}
Similarly 
\begin{align*}
	\Tl(f \cdot Exp^\ell(\delta^{-1}))
&= - \sum_{m\geq 0,i} \Res_{\Sigma }\Big( \grf^-_{m,i}(t)g(t)dt \Big)  h(v) \grf^+_{-m-1,n-i-1}(v) \\
&= - h(v) \sum_{m\geq 0,i} \gra_{-m-1,n-i-1} \grf^+_{-m-1,n-i-1}(v)
\end{align*}
Putting together the two pieces we get the thesis. 
\end{proof}

\subsection{Operations on modules and \texorpdfstring{$\calD$}{D}-modules}\label{sez:Dmoduli}

Given $\pi:A\surjmap B$ a surjective map as before, the associated diagonal embedding
$\Delta(\pi):X^B\hookrightarrow X^A$ induces maps among quasi coherent sheaves and $D$-modules on the two spaces. We can rephrase these constructions at the level of $\pbarOm{A}$ and $\pbarDm{A}$ modules without any change. \index{$\pbarD$-modules}

In the previous Section we have already defined the map $\Delta(\pi)^\sharp: \pbarOm{A} \to \pbarOm{B}$: the continuous morphism of topological algebras induced by $\otimes^*_a f_a \mapsto \otimes_b^* ( \prod_{a \in I_b} f_a)$. We denote by $\calI_{\Delta(\pi)}$ the kernel of this map. 

\subsubsection{Inverse image of \texorpdfstring{$\calD$}{D} modules} \label{ssez:inversoDmoduli}
Given a $\pbarOm{B}$ module $\calM$ we define $\Delta(\pi)_*\calM$ \index{$\pbarD$-modules!inverse image of $\pbarD$ modules}
as the $\pbarOm{A}$-module that as a sheaf is equal to $\calM$ and where the action of $\pbarOm{A}$ is given by the map $\Delta(\pi)^\sharp$. Similarly, given a $\pbarOm{A}$-module $\calN$, we define 
$\Delta(\pi)^*\calN$ as $\calN/\calI_{\Delta(\pi)}\calN$. In case $\calN$ is a topological sheaf we equip $\Delta(\pi)^*\calN$ with the quotient topology. 

If $\calN$ is a left $\pbarDm{A}$-module, then, as in the usual case, $\Delta(\pi)^*\calN$ has a natural structure of left $\pbarDm{B}$-module. 

If $\calN$ is a right $\pbarDm{A}$-module, we define also the right $\pbarDm{B}$-module $\Delta(\pi)^!\calN$ as the sheaf
$$
\Delta(\pi)^!\calN=\{n\in \calN : n\cdot \calI_{\Delta(\pi)}=0\}.
$$
This is clearly a $\pbarOm{B}$ module. The action of $\pbarDm{B}$ is defined as follows. Let $\pbarDm A(\pi)$ be the subalgebra of $\pbarDm A$ operators $P$ such that $P(\calI_{\Delta(\pi)})\subset \calI_{\Delta(\pi)}$. Notice that $\calI_{\Delta(\pi)}\cdot \pbarDm A$ is a bilateral ideal of $\pbarDm A(\pi)$ and that the quotient $\pbarDm A(\pi)/ \calI_{\Delta(\pi)}\cdot \pbarDm A$ acts on $\pbarOm B$ and it is
isomorphic to $\pbarDm B$. Finally notice that the action of $\pbarDm A(\pi)$ preserves $\Delta(\pi)^!\calN$ and that the ideal $\calI_{\Delta(\pi)}\cdot \pbarDm A$ acts as zero on $\Delta(\pi)^!\calN$, hence induces an action of $\pbarDm B$ on $\Delta(\pi)^!\calN$.

If $\calN$ is a topological sheaf, then we equip $\Delta(\pi)^!\calN\subset\calN $ with the subsheaf topology. If $\calN$ is complete, then it is also complete. 

\subsubsection{Direct image of \texorpdfstring{$\calD$}{D}-modules}\label{ssez:direttaD}
 If $\calM$ is a right  $\pbarDm{J}$- module we define the direct image $\Delta(\pi)_!\calM$ following exactly the same steps of the usual situation. We need to recall the definition in order \index{$\pbarD$-modules!direct image of $\pbarD$ modules}
to describe the topology. 

Let $\calD_\pi=\Delta(\pi)^*(\pbarDm{A})$. This is a left $\pbarDm{B}$-module and a right $\pbarDm{A}$ module. We have a local description of this bimodule similar to the one given for $\pbarDm{A}$: choose coordinates $x=(x_b)_{b\in B}$ and $y=(y_c)_{c\in A\senza B}$ such that the ideal $\calI_{\Delta(\pi)}$ is generated by the coordinates $y_c$ and $x_b$ restricts to coordinates of $X^B$. Then, as in the classical case, we have:
$$ 
\calD_\pi=\bigoplus_{\grb,\grg} \pbarOm{B}\partial_x^\grb \partial_y^\grg=\bigoplus_{\grg} \pbarDm{B} \partial_y^\grg.
$$
Define also $\calD_\pi^{\leq n}$ as the left $\pbarDm{B}$-submodule of $\calD_\pi$ generated by the image of $\Delta(\pi)^*((\pbarDm{A})^{\leq n})$. In the local description above we have 
$$
\calD^{\leq n}_\pi=\bigoplus_{\grb,|\grg|\leq n} \pbarOm{B}\partial_x^\grb \partial_y^\grg=\bigoplus_{|\grg|\leq n} \pbarDm{B} \partial_y^\grg.
$$
If $\calM$ is a right $\pbarDm{B}$ module then set
$$
\Delta(\pi)_!(\calM) = \calM\otimes_{\pbarDm{B}} \calD_\pi \quad \mand \quad \Delta(\pi)^{\leq n}_!(\calM)=\calM\otimes_{\pbarDm{B}}\calD^{\leq n}_\pi.
$$
The first one is is a right $\pbarDm{A}$-module. Notice that $\calD_\pi^{\leq 0}\simeq \pbarDm{B}$ and 
$\Delta(\pi)^{\leq 0}_!(\calM)=\Delta(\pi)_*(\calM)$ and it generates $\Delta(\pi)_!(\calM)$ as a $\pbarDm{A}$-module. Indeed
$\Delta(\pi)^{\leq n}_!(\calM) =\Delta(\pi)_*(\calM)\cdot (\pbarDm A)^{\leq n}$ 
and
$\Delta(\pi)_!(\calM)$  is the colimit of the sequence of inclusions
$\Delta(\pi)^{\leq n}_!(\calM)\subset \Delta(\pi)^{\leq n+1}_!(\calM)$.  

We now assume that $\calM$ is a right $\pbarD^B$-module \cgmod module and 
we equip $\Delta(\pi)_!(\calM)$ with the structure of a $\pbarD^A$-module \cgmod module.
We describe first the topology of $\Delta(\pi)^{\leq n}_!(\calM)$ and then we define the 
topology of $\Delta(\pi)_!(\calM)$ as the colimit topology. A \fsonoz of $\Delta(\pi)^{\leq n}_!(\calM)$ is 
constructed as $\calU\cdot (\pbarDm{A})^{\leq n}$ where $\calU$ is a \noz of $\calM=\Delta(\pi)_*(\calM)\subset \Delta(\pi)_!(\calM)$. 

Locally we have that $\Delta(\pi)^{\leq n}_!(\calM)$ is isomorphic 
to the direct sum of a finite number of copies of $\calM$. To check this recall that $(\pbarDm{B})^{\leq n}$ is $w$-compact, hence for all 
\noz $\calU$ of $\calM$ there exists a \noz $\calV$ of $\calM$ such that $\calV \cdot (\pbarDm{B})^{\leq n}\subset \calU$.

As in the usual case, we have the following adjunctions:
\begin{align*}
\Hom_{\pbarOm{A}} \left(\Delta(\pi)_*\calF,\calG\right)&=
\Hom_{\pbarOm{B}} \left(\calF,\Delta(\pi)^!\calG\right)
\quad \mand\\
\Hom_{\pbarDm{A}} \left(\Delta(\pi)_!\calM,\calN\right)&=
\Hom_{\pbarDm{B}} \left(\calM,\Delta(\pi)^!\calN\right)
\end{align*}
if $\calF$ is  a $\pbarOm{B}$-module, $\calG$ a $\pbarOm{A}$-module,
$\calM$ a right $\pbarDm{B}$-module and $\calN$ a right $\pbarDm{A}$-module. In the topological case, it is easy to check that the adjunction maps $\calM\lra \Delta(\pi)^!\Delta(\pi)_!\calM$ and $\Delta(\pi)_!\Delta(\pi)^!\lra \calN$
are continuous, hence we have also
\begin{align*}
\Homcont_{\pbarOm{A}} \left(\Delta(\pi)_*\calF,\calG\right)&=
\Homcont_{\pbarOm{B}} \left(\calF,\Delta(\pi)^!\calG\right)
\quad \mand\\
\Homcont_{\pbarDm{A}} \left(\Delta(\pi)_!\calM,\calN\right)&=
\Homcont_{\pbarDm{B}} \left(\calM,\Delta(\pi)^!\calN\right).
\end{align*}
These maps have the usual functoriality properties, i.e. $\Delta( \pi'\circ\pi)_!=\Delta( \pi)_!\circ \Delta( \pi')_!$ and similarly for $\Delta( \pi)^!$.

\begin{remark} 
	If $\calF$ is $\calO_X$ module on $X$ in Section \ref{ssec:definizionefascibarF} we defined the topological sheaf 
	$\pbarF{\calF}$ on $\Sigma$. If $\calF$ is a a $\calD_{X/S}$ module then $\pbarF{\calF}$  is a a $\pbarD$-module. Similar construction can be given in the case of $X^A$ obtaining a  $\pbarDm{A}$ module. The definition given above commutes with the usual operation among $\calD_{X^A/S}$-modules.
\end{remark}

\begin{remark}[Behaviour under pullback]\label{rmk:pullbackpushdifferential}
	Consider a morphism $\varphi : S' \to S$ and let $\Sigma'$ be the collection of sections of $X'=X\times_S S' \to S'$ induced by $\Sigma$. Denote by $\Delta_{\oSigma,*},\Delta_{\oSigma',*}$ the pushforward for $\Sigma$ and $\Sigma'$ case respectively; analogously, write $\Delta_{\oSigma,!},\Delta_{\oSigma',!}$ for the corresponding $\calD$-module pushforward. Then it follows by $\hat{\varphi}^*\pbarOm{n} = \calO_{\overline{\Sigma}'}^n$ (resp. Remark \ref{rmk:pullbacktangentdifferential}) that for any $\pbarOm{B}$-modules (resp. $\pbarDm{B}$-module $\calF$) and any surjection $\pi : A \to B$ we have 
	\[
		\hat{\varphi}^*\Delta_{\oSigma,*}(\pi)\calF \simeq \Delta_{\oSigma',*}(\pi)\left(\hat{\varphi}\calF\right) \quad \left( \text{resp. }\hat{\varphi}^*\Delta_{\oSigma,!}(\pi)\calF \simeq \Delta_{\oSigma',!}(\pi)\left(\hat{\varphi}\calF\right) \right).
	\]
\end{remark}

\subsubsection{The case of the ring \texorpdfstring{$\pbarOpoli$}{functions with poles}}
Similar considerations and constructions can be given for the sheaves  $\pbarOpolim{A}$ and $\pbarDpolim{A}$. We denote with the same symbol $\Delta(\pi)^\sharp$ the map among the rings $\pbarOpolim{A}$ and $\pbarOpolim{B}$ and with $\calJ_{\Delta{(\pi)}}\subset \pbarOpolim{A}$ its kernel \index{$\calJ_{\Delta{(\pi)}}$}. We notice that this ideal of $\pbarOpolim A$ is generated by $\calI_{\Delta(\pi)}$. 
In particular, if $\calN$ is a right $\pbarDpolim{A}$ \cgmod module  then the topological sheaves $\Delta(\pi)^!\calN$ and
$\Delta(\pi)_!\calM$ is exactly the same as the one defined for $\pbarDm{A}$ \cgmod module. 

A minor difference between the case with poles and the case without poles is in the definition of the topology on $\Delta(\pi)_!\calM$. 
In this case with poles the sheaves $(\pbarDpolim{B})^{\leq n}$ are not anymore $w$-compact. We can define $\Delta(\pi)_!\calM$ using an \exhaustion of $\pbarDpolim{B}$ or, equivalently, we can notice that  if $\calM$ is a right $\pbarDpolim{B}$ \cgmod module then it is also a a right $\pbarDm{B}$ \cgmod module, and we can notice that on the sheaf $\Delta(\pi)_!\calM$ defined in the previous section there is a \cgmod right action of $\pbarDpolim{A}$.

\subsubsection{Kashiwara's Theorem}\label{sssec:Kashiwara} In the quasi coherent setting Kashiwara's Theorem gives an equivalence between $\calD$-modules on a variety $\calX$ supported on a smooth subvariety $\calY$ and the $\calD$-modules on $\calY$. The same argument proves the following Proposition.

\begin{proposition}\label{prop:Kashiwara} Assume (as we always do) that $S$ is topologically noetherian. Let $\pi:A\surjmap B$ be a surjective map among finite sets.
\begin{enumerate}[\indent a)]
\item Let $\calN$ be a $\pbarDm B$ module, then $\Delta(\pi)^! \Delta(\pi)_!\calN \simeq \calN$ 
\item Let $\calM$ be a right $\pbarDm{A}$-module and define the subpresheaf 
$$ \calM^{\pi\loc}(U)=\{\grf\in\calF(U)\,:\, \text{exists }m>0\text{ such that } \grf\cdot \calI_{\Delta(\pi)}^m=0\} . $$
on an open subset $U$ of $S$. 
Then $\calM^{\pi\loc}$ \index{$\calM^{\pi\loc}$} is a sheaf and it is isomorphic to $\Delta(\pi)_!\Delta(\pi)^!\calM$.
\end{enumerate}
Completely analogous statements hold for $\pbarDpolim{B}$ and $\pbarDpolim{A}$-right modules. 
\end{proposition}
The proof of this Proposition is equal to the proof of Kashiwara' Theorem (see for example \cite{HTTDmoduli}, Section 1.6). We sketch the proof of point b) whose statement looks different from the one on quasi coherent $\calD$-modules. The proof gives also a description of this sheaf. 

\begin{proof}[Sketch of proof]The fact that it is a sheaf follows from the quasi compactness of open subsets of $S$. 
To prove the second statement, by functoriality we can reduce to the case where $B$ has one element 
less than $A$. We prove that there is a natural isomorphism
$$ \calHom_{\pbarDm{A}} \left(\calM^{\pi\loc},\calN\right)\simeq 
\calHom_{\pbarDm{B}} \left(\Delta(\pi)^!\calM,\Delta(\pi)^!\calN\right) $$
This implies that $\calM^{\pi\loc}$ has the correct universal property. Notice that $\Delta(\pi)^!\calM$ is a subsheaf of $\calM^{\pi\loc}$ hence there is a natural morphism from the left hand side to the right hand side given by restriction. To prove that it is an isomorphism we can cheack the claim locally. So we can assume that we have local coordinates $x_1,\dots,x_n,y$ for $X^A$ such that $\calI_{\Delta(\pi)}$ is generated by $y$.  Consider the operator $h=\partial_y \cdot y$. As in the case of the proof of Kashiwara's Theorem (see for example \cite{HTTDmoduli}, Section 1.6), we have
\begin{equation}\label{eq:Kashiwara} \calM^{\pi\loc}=\bigoplus_{i\geq 1} \calM^{\pi\loc}(i) \end{equation}
where $\calM^{\pi\loc}(i)$ is the $h$-eigenspace of eigenvalue equal to $i$ and $\partial_y:\calM^{\pi\loc}(i)\lra\calM^{\pi\loc}(i+1)$ is an isomorphism for all $i\geq 1$. By noticing that $\calM^{\pi\loc}(1)$ is equal to $\Delta(\pi)^!\calM$ we get the claim. 

The case of $\pbarDpolim{A}$-modules is completely analogous. Following exactly the same steps we prove that 
$\Delta(\pi)_!\Delta(\pi)^!\calM$ is isomorphic to the space of sections annihilated by some power of $\calJ_{\Delta(\pi)}\subset \pbarOpolim{A}$. However this ideal is generated by $\calI_{\Delta(\pi)}$ hence the two conditions are equivalent. 
\end{proof}

\begin{remark}
Notice that we proved that, as a sheaf $\calM^{\pi\loc}$ is isomorphic to $\Delta(\pi)_!\Delta(\pi)^!\calM$. The given proof shows also that 
for a \QCC sheaf, if we define $\calM^{\pi\loc,\leq m}$ \index{$\calM^{\pi\loc,\leq m}$} the subsheaf of sections killed by $\calI_\Delta^m$ then 
$$
\Delta(\pi)^{\leq m-1}_!\Delta(\pi)^!\calM \simeq \calM^{\pi\loc,\leq m}
$$
as a topological sheaf. So we have a continuous map from $\Delta(\pi)_!\Delta(\pi)^!\calM$  to
$\calM^{\pi\loc}$. However in general this does not need to be an isomorphism. 

Conversely the isomorphism $\Delta(\pi)^! \Delta(\pi)_!\calN \simeq \calN$ is always an isomorphism of topological sheaves. 
\end{remark}

\subsubsection{The de Rham residue map}\label{sec:derhamresidue}
The decomposition of equation \eqref{eq:Kashiwara} is local, depending on the choice of $\partial_y$ or equivalently on the coordinate $x_1,\dots,x_n$. In particular the projection onto $\calM^{\pi\loc}(1)=\Delta(\pi)^!\calM$ is not well defined. However, in the case of the inclusion $\Delta\subset X^2$, we have two canonical choices and we choose one of them. 

\begin{definition}Let $\tsigma \otimes 1 \subset \calD^2_{\overline{\Sigma}}$ be the sheaf of derivations of the form $\partial \otimes 1$. Define the de Rham cohomology with respect to the first component of a right $\pbarDm{2}$-module $\calM$ as
$$
h^0_r(\calM)=\frac{\calM}{\calM\cdot \tsigma\otimes 1}
$$
This has a structure \index{$h^0_r$} of right $\pbarD$-module, where the action of an operator $D\in\pbarD$ is induced by the action of $1\otimes D\in \pbarDm{2}$ on $\calM^{\pi-loc}$.
\end{definition}

The discussion in the proof of Proposition \ref{prop:Kashiwara} implies the following result. 
In the case of the inclusion of the diagonal $\Delta \subset X^2$ we write $\calM^{\text{loc}}$ to denote   $\calM^{\pi\loc}$ for the appropriate $\pi$. 

\begin{corollary}
Let $\calM$ be a right $\pbarDm{2}$-module. Then the inclusion 
$\Delta^!\calM\subset \calM^{\oloc}$ induces an isomorphism of right $\pbarD$-modules
$$
\Delta^!\calM \lra h^0_r(\calM^{\oloc})\;\left(\simeq h^0_r(\Delta_!\Delta^!\calM)\right).
$$
\end{corollary}

\begin{proof}
We prove that $\calM^{\oloc}=\Delta^!\calM\oplus \calM^{\oloc}\cdot  (T_{\oSigma}\otimes 1)$.
Let $t$ be a local coordinate over $X$ and set $u=t\otimes 1$ and $v=1\otimes t$ and let $\partial_u$ and $\partial_v$ the dual choice of derivation. Set also $y=u-v$ and $x_1=v$. If we denote by $\partial_y$ and $\partial_{x_1}$ the dual choice of derivation then we have $\partial_y=\partial_u$ and $\partial _{x_1}=\partial_u+\partial_v$. Let also $h=\partial_y \cdot y$. Then in the notation of formula \eqref{eq:Kashiwara} we  have that 
$$ \bigoplus_{i\geq 2}\calM^{\oloc}(i) = \calM^{\oloc}\cdot  (T_{\oSigma}\otimes 1) $$
which implies that $\Delta^!\calM$ and $h^0_r(\calM^{\loc})$ are isomorphic. We need to check that this is $\pbarD$-equivariant. This 
goes down to the definition of the action of $\pbarD$ on $\Delta^!\calM$ (see Section \ref{ssez:inversoDmoduli}). Indeed by construction $\partial_t$ acts like $\partial_u+\partial _v$ in 
$\calM^{\oloc}=\Delta_!\Delta^!\calM$, hence like $\partial_v$ on $h_r^0(\calM^{\oloc})$.
\end{proof}

As a consequence we can define the following residue map. 

\begin{definition}\label{def:dR}
For any right $\pbarDm{2}$-module $\calM$, the projection $\calM^{\oloc}\lra h^0_r(\calM^{\oloc})$ determines a de Rham residue map\index{de Rham residue}\index{$dR_r$}
$$
dR_r:\calM^{\oloc}=\Delta_!\Delta^!\calM \lra \Delta^!\calM.
$$ 
which \index{$dR_r$} is $\pbarD$ equivariant, where, as above, a differential operator $D\in\pbarD$ acts on $\calM^{\text{loc}}$  as $1\otimes D\in \pbarDm{2}$.
\end{definition}

\subsubsection{Local description}\label{ssez:tuvy} For our convenience we describe in one example the combinatorics of our constructions. Assume that $S=\Spec \mC$ and that $X=\Spec \mC[t]$. Let $\Delta(t)=(t,t)$ be the diagonal map. Let $u=t\otimes 1$, $v= 1\otimes t$ and let $\partial _u$, $\partial_v$ be the dual choice of derivations. The ideal $\calI_\Delta$ is generated by $y=u-v$ and let $\partial_y$ and $\partial'_v$
the derivations dual to $v$ and $y$, so that $\partial_y=\partial_u$ and $\partial_v'=\partial_u+\partial_v$. 

We start by describing the module $\Delta^!N$. If $N$ is a $A[u,v,\partial_u,\partial_v]$-right module, then $M=\Delta^! N$ is the submodule of $N$ annihilated by $y$. The subalgebra of $\mC[u,v,\partial_v',\partial_y]$ which preserves the ideal generated by $y$ is generated by $y$, $v$, $\partial_v'$ and $y\partial_y$. The quotient $\mC[y,v]/(y)$ is isomorphic to $\mC[t]$ via the map $v\mapsto t$. In particular the derivation $\partial'_v=\partial_u+\partial_v$ acts as $\partial_t$ on $\mC[t]$. 
Hence $\partial_t$ acts like $\partial_u+\partial_v$ on $M$ and $t$ acts like multiplication by $v$ (or equivalently by $u$).

The description of $N =\Delta_!M$, if $M$ is a $\mC[t,\partial_t]$ right module is even easier. We have that $N=\Delta_!M$ is the $\mC[v,y,\partial_v',\partial_y]$-right module $\bigoplus_{n\geq 0} M\partial_y^n$ where the action is given by 
\begin{align*}
m\partial_y^n \cdot v &=  (mt)\partial_y^n, \qquad 
m\partial_y^n \cdot y =  -n m\partial_y^{n-1}, \\
m\partial_y^n \cdot \partial'_v &= (m\partial_t) \partial_y^n,  \qquad 
m\partial_y^n \cdot \partial_y = m\partial_y^{n+1}.  
\end{align*}
As a consequence \begin{align*}m\partial_y^n \cdot \partial_v &= m\partial_y^n \cdot (\partial_v' -\partial_y)= (m\partial_t)\partial_y^{n} - m\partial_y^{n+1} \text{ and } \\
m\partial_y^n \cdot u &= m\partial_y^n \cdot (y+v)= (mt)\partial_y^{n} - n m\partial_y^{n-1}\end{align*}.

In particular this describes $N^{\text{loc}}=\Delta_!\Delta^!N$. Notice that $\partial_y=\partial_u$ hence $N^{\text{loc}}\cdot \tsigma \otimes 1=\bigoplus_{n\geq 1} M\partial_ y^n$
and
$$
h^0_r (N^{\text{loc}})=\frac {N^{\text{loc}}}{N^{\text{loc}}\cdot \tsigma \otimes 1} = M
$$
where $u$ and $v$ acts like $t$ and $\partial_v$ acts like $\partial_t$.

\subsubsection{External tensor products}\label{ssec:prodottiesterni}
Let $I=J\sqcup H$ and let $\calM_J$, resp.\ $\calM_H$, be a $\pbarOm{J}$-topological sheaf, resp.\ a $\pbarOm{H}$-topological sheaves.
Then, by Remark \ref{oss:prodottoalgebremoduli}, the module $\calM_I=\calM_J\tensor{\rightarrow}\calM_H$ is a $\pbarOm{I}$-topological sheaf.

Notice that this remark can be extended to $\calD$-modules in the following way. Notice that $\pbarDm{I}=\pbarDm{J}\tensor{cg}\pbarDm{H}$, hence by
 Remark \ref{oss:algebreregolariemoduli} if the sheaves $\calM_J$ and $\calM_H$ are \cgmod modules of the \ccg-topological algebras $\pbarDm{J}$ and $\pbarDm{H}$, then the tensor products  
$\calM_J\tensor{\rightarrow}\calM_H$ and 
$\calM_J\tensor{*}\calM_H$ are in a natural way $\pbarDm{I}$ \cgmod module.
To stress that we consider these tensor products as $\pbarDm{I}$-modules (or $\pbarOm{I}$-modules) we use the notations
$\calM_J\exttensor{\rightarrow}\calM_H$ and 
$\calM_J\exttensor{*}\calM_H$. \index{external tensor product $\exttensor{}$}

\subsubsection{Poles along the diagonal and pushforward of \texorpdfstring{$\calD$}{D}-modules} In the following Proposition we recall some relations which relate admitting poles along the diagonals and taking $\calD$-module pushforward. If $\calM$ is topological $\pbarO$-sheaf then, since $\pbarOmega$ is compact, we have  $\pbarOmega\exttensor{*}\calM=\pbarOmega\exttensor{\rightarrow}\calM$, and as explained above this is a topological-$\pbarOq$ sheaf and, if $\calM$ is right $\pbarD$-topological sheaf is also right $\pbarDm{2}$-sheaf and the action, restricted to $(\pbarDm{2})^{\leq n}$ is continuous for all $n$.

The following Proposition rephrase the description of pushforward of right $\calD$-modules along the diagonal for topological modules. 

\begin{proposition}\label{prop:defD}\
    The following hold:
    \begin{enumerate}[\noindent a)]
      \item For every complete $\pbarO$-topological sheaf $\calM$ we have the following isomorphisms of topological modules.
    	$$\Delta^!\left(\frac{\pbarOmega\exttensor{*}\calM(\infty\Delta)}{\pbarOmega\exttensor{*} \calM}\right)\simeq \calM\quad\mand\quad 
      	\Delta_*(\calM)\simeq\frac{\pbarOmega\exttensor{*}\calM(\Delta)}{\pbarOmega\exttensor{*} \calM}$$
      \item If $\calM$ is a complete right $\pbarD$-topological sheaf then the first isomorphism of point a) induces a canonical map 
        \[
            D(\calM) : \Delta_! \calM \lra \frac{\pbarOmega\exttensor{*}\calM(\infty\Delta)}
            {\pbarOmega\exttensor{*} \calM}
        \]
        which is an isomorphism of right topological $\pbarDm{2}$-modules. \index{$D(\calM)$}
    \end{enumerate}
\end{proposition}

\begin{proof}[Sketch of proof]
	The proof follows exactly the same steps of the analogous statement in the case of a $\calD$ modules on a curve $X$. 
    We just recall as the map $\pbarOmega \exttensor{*}\calM(\Delta)\lra \Delta_*\calM$ is defined. 
    Let $d_2:\pbarOq\lra \pbarOmega$ be the continuous extension of  
    $d_2(f\otimes g)=fdg$. This induces an isomorphism $\calI_\Delta/\calI^2_\Delta \simeq \pbarOmega$. It is easy to check that the map
    \begin{equation}
    \pbarOmega\exttensor{*}\calM(\Delta)=\left(\frac {\calI_\Delta}{\calI^2_\Delta}\exttensor{*} \calM\right)\otimes_\pbarOq \Hom_{\pbarOq}\left(\calI_\Delta, \pbarOq \right)\lra \Delta_*(\calM)
    \end{equation}
    	given by $\gra\boxtimes m \otimes \grf\mapsto \Delta^\sharp(\grf(\gra)) m$ is well defined and factors through the quotient by $\pbarOmega\exttensor* \calM$. The fact that the induced map 
    	\begin{equation}\label{eq:definizioneD}
    	\Phi_\calM :\frac{\pbarOmega\exttensor{*}\calM(\Delta)}{\pbarOmega\exttensor{*}\calM}\lra \Delta_*(\calM)
    	\end{equation}
    	is an isomorphism can be checked locally. 
    	
    	The isomorphism on the left hand side in point a) comes from $\Phi_{\calM}$ as well. Finally, $D(\calM)$ is obtained as the morphism obtained by adjunction from the inverse of the last isomorphism.
\end{proof}

We will need also the following functoriality property which follows immediately from the definition. 

\begin{lemma}\label{lem:pq}
 Let $p : I \twoheadrightarrow J$ and $q : J \twoheadrightarrow K$ be surjections of finite sets and let $\calM$ be a right $\pbarDm{J}$-module. Then there is a canonical isomorphism of right $\pbarDm{I}$-modules:
\[
\Delta(p)_! \left(\calM\big(\infty\Divisorepi(q)\big)\right) \simeq \big(\Delta(p)_! \calM \big)\Big(\infty\Divisorepi(q\circ p)\Big).
\]
\end{lemma}

\section{Distributions and fields}\label{sec:distributionandfields}

In what follows we fix a family of smooth curves $X \to S$, a finite set $I$ and collection of sections $\Sigma = \{ \sigma_i \}_{i \in I}$, keeping the notation of the previous sections. We will fix a new object as well: $\calU$ ,\index{$\calU$}which we assume to be an associative, unital, complete topological $\calO_S$ algebra with topology generated by left ideals (and therefore a $\otimesr$-topological algebra). For some results we will assume also that $\calU$ is locally a \QCCF sheaf. 

\begin{definition}\label{def:campi}With the notation above, we  define the sheaf of $\calU$-valued {\em fields} on $\pbarOpoli$ as the sheaf of continuous maps: 
	$$
	\mF^1_{\Sigma,\calU}=\calHom^{\text{cont}} _{\calO_S} \left(\pbarOpoli,\calU\right). 
	$$\index{$\mF^1_{\Sigma,\calU}$}
	More in general, for a finite set $A$, define the sheaf of $\calU$-valued $A$-fields on $\pbarOpoli$ as
	\[
	\mF^A_{\Sigma,\calU}=\calHom^{\text{cont}} _{\calO_S} \left(\pbarOpolim{A},\calU\right).
	\]\index{$\mF^A_{\Sigma,\calU}$}
	Sometimes we denote by $\mF^n_{\Sigma,\calU}$ the space of $A$-field where $A$ has $n$ elements and in particular for $A=\{ 1,\dots,n\}$.
	
	The space of $A$-fields is naturally a $\pbarDpolim{A}$ right module, dualizing the left actions of  $\pbarDpolim{A}$ on $\pbarOpolim{A}$:
	$$
	(X\cdot P) (g)=X(P\cdot g)
	$$
	for any $P \in \pbarDm{A},g\in \pbarOpolim{A}$ and $X\in\mF^A_{\Sigma,\calU}$. 
\end{definition}

A \fsonoz of the space of fields is given by the subspaces $\mH_{k,\calV}$ consisting of those fields $X$ which verify $X(\pbarOm{A}(k))\subset \calV$ where  
$\pbarOm{A}(k)=\tensor{*}_A\pbarO(k\Sigma)$ and $\calV$ is a \noz of $\calU$. This topology makes the space of field a topological $\pbarOm{A}$ module and a  $\pbarOpolim{A}$ \ccgmod module as well as a $\pbarDpolim{A}$ \ccgmod module. In the case of $\mF^1$ this is also a topological $\pbarOpoli$-module, not only a \ccgmod module. 

The subsheaves $\pbarOm{A}(k)$ are an \exhaustion of $\pbarOpolim{A}$ by (\QCCF) compact modules, hence, by Remark \ref{rmk:completezzaHom} if $\calU$ is complete also the space of fields is complete. Finally, by Remark \ref{rmk:sectionshom} it is easy to describe its sections. 

\begin{lemma}\label{lem:sezionicampi}Let $U = \Spec(A) \subset S$ is an affine well covered open subscheme of $S$, the
	\[
	\mF^A_{\Sigma,\calU}(U) = \Hom^{\text{cont}}_A\left(\pbarOpolim{A}(U), \calU(U)\right).
	\]
\end{lemma}
\begin{remark}[Local description]\label{rmk:localdescfield}
	Using the description given in Section \ref{ssec:descrizionelocalepbarO} together with Lemma \ref{lem:sezionicampi} we see that for well covered open subset $U \subset S$ the space $\mF^1_{\Sigma,\calU}(U)$ coincides with the space of fields considered in Definition 3.0.4 of \cite{cas2023}.
\end{remark}

\subsection{The case of \texorpdfstring{$\calU$}{U} \QCCF}\label{ssec:campiQCCF}
Recall that on any well covered affine open subset $U$ the sheaf $\pbarOpolim{A}$ has an \exhaustion by proj-compact modules which have a complement. Hence, by Remark \ref{oss:morfismidifasci1} if $\calU$ is a \QCC sheaf on any well covered open affine subset we have 
$$
\mF^A_{\Sigma,\calU} = \qcc{H} \quad\text{ where }\quad  H=\Homcont_A\left(\pbarOpolim{A}(U), \calU(U)\right)
$$
in particular the sheaf of fields is itself a \QCC sheaf.
If, moreover, we assume $\calU$ to be a \QCCF sheaf also the space of fields is \QCCF. 

\begin{lemma}Assume that $\calU$ is a \QCCF sheaf, then also the space of $A$-fields is a \QCCF sheaf.
\end{lemma}
\begin{proof}
	By what noticed above it is enough to prove that over a well covered affine subset $U=\Spec A$ such that $\calU(U)$ is a \QCCF-module, then 
	$$
	\Homcont\left(\pbarOpolim{A}(U),\calU(U)\right)
	$$
	is a \QCCF module. Let us fix $k\in \mZ$ and $V\subset \calU(U)$ a \noz such that $\calU(U)=V\oplus L$ with $L$ a free discrete module. 
	Notice that $\pbarOm{A}(k)(U)$ has a complement in $\pbarOm{A}(U)$. Hence we have the following exact sequence
	$$
		0 \to \mH_{k,V}(U) \to \Homcont\left(\pbarOpolim{A}(U),\calU(U)\right) \to \Homcont\left(\pbarOm{A}(k)(U),L\right) \to 0.
	$$
	Hence it is is enough to prove that the last module in this sequence is a free $A$-module. Recall from Section \ref{ssec:descrizionelocalepbarO} that the modules $\pbarOm{A}(k)(U)$, are a product of countable many copies $Ae_i$ of $A$. Hence, being $L$ discrete, we deduce that $\Homcont\left(\pbarOm{A}(k)(U),L\right)$ is isomorphic to the sum of countably many copies of $L$. 
\end{proof}

\subsection{Local fields}\label{ssez:campilocali}

In this section we introduce local fields and prove some of their basic properties.

\begin{definition}
	Let $\pi :A \surjmap B$ be a surjection of finite sets. Recall that, as defined in the previous section, we denote by $\calJ_{\Delta(\pi)}$ the kernel of  $\Delta(\pi)^\# : \pbarOpolim{A} \to \pbarOpolim{B}$. We define the sheaf of $\pi$\emph{-local} $A$-fields as
	\[
	\mF^{A,\pi\loc}_{\Sigma,\calU} \stackrel{\text{def}}{=} \limind_n \calHom _{\calO_S} \left(\,\frac{\pbarOpolim{A}}{\calJ_{\Delta(\pi)}^n},\, \calU\right).
	\]
	\index{$\mF^A_{\Sigma,\calU}$:$\mF^{A,\pi\loc}_{\Sigma,\calU}$}this is the subsheaf  of $A$-fields  killed by a sufficiently high power of $\calJ_{\Delta(\pi)}$. 
\end{definition}

Recall from Section \ref{sssec:Kashiwara} the description of $\Delta(\pi)_!\Delta(\pi)^!\calM$ as $\calM^{\pi\loc}$. In the case of the standard diagonal $\Delta:X\lra X^2$, we introduced in Definition \ref{def:dR} 
the de Rham residue map $ dR_r:\calM^{loc}\lra \Delta^!\calM$ which is $\pbarD$-equivariant. In the following Lemma we describe these objects in the case of the sheaf of  fields.

\begin{lemma}\label{lem:campilocali}
	There are natural isomorphisms 
	\[
	\Delta(\pi)^!\left(\mF^A_{\Sigma,\calU}\right)\simeq \mF^B_{\Sigma,\calU}\ \qquad\mand \qquad \Delta(\pi)_!\left(\mF^B_{\Sigma,\calU}\right) \simeq \mF^{A,\pi\loc}_{\Sigma,\calU}.
	\]
	of right $\pbarDm{A}$-modules. In addition, 
	the de Rham residue map $dR_r:{\mF^{2\loc}_{\Sigma,\calU}}\lra {\mF^{1}_{\Sigma,\calU}}$ coincides with $\piuno$ \index {$\piuno$}, defined as
\begin{equation}\label{eq:piuno}
\piuno: \mF^2_{\Sigma,\calU} \to \mF^1_{\Sigma,\calU} \qquad \big(\piuno (X)\big) (f)=X(1\otimes f).
\end{equation}
\end{lemma}

\begin{proof}In the first isomorphism the left hand side is given by $A$-fields killed by $\calJ_{\Delta(\pi)}$, hence using that $\pbarOpolim{B}\simeq \pbarOpolim{A}/\calJ_{\Delta(\pi)}$ it immediately follows that $\Delta(\pi)^!\mF^A_{\Sigma,\calU} \simeq \mF^B_{\Sigma,\calU}$.
	
	The second isomorphism follows from the first using Kashiwara's Theorem (see section \ref{sssec:Kashiwara}).
	
	To prove the claim on the de Rham residue map (c.f. \ref{sec:derhamresidue}) it is enough to check that $\piuno$ vanishes on ${\mF^{2\loc}_{\Sigma,\calU}}\cdot (T_{\oSigma}\otimes 1)$
	and that pre-composing $\piuno$ with the immersion $\mF^1_{\Sigma,\calU} = \Delta^!\mF^2_{\Sigma,\calU} \subset \mF^2_{\Sigma,\calU}$ gives us the identity. The second claim is trivial, we check the first one:
	\[
	\big(   \piuno(X \cdot (\partial \otimes 1 )) \big)(f)=X \big((\partial \cdot 1)\otimes f\big)=0.\qedhere
	\]
\end{proof}

\subsection{Tensor products of the space of fields}\label{sssec:prodottiordinaticampi} 
Being the space of $A$-field a \ccgmod module over $\pbarDpolim{A}$ we can form its tensor product and if $A=\sqcup_{\ell=1}^k A_\ell$ then the sheaf
$$
\mF^{A_1}_{\Sigma,\calU}\exttensor{*} \dots \exttensor{*} \mF^{A_k}_{\Sigma,\calU}
$$
is a \ccgmod right $\pbarDpolim {A_1}\tensor{cg}\cdots \tensor{cg}\pbarDpolim {A_k}$-module (see \ref{ssec:prodottiesterni}).

By the results explained in Section \ref{ssec:prodottiordinatiOeD} then there is a natural action of 
$\pbarOpolim{A_1}\tensor{\ra}\cdots \tensor{\ra}\pbarOpolim{A_k} $
on $
\mF^{A_1}_{\Sigma,\calU}\tensor{\ra} \dots \tensor{\ra} \mF^{A_k}_{\Sigma,\calU}
$. This action is continuous in the second variable, that is for all 
$f\in\pbarOpolim{A_1}\tensor{\ra}\cdots \tensor{\ra}\pbarOpolim{A_k} $ the multiplication by $f$ is a continuous map on $
\mF^{A_1}_{\Sigma,\calU}\tensor{\ra} \dots \tensor{\ra} \mF^{A_k}_{\Sigma,\calU}
$. 
If, moreover, $\calU$ is a \QCCF sheaf, the action is continuous in both variables. Similar remarks hold for the action of differential operators. 

To stress that this tensor product is equipped naturally with this action we write it as 
$$
\mF^{A_1}_{\Sigma,\calU}\exttensor{\ra} \dots \exttensor{\ra} \mF^{A_k}_{\Sigma,\calU}.
$$

\subsection{Product of fields}We now define the product of fields. 

\begin{definition} For a fixed decomposition $A=B\sqcup C$ of a finite set we define the right multiplication map
	\[
	m^{B,C}_r : \mF^B_{\Sigma,\calU} \otimes \mF^C_{\Sigma,\calU} \to \mF^{A}_{\Sigma,\calU}
	\]
	as follows\index{$m^{B,C}_r$}.  Given a $B$-field $X : \pbarOpolim{B} \to \calU$ and a $C$-field $Y : \pbarOpolim{C} \to \calU$ , their product $XY= m^{B,C}_r(X\otimes Y)$ is defined to be the composition
	\[
	XY :\pbarOpolim{A}\lra  \pbarOpolim{B} \tensor{*} \pbarOpolim{C} \to \calU\tensor{*}\calU \to \calU .
	\]
	where the composition of the last two maps is the continuous extension of $f\otimes g \mapsto X(f)Y(g)$.
\end{definition}

\begin{remark}\label{oss:estensioneXY}
	We notice that given $X$ and $Y$ as above their product extends to a continuous morphism $\pbarOpolim{B}\tensor{\ra}\pbarOpolim{C}\lra \calU$.
	Indeed by functoriality we get a continuous morphism  $\pbarOpolim{B}\tensor{\ra}\pbarOpolim{C}\lra \calU\tensor\ra \calU$ and the multiplication  
	map $\calU\tensor\ra \calU$ since we are assuming that $\calU$ is $\ra$-topological algebra. 
\end{remark}

\begin{lemma}\label{lem:estensionem} The multiplication map $m_r$ defined above is $\ra$-continuous so it extends to a continuous morphism
	\[
	m^{B,C}_r : \mF^B_{\Sigma,\calU}\otimesr\mF^C_{\Sigma,\calU} \to \mF^{A}_{\Sigma,\calU}.
	\]
	In particular it is $^*$-continuous and therefore defines a morphism
	\[
	m^{B,C}_r : \mF^B_{\Sigma,\calU}\tensor{*} \mF^C_{\Sigma,\calU} \to \mF^{A}_{\Sigma,\calU}.
	\]
\end{lemma}

\begin{proof}
	We have to prove that given a \noz $W \subset \mF^{A}_{\Sigma,\calU}$ 
	then 
	
	i) there exists a \noz $V \subset \mF^{C}_{\Sigma,\calU}$ such that $m_r^{B,C}(\mF^B\otimes V)\subset W$
	
	ii) given any section $Y_0 \in \mF^C_{\Sigma,\calU}$  there exists an open submodule $U \subset \mF^B_{\Sigma,\calU}$ such that $m^{B,C}_r\left(U \otimes Y_0\right) \subset W$ 
	
We may assume that $W = \{ Z : Z(C_1\otimes C_2) \subset W' \}$ with $C_1 \subset \pbarOpolim{B}$ and $C_2 \subset \pbarOpolim{C}$ compact and $W'$ an open left ideal in $\calU$, since the closures of $C_1\otimes C_2$ in $\pbarOpolim A$ as $C_1,C_2$ vary are a $w$-compact exhaustion of $\pbarOpolim A$. 
	
	Set $V = \{ Y : Y(C_2) \subset W'\}$. Using that $W'$ is a left ideal one can check that
	\[
	m^{B,C}_r\left( \mF^B_{\Sigma,\calU} \otimes V \right) \subset W.
	\]
	
	Moreover, since $C_2$ is compact, the image $Y_0(C_2)$ is finite modulo $W'$. Hence, there exists  an open left ideal $W'' \subset \calU$ such that $W''\cdot Y_0(C_1) \subset W'$. Now if we set $U = \{ X : X(C_1) \subset W'' \}$ it is easy to check that
	\[
	m^{B,C}_r \left( U \otimes Y_0 \right) \subset W'. \qedhere
	\]
\end{proof}

\begin{definition}
	Switching the factors we consider another product
	\begin{align*}
	m^{B,C}_\ell : \mF^B_{\Sigma,\calU} \otimes \mF^C_{\Sigma,\calU} &\to \mF^{A}_{\Sigma,\calU}, \\
	X\otimes Y &\mapsto (f\otimes g \mapsto Y(g)X(f)).
	\end{align*}\index{$m^{B,C}_\ell$}In a similar way of $m^{B,C}_r$, the product $m^{B,C}_\ell$ admits an extension
	\(
	m^{B,C}_\ell : \mF^B_{\Sigma,\calU} \otimesl \mF^C_{\Sigma,\calU} \to \mF^{A}_{\Sigma,\calU}.
	\)
	obtained in the same manner of the extension of $m^{B,C}_r$.
\end{definition}

\subsubsection{Introducing poles}Let $A=B\sqcup C$ and recall the definition of $\nabla_{B=C}$ from Section \ref{sssez:art2notazionidiagonali}. 
We noticed that  $\mF^B_{\Sigma,\calU}\exttensor{*} \mF^C_{\Sigma,\calU}$ is a cg-topological $\pbarOq$-module and that
$\mF^B_{\Sigma,\calU}\exttensor{\ra} \mF^C_{\Sigma,\calU}$ is a $\pbarOpolim B \tensor{\ra}\pbarOpolim C$-module such that the action is continuous in both variables. 
By Lemma \ref{lem:extensiondiagonal} it is also a $\pbarOm A(\infty \nabla_{B=C})$-module. Moreover we have a linear continuous morphism from $\mF^1_{\Sigma,\calU}\exttensor{*} \mF^1_{\Sigma,\calU}$  to $\mF^1_{\Sigma,\calU}\exttensor{\ra} \mF^1_{\Sigma,\calU}$.  We can extend this morphism (see Section \ref{ssec:localization}) to get a continuous morphism
$$
\mF^B_{\Sigma,\calU}\exttensor{*}\, \mF^C_{\Sigma,\calU} (\infty\nabla_{B=C}) \lra \mF^B_{\Sigma,\calU}\exttensor{\ra} \mF^C_{\Sigma,\calU}.
$$
We can compose this morphism with $m_r$. In this way we restrict $m_r$ to a product on the sheaf $\mF^B_{\Sigma,\calU}\boxtimes^{*} \mF^C_{\Sigma,\calU} (\infty\nabla_{B=C}) $. We denote this map also by $m_r$. 


We conclude with a remark on how we may describe these maps and on how these constructions are related to the one given in Definition 3.1.7 of \cite{cas2023}. 

\begin{lemma}\label{rmk:descrprod}
	Let $X\in \mF^B_{\Sigma,\calU}$ and $Y \in \mF^C_{\Sigma,\calU}$ and $f \in \pbarOm{A}(\infty\nabla_{B=C})$ then 
	$$ m^{B,C}_r(( X\boxtimes Y)\cdot f) (\ell)= m^{B,C}_r(X\boxtimes Y)(Exp^r_\pi(\ell f))$$ 
	for all $\ell \in \pbarOpolim A$.	
	Here $m^{B,C}_r(X\boxtimes Y):\pbarOpolim B\tensor{\ra}\pbarOpolim C\lra \calU $ is the continuous extension explained in Remark \ref{oss:estensioneXY}. And analogous statement holds for $m^{B,C}_\ell$.
\end{lemma}

\begin{proof}
	We prove a slightly stronger statement by proving that the same claim holds for all $f\in \pbarOpolim B\tensor{\ra}\pbarOpolim C$. The fields $\ell\mapsto m^{B,C}_r((X\boxtimes Y)\cdot f)(\ell)$ and $\ell\mapsto m_r^{B,C}(X\boxtimes Y)(\ell f)$ are continuous in $\ell$ hence we can choose $\ell \in \pbarOpolim A$ of the form $\ell=\gra\otimes \grb$. Moreover once $\ell$ is fixed we notice that the maps
	$$
	f\mapsto m^{B,C}_r(( X\boxtimes Y)\cdot f) (\ell)\qquad\mand \qquad f \mapsto m^{B,C}_r(X\boxtimes Y)(\ell f)
	$$
	are continuous also in $f$ (see Lemma \ref{lem:BCMN2}). Hence it is enough to prove the claim for $f=g\otimes h$ with $g,h\in  \pbarOpoli$. In this case both the left hand side and the right hand side are equal to $ X(\gra g)Y(\grb h)$.
\end{proof}

\section{Ordered operad and operad with poles}\label{sec:operads}

We use notation and terminology of \cite{BDchirali}, Chapter 1. In the present Section we will prove some commutation relations for the morphisms $m_r$ and $m_\ell$ introduced above. We will use the formalism of operads to make the exposition more concise. We will use the notation on finite sets introduced in Section \ref{ssez:artquestonotazionifiniteset}.

Define the following category of modules indexed by finite sets. 

\begin{definition}\label{def:Obullet}
A $\calO^\bullet$-module\index{$\calO^\bullet$-module} is a collection of complete topological $\calO_S$-module $M^A$, indexed by $A \in \ffset$, such that 
each $M^A$ is a $\pbarOm{A}$ topological-module whose topology is $\pbarOm{A}$-linear. A morphism $f : M^\bullet \to N^\bullet$ is a collection of $\pbarOm{A}$-linear and continuous maps $f^A : M^A \to N^A$. We call this category $\calO^\bullet$-mod and refer to it as the category of $\calO^\bullet$ modules.

An $I$-family of $\calO^\bullet$ modules is a map $I\ni i\lra M^\bullet_i\in \calO^\bullet$-mod. Usually we will denote these families by $M^\bullet_i$, we will use the notation $i\to M_i^\bullet$ occasionally to make some formulas more transparent. 
\end{definition}

In what follows finite sets will be used to index both ``graded" objects (indexes for object in the category $\calO^\bullet$-mod) as well as collections of objects (finite collections of objects in $\calO^\bullet$ that we be used to describe the operad structure). We reserve the letters $A,B$ for the former and the letters $I,J$ for the latter.  

\begin{ntz}\label{notaz:prodottiesterniOmod}
Let $q : A \twoheadrightarrow I$ be a surjective morphism of finite sets, let $\tau : \{ 1,\dots,n\} \to I$ be an ordering of $I$ and let $i\to M^\bullet_i$ be a $I$-family of $\calO^\bullet$-modules.  We introduce the following notations:
	\begin{align*}
      \exttensor{*}_q M^\bullet _i   & = \exttensor{*}_{i\in I} M_i^{A_i} \\
      \exttensor{\ra}_{\tau,q} M^\bullet _i  & =  M_{\tau(1)}^{A_{\tau(1)}} \exttensor{\ra} \cdots \exttensor{\ra} M_{\tau(n)}^{A_{\tau(n)}}     
    \end{align*}
    \index{external tensor product $\exttensor{}$:$\exttensor{*}_q$}\index{external tensor product $\exttensor{}$:$\exttensor{\ra}_{\tau,q}$}when we want to make the indexing set $I$ more explicit, we will write $\exttensor{*}_q \{M^\bullet _i\}_{i \in I}$ in place of $\exttensor{*}_q M^\bullet _i$. The definitions of the modules on the right hand side were given in Section \ref{ssec:prodottiesterni} and it is easy to check that their topology is $\pbarOm A$-linear.

  If, in addition, we assume that $M^{A}_i$ is a $\pbarOpolim{A}$-module, then, 
    by Remarks \ref{oss:prodottoalgebremoduli} and \ref{oss:algebreregolariemoduli}, and the isomorphism $\pbarOpolim{A}\simeq {\tensor{cg}}\pbarOpolim{A_i}$ these are $\pbarOpolim{A}$ \cgmod modules. Moreover, by Lemma \ref{lem:BCMN2}, there is an action of 
    $\pbarOpolim{A_{\tau(1)}} \tensor{\ra} \cdots \tensor{\ra}\pbarOpolim{A_{\tau(n)}}$ on $\exttensor{\ra}_{\tau,q} \{ M^\bullet _i \}_{i \in I}$ which is continuous in the second variable.
     Similar considerations hold for the action of the algebras $\pbarDm A$ and $\pbarDpolim A$. 
    
    Finally there is a continuous map 
    \begin{equation}\label{eq:mda}
    \imath_{\tau,q}:
    \exttensor{*}_q M^\bullet _i \lra \exttensor{\ra}_{\tau,q} M^\bullet _i
    \quad\text{given by}\quad \otimes X_i \mapsto X_{\tau(1)}\otimes\dots\otimes X_{\tau(n)}.
    \end{equation}
\end{ntz}

\subsection{The pseudo tensor structure with poles\texorpdfstring{: $P^\Delta$}{}}

    Consider a $I$-family $\{M^\bullet_i\} \in \calO^\bullet$-mod. We define $\exttensor{\Delta}_{i\in I} M^\bullet_i \in \calO^\bullet\text{-mod}$ as follows. To a finite set $A$ and a surjection $q  : A \twoheadrightarrow I$, we consider the topological modules
	\begin{align*}
	\left(\exttensor{\Delta}_{i\in I} M^\bullet_i\right)^A_q &= (\exttensor{*}_q M^\bullet_i)(\infty\nabla(q)) = (\exttensor{*}_{i \in I } M_i^{A_i})(\infty(\nabla(q)))\\
	\left(\exttensor{\Delta}_{i\in I} M^\bullet_i\right)^A &= \bigoplus_{q : A \twoheadrightarrow I} \left(\exttensor{\Delta}_{i\in I} M^\bullet_i\right)^A_q = \bigoplus_{q: A \twoheadrightarrow I} (\exttensor{*}_{i \in I } M_i^{A_i})(\infty(\nabla(q))).
	\end{align*}
	where recall the definition \index{external tensor product $\exttensor{}$:$\exttensor{\Delta} M_i^\bullet$} of $\nabla(q)$ and the topology on the right hand side module from Section \ref{ssec:polimolti}. Notice that, in particular, if there are not surjective maps from $A$ to $I$ this module is zero. 
	It follows that $(\exttensor{\Delta} M^\bullet_i) \in \calO^\bullet\text{-mod}$ .
	This tensor product is a pseudo tensor product on the category of $\calO^\bullet$-modules, as defined in \cite{BDchirali} Section 1.1.3. This means that given a surjective map 
	$\pi:J\surjmap I$ we have a natural map
	\begin{equation}\label{eq:epsilon}
	\gre_\pi : \exttensor{\Delta }_{j\in J} M^\bullet_j\lra \exttensor{\Delta}_{i\in I}\left(\exttensor{\Delta}_{j\in J_i} M^\bullet_j\right)
	\end{equation}
	which is given by the natural maps among multiple tensor products (see Remark \ref{oss:amtp}) and the fact that we are localizing along the same diagonals $\Delta_{a=b}$ for $q(a)\neq q(b)$.
	To such a data is associated a pseudo tensor structure (as defined in \cite{BDchirali}, Section 1.1.2) as follows. For each finite sets $I$ and $A$, for each $I$-family $\{M_i^\bullet\}$ of $\calO^\bullet$-modules and for each $N^\bullet\in \calO^\bullet$-mod define the abelian group
    \begin{align*}
    P^{\Delta}_I\left( \{ M^\bullet_i\}, N^\bullet \right)^A &= \Homcont_{\pbarOpolim{A}}\left( \big(\exttensor{\Delta}_{i\in I}M^\bullet_i\big) ^A,N^A \right)  \\
    &= \bigoplus_{q : A \twoheadrightarrow I} \Homcont_{\pbarOpolim{A}}\left( (\exttensor{*}_{i \in I} M_i^{A_i})(\infty\nabla(q)),N^A \right).
	\end{align*}
    and\index{$P^\Delta$} define $P^{\Delta}_I\left( \{ M^\bullet_i\}, N^\bullet \right)^A_q$ as the summand corresponding to $q:A\twoheadrightarrow  I$. If $A \stackrel{q}{\twoheadrightarrow} J \stackrel{\pi}{\twoheadrightarrow} I$ are surjections of finite non-empty sets and $\{M^\bullet_j\}_{j \in J},\{N^\bullet_i\}_{i \in I},L^\bullet$ are object in $\calO^\bullet$-modules we have composition maps  
   \begin{equation}\label{eq:PDelta1}
    P^\Delta_I\left(\{N^\bullet_i\},L^\bullet\right)^A_{\pi\circ q} \times \left( \prod_{i \in I} P^\Delta_{J_i}\left(\{M^\bullet_{j}\},N^\bullet_i\right)^{A_i}_{q_i} \right) \to P^\Delta_J\left(\{M^\bullet_j\},L^\bullet\right)^A_q. 
    \end{equation}
    where $q_i:A_i\twoheadrightarrow J_i$ is the restriction of $q$. The composition is obtained as follows: if $\psi_i\in P^\Delta_{J_i}(\{M_j\},M^\bullet)_{q_i}^{A_i}$ then by localizing $\tensor{*}\psi_i$ we get a map
    $$
    \Psi: \exttensor{\Delta}_{i\in I} \left( \exttensor{\Delta}_{j\in I_j} M^{A_j}_j\right)\lra \exttensor{\Delta}_{i\in I}N_i^{A_i}
    $$
    If $\grf\in P^\Delta_I(\{N^\bullet_i\},L^\bullet)_{\pi\circ q}^A$  then the composition of $\grf$ and $\{\psi_i\}$ is defined as 
    $$
    \grf\circ\{\psi_i\}=\grf\circ \Psi\circ \gre_{\pi}.
    $$
    In this way we have defined a pseudo tensor structure graded by finite sets. To obtain a pseudo tensor structure in the sense of \cite{BDchirali}, Section 1.1.2, define $$
    P_I^\Delta(\{M^\bullet_i\}, N^\bullet)=\prod_{q\surjmap I} P_I^\Delta(\{M^\bullet_i\}, N^\bullet)_q^A
    $$ over all finite sets $A$ and $q\surjmap I$. 
    The composition maps are equal to the maps \eqref{eq:PDelta1}. More precisely the composition maps will be maps
    $$
    P^\Delta_I\left(\{N^\bullet_i\},L^\bullet\right) \times \left( \prod_{i \in I} P^\Delta_{J_i}\left(\{M^\bullet_{j}\},N^\bullet_i\right) \right) \to P^\Delta_J\left(\{M^\bullet_j\},L^\bullet\right). 
    $$ 
    constructed as follows: if we write these spaces as a product, then all factors of this map corresponding to indexes $A,q$ as in \eqref{eq:PDelta1} will be equal to the composition maps we have already constructed and all the other factors will be zero.  

    \subsubsection{The action of bijections}\label{sssez:bigezioni1} Finally, there is an action of the groupoid of bijections on the pseudo tensor structure. If $\rho:I\lra I'$ is a bijection and, $\{i\to M_i^\bullet \}$ is an $I$-family in $\calO^\bullet$ then the action of $\rho$ is given by the morphism
     $$    \Big( \exttensor{\Delta}_{i\in I}\{i\to M_i^\bullet\}\Big)^A_q     \lra     \Big(\exttensor{\Delta}_{i'\in I'}\{i'\to M_{\rho^{-1}(i')}^\bullet\}  \Big)^A_{\rho\circ q}   $$
      which sends $(\otimes_{i\in I} X_i)\cdot f$ to $(\otimes_{i'\in I'} X_{\rho^{-1}(i')})\cdot f$, . where $f\in \pbarO(\nabla (q))$.
     Notice that $\nabla(q)=\nabla(\rho\circ q)$ so the map is well defined. Let us notice that this action does not much since the two modules are exactly equal, however they are indexed by a different map from $A$ to $I$. Thus, the action we are considering involves only a change of indices.
      
    The action on the pseudo tensor structure is obtained by pre-composition with the inverse of this morphism. 
     
    The composition maps are equivariant with respect to this action: more precisely given $\pi:J\surjmap I$, $\pi':J'\surjmap I'$ and $\rho_J:J\lra J'$, and $\rho_I:I\lra I'$ such that $\pi'\circ\rho_J=\rho_I\circ \pi$ then $\rho_J$ induces bijections $\rho_i:J_i\lra J'_{\rho_I(i)}$ and the composition map is equivariant if $I$ acts on $P_I$ with $\rho_I$, on $P_{J_i}$ with $\rho_i$ and on $P_J$ with $\rho_J$. 
     
\subsection{The ordered pseudo tensor structure \texorpdfstring{: $P^{\tensor{\ra}}$}{}} Similarly, we define the ordered pseudo tensor structure. We indicate only the main differences. Recall Notation \ref{notaz:prodottiesterniOmod}.

Given an ordering $\tau$ of a finite set $I$ and a finite set $A$ with a surjection $q : A \twoheadrightarrow I$, we consider the modules 
\begin{align*}
(\exttensor{\ra}_{i \in I} M^\bullet_i)^A_q &= \bigoplus_{\substack \tau \text{ ordering of } I} \exttensor{\ra}_{\tau,q} M^\bullet_i \\
(\exttensor{\ra}_{i \in I} M^\bullet_i)^A &= \bigoplus_{q : A \twoheadrightarrow I} (\exttensor{\ra}_{i \in I} M^\bullet_i)^A_q =  \bigoplus_{\tau \text{ ordering of } I, q : A \twoheadrightarrow I} \exttensor{\ra}_{\tau,q} M^\bullet_i 
\end{align*}
The map $\gre_\pi$ in this case is induced by the associativity constraints of the $\tensor{\ra}$ tensor product. As above we define the associated pseudo tensor structure as 
\begin{align*}
P^{\otimesr}_I\left( \{ M^\bullet_i\}, N^\bullet \right)^A &= \Homcont_{\pbarOpolim{A}}\left( (\exttensor{\ra}_{i \in I} M^\bullet_i)^A,N^A \right) \\
&=\bigoplus_{\tau, q}\Homcont_{\pbarOpolim{A}}\left( (\exttensor{\ra}_{\tau,q} M^\bullet_i)^A,N^A \right)
\end{align*}
where\index{$P^{\otimesr}$} $\tau$ is an ordering of $I$ and $q:A\surjmap I$. 
We denote with $P^{\otimesr}_I\left( \{ M^\bullet_i\}, N^\bullet \right)^A_{\tau,q}$ the summand corresponding to $\tau$ and $q$ and by 
$P^{\otimesr}_I\left( \{ M^\bullet_i\}, N^\bullet \right)^A_{q}$ the summand corresponding to $q$ (which is a sum over all possible $\tau$). 
In this case the analogue of the composition maps \eqref{eq:PDelta1} refines, taking care of orderings, as follows. Consider surjections $q : A \twoheadrightarrow J$, $\pi : J \twoheadrightarrow I$, then we have a morphism
$$
   P^{\otimesr}_I\left(\{N^\bullet_i\},L^\bullet\right)^A_{\tau,\pi\circ q} \times \left( \prod_{i \in I} P^{\otimesr}_{J_i}\left(\{M^\bullet_{j}\},N^\bullet_i\right)^{A_i}_{\grs_i,q_i} \right) \to P^{\otimesr}_J\left(\{M^\bullet_j\},L^\bullet\right)^A_{\tau\star\{\grs_i\},q} 
$$
where $\tau$ is an ordering of $I$, $\grs_i$ is an ordering of $J_i$ and $\tau \star \{\sigma_i \}$ is the induced lexicographic ordering on $J$. To make explicit the role of orderings given $\psi:N_{\tau(1)}\exttensor\ra \cdots \exttensor\ra  N_{\tau(k)}\lra L$ and $\grf_i:M_{\grs_i(1)}\exttensor\ra \cdots \exttensor\ra  M_{\grs_i(k_i)}\lra N_i$ then the composition is constructed as
$$
\psi\circ \{\grf_i\}=\psi \circ\Big(\grf_{\tau(1)}\exttensor{\ra}\cdots \exttensor{\ra}\grf_{\tau(k)}\Big).
$$
As before to get a pseudo tensor structure we define $P_I^{\tensor\ra}$ as the product over all finite sets $A$ of $(P_I^{\tensor\ra})^A$. The composition map of this pseudo tensor structure is constructed as in the previous case, by declaring to be zero the composition of maps which are not of the type above. 

\subsubsection{The action of bijections in the case othe ordered tensor structure}\label{sssez:bigezioni2}
Also in this case there is an action of the groupoid of bijection. Given $\rho:I\lra I'$ be a bijection we notice that we have an equality
$$
\exttensor{\ra}_{\tau,q} \{i\to M_i^{\bullet}\} = \exttensor{\ra}_{\rho\circ\tau,\rho\circ q} \{i'\to M_{\rho^{-1}(i')}^{\bullet}\}
$$
meaning that the two tensor products are exactly the products of the same modules with the same ordering. The action of $\rho$ on $\exttensor{\ra}_I \{i\to M_i^{\bullet}\}$ is defined as the sum of these equalities:
\begin{align*}
\exttensor{\ra}_I \{i\to M_i^{\bullet}\}= \bigoplus_{\tau,q} \exttensor{\ra}_{\tau,q} \{i\to M_i^{\bullet}\} &= \bigoplus_{\tau,q}  \exttensor{\ra}_{\rho\circ\tau,\rho\circ q} \{i'\to M_{\rho^{-1}(i')}^{\bullet}\} \\ &=
\exttensor{\ra}_{I'} \{i'\to M_{\rho^{-1}(i')}^{\bullet}\} .
\end{align*}
Let us say that also in this case, this action does not affect modules but is just a change of indices. The action on $P_I^{\tensor \ra}$ is obtained by duality. Also in this case composition maps are equivariant.

\subsubsection{The map among the two tensor structure}
	\label{rmk:polisuOmod}
	Assume from now on that the modules $M^A$ are $\pbarOpolim{A}$ \ccgmod modules. 
	
	Under this assumption, by what noticed in Notation \ref{notaz:prodottiesterniOmod} all the pseudo tensor product we have defined are $\pbarOpolim {A}$ \ccgmod module. Moreover all the action and maps among these modules we have constructed are $\pbarOpolim A$-equivariant. Similar considerations also 
	hold in the case  $\pbarOpoli$ is replaced by $\pbarDpoli$.

By Lemma  \ref{lem:espansionediagonale}, given a surjection $q: A\surjmap I$ and an ordering $\tau:\{1,\dots,k\}\ra I$, there is a morphism of algebras
$$ \pbarOm A(\infty\nabla(q))\lra \pbarOpolim {A_{\tau(1)}} \tensor \ra \cdots \tensor \ra  \pbarOpolim {A_{\tau(k)}}
=\exttensor{\ra} _{\tau,q} \left\{  \pbarOpolim \bullet\right\}_{i\in I} $$
This implies that given  $\pbarOpolim	{A_i}$ \cgmod modules $M^\bullet_i$
the morphism \eqref{eq:mda} induces a map
\begin{equation}\label{eq:mappatratensori}
\imath^A_{\tau,q}:\left(\exttensor{\Delta}_{i\in I} M_i^ \bullet\right)^A_q=
\exttensor{*}_{i \in I} M_i^{A_i}(\infty \nabla(q)) \lra 
M_{\tau(1)}^ {A_{\tau(1)}} \tensor \ra \cdots \tensor \ra  M_{\tau(k)}^ {A_{\tau(k)}}=\exttensor{\ra} _{\tau,q}  M_i^ \bullet
\end{equation}

\begin{definition}\label{def:mappatratensori}
Let $M_i^\bullet$ be an element of $\calO^\bullet$-mod such that 
	$M_i^A$ is an $\pbarOpolim A$ \cgmod module, then 
the morphisms \eqref{eq:mappatratensori} induces a $\calO^\bullet$-morphism of pseudo tensor products
$$
\imath_I^\bullet:\exttensor{\Delta}_{i\in I}M^\bullet_i\lra
\exttensor{\ra}_{i\in I} M^\bullet_i
$$
constructed as follows. If we fix $A$ and we write the left hand factor as a sum over all possible $q:A\surjmap I$ and the right hand factor as a sum as a sum over all possible $q:A\surjmap I$ and ordering $\tau$ then the morphism 
\[
	\left( \exttensor{\Delta}_{i \in I} M^\bullet_i\right)^A \to \left( \exttensor{\ra} M^\bullet_i \right)^A
\]
preserves the $q$ components and on each of them it reads as
\[
	\oplus_{\tau} \imath^A_{\tau,q} : \big(\exttensor{\Delta}_{i \in I} M^\bullet_i\big)^A_q \to \bigoplus_{\tau} \exttensor{\ra}_{\tau,q}M^\bullet_i = (\exttensor{\ra}_{i \in I} M^\bullet_i)^A_q.
\]
By construction this morphism commutes with the composition maps  and it is equivariant with respect to the  action of bijections of finite sets. By duality this determines a morphism 
\begin{equation}\label{eq:jDf}
\jmath^\bullet_I:P^{\tensor\ra}_I\left( \{ M^\bullet_i\}, N^\bullet \right)\lra  P^{\Delta}_I\left( \{ M^\bullet_i\}, N^\bullet \right),
\end{equation}
which is then a morphism of pseudo tensor structures.
\end{definition}

\subsection{The Operads} \label{ssec:operads} In what follows $P$ will stand for $P^\Delta$ or for $P^{\tensor{\ra}}$.
Given a graded topological module $M^\bullet$ we will write $P(M^\bullet)$ for the corresponding operads and $P_I(M^\bullet)$ will denote the collection of $I$-operations, that is $P_I(\{M^\bullet_i\},N^\bullet)$ with $N^\bullet =M^\bullet_i = M^\bullet$ for all $i \in I$.  For a given surjection $q : A \twoheadrightarrow I$ we will write $P_I(M^\bullet)^A_q$ for the summand relative to $q$ in $P_I(M^\bullet)$. The identity of our operad $Id_P$ is given by the identity of $M^\bullet$ in  $P_I(\{M^\bullet\})$ for every $I$ with only one element. 
The composition laws of $P_I$ can be described as before, for surjections $A \stackrel{q}{\twoheadrightarrow} J \stackrel{\pi}{\twoheadrightarrow} I$ we have a composition map
\[
P_I(M^\bullet)^A_{\pi\circ q} \times \left( \prod_{i \in I} P_{J_i}(M^\bullet)^{A_i}_{q_i} \right) \to P_J(M^\bullet)^A_q
\]
and in the case of the ordered pseudo tensor structure can be refined also taking care of the orderings of $I$ and $J$. 
These compositions are equivariant under the action of $\Aut(I)$ (and more generally bijections of finite sets), so these are symmetric operads. 

The map $j_I$ of formula \eqref{eq:jDf} induces $\Aut(I)$-equivariant morphisms on $I$-operations (still denoted by $\jmath_I$) commuting with all composition maps and therefore a morphism of operads
$$
\jmath^\bullet_I:P^{\tensor{\ra}}_I(M^\bullet)\lra P^\Delta_I(M^\bullet),\qquad \jmath^\bullet:P^{\tensor{\ra}}(M^\bullet)\lra P^\Delta(M^\bullet).
$$

\subsection{Comparison with the chiral pseudo tensor structure of \texorpdfstring{\cite{beilinsonTopologicalAlgebras}}{Beilinson} and the chiral pseudo tensor structure of \texorpdfstring{\cite{BDchirali}}{Beilinson and Drinfeld}}
Our pseudo tensor structures are a graded version of the chiral pseudo structures in \cite{beilinsonTopologicalAlgebras} and \cite{BDchirali}. We denote them respectively by $P^{chB}$ and $P^{chBD}$. The first one is defined as
$$
P_I^{chB}(\{M_i\},N) =\bigoplus _\tau \Homcont \left( M_{\tau(1)}\tensor\ra \cdots \tensor\ra M_{\tau(n)}, N \right)
$$
where\index{$P^{chB}$} $\tau$ ranges over all possible orderings of $I$. It is clear from the definition that this pseudo tensor structure is very similar to our ordered pseudo tensor structure. 

\subsubsection{The chiral pseudo tensor structure of Beilinson and Drinfeld}\label{ssec:filteredbeilinsondrinfeld}
We now define an analogue of the chiral pseudo structures of Beilinson and Drinfeld  \cite{BDchirali} in the case of topological modules. In order to be able to define the composition laws we consider filtered $\pbarD$ \ccgmod module. Given a topological module $M$ (we do not require completeness), which is also a right $\pbarDm A$ \ccgmod module we define a filtration on $M$ to be an increasing sequence of $\pbarO$ submodules $M(n)$ of $M$ such that $M=\limind M(n)$ as a sheaf (we do not require them to be equal as topological sheaves) and such that for all $n$ and $m$ there exists $k$ such that $(\pbarDm A )^{\leq m}\cdot M(n)\subset M(k)$ (look at Section \ref{ssec:operatoridifferenziali} for the definition of $\pbarD^{\leq m}$ and $(\pbarDm{A})^{\leq m}$). If $M_i$ is a filtered $\pbarDm {A_i}$ module we define their tensor product 
$$
M_1\exttensor{\fil}M_2 \exttensor{\fil}\cdots \exttensor{\fil}M_h=\limind_n
M_1(n)\exttensor{*}M_2(n) \exttensor{*}\cdots \exttensor{*}M_h(n).
$$
Notice\index{external tensor product $\exttensor{}$:$\exttensor{\fil}$} that if the $M_i$ are right $\pbarDm{A_i}$ \cgmod modules then their filtered tensor product is 
a $\pbarDm A$ \ccgmod module where $A=\sqcup A_i$. Notice also that if $M$ is a  $\pbarDm A$ module equipped with a filtration then also $M(\infty \Delta_{a=b})$ and 
$\Delta(\pi)_!(M)$  are equipped with filtrations. In the first case take the image of $M(n)(n\infty \Delta_{a=b})$ and in the second case take the image of $M(n)\otimes (\pbarDm{B})^{\leq n}$. If $M$ and $N$ are filtered modules we say that a continuous homomorphism $\grf:M\lra N$ is filtered if for all $n$ there exists $h$ such that $\grf(M(n))\subset N(h)$. We denote the space of filtered continuous homomorphism by $\Hom^{\fil}$. 

\begin{remark}
    If the module $M$ has an \exhaustion by $w$-compact submodules and we choose as filtration the elements of this exhaustion then the tensor product above coincides with the cg-tensor product and all constructions produce modules with an \exhaustion, and all continuous homomorphism are automatically filtered.
\end{remark}

\begin{remark}
	One can consider also the trivial filtration on a module $M$, i.e. $M(n)=M$ for all $n$, however, also in this case we will have non trivial filtration on $M(\infty \nabla)$ and $\Delta_!M$, hence the filtered morphism among these modules will be not all continuous morphisms. 
\end{remark}

In order to define the composition of morphisms we will need the following two Lemmas. 

\begin{lemma}\label{lem:chBD2}Let $\pi:J\surjmap I$ be a surjective map and let
	$N$ be a right $\pbarDm{I}$  \cgmod module, then
	we have a natural continuous morphism
	$$
	\eta:\Big(\Delta(\pi)_! N\Big) (\infty\nabla(\pi))\lra \Delta(\pi)_! \Big( N\big(\infty\nabla(I)\big)\Big).
	$$
	Moreover, if $N$ is filtered then $\eta$ is filtered. 
\end{lemma}
\begin{proof}
	By functoriality we have a morphism from $\Delta(\pi)_! N$ to  $\Delta(\pi)_! ( N(\infty\nabla(I)))$. To prove that there exists $\eta$ as in the claim of the Lemma we need to prove that, given $i,j\in J$ such that $\pi(i)\neq\pi(j)$, a function which locally generates $\Delta_{i=j}$ acts bijectively on the right-hand side. This implies the existence of the map $\eta$.
	
    We can check this statement locally. We can assume that $S=\Spec A$ is well covered. Let $t$ be a local coordinate, we denote with $x_i$ the image of the coordinate $t$ along the $i$-th immersion of $\pbarO \to \pbarOm{I}$, and similarly we denote by $y_j$ the coordinates on $\pbarOm{J}$. Then we need to prove that the action of $y_i-y_j$ is invertible on the right hand side whenever $\pi(i)\neq \pi(j)$. If we set  $M=N\big(\infty\nabla(I)\big)$ then 
	$$\Delta(\pi)_! M = \bigoplus_{|\gra|\leq k} M\partial ^\gra $$
	where the derivatives are with respect to a subset of the variables $y_j$ (we do not need to be precise about this point). Up to filter this module according to the order of the differential operator $\partial^\gra$, then the action of $(y_i-y_j)$ on $M \partial  ^\gra$ is given by
	$$
	m\partial  \cdot (y_i-y_j)= \pm m\cdot (x_{\pi(i)}-x_{\pi(j)}) \partial^\gra
	$$
	hence it is invertible if $\pi(i)\neq\pi(j)$. 
	
	Finally we check that if  $N$ is filtered then $\eta$ is filtered as well. The filtration $\big((\Delta(\pi)_!N)(\infty\nabla(\pi))\big)(k)$ of $(\Delta(\pi)_!N)(\infty\nabla(\pi))$ is defined as the the image of 
	$\big(N(k)\otimes (\pbarDm J)^{\leq k}\big)(k\nabla(\pi))$, while the filtration $\Big(\Delta(\pi)_!\big(N(\infty\nabla(I)\big)\Big)(k)$ of $\Delta(\pi)_!\big(N(\infty\nabla(I)\big)$ is defined as the image of 
	$\big(N(k)(k\nabla(I))\big)\otimes (\pbarDm J)^{\leq k} $.
	Notice
	that the degree of the pole possibly increases by one when we commute $(y_i-y_j)^{-1}$ with a derivative. Hence $\eta\Big(\big((\Delta(\pi)_!N)(\infty\nabla(\pi))\big)(k) \Big) $   is contained in the image of $\big(N(k)(k^2\nabla(I))\big)\otimes (\pbarDm J)^{\leq k} $ and in particular
	\[
	\eta\Big(\big((\Delta(\pi)_!N)(\infty\nabla(\pi))\big)(k)\Big)\subset \Big(\Delta(\pi)_!\big(N(\infty\nabla(I))\big)\Big)(k^2).
	\qedhere\]
\end{proof}

The following Lemma is not true if we do not work in a filtered setting (the reason for this is the same as the one explained in Remark \ref{oss:prodottiregolarifreccetta}, that is: the $*$-tensor product does not commute with infinite direct sums)

\begin{lemma}\label{lem:chBD1}
	Let $\pi:J\surjmap I$, and let $N_i$ be a right $\pbarD$ filtered-\cgmod module. 
	Then we have a natural filtered homomorphism 
	$$
	\exttensor{\fil}_{i\in I}\left(\Delta(J_i)_! N_i\right)  \lra \Delta(\pi)_!\left(\exttensor{\fil}_{i\in I}N_i\right).
	$$
	 By localizing this map we get a filtered homomorphism
	$$
	\vartheta:	\left(\exttensor{\fil}_{i\in I}\left(\Delta(J_i)_! N_i\right) \right) (\infty\nabla (\pi)) \lra \left(\Delta(\pi)_!\left(\exttensor{\fil}_{i\in I}N_i\right)\right) ( \infty\nabla(\pi)).
	$$
\end{lemma}
\begin{proof}
	It is enough to prove the first statement. If $N$ is filtered we denote by $\Delta(\pi)^{\leq h}_!N(k)$ the image of $N(k)\otimes \calD_\pi^{\leq h}$ in 
	$\Delta(\pi)_!N$. We notice that we have compatible morphisms 
	$$
	\theta : 	\exttensor{}_{i\in I}\left(\Delta(J_i)^{\leq k}_! N_i(h)\right)  \lra \Delta(\pi)^{k'}_!\left(\exttensor{}_{i\in I}N_i(h)\right)
	$$
	where $k'=k\cdot |I|$. Indeed the left hand side is the image of 
	$$
	\exttensor{}N_i(h)\otimes (\pbarDm{J_i})^{\leq k}
	$$
	and $(\pbarDm{J_1})^{\leq k}\otimes \cdots \otimes (\pbarDm{J_k})^{\leq k}\subset (\pbarDm{J})^{\leq k'}$. It is easy to check locally that this map is continuous with respect to the $**$ topology on both sides, proving the claim of the Lemma. 
\end{proof}

The topological version of the pseudo tensor structure of Beilinson and Drinfeld is defined as follows. Let $M_i$ and $N$ be right $\pbarD$ \ccgmod module with filtrations then
$$
P_I^{chBD}(\{M_i\},N) = \Hom^{\fil}_{\pbarDm I} \left( (\exttensor{\fil} M_i) (\infty \nabla(I)), \Delta{(I)}_! N \right).  
$$
where $\nabla(I)\subset X^I$ is the\index{$P^{chBD}$} big diagonal and $\Delta{(I)}:X\lra X^I$ is the small diagonal (see Section \ref{sssez:art2notazionidiagonali} for these notations).

The composition is defined following the same steps of the other pseudo tensor structure. Let  $\pi:J \surjmap I$, $\grf\in P_{I}^{chBD}(\{N_i\},L)$ and $\psi_i\in P_{J_i}^{chBD}(\{M_j\},N_i)$. We define the restriction of $\grf\circ \{\psi_i\}$ to $\exttensor{*}_JM_j(n\nabla(J)$ as the composition of the following maps
$$
\xymatrix{
	\left(\exttensor{\fil}_{j\in J}M_j\right)(\infty\nabla(J)) \ar[rr]^-{\gre_\pi}
	&&\left(\exttensor{\fil}_{i\in I}\left( \left(\exttensor{\fil} _{j\in J_i} M_j\right)(\infty\nabla(J_i))\right)\right)(\infty\nabla(\pi))}
$$
$$
\xymatrix{
	\left(\exttensor{\fil}_{i\in I} \left(\left(\exttensor{\fil} _{j\in J_i} M_j\right)(\infty\nabla(J_i))\right)\right)(\infty\nabla(\pi))
	\ar[r]^-{\Psi}
	&\left(\exttensor{\fil}_{i\in I} \left( \Delta(J_i)_! N_i\right) \right)(\infty\nabla(\pi))}
$$
$$
\xymatrix{
	\left(\exttensor{\fil}_{i\in I} \left( \Delta(J_i)_! N_i\right) \right)(\infty\nabla(\pi))
	\ar[rr]^-{\vartheta}
	&&
	\left(\Delta(\pi)_! \left( \exttensor{\fil}_{i\in I}  N_i\right) \right)(\infty\nabla(\pi))}
$$
$$
\xymatrix{
	\left(\Delta(\pi)_! \left( \exttensor{\fil}_{i\in I}  N_i\right) \right)(\infty\nabla(\pi))
	\ar[rr]^-{\eta}
	&&\Delta(\pi)_! \left(\left( \exttensor{\fil}_{i\in I}  N_i \right) (\infty\nabla(I))\right)}
$$
$$
\xymatrix{
	\Delta(\pi)_! \left(\left( \exttensor{\fil}_{i\in I}  N_i \right) (\infty\nabla(I))\right)
	\ar[rr]^-{\Delta(\pi)_!(\grf)}
	&&\Delta(\pi)_! \Delta(I)_!L = \Delta(J)_!L
}
$$
where $\gre_\pi$ is the filtered analogue of the morphism \eqref{eq:epsilon}, $\Psi$ is the localization with respect to $\nabla (\pi)$ of the tensor product of the morphisms $\psi_i$, the morphism $\vartheta$ was defined in Lemma \ref{lem:chBD1} and the morphism $\eta$ was defined in Lemma \ref{lem:chBD2}. Since all morphisms are filtered the composition is also filtered.

The action of the groupoid of bijections is defined similarly to the previous cases, hence, for a filtered module, we have symmetric operads $P^{chB}(M)$ and $P^{chBD}(M)$ where the unity is given by $id_M:M\lra M\in P_1(M)$.


\subsubsection{Comparison between the composition in \texorpdfstring{$P^\Delta$}{} and the composition in \texorpdfstring{$P^{chBD}$}{the Beilinson and Drinfeld's chiral pseudo tensor structure}}\label{sssec:chBDDelta}
We want to compare the composition in the pseudo tensor structure $P^\Delta$ and the composition in the pseudo tensor structure $P^{chBD}$. If $M$ is a right $\pbarD$ \ccgmod module 
we define $M_!^\bullet\in \calO^\bullet$ as
$M_!^A=\Delta{(A)}_! M$. Notice that given $M,N$ as above we have

\begin{align*}
	P^\Delta_I(\{M^\bullet_{!,i} \}_{i \in I}, N_!^\bullet)^A &= \bigoplus_{q : A \twoheadrightarrow I} P^\Delta_I(\{M^\bullet_{!,i} \}_{i \in I},N_!^\bullet)^A_q \\
	&= \bigoplus_{q : A \twoheadrightarrow I} \Homcont_{\pbarDm{A}} \left((\exttensor{*}_q \Delta(A_i)_!M_i)(\infty\nabla(q)),\Delta_!(A)N\right).
\end{align*}
while
\[
	P^{chBD}_I(\{M_i\}_{i \in I},N) = \Hom^{\mathrm{fil}}_{\pbarDm{I}}\left( (\exttensor{\mathrm{fil}}M_i)(\infty\nabla(I)),\Delta(I)_!N \right).
\]
In particular, if the $M_i$ are filtered the natural map
$\mathrm{inc}_I:\exttensor{\fil}M_i(\infty \nabla)\to \exttensor {*}M_i(\infty \nabla)$ induces, by precomposition, a morphism from
\[
	P^\Delta_I(\{M^\bullet_{!,i} \}_{i \in I}, N_!^\bullet)^I_{\id : I \to I} = \Homcont_{\pbarDm{I}}\left( (\exttensor{*} M_i)(\infty\nabla(I)),\Delta(I)_!N \right)
\]
to
\[
	\Hom^{\mathrm{fil}}_{\pbarDm{I}}\left( (\exttensor{\mathrm{fil}}M_i)(\infty\nabla(I)),\Delta(I)_!N \right) = P^{chBD}_I(\{ M_i \}_{i \in I},N).
\]

Let $\pi:J\surjmap I$ be a surjective map, $\tilde \grf\in P_I^{chBD}(\{N_i\},L)$, $\tilde \psi_i\in P_{J_i}^{chBD}(\{M_j\},N_i)$, $\grf\in P_I^{\Delta}(\{N_i^\bullet\},L^\bullet)^J_\pi$, $\psi_i\in P_{J_i}^{\Delta}(\{M_J^\bullet\},N_i^\bullet)^{J_i}_{id:J_i\lra J_i}$,
and assume that 
$$
\tilde \psi_i= \psi\circ \mathrm{inc}_{J_i}
\quad\mand\quad
\Delta(\pi)_!(\tilde \grf)\circ \eta \circ \vartheta=\grf\circ \mathrm{inc}_I
$$
then, by construction,
$$
\tilde \grf \circ \{\tilde \psi_i\} = \grf\circ\{\psi_i\}\circ \mathrm{inc}_J.
$$

The action of bijections, on the two sides has also a similar compatibility, and the unity is given in both cases by the identity map. 
 
\subsection{Lie structures} To give a morphism from the Lie operad to a symmetric operad is equivalent 
to give a Lie structure: that is for every set $I$ with two elements an element $\mu_I\in P_I(M)$ such that 
\begin{itemize}
\item $\grs \cdot \mu_I =\gre(\grs) \mu_J$ where $\grs$ is a bijection from $I$ to $J$ (see Section \ref{ssez:artquestonotazionifiniteset} for the definition of finite sets and recall that our finite sets are naturally ordered);
\item Let $\pi_{2=3}:[3]\surjmap [2]$ and define  $\nu=\mu_{[2]} \circ \{Id_M,\mu_{\{2,3\}}\}\in P_{[3]}(M)$ where $Id_M\in P_{[1]}(M)$ is the identity of the operad. Let $\rho$ be the the cyclic permutation of the set $[3]$, then 
$\nu +\rho\cdot \nu +\rho^2\cdot \nu=0$.
\end{itemize}

We want to be more explicit about the second formula and expand a bit on it. By the equivariance of the composition, which we assume for symmetric operads, we have 
\begin{align*}
\rho\cdot \nu & =\mu_{[2]} \circ \{Id_M,\grs\cdot \mu_{\{1,3\}}\}\\
\rho^2\cdot \nu & =\mu_{[2]} \circ \{Id_M, \mu_{\{1,2\}}\}
\end{align*}
where $\grs$ is the transposition of $\{1,3\}$ and the $Id_M$ which appear in the first formula is an element of $P_{\{2\}}(M)$ and the one in the second formula is an element of $P_{\{3\}}(M)$. 

\subsubsection{Notation: finite sets vs natural numbers} Most often operad are indexed by natural numbers, like $P_n$. This make the notation a little bit lighter. However it is sometimes useful to work with finite sets to record all the information and to avoid possible misunderstanding. For this reason when it is not too heavy we prefer to use finite sets. 

\subsubsection{Lie structure in the ordered operad}
In the case of the ordered pseudo tensor structure we will use the following remark by Beilinson in \cite{beilinsonTopologicalAlgebras}. He considers the projection from $P_{[2]}$ onto the first component:
$$\begin{tikzcd}[column sep=tiny]
	p_1: P^{chB}_{[2]}(M) & = &\Homcont(M \tensor{\ra}M,M)_{1<2}\oplus\Homcont(M\tensor{\ra}M,M)_{2<1} \arrow{d}{}\\ && \Homcont(M\tensor{\ra}M,M).
\end{tikzcd}
$$
Beilinson noticed the following result. 

\begin{proposition}[Section 1.4 in \cite{beilinsonTopologicalAlgebras}]\label{prop:beilinsontopalg14}
	Let $M$ be a topological module then the projection 
	$$
	p_1: P^{chB}_{[2]}(M) \lra \Homcont(M\tensor{\ra}M,M).
	$$
	establishes an equivalence between Lie structures on the operad $P^{chB}M$ and associative products  $M\tensor\ra M\lra M$.  
\end{proposition}

This result extends to graded topological modules and in a similar way to $\calO^\bullet\tmod$. So to give a Lie structure on $P^{\tensor\ra}(M^\bullet)$ is equivalent to give an associative product $M^\bullet \tensor\ra M^\bullet\lra M^\bullet$. 

\subsubsection{Lie structures in \texorpdfstring{$P^\Delta$}{the operad with poles} and Lie structures in \texorpdfstring{$P^{chBD}$}{Beilinson and Drinfeld's chiral operad}}
The morphism $\jmath^\bullet$ introduced in Section \ref{ssec:operads} is a morphism of symmetric operad, if we have a Lie structure in the operad $P^{\tensor{\ra}}$ then, by applying $\jmath^\bullet$ we get a Lie structure in the operad $P^{\Delta}$. 

We want now to compare Lie structure in the operad $P^\Delta$ and Lie structure in the operad $P^{chBD}$. Let $M$ be a right $\pbarD$ filtered-\cgmod module and define $M^\bullet_!$ as in Section \ref{sssec:chBDDelta}. Let $\mu^{\bullet,\bullet}\in P_2^{\Delta}(M^\bullet_!)$ be a Lie structure, which is also  $\pbarDm{\bullet}$ equivariant. This means that for all surjective map $\pi:A\lra [2]$, setting $B=A_1$ and $C=A_2$, we have maps 
$$
\mu^{B,C}=\mu^{B,C}_{[2]}: \big(\Delta(B)_!M \exttensor * \Delta(C)_!M\big) (\infty \nabla_{B=C})\lra \Delta(A)_!M 
$$
and that these maps define a Lie structure. In particular we set 
\begin{equation}\label{eq:mu11}
\mu=\mu^{\{1\},\{2\}}=:M\exttensor {*} M(\infty \Delta)\lra \Delta_!M
\end{equation}
and 
\begin{equation}\label{eq:mu12}
\mu^{1,2}=\mu^{\{1\},\{2,3\}}:M\tensor * \Delta_! M(\infty \nabla_{\{1\}=\{2,3\}})\lra \Delta([3])_!M.
\end{equation}
Sometimes we will use $\mu^{1,2}$ to denote any morphism $\mu^{B,C}$ where $B$ has one element and $C$ has two elements or vice-versa. When this could lead to some ambiguities we will be more explicit.

\begin{lemma}\label{lem:chBDDeltaLie}
Let $M$ be a right $\pbarD$ filtered-\ccgmod module, and 
let $\mu^{\bullet,\bullet}$ be a Lie structure in $P_{[2]}^{\Delta}(M^\bullet _!)$ which is also $\pbarDm{\bullet}$ equivariant. Define $\tilde \mu=\mu\circ \mathrm{inc}_2$ the restriction of $\mu$ to $M\exttensor {\mathrm{fil}} M(\infty \Delta)$. Assume that $\tilde\mu$ is filtered and that 
$$
\Delta(\pi)_!(\tilde \mu)\circ \eta\circ \theta =\mu^{1,2}\circ \mathrm{inc}_3.
$$ 
Then $\tilde \mu$ defines a Lie structure in the operad $P^{chBD}(M)$. 
\end{lemma}
\begin{proof}
	We need to check the two conditions defining a Lie structure. The first one is trivial. Let $\tilde \nu=\tilde \mu\circ (Id_M,\tilde \mu)$ and $\nu= \mu^{1,2}\circ (Id_M,\mu)$. By the discussion in Section \ref{sssec:chBDDelta} we have 
	$$
	\tilde \nu = \nu \circ \mathrm{inc}_3. 
	$$
	By assumption we know that $\nu+\rho\cdot \nu+\rho^2\cdot \nu=0$. The action of bijection in the two operads is defined in the same way, hence if we restrict this equality to $M\exttensor{\fil}M\exttensor{\fil}M(\infty \nabla(3))$ we get  
\[\tilde\nu+\rho\cdot \tilde\nu+\rho^2\cdot \tilde \nu=0\]
as desired.
\end{proof}


\section{Chiral products in the space of fields}\label{sec:chiralprod}

Following Beilinson \cite[Section 1.4]{beilinsonTopologicalAlgebras}, we define
a non unital chiral structure on $M^\bullet \in \calO^\bullet$-mod as a Lie structure on the operad $P^{\tensor{\ra}}(M^\bullet)$, and thank to Proposition \ref{prop:beilinsontopalg14} to give such a structure is equivalent to give a continuous associative product $M^\bullet\tensor\ra M^\bullet\lra M^\bullet$. We apply these considerations to the case of the $\calO^\bullet$-module of fields:
\[
\mF^\bullet = \big( A \mapsto \mF^A_{\Sigma,\calU}\big).
\]

\smallskip
 
To shorten the notation we will write $\mF^A$ in place of $\mF^A_{\Sigma,\calU}$. We will see that the right and left multiplication maps $m^{A,B}_r,m^{A,B}_\ell$ induce a morphism of operads $m : \Lie \to P^{\otimesr}(\mF^\bullet)$.

\begin{lemma}\label{lem:asssociativeprodfields}
	The collection of right multiplication maps $m^{A,B}_r$ defines an associative product
	\[
	\mF^\bullet \otimesr \mF^\bullet \to \mF^\bullet.
	\]
\end{lemma}

\begin{proof}
	In Lemma \ref{lem:estensionem} we noticed that the product $m^{A,B}_r$ is right continuous and therefore defines a morphism $\mF^\bullet \otimesr \mF^\bullet \to \mF^\bullet$. To check associativity we may work at the level of non completed tensor product, so that we reduce to show that the following diagram commutes
	\[\begin{tikzcd}
	{\mF^A\otimes(\mF^B\otimes\mF^C)} & {=} & {(\mF^A\otimes\mF^B)\otimes\mF^C} \\
	{\mF^A\otimes\mF^{B\sqcup C}} && {\mF^{A\sqcup B }\otimes\mF^{C}} \\
	{\mF^{A\sqcup B\sqcup C}} & {=} & {\mF^{A\sqcup B\sqcup C}}
	\arrow["{\id\otimes m^{B,C}_r}", from=1-1, to=2-1]
	\arrow["{m^{A,B}_r\otimes\id}", from=1-3, to=2-3]
	\arrow["{m^{A,B\sqcup C}_r}", from=2-1, to=3-1]
	\arrow["{m^{A\sqcup B,C}_r}", from=2-3, to=3-3]
	\end{tikzcd}\]
	So take $X \in \mF^A, Y \in \mF^B, Z\in \mF^C$, write $\pbarOpolim{A\sqcup B\sqcup C}=\pbarOpolim{A}\,\tensor{cg}\,\pbarOpolim{B}\,\tensor {cg} \,\pbarOpolim{C}$ and take $f\in\pbarOpolim{A},g\in\pbarOpolim{B}$ and $h\in\pbarOpolim{C}$. Evaluating the vertical arrows in the left hand side in the diagram above on $X\otimes Y\otimes Z$ one finds the $A\sqcup B\sqcup C$-field 
	\[
	X(YZ) : f \otimes g \otimes h \mapsto X(f)\cdot(YZ) (g\otimes h) = X(f)\cdot(Y(g)\cdot Z(h))
	\]
	analogously the result of the evaluation of the right arrows on $X\otimes Y \otimes Z$ is 
	\[
	(XY)Z : f\otimes g \otimes h \mapsto (X(f)\cdot Y(g))\cdot Z(h).
	\]
	Here $\cdot$ stands for the product on $\calU$, which, being associative, makes the collection of the $m^{A,B}_r$ associative as well.
\end{proof}

Thus, the associative multiplication $m^{\bullet,\bullet}_r : \mF^\bullet \otimesr \mF^\bullet \to \mF^\bullet$ 
corresponds 
to the Lie structure in $P^{\tensor\ra}(\mF^\bullet)$ given by the collection of morphisms
\[
{\mu}^{\bullet,\bullet}=(m^{A,B}_r,-m^{A,B}_\ell) : \mF^A\tensor\ra\mF^B \oplus \mF^A\tensor\leftarrow \mF^B \to \mF^{A\sqcup B}.
\]
Notice\index{$\mu^{\bullet,\bullet}$} that $\mF^\bullet$ is naturally also a \cgmod $\pbarDpolim{\bullet}$ module and that $\mu^{\bullet,\bullet}$ is a also $\pbarDpolim{\bullet}$ equivariant. Here we did a small change in our notations using the obvious isomorphism $\mF^A\otimesl \mF^B\simeq \mF^B\otimesr \mF^A$. 
If we compose $m$ with the morphism of operads $\jmath^\bullet$ (see Section \ref{ssec:operads}) or equivalently we compose $\vec\mu$ with $\imath^\bullet_2$ then we get the Lie structure on the operad $P^\Delta(\mF^\bullet)$ given by the collection of morphisms
\[
\mu^{A,B} = m^{A,B}_r - m^{A,B}_\ell.
\]
We define $\mu$ and $\mu^{1,2}$ as in formulas \eqref{eq:mu11} and \eqref{eq:mu12} and similarly we define $m_r$, $m_r^{1,2}$, $m_\ell$ and $m_\ell^{1,2}$.

We would like to say that $\mu$ defines a Lie structure on $P^{chBD}(\mF^1)$, however this is not possible since the image of $\mu$ is not contained in $\Delta_!\mF^1$. We will deal with this problem in the next Section. Nevertheless we can prove an analogue of the condition expressed in Lemma \ref{lem:chBDDeltaLie}. Notice first that if we consider $\mF^1$ with the trivial filtration, by Lemmas \ref{lem:chBD1} and \ref{lem:chBD2}, for all $h$ we have a natural map 
$$\gra^h :\big(\mF^1 \exttensor * \Delta_!^{\leq h} \mF^1\big) (\infty \nabla (\pi_{2=3}))
\lra \Delta(\pi_{2=3})_!\big(\mF^1\exttensor * \mF^1(\infty\Delta)\big).$$

\begin{lemma}\label{lem:mu12}
    On the subsheaf $\big(\mF^1 \exttensor * \Delta_!^{\leq h} \mF^1\big) (\infty \nabla (\pi_{2=3}))$ of $\big(\mF^1 \exttensor *  \mF^2\big) (\infty \nabla (\pi_{2=3}))$ we have 
	$$
    \mu^{1,2}=\Delta(\pi_{2=3})^{\leq h}_!(\mu)\circ\gra^h
	$$
\end{lemma}
\begin{proof}We denote $\pi_{2=3}$ by $\pi$. 
	It is enough to prove the statement for $h=0$. Indeed it is easy to check locally that 
	$$\big(\mF^1 \exttensor * \Delta_!^{\leq h+1} \mF^1\big) (\infty \nabla (\pi))=
	\big(\mF^1 \exttensor * \Delta_!^{\leq h} \mF^1\big) (\infty \nabla (\pi))\cdot (\pbarDm 3) ^{\leq 1}.
	$$
	Hence using that these maps are $\pbarD\boxtimes\pbarDm{2}$-equivariant, the claim for a generic $h$ follows by induction from the claim for $h=0$. 
	
	For $h=0$ we have $\Delta_!^{\leq 0} \mF^1=\Delta_* \mF^1$, hence we need to prove that for $X,Y\in \mF^1$, $f\in \pbarOm 3(\infty \nabla(\pi))$ and $h\in \pbarOpolim 3$ we have 
	\begin{equation}\label{eq:mumu}
    \bigg(\big(\Delta(\pi)_*(\mu)\big)\big(f\cdot (X\otimes \tilde Y)\big)\bigg)(h)=
    \bigg(\mu^{1,2}\big(f\cdot (X\otimes \tilde Y)\big)\bigg)(h)
    \end{equation}
    where $\tilde Y=\Delta_*(Y):\pbarOpoliq\lra \calU$ is given by $\ell\mapsto Y(\Delta^\sharp(\ell))$. 
    We compute the two sides of formula \eqref{eq:mumu} using Lemma \ref{rmk:descrprod}. 
    On the left hand side we have
    \begin{align*}
    \bigg(\mu^{1,2}\big(f\cdot (X\otimes \tilde Y)\big)\bigg)(h)& =
    m^{1,2}_r\big(f\cdot (X\otimes \tilde Y)\big)(h)
    -
    m^{1,2}_\ell\big(f\cdot (X\otimes \tilde Y)\big)(h)\\
    &= m^{1,2}_r(X\otimes \tilde Y)(Exp_\pi^r(fh))-m^{1,2}_\ell(X\otimes \tilde Y)(Exp_\pi^\ell(fh))    \end{align*}
    and on the right hand side we have
    \begin{align*}
    \bigg(\Delta(\pi)_*(\mu)&\big(f\cdot (X\otimes \tilde Y)\big)\bigg)(h)=
    \mu\Big(\big(\Delta(\pi)^\sharp(f)\big) (X\otimes Y)\Big)\big(\Delta(\pi)^\sharp(h)\big) \\
   &= m_r\big(X\otimes Y\big)\Big(Exp^{r}_\pi\big(\Delta(\pi)^\sharp(fh)\big)\Big)
    -
    m_\ell\big(X\otimes Y\big)\Big(Exp^{\ell}_\pi\big(\Delta(\pi)^\sharp(fh)\big)\Big)\\
  &= m_r\big(X\otimes Y\big)\Big(\Delta(\pi)^\sharp Exp_\pi ^{r}\big((fh)\big)\Big)
  -
  m_\ell\big(X\otimes Y\big)\Big(\Delta(\pi)^\sharp Exp_\pi^{\ell}(fh)\Big)
  \end{align*}	
Hence it is enough to prove that for all $g\in \pbarOpoli \tensor\ra \pbarOpolim 2$ we have
$$
m^{1,2}_r(X\otimes \tilde Y)(g)=
m_r(X\otimes Y)\big(\Delta(\pi)^\sharp g\big),
$$
where $\Delta(\pi)^\sharp : \pbarOpoli \otimesr \pbarOpoliq \to \pbarOpoli \otimesr \pbarOpoli$, and similarly for the left product. By continuity it is enough to prove the claim form 
$g=g_1\otimes g_2\otimes g_3$ with $g_i\in \pbarOpoli$. 
In this case for the right product we get
$$
m^{1,2}_r(X\otimes \tilde Y)(g)=X(g_1)Y(g_1g_2)=
m_r(X\otimes Y)\big(\Delta(\pi)^\sharp g\big)
$$
and similarly for the left product. 
\end{proof}

As we will see that  this result implies Dong's Lemma (see Lemma \ref{lem:Dong2}).

\begin{example}\label{oss:nonsbagliarsi2}
When we use the notation $\mF^1$ or $\mF^n$ and not $\mF^A$ some information could be hidden. Let us make an example of that. Let $\rho$ be the cyclic permutation of $1$ , $2$ and $3$ and let 
$\nu=\mu^{1,2}\circ \{Id_M,\mu\}$. We want to make explicit what the map $\rho \cdot \nu$ is. Recall that 
$$\nu\in \Big(P^\Delta_{[3]}(\mF^1)\Big)^{[3]}$$ 
has six different components, one for each bijection $q:[3]\lra [3]$. More precisely, since $\nu=\mu^{1,2}\circ \{Id_M,\mu\}$ we need to consider the composition associated to $\pi=\pi_{2=3}\circ q:[3]\lra [2]$ and the $q$ component $\nu_q$ is equal to 
$$
\mu^{\pi^{-1}(1), \pi^{-1}(2)}\circ\{ Id_{\mF^{\pi^{-1}(1)}}, \mu^{\{a\},\{b\}}\}
$$
where $\{a,b\}=\pi^{-1}(2)$ and $q(a)<q(b)$ and $\mu^{\{a\},\{b\}} : \mF^{\{a\}} \exttensor{*} \mF^{\{b\}}(\infty\Delta) \to \mF^{\{a,b\}}$ Recall from the description of the action of bijections on the pseudo tensor structure with poles that $\rho$ sends the $q$ component to the $\rho^{-1}\circ q$ component. In particular the
$q=\id$ component of $\rho\circ \nu$ is equal to $\nu_\rho$, and more explicitly to 
$$
(\rho\cdot\nu)_{\id}=\mu^{\{3\},\{1,2\}}\circ\{Id_{\mF^{\{3\}}},\mu^{\{1\},\{2\}}\}
$$

The latter map is a morphism $\mF^1\exttensor{*}\mF^1\exttensor{*} \mF^1 (\infty\nabla([3])) \to \mF^3$. These three copies of $\mF^1$ corresponds indeed to different sets of one element, for example $\{1\}$, $\{2\}$ and $\{3\}$, the correct interpretation of $\rho\cdot \nu$ is straightforward with our notation if we recall this fact and we keep using sets in our notation. 

If we do not use finite sets and we want to remember their meaning just by placing them in the right position we have 
$$\Big((\rho\cdot \nu)_{\id} (X_1\otimes  X_2 \otimes X_3)\Big)
(f_1\otimes f_2 \otimes f_3)=
\mu^{1,2}\Big( X_3 \otimes \mu(X_1\otimes X_2)\Big) (f_3\otimes f_1\otimes f_2)
$$
\end{example}

\subsection{Mutually local fields}\label{ssez:campimutualmentelocali}
Although we have a Lie structure in $P^\Delta (\mF^\bullet)$ we cannot apply Lemma \ref{lem:chBDDeltaLie} to 
deduce that we have a Lie structure in $P^{chBD}(\mF^1)$. The main reason is that the image of $\mu$ is not contained in $\Delta_!\mF^1$. For this reason we turn now our attention to the case in which we have a $\pbarD$-submodule $V \subset \mF^1$ consisting of \emph{mutually local fields}.

\begin{definition}\label{def:mutuallylocal}
	Two fields $X,Y$ are said to be mutually local if $\mu(X\boxtimes Y) \in \mF^{2\loc}$, where $\mF^{2\loc}$ is the sheaf defined in Section \ref{ssez:campilocali}.
	Two subsheaves $V_1,V_2$ of $\mF^1$ are said to consist of \emph{mutually local fields} if
	\[
	\mu(V_1\boxtimes V_2) \subset \mF^{2\loc}
	\]
	Since $\mu(V_1\boxtimes V_2) = - \mu(V_2\boxtimes V_1)$ this definition is symmetric. Similarly, we say that a subsheaf $V\subset \mF^1$ consists of \emph{mutually local fields} if $\mu(V\boxtimes V) \subset \mF^{2\loc}$.
\end{definition}

As in the theory of vertex algebras mutually local fields behave well under applying $\mu$: we have an analogue of Dong's Lemma.

\begin{lemma}(Dong's Lemma)\label{lem:Dong2}
Let $X,Y,Z\in \mF^1$ be mutually local fields then 
\begin{enumerate}[\indent a)]
    \item if $f\in \pbarOpolim 2 (\infty \Delta)$ then $\mu((X\otimes Y) \cdot f)\in \Delta_!\mF^1$. More precisely if $\mu(X\otimes Y)\in \Delta^{\leq h}_!\mF^1$ and
    $f\in \pbarOpolim 2 (n\Delta)$ then $\mu((X\otimes Y)\cdot f)\in \Delta^{\leq h+n}_!\mF^1$, in particular, it is zero if $h+n<0$. 
    \item if $D\in (\pbarDm 2)^{\leq n}$ and $\mu(X\otimes Y)\in \Delta^{\leq h}_!\mF^1$ then $\mu((X\otimes Y) \cdot D)\in \Delta^{h+n}_!\mF^1$. In particular, using also point a), if $P\in\pbarDpoli$ then $X\cdot P$ and $Y$ are mutually local.
    \item if $f\in \pbarOpolim 2 (\infty \Delta)$ then $X$ and $\piuno \mu((Y\otimes Z)\cdot f)$ are mutually local. More precisely if $f\in \pbarOpolim 2 (n \Delta)$ and $\mu(X\otimes Y), \mu(X\otimes Z), \mu(Z\otimes Y), \in \Delta^{\leq h}_!\mF^1$ then there exists $N$ which depends only on $h$ and $n$ such that 
    $\mu(X\otimes \piuno \mu((Y\otimes Z)\cdot f))\in \Delta^{\leq N}_!\mF^1$.
\end{enumerate}
\end{lemma}
\begin{proof}
    a) Recall that by the local description $\Delta_!^{\leq k}\mF^1$ is the sheaf of $2$-fields annihilated by $\calJ_{\Delta}^{k+1}$, where as usual $\calJ_\Delta$ is the sheaf defining the diagonal in $\pbarOpoliq$. Thus, we need to prove that $h+n$ satisfies $\mu(f(X\otimes) Y) \cdot  \calJ_{\Delta}^{h+n+1} = 0$.  By direct computation, using the fact that $\mu$ is a morphism of $\pbarOpoliq$-modules, we have
	\[
	\mu (f\cdot( X \otimes Y)) \cdot \calJ_{\Delta}^{h+n+1}  = \mu (\calJ_{\Delta}^{h+n+1} f \cdot (X \otimes Y)) \subset \mu (\calJ_{\Delta}^{h+1}\cdot ( X \otimes Y)) = \mu (X\otimes Y)\cdot  \calJ_{\Delta}^{h+1} = 0. 
	\]

    b) This follows from the fact that $\mu$ is $\pbarDpolim 2$ equivariant and $\Delta_!\mF^1$ is a 
$\pbarDpolim 2$ submodule of $\mF^2$.

    c) Let us start by recalling that by Lemma \ref{lem:campilocali} we have $\mF^{2\loc}=\Delta_!\mF^1$ and that $\piuno:\mF^2\lra \mF^1$ was defined in formula \eqref{eq:piuno} by $\piuno(X)(\gra)=X(1\otimes \gra)$. We introduce also $\tilde \piuno:\mF^3\lra \mF^2$ defined as  
$\tilde\piuno (W)(\gra\otimes \grb)=W(\gra\otimes1\otimes \grb)$. We notice that for all
\begin{equation}\label{eq:RtildeRmu}
  m_r\Big(X\otimes \piuno(X' )\Big)=\tilde\piuno \Big(m_r^{1,2}\big(X\otimes X'\big)\Big)
  \qquad 
  m_\ell\Big(X\otimes \piuno(X' )\Big)=\tilde\piuno \Big(m_\ell^{1,2}\big(X\otimes X'\big)\Big)
\end{equation}
for all $X\in \mF^1$ and $Y\in \mF^2$. We check the first formula. By continuity it is enough to check the equality of the two fields over elements of the form $\gra\otimes \grb$
with $\gra,\grb \in \pbarOpoli$. In this case we have:
$$
  m_r\Big(X\otimes \piuno(X' )\Big) (\gra\otimes\grb)=X(\gra)\cdot X'(1\otimes \grb)=
  \tilde\piuno\Big(m_r^{1,2}\big(X\otimes X'\big)\Big)(\gra\otimes\grb).
$$
	
	Let $W=\mu^{1,2}(X\otimes \mu((Y\otimes Z) \cdot f))$. Considering the two formulas \eqref{eq:RtildeRmu} and setting $X'=\piuno(\mu((Y\otimes Z)\cdot f)$ we deduce 
    $$\mu\Big(X\otimes \piuno(\mu((Y\otimes Z)\cdot f))\Big)=\tilde\piuno \Big(\mu^{1,2}\big(X\otimes \mu((Y\otimes Z) \cdot f)\big)\Big).$$
 Hence, to prove the Lemma, it is enough to prove that there exists $N$ such that  $\calJ^N_{\Delta_{1=3}}$ annihilates $W$. We prove that there exists $M$ (depending only on $h$ and $n$) such that $\calJ_{\Delta_{2=3}}^M\cdot W=0$ and 
    $\calJ_{\Delta_{1=2}}^M\calJ_{\Delta_{1=3}}^M \cdot W=0$. Since 
    $\calJ_{\Delta_{1=3}}^{3M+1}\subset \calJ_{\Delta_{1=3}}^M(\calJ_{\Delta_{2=3}}^M\calJ_{\Delta_{1=2}}^M)$ this implies our thesis. 

    By point a) $\mu(Y\otimes Z \cdot f) \in \Delta^{n+h}_!\mF^1$ and by Lemma \ref{lem:mu12} this implies $W \in \Delta(\pi_{2=3})^{\leq m}_!(\mF^2)$ for some $m$ and in particular $\calJ^{m+1}_{\Delta_{2=3}} \cdot W=0$. Moreover the map 
    $\gra^h$ which appears in Lemma \ref{lem:mu12}, as a consequence of Lemma \ref{lem:chBD2} and Lemma \ref{lem:chBD1}, is filtered so $m$ can be chosen to depend only on $h$ and $n$.  

    Being $\mu^{\bullet,\bullet}$ a Lie structure we have 
    \begin{equation}\label{eq:dongjacobi}
    W=-\rho\cdot \nu ((X\otimes Y\otimes Z) \cdot (1\otimes f))
      - \rho^2\cdot \nu ((X\otimes Y\otimes Z) \cdot (1\otimes f))
    \end{equation}
    where $\nu=\mu^{1,2}\circ(Id\otimes \mu):\mF^1\otimes \mF^1\otimes\mF^1(\infty \nabla(3))\lra \mF^3$ and $\rho$ is the cyclic permutation. Since $1\otimes f \in \pbarOpolim 3(n \Delta_{2=3})$ and $X$ and $Y$ are mutually local, by Lemma \ref{lem:mu12} applied to $\pi_{1=2}$ instead that $\pi_{2=3}$ we get that $\rho\cdot \nu ((X\otimes Y\otimes Z) \cdot (1\otimes f))\in \Delta(\pi_{1=2})^{\leq m_1}_!(\mF^2)$ and similarly 
$\rho^2\cdot \nu ((X\otimes Y\otimes Z) \cdot (1\otimes f))\in \Delta(\pi_{1=3})^{\leq m_1}_!(\mF^2)$ and as in the previous case we can chose $m_1$ to depend only on $h$ and $n$. 

Hence  $\calJ_{\Delta_{1=2}}^{m_1}\calJ_{\Delta_{1=3}}^{m_1} $
annihilates the righthand side of formula \eqref{eq:dongjacobi}, proving our claim.
\end{proof}

\subsection{Unity and chiral algebra structure}
Recall from \cite[Def. 3.3.3]{BDchirali} the following definition.

\begin{definition}[Chiral algebra]
	A chiral algebra\index{chiral algebra} is a 
	right $\pbarD$ filtered-\ccgmod module $M$ equipped with a Lie structure $\mu\in P^{chBD}(M)$ and a $\pbarD$ equivariant morphism, called unity, $\mathbf{1}_M:\pbarOmega\lra M$\index{chiral unit map, $\mathbf{1}$} such that the composition
	$$
	\mu \circ (\unoU\boxtimes id):\pbarOmega\exttensor * M(\infty \Delta)\lra \Delta_!M
	$$ is zero on $\pbarOmega\exttensor * M$ and, on the quotient $\pbarOmega\exttensor * M(\infty \Delta)/\pbarOmega\exttensor* M$, induces the inverse of the isomorphism $D(M)$ of Proposition \ref{prop:defD}.
\end{definition}

\begin{remark}
	Since $\pbarOmega$ $w$-compact, the map $\unoU$ is automatically filtered. 
\end{remark}

In the case of fields we have a canonical unity. 

\begin{definition}
	We define the \emph{unit map}
	\[
	\mathbf{1} : \pbarOmega \to \mF^1_{\Sigma,\calU}
	\]
	as follows. To a local $1$-form $\omega \in \pbarOmega$ we associate the field 
	\(
	f \mapsto   \left( \Res_\Sigma f\omega\right) \cdot 1_\calU
	\),
	where $1_\calU$ is the unit in $\calU$ and $\Res_\Sigma$ is the residue map defined in Section \ref{sec:residuo}.
\end{definition}

\begin{proposition}\label{prop:chiralunit}
	The following hold
	\begin{enumerate}
		\item $\unoU $ is a continuous map. 
		\item The composition
		\[\xymatrix{
			\pbarOmega \boxtimes \mF^1\ar[rr]^{\mathbf{1}\boxtimes\id} & & \mF^1\boxtimes \mF^1 \ar[rr]^{\mu}&& \mF^2}
		\]
		vanishes;
		\item The resulting map induces an isomorphism
		\[
		\frac{\pbarOmega \exttensor * \mF^1 (\infty\Delta)}{\pbarOmega \exttensor * \mF^1} \to {\Delta_! \mF^1}
		\]
		which coincides with the inverse of the canonical map $D(\mF^1)$ of Proposition \ref{prop:defD}.
		\item For any $\calD_{\overline{\Sigma}}$-submodule $V\subset\mF^1$ we have that $\mu \circ (\unoU\boxtimes \id)\left(\pbarOmega \exttensor * V (\infty\Delta)\right) \subset \Delta_! V$ and that
		\[
		\mu \circ (\unoU\boxtimes id): \frac{\pbarOmega \exttensor * V (\infty\Delta)}{\pbarOmega \exttensor * V} \to \Delta_! V
		\]
		is the inverse of the canonical isomorphism $D(V)$.
	\end{enumerate}
\end{proposition}

\begin{proof} 
	The continuity of $\unoU$ can be checked locally, hence we choose a well covered open subset of $S$ so that $\pbarOmega =\pbarO dt$. Then the continuity of 
	$\unoU$ follows from $\pbarO(n)\cdot\pbarO(m)\subset \pbarO(n+m)$. 
	
	Point (2) directly follows from the fact that $\calO_S \cdot 1_\calU$ commutes with all elements in $\calU$.
	
	We prove (3). We set $\nu=\mu \circ (\unoU\boxtimes \id)$. Point (2) proves that $\pbarOmega$ and $\mF^1$ are mutually  local, hence by \ref{lem:Dong2}, the image of  $\nu$ is contained in   $\Delta_! \mF^1$. More precisely, since $\mu(\pbarOmega\exttensor* \mF^1)=0$, the same proof of \ref{lem:Dong2} shows that 
	$\mu\big(\pbarOmega\exttensor*\mF^1(n\Delta)\big)$ is annihilated by $\calJ_\Delta^n$, hence is contained in 
	$\Delta_!^{\leq n-1}\mF^1$. In particular $\nu$ induces a morphism	
	\[
	\nu_0: \frac{\pbarOmega \exttensor * \mF^1 (\Delta)}{\pbarOmega \exttensor * \mF^1} \to \Delta_* \mF^1
	\]
	which is the subspace of $\mF^2$  killed by  $\calJ_\Delta$. Since, in general, $\Delta_!\calM$ is generated as $\pbarDm 2$-module by $\Delta_*\calM$, a consequence of Proposition \ref{prop:defD} (b) we have that 
	${\pbarOmega \exttensor * \mF^1 (\infty\Delta)}/{\pbarOmega \exttensor * \mF^1}$ is generated as a $\pbarDm 2$ module by $\pbarOmega \exttensor * \mF^1 (\Delta)/{\pbarOmega \exttensor * \mF^1}$. Hence it is enough to prove that the inverse of $D(\mF^1)$ agrees with $\nu_0$ on $\pbarOmega \exttensor * \mF^1 (\Delta)/{\pbarOmega \exttensor * \mF^1}$. 
	
	By construction, the restriction of $D(\mF^1)$ to $\Delta_*\mF^1$ is equal to the inverse of the isomorphism $\Phi_{\mF^1}$ defined in equation \eqref{eq:definizioneD}, so that in order to prove that  $D(\mF^1)$ is equal to  $\nu$ it is the enough to prove that $\nu_0$ and $\Phi_{\mF^1}$ are equal. We will see that this essentially follows from the Cauchy Formula of Proposition \ref{prop:Cauchy}. To check this fact we can argue locally, so assume that $S$ is well covered and that $t$ is a local coordinate. Hence $\pbarOmega =\pbarO dt$ and 
	$$\pbarO(\Delta) = \frac 1 \delta \pbarO\subset \pbarO(\infty\Delta) $$
	where $\delta=t\otimes 1 - 1 \otimes t$. It is easy to check that the  map $\Phi_{\mF^1}$ in these coordinates is given by 
	\begin{equation} \label{eq:Dmu}
	\Phi_{\mF^1} \left( (f dt \boxtimes  X) \cdot \frac {1}\delta \right) = X \cdot f 
	\end{equation}
	where $f\in \pbarO$ and $X\in \mF^1$. Here the one field on the right is seen as a  $2$-field  via the precomposition with the multiplication map $\Delta^\sharp:\pbarOpoliq \ra \pbarOpoli$. 

    We now evaluate the field $\nu_0( (f dt \boxtimes  X) \cdot \frac {1}\delta)$ on an element 
    $k\in \pbarOpolim{2}$. By  definition, and by Lemma 6.4.5,  we have
\begin{align*}
    \nu_0&\left((f dt \boxtimes  X) \cdot \frac{1}{\delta})\right) (k)=
    \mu\left( (\unoU(f dt) \boxtimes  X) \cdot \frac {1}{\delta}\right) (k)\\ &=
    m_r\left((\unoU(f dt) \boxtimes  X)\cdot \frac {1}{\delta}\right)(k)
    -
    m_\ell\left((\unoU(f dt) \boxtimes  X)\cdot \frac {1}{\delta}\right)(k)\\
     & =
     m_r\Big(\unoU(dt)\boxtimes X\Big) \big((f\otimes 1)k\cdot  Exp^r(\delta^{-1})\big)
     -
     m_\ell\Big(\unoU(dt)\boxtimes  X\Big)\big((f\otimes 1)k\cdot Exp^\ell(\delta^{-1})\big).\\
\end{align*}
Recall now the definition of the operator $T_r$ and $T_\ell$ from Section \ref{ssez:Cauchyformula} and notice that, by definition, we have 
$$
   m_r\Big(\unoU(dt)\boxtimes X\Big) (\gra) = X\big(T_r(\gra)\big)
   \quad\mand\quad
    m_\ell \Big(\unoU(dt)\boxtimes X\Big) (\grb) = X\big(T_\ell(\grb)\big)
$$
for all $\gra \in \pbarOpoli\tensor{\ra}\pbarOpoli$ and for all $\grb \in \pbarOpoli\tensor{\la}\pbarOpoli$. Hence, by Proposition \ref{prop:Cauchy}, in our case we obtain 
\begin{align*}
\nu_0\left((f dt \boxtimes  X) \cdot \frac{1}{\delta}\right) (k)&= 
X\Big(T_r\big((f\otimes1)k\cdot Exp^r(\grd^{-1})\big) - T_\ell\big((f\otimes1)k\cdot Exp^\ell(\grd^{-1})\big) \Big)\\
&=X(\Delta^\sharp (f\otimes1)k)= X(f\cdot\Delta^\sharp (k))
\end{align*}
Hence we conclude $\nu_0=\Phi_{\mF^1}$.
 
	To prove (4) notice that by formula \eqref{eq:Dmu} we have that $\nu (\pbarOmega \exttensor *V(\Delta))\subset \Delta_*V$. Hence $\nu (\pbarOmega \exttensor * V(\Delta))\subset \Delta_!V$ and moreover the restriction $\nu_V$ of $\nu $ to $\pbarOmega \exttensor * V(\Delta)/\pbarOmega \exttensor * V$ is equal to $\Phi_V$. Hence $\nu_V$ is the inverse of $D(V)$.
\end{proof}

Putting together Lemma \ref{lem:mu12} and Proposition \ref{prop:chiralunit} we get the following theorem.

\begin{theorem}\label{thm:chiralalgebra}
	Let $V \subset \mF^1$ be a  $\pbarD$-submodule consisting of mutually local fields. 
	Assume that $V$ is filtered, that
	\[
	\mu(V\boxtimes V (\infty\Delta)) \subset \Delta_! V \subset \Delta_!\mF^1
	\]
	and that the map $\mu$ is filtered. Assume also that the unit map $\mathbf{1} : \pbarOmega \to \mF^1$ has image contained in $V$. Then $V$ with unit map $\mathbf{1}$ and chiral product $\mu : V\boxtimes V (\infty\Delta) \to \Delta_! V$ is a chiral algebra.
\end{theorem}

\begin{remark}
	Notice that if $V$ is a \ccg-module, then the assumption about filtration are automatically satisfied choosing any \exhaustion of $V$. 
\end{remark}

\subsection{Chiral algebra generated by mutually local fields}\label{ssez:generatechiral}
The scope of this Section is to prove that given a set of mutually local fields, we can construct a chiral algebra generated by these fields. The fact that we are working with topological modules forces us to slightly change the classical construction, in particular we will construct fourt different possible chiral algebras depending on how we want to treat the topologies. As in the classical case, the main ingredient we need is Dong's Lemma.

\subsubsection{Basic construction} Let $G\subset \mF^1$ be a $\calO_S$ subsheaf and assume that 
$
\mu(G\otimes G) \subset \Delta^{\leq h}_! \mF^1.
$
We define inductively 
\begin{align*}
\vbasic(0)&=G+\calO_S \unoU \qquad 
\mand \\
\vbasic(n+1)&=  \vbasic(n)\cdot \pbarD^{\leq 1}+ \piuno \Big(\mu (\vbasic(n)\otimes \vbasic(n) (\Delta)) \Big)
\end{align*}
By \index{$\vbasic$}\index{$\vbasic(n)$} Lemma \ref{lem:Dong2} for each $n$ there exists $h_n$ such that 
$\mu(\vbasic(n)\otimes \vbasic(n)) \subset \Delta^{\leq h_n}_! \mF^1$. 

\begin{lemma}\label{lem:RmuD} For all $n,m$ we have
$\piuno \mu(\vbasic(n)\otimes \vbasic(n) (m\Delta))\subset \vbasic(n+2m).$
\end{lemma}
\begin{proof}
We check the claim locally. Let $S$ be well covered and let $t$ be a local coordinate. Let $u=t\otimes 1$ and $v=1\otimes t$ so that the ideal of the diagonal is generated by $\delta=u-v$. We check that for all $X,Y\in \mF^1$ we have 
\begin{equation}\label{eq:RmuD}
\piuno \mu ((X\otimes Y) \, \delta^{-m}) \cdot \partial _t =
\piuno \mu ((X\otimes (Y\cdot \partial_t)) \, \delta^{-m}) - 
m\piuno \mu ((X\otimes Y) \, \delta^{-m-1}). 
\end{equation}
Indeed let 
$Exp^r (\delta ^{-m})=\sum_{i} a_i\otimes b_i$ in $\pbarOpoli\tensor\ra \pbarOpoli$. Then 
$mExp^r(\delta^{-m-1})=\partial _v Exp^r(\delta^{-m})=\sum a_i\otimes b_i'$.
Similarly write $Exp^\ell (\delta ^{-m})=\sum_{i} c_i\otimes d_i$. Then 
\begin{align*}
\left(\piuno \mu ((X\otimes Y) \, \delta^{-m}) \cdot \partial _t\right) (f) &= 
\sum X(a_i)Y(b_if')-\sum Y(d_if')X(c_i) \\
&=\sum X(a_i)Y((b_if)')-\sum Y((d_if)')X(c_i)\\ 
&\qquad  -\sum X(a_i)Y(b'_if)+ \sum Y(d'_if)X(c_i) \\
&=\left(\piuno \mu ((X\otimes (Y\cdot \partial_t)) \, \delta^{-m}) \right) (f) - m\left(\piuno \mu ((X\otimes Y) \, \delta^{-m-1}) \right) (f) 
\end{align*}
Now the claim of the Lemma follows by induction on $m$ and formula \ref{eq:RmuD}. 
\end{proof}

Hence, the subsheaf $\vbasic=\vbasic(G)=\bigcup \vbasic(n)$ is stable under the action of $\pbarDpoli$ and satisfies $\mu(\vbasic\otimes \vbasic(\infty \Delta))\subset \Delta_!\vbasic$. In addition the map $\mu$ is filtered so by Theorem \ref{thm:chiralalgebra} $V_{\text{basic}}$ is a chiral algebra. 

We can define a variation of this construction by defining
$$
\vbasicpiu=\limind \vbasic(n).
$$
The sheaves $\vbasic$ and $\vbasicpiu$\index{$\vbasicpiu$} coincide, however, in general their topology can be different. This is due to the following fact: the topology of $\vbasic$ is induced by that of $\mF^1$, in particular it has a countable system of neighborhoods of $0$; on the other hand $\vbasicpiu$ is typically a direct sum of complete topological modules which does not have a countable system of neighborhoods of $0$.

\subsubsection{Completed construction}
Notice that $\Delta^{\leq m}_!\mF^1$ is a closed subsheaf of $\mF^2$. In particular if 
$\mu(G\otimes G) \subset \Delta^{\leq h}_! \mF^1$ then 
the closure $\overline G$ of $G$ satisfies the same property. 
We define inductively 
\begin{align*}
	\vcom(0)&=\overline{G+\calO_S \unoU} \qquad 
\mand \\
\vcom(n+1)&= \overline{ \vcom(n)\cdot \pbarD^{\leq 1}+ \piuno \Big(\mu (\vcom(n)\otimes \vcom(n) (\Delta)) \Big)}
\end{align*}
As in the previous case if we define $\vcom=\vcom(G)=\bigcup \vcom(n)$\index{$\vcom$} it is a chiral algebra and we can define a variation of this construction by defining
$$
\vcompiu=\limind \vcom(n).
$$
\index{$\vcompiu$}

\section{Factorization structures}\label{sec:factorizationstructures}
Here we introduce factorization patterns in our geometric setting; our notion of a factorization algebra is a translation in our $X,S,\Sigma$ geometric setting of the notion of factorization algebra of \cite{BDchirali}. Reading this Section it will be extremely important to recall the notation introduced in \ref{ssez:artquestonotazionifiniteset}. We recall briefly some of that notation and we introduce some further notations regarding diagonals.

\begin{ntz}\label{ntz:factorization}
	Let $I$ and $J$ be finite sets, suppose that $\Sigma : S \to X^J$ consists of $J$ sections and let $\pi : J \twoheadrightarrow I$ be a surjection. Denote by $j_{J/I} : U_\pi=U_{J/I} \to S$ be the open immersion corresponding to the open subset $U_{J/I} = \{ x \in S : \sigma_{i_1}(x) \neq \sigma_{i_2}(x) \text { if } \pi(i_1) \neq \pi(i_2) \}$\index{$j_{J/I} : U_{J/I} \to S$} and $i_{J/I} : V_{J/I} \to S$\index{$i_{J/I} : V_{J/I} \to S$} the closed immersion corresponding to the closed subscheme $V_{J/I} = \{ x \in S : \sigma_{i_1}(x) = \sigma_{i_2}(x) \text { if } \pi(i_1) = \pi(i_2) \}$. For $i \in I$ let $\Sigma_{J_i}$ be the subset of sections of $\Sigma$ corresponding to $J_i = \pi^{-1}(i)$. By definition the various $\Sigma_{J_i}$ are disjoint from one another when restricted to $U_{J/I}$. On the other hand when restricted to $V_{J/I}$ all sections inside $\Sigma_{J_i}$ coincide, so that $\Sigma$ on $V_{J/I}$ factors through a collection of sections $\Sigma_I : V_{J/I} \to (i_{J/I}^*X)^I$. Recall also that the choice of a coordinate $t$ on $\Sigma$ induces a choice of a coordinate $t_i$ for $\calO_{\overline{\Sigma}_{J_i}}$ and a coordinate $t_I$ for $\calO_{\overline{\Sigma}_{I}} : V_{J/I} \to (i_{J/I}^*X)^I$. We will write $\Sigma_J,\Sigma_I$ in place of $\Sigma$ to stress the dependence on the finite set $I$ when needed, following the above conventions.
\end{ntz}

\begin{definition}\label{def:factorizationcompletesheafsigma}
		Consider $\calA = \{ \calA_{\Sigma}\}_{S,X,\Sigma}$ to be a collection of complete topological $\calO_S$-modules (often QCC) sheaves on $S$ indexed by topologically noetherian schemes $S$, families of curves $X \to S$ and collection of sections $\Sigma : S \to X^I$ (we let the index $I$ vary as well).
		A \emph{pseudo-factorization algebra structure} on $\calA$\index{pseudo-factorization algebra} is the datum, for any surjection of finite sets $\pi : J \twoheadrightarrow I$, of continuous morphisms of $\calO$-modules
		\begin{align*}
			\Ran^{\calA}_{J/I} &: \hat{i}_{J/I}^*\left( \calA_{\Sigma_J} \right) \rightarrow \calA_{\Sigma_I}, \\
			\fact^{\calA}_{J/I} &: \hat{j}_{J/I}^*\left( \bigotimes^!_{i \in I} \calA_{\Sigma_{J_i}} \right) \rightarrow \hat{j}_{J/I}^*\left( \calA_{\Sigma_J} \right) .
		\end{align*}
		Where $\hat{i}_{J/I}^*$ and $\hat{j}_{J/I}^*$ denote the completed pullback along $i_{J/I} : V_{J/I} \to S$ and $j_{J/I} : U_{J/I} \to S$. We require these morphisms to be compatible with one another when considering surjections $K \twoheadrightarrow J \twoheadrightarrow I$ in analogy with \cite[3.4.4]{BDchirali}, that is:
		\begin{itemize}
			\item $\Ran^{\calA}_{K/I} = \Ran^{\calA}_{J/I}\circ\left(\hat{i}_{J/I}^*\Ran^{\calA}_{K/J}\right)$;
			\item $\fact^{\calA}_{K/J} = \left( \hat{j}^*_{K/J} \fact^{\calA}_{K/I} \right) \circ \left(\hat{j}_{K/I}^* \bigotimes^!_{i \in I} \fact^{\calA}_{K_i/J_i} \right)$ ;
			\item $ \left( \hat{j}_{J/I}^*\Ran^{\calA}_{K/J} \right) \circ \left( \hat{i}_{K/J}^*\fact^{\calA}_{K/I} \right) = \fact^{\calA}_{J/I} \circ\left(\hat{j}_{J/I}^*\bigotimes^!_{i \in I} \Ran^{\calA}_{K_i/J_i}\right) $. 
		\end{itemize}
		We say that the collection $\calA_{\Sigma}$ is a \emph{completed topological factorization algebra} \index{completed topological factorization algebra}if the morphisms $\Ran$ and $\fact$ are isomorphisms.
	\end{definition}

	\begin{remark}
		Often the sheaves $\calA_{\Sigma}$ are endowed with more algebraic structures (e.g. they are associative or Lie algebras) and we will consider morphisms $\Ran,\fact$ which preserve the appropriate kind of structure. Let us say that the term \emph{algebra} in \emph{factorization algebra} has nothing to do with this extra structure.
	
		Moreover, note that the $\otimes^!$ product appearing in the definition may be replaced by any other symmetric monoidal operation on complete topological sheaves, for instance by $\times$ or $\oplus$. We will refer to such structures on a family of sheaves $\calA_\Sigma$ as a (pseudo) factorization algebra structure with respect to $\times,\oplus$.
	\end{remark}
	
	\begin{example}\label{ex:factstructurepbaro}
		The collection of sheaves $\pbarO$ is a completed topological factorization algebra with respect to $\times$. In order to prove this, notice that $\hat{i}^*_{J/I}(\pbarO),\hat{j}^*_{J/I}(\pbarO)$ are the sheaves obtained by the same construction of $\pbarO$ for the family of curves $X \times_S V_{J/I} \to V_{J/I}$ and $X \times_S U_{J/I} \to U_{J/I}$ respectively. To show the claim it is then enough to show that
		\begin{enumerate}
			\item Given a surjection $J \twoheadrightarrow I$ such that $\sigma_{j_1} = \sigma_{j_2} : S \to X$ anytime $\pi(j_1) = \pi(j_2)$ then $\calO_{\overline{\Sigma}_J} = \calO_{\overline{\Sigma}_I}$. This is obvious since the sections are the same but with higher multiplicity, and such information disappears when taking the completion;
			\item Given a surjection $J \twoheadrightarrow I$ such that $\sigma_{j_1}(S) \cap \sigma_{j_2}(S) = \emptyset$ anytime $\pi(j_1) \neq \pi(j_2)$ then $\calO_{\overline{\Sigma}_J} = \times_{i \in I }\calO_{\overline{\Sigma}_{J_i}}$. This is easy to check since any quotient $p_*(\calO_X/\calI_{\Sigma_J}^n)$ splits as $\times_{i \in I} p_*(\calO_X/\calI_{\Sigma_{J_i}}^n)$.
		\end{enumerate}
	\end{example}

\subsection{Pullbacks and factorization of fields}\label{ssec:pullbackcampi}
Let $f:S'\lra S$ be a morphism and assume that both $S$ and $S'$ are topologically noetherian. Denote by $p':X'\lra S'$ the pullback of $p$. If $\Sigma$ is a set of sections of $p$ we denote by $\Sigma'$ the Sections of $p'$ induced by $\Sigma$. Recall that in Section \ref{ssec:pullbackcompletati} we defined a completed version of pullback, In the remaining of this Section, we will refer to completed pullbacks simply as pullbacks.

\begin{remark} We do some general remarks about pullbacks of \QCC sheaves.
In Section \ref{ssec:pullback1} we noticed that the pullback of $\pbarO$ and $\pbarOpoli$ are isomorphic to $\calO_{\overline{\Sigma'} }$ and $\calO_{(\overline{\Sigma'})^* }$.
By Remark \ref{ssec:pullbackcompletati} pullback of \QCC sheaves is a \QCC sheaves and by Proposition \ref{prop:pullbacktensorproduct} pullback commutes with all completed tensor product.  This implies that the analogous claim is true for the sheaves $\pbarOm I$ and $\pbarOpolim I$.  

Moreover if $S$ is affine and well covered, $t$ is a local coordinate over $S$ and $S'$ is affine, then the pullback of $t$ is a local coordinate over $S'$. From the local description of $\pbarD$ and $\pbarDpoli$ and by Remark \ref{oss:prodottiregolarifreccetta} it follows that the pullback of these sheaves is respectively equal to $\calD_{\overline{\Sigma'}}$ and
$\calD_{(\overline{\Sigma'})^*}$. 
\end{remark}

Let now $\calU$ be a topological algebra whose topology is generated by left ideals. 
Let $\calU'=\hat f ^*\calU$ the topological pullback of 
$\calU$. It is easy to check that $\calU'$ is also a algebra whose topology is generated by left ideals and if $\calU$ is \QCCF, then $\calU'$ is \QCCF as well. 

Given a collection of sections $\Sigma$, indexed by the finite set $I$, let $\Sigma'$ be the pullback of these sections. We denote by $\mu^{\bullet,\bullet}$ the product of the fields over $S$ with values in $\calU$ and by $(\mu')^{\bullet,\bullet}$ the product of the fields over $S'$ with values in $\calU'$. 
Given a field $X:\pbarOpolim I\lra \calU$ its pullback is a field
$\hat f^*X:\calO_{(\overline{\Sigma '})^*}^I \lra \calU'$. This induces a continuous morphism 
$$
\hat f^*_\mF: \hat f^* \Big( \mF^I_{\Sigma,\calU}\Big) \lra \mF^I_{\Sigma',\calU'}
$$
given by $g\otimes X\mapsto g \cdot \hat f^*X$ for $g\in \calO_{(\overline{\Sigma '})^*}^I$.

\begin{lemma}\label{lemma:pullbackoffields} Assume that $\calU$ is a topological algebra whose topology is generated by left ideals. 
\begin{enumerate}[\indent a)]
   \item If $\calU$ is \QCCF then $\hat f^* (\mF^I_{\Sigma',\calU'})\simeq \mF^I_{\Sigma',\calU'}$
   \item $\hat f_\mF^* \circ (\mu^{\bullet,\bullet})=(\mu')^{\bullet,\bullet} \circ (\hat f_\mF^*\times \hat f_\mF^*)$
\end{enumerate}
\end{lemma}
\begin{proof}
a)  To prove that $\hat f ^*_\mF$ is an isomorphism we can argue locally and by what noticed in Section \ref{ssec:pullbackcompletati} we can argue at the level of modules. So we need to prove that 
$$
\limpro  A'\otimes_A  \frac{\Homcont_A (\pbarOpolim I, \calU)}{\mH_{K,V}} \simeq 
\Homcont_A (  \calO^I_{(\overline{\Sigma'})^*} , \calU')
$$
where the limit is over the compact submodules $K$ of $\pbarOpolim I$
and the \noz $V$ of $\calU$. Under our assumption we can choose $\mH_{K,V}$ (see \ref{ssec:omomorfismi}) so that $K=\pbarOm I (n)$, that has a free complement in $\pbarOpolim I$, and $V$ has a complement in $\calU$. We have the following short exact sequence
$$
	0 \to \mH_{K,V} \to \Homcont_A(\pbarOpolim{I},\calU) \to \Homcont_A(K,\calU/V) \to 0
$$
Moreover $\Homcont\left(K,\calU/ V \right) =
\limind _U \Hom_A (K/U,   \calU/ V) $ where $U$ is a \noz of $K$ and in our case we can assume that $K/U$ is free and finitely generated. In particular $\Hom_A (K/U,   \calU/ V) $ is free, The filtered colimit of flat modules is flat, so  $\Homcont\left(K,\calU/ V \right)$ is a flat $A$ module. Hence tensoring with $A'$ we still get an exact sequence. Moreover tensoring with $A'$ commutes with colimits, hence 
\begin{align*}
A'\otimes _A\Homcont_A\left(K,\frac \calU V \right) &= 
\limind_U \Hom_{A'}\left( A\otimes_A \frac K U , A'\otimes_A \frac \calU V\right)\\ &=
\Homcont_{A'} \left(K',A'\otimes \frac \calU V\right)=\mH_{K',V'}
\end{align*}
where $K'=\calO_{\overline{\Sigma'}}(n)$ and $V'$ is the completion of $A'\otimes V$ in $\calU'$. Hence we have 
$$
A'\otimes_A \frac{\Homcont_A (\pbarOpolim I, \calU)}{\mH_{K,V}} \simeq 
\frac{\Homcont_{A'} (\calO_{(\overline{\Sigma'})^*} ^I, \calU')}{\mH_{K',V'}}
$$
since the colimit over all $K'$ is equal to $\calO_{(\overline{\Sigma'})^*} ^I$, 
taking the limit over $K$ and $V$ we get the desired isomorphism. 

b) The morphism $\mu^{I,J}$ and $(\mu')^{I,J}$ are defined exactly in the same way via the morphisms $m_r,m_l$ which themselves may be described via Lemma \ref{rmk:descrprod}. It is then enough to notice that the right and left expansion $Exp^r$ and $Exp^\ell$ commute with pullbacks and that $\hat f^*(\calJ_\Delta)$ is the ideal defining the diagonal in $\calO_{(\overline{\Sigma'})^*} ^2$.
\end{proof}

\subsubsection{Factorization of fields}\label{ssec:fattorizzazionecampigenerale}
Let $\calU$ be a (pseudo)-factorization algebra and assume in addition that $\calU$ is a \QCCF $\otimesr$ associative algebra and that the map $\Ran$ and $\fact$ are of $\otimesr$ algebras.
This induces a pseudo factorization structure on the space of fields. 
Let $f:S' \to S$ be a quasi-separated morphism of schemes and define $\calU'$, $p'$, $\Sigma'$ as in the discussion above. Let $\pi:J\surjmap I$ be a surjective map among finite sets. The behaviour of the space of fields described in the previous Lemma induces an isomorphisms 
$$
\Ran^{\mF^\bullet}_{J/I}:\hat i ^*_{J/I} \mF^\bullet_{\Sigma,\calU} 
\lra \mF^\bullet_{\Sigma',\calU'} 
$$
Moreover we have a similar isomorphism for the restriction $j_{J/I}$ to the open subset. Using the factorization properties of $\pbarOpoli$ now notice that over the open set $U_{J/I}$ the sheaf $\calO^\bullet_{\bar \Sigma^*_J}$
splits as a product $\prod_{i\in I} \calO^{\bullet}_{\bar \Sigma^*_{J_i}}$ and the algebra $\calU_{\Sigma_J}$ factors as $\tensor ! _{i\in I} \calU_{\Sigma_{J_i}}$. Hence we have a map 
$$
\fact^{\mF^\bullet}_{J/I}: \prod_{i\in I}\hat j ^*_{J/I} \mF^\bullet_{\Sigma_{J_i},\calU_{\Sigma_{J_i}}} 
\lra 
\hat j ^*_{J/I} \mF^\bullet_{\Sigma_J,\calU_{\Sigma_J}} 
$$
which sends the $I$-tuple of fields $(X_i)$ to the field 
$X((f_i)_{i\in I})=\sum X_i(f_i)$ where we map $\calU_{\Sigma_{J_i}}$ into $\calU_{\Sigma_J}$ via $u_i\mapsto 1\otimes \cdots u_i\otimes \cdots 1$. It is easily seen that this map is continuous and that if $\calU$ is a \QCCF sheaf then it is injective. 

These maps are compatible with each other in a factorization structure fashion, so that $\mF^1_{\Sigma,\calU}$ is a complete topological pseudo factorization algebra (with respect to $\prod$). These maps preserve the product $\mu^{\bullet,\bullet}$.

Finally notice if $\calU$ is not \QCCF we cannot apply point a) of the previous lemma, but we still have a pseudo factorization structure on the space of fields compatible with the chiral product $\mu$. 

\section{The chiral algebra \texorpdfstring{$\VkSg$}{V Sigma}}\label{sez:chiralig}

In this Section we introduce a version of the affine Lie algebra $\hat{\gog}$ in our $X,S,\Sigma$ setting. The usual affine Lie algebra corresponds to the case of $S$ a point and one single Section. Generalizations of these constructions were considered and studied by many authors, certainly by Gaitsgory and Raskin \cite{Raskin} (see also \cite{fortuna2022local}, \cite{cas2023} where the center at the critical level of their enveloping algebra was computed).

\subsection{The affine algebra \texorpdfstring{$\hat{\gog}_\Sigma$}{g}}\label{sec:thecaseoftheaffinealgebra}

We apply the results of the previous sections to the case of the affine algebra $\hat{\gog}_{\Sigma}$. We keep notation from the previous sections when considering $S,X,\Sigma$. From now on let $\gog$ be a finite dimensional Lie algebra over $\mC$ and let $\kappa$ be a symmetric invariant bilinear from on it. We refer to $\kappa$ as the \emph{level}. When $\gog$ is simple we identify $\kappa$ with the number $k\in \mC$ such that $\kappa = k \kappa_{\gog}$ where $\kappa_{\gog}$ is the Killing form: $\kappa_{\gog}(X,Y) = \text{tr}(\text{ad}(X)\text{ad}(Y))$.

\begin{definition}\label{def:constructiongogsigma}
	We define the affine Lie algebra relative to $\Sigma$ of level $\kappa$ as the sheaf of $\calO_S$-Lie algebras
	\[
	\gsig{\kappa} \stackrel{\text{def}}{=} \gog \otimes_\mC \pbarOpoli \oplus \calO_S\mathbf{1}.
	\]
	The\index{$\hat{\gog}_{\Sigma,\kappa}$} bracket is defined declaring $\mathbf{1}$ to be central while for $X,Y \in \gog$ and $f,g \in \pbarOpoli$ we set
	\[
	[X\otimes f, Y\otimes g] = [X,Y]\otimes (fg) + \kappa(X,Y)\mathbf{1}\Res_\Sigma gdf.
	\]
	For $n \geq 0$ the subsheaves
	\[
	\hat{\gog}_{\Sigma,\kappa}^{\geq n} = \gog\otimes_\mC \pbarO(-n\Sigma)
	\]
	are Lie subalgebras and determine a topology on $\hat{\gog}_{\Sigma,\kappa}$ which makes it into a \QCCF, ccg sheaf.
\end{definition}

We construct a completion of the enveloping algebra of $\gsig{\kappa}$. Let $\calU^0(\gsig{\kappa})$ be the sheaf of universal enveloping algebras of $\gsig{\kappa}$. And let $\calU^1(\gsig{\kappa})$ be the quotient of $\calU^0(\gsig{\kappa})$ by the two sided ideal generated by $1 - \mathbf{1}$, where $1$ is the unit of $\calU^0(\gsig{\kappa})$ and $\mathbf{1}$ is the central element in $\gsig{\kappa}$. Finally consider the completion
\[
\ugsig{\kappa} = \varprojlim_n\frac{\calU^1(\gsig{\kappa})}{\calU^1(\gsig{\kappa})\hat{\gog}_{\Sigma}^{\geq n}}
\]
\index{$\ugsig{\kappa}$}
\begin{definition}
	We define the completed universal enveloping algebra of $\gsig{\kappa}$ as a the sheaf of associative algebras $\ugsig{\kappa}$ constructed above.
\end{definition}

\begin{remark}\label{rmk:polyPBW}
	It's easy to check that the product on $\calU^1(\gsig{\kappa})$ is continuous with respect to the topology given by the open neighborhoods of zero $\calU^1(\gsig{\kappa})\hat{\gog}_{\Sigma}^{\geq n}$. Since these are left ideals this product makes $\ugsig{\kappa}$ into a $\otimesr$-topological algebra.
\end{remark}

\subsubsection{Local structure and PBW}\label{sssec:PBWg}
We study the local structure of the topological algebra $\ugsig{\kappa}$ and describe a topological basis for this algebra. We assume that $S=\Spec A$ is well covered and that $t$ is a local coordinate.  Recall that the PBW theorem holds for a Lie algebra $L$ over a ring $A$ if $L$ is a free $A$-module. The rings $R=\pbarO(S)$ and $K=\pbarOpoli(S)$ are not free over $A$, for this reason we introduce a polynomial version of $\gsig{\kappa}$. If we further restrict $S$, by Remark \ref{oss:unsolof}, we can assume that there exists $f=\prod(t-b_i)$
$$
R=\pbarO=\limpro \frac {A[t]}{ (f^n)}
\quad\mand\quad
K=R[f^{-1}].
$$
We define $R^{\text{poly}}=A[t]$, $K^{\text{poly}}=R^{\text{poly}}[f^{-1}]$ and $R^{\text{poly}}(n)=f^{-n} R^{\text{poly}}\subset K^{\text{poly}}$. The submodules $R^{\text{poly}}(n)$
define a topology on $R^{\text{poly}}$ and $K^{\text{poly}}$ and their completions are equal to $R$ and $K$. We define an analogue $\gsig{\kappa}^{\text{poly}}$ of $\gsig{\kappa}$ by replacing $K$ with $K^{\text{poly}}$. Define also $\gsig{\kappa}^{\text{poly},\geq n}$, $\calU^0(\gsig{\kappa}^{\text{poly}})$,
$\calU^1(\gsig{\kappa}^{\text{poly}})$ and $\calU(\gsig{\kappa}^{\text{poly}})$ exactly in the same way. 

The Lie algebra $\gsig{\kappa}^{\text{poly}}$ is free over $A$, hence  
$\calU^0(\gsig{\kappa}^{\text{poly}})$ and $\calU^1(\gsig{\kappa}^{\text{poly}})$ 
have a PBW basis and in particular are free over $A$. Moreover notice that $K^{\text{poly}}$ has a basis compatible with the \noz $R^{\text{poly}}(n)$ and this implies that $\calU^1(\gsig{\kappa}^{\mathrm{poly}})$ has a basis compatible with the submodules 
$\calU^1(\gsig{\kappa}^{\text{poly}}) \cdot \gsig{\kappa}^{\text{poly},\geq n}$. In particular 
we deduce that $\calU(\gsig{\kappa}^{\text{poly}})$ is a \QCCF topological algebra. 

We finally notice that $\calU(\gsig{\kappa}^{\text{poly}})$ and $\calU(\gsig{\kappa})$ are isomorphic. Indeed we have natural maps $K^{\text{poly}}\lra K$ which induces a map of Lie algebras from $\gsig{\kappa}^{\text{poly}}$ to $\gsig{\kappa}$ and from 
$\calU(\gsig{\kappa}^{\text{poly}})$ to $\calU(\gsig{\kappa})$. We now construct a morphism from $\calU(\gsig{\kappa})$ to $\calU(\gsig{\kappa}^{\text{poly}})$. Indeed 
notice that $\gsig{\kappa}$ is the completion of $\gsig{\kappa}^{\text{poly}}$ with respect to the submodules 
$\gsig{\kappa}^{\text{poly},\geq n}$ and that the map
from $\gsig{\kappa}^{\text{poly}}$ to $\calU(\gsig{\kappa}^{\text{poly}})$ is continuous. Hence  induces a map of Lie algebras from $\gsig{\kappa}$ to $\calU(\gsig{\kappa}^{\text{poly}})$ and then a map from $\calU(\gsig{\kappa})$
to $\calU(\gsig{\kappa}^{\text{poly}})$. 

This in particular proves that $\calU(\gsig{\kappa})$ is a \QCCF module and that has a PBW topological basis.

\subsubsection{Compatibility by pullbacks and factorization properties}\label{ssec:pullbackgU}
We now study the compatibility of our construction under pullbacks. 
Let $f:S'\lra S$ be a morphism of quasi separated compact schemes and denote by $p:X'\lra S'$ the pullback of $p$. If $\Sigma $ is a set of sections denote by $\Sigma'$ the pullback of these sections as in Section \ref{ssec:pullbackcampi}. 
The next Lemma shows that our construction behaves well under pullback. 

\begin{proposition}\label{prop:pullbackgu}
The following hold
	\begin{enumerate}[\indent a)]
            \item There exists a natural isomorphism of \QCCF Lie algebras
            \(
                \hat{f^*}\hat{\gog}_{\Sigma,\kappa} \simeq \hat{\gog}_{\Sigma',\kappa};
            \)
        \item There exists a natural isomorphism of \QCCF $\otimesr$-algebras
        \(
            \hat{f}^*\ugsig{\kappa} \simeq \calU_\kappa\left(\hat{\gog}_{\Sigma'}\right);
        \)
        \item There exists a natural isomorphism of $\pbarDpolim \bullet$ right modules between $\hat f^*\mF^\bullet _{\Sigma, \calU_\kappa(\hgog_\Sigma)}$
        and $\mF^\bullet _{\Sigma', \calU_\kappa(\hgog_{\Sigma'})}$. Moreover $f^*\mu^{\bullet, \bullet}=(\mu')^{\bullet,\bullet}$.
     \end{enumerate}
\end{proposition}
\begin{proof}
    Point $a)$ follows by the base change properties of $\pbarOpoli$ and $\Res$. From point a) it follows that we have a map from $\calU_\kappa\left(\hat{\gog}_{\Sigma'}\right)$  
    to $\hat{f}^*\ugsig{\kappa}$. To show that it is an isomorphism we can argue locally using the basis described in Section \ref{sssec:PBWg}. Point c) follows from Lemma \ref{lemma:pullbackoffields} and point b).
\end{proof}

In the case when $f = i_{J/I}$ or $f = j _{J/I}$ (c.f. notation \ref{ntz:factorization}) the isomorphisms above induce the following factorization structure. 

\begin{proposition}\label{prop:factpropertiesgogcalugog}
Let $\pi: J \twoheadrightarrow I$ a surjective map of finite sets:
	\begin{enumerate}[\indent a)]
		\item There are natural isomorphisms of \QCCF, $\calO_S$-linear Lie algebras
			\begin{align*}
				\Ran^{\gog}_{J/I} &: \hat{i}_{J/I}^*\hat{\gog}_{\Sigma_J,\kappa} \to \hat{\gog}_{\Sigma_I,\kappa}, \\ \fact^{\gog}_{J/I} &: \hat{j}_{J/I}^*\left( \bigoplus_{i\in I} \hat{\gog}_{\Sigma_{J_i},\kappa} \right) \bigg/ \left( \sum_{i \in I} \mathbf{1}_i = \mathbf{1} \right) \to \hat{j}_{J/I}^*\hat{\gog}_{\Sigma_J,\kappa},
			\end{align*}
			which induce a factorization structure on $\hat{\gog}_{\Sigma,\kappa}$;
		\item There are natural isomorphisms of \QCCF, $\calO_S$-linear $\otimesr$-algebras
			\begin{align*}
				\Ran^{\ugsig{\kappa}}_{J/I} &: \hat{i}_{J/I}^*\calU_\kappa(\hat{\gog}_{\Sigma_{J}}) \to \calU_\kappa(\hat{\gog}_{\Sigma_{I}}), \\ \fact^{\calU(\hat{\gog})}_{J/I} &: \hat{j}_{J/I}^*\left( \bigotimes^!_{i\in I} \calU_\kappa(\hat{\gog}_{\Sigma_{J_i}}) \right) \to \hat{j}_{J/I}^*\calU_k(\hat{\gog}_{\Sigma_{J}})
			\end{align*}
			which make $\ugsig{\kappa}$ into a complete topological factorization algebra.
             \item There is a natural isomorphism
                \[
				\Ran^{\mF^1}_{J/I} : \hat{i}_{J/I}^*\mF^1_{\Sigma_J,\gog}\to \mF^1_{\Sigma_I,\gog}
			\]
            and a natural injection 
			\[
				\fact^{\mF^1}_{J/I} : \hat{j}_{J/I}^*\left( \prod_{i\in I} \mF^1_{\Sigma_{J_i},\gog} \right) \to  \hat{j}_{J/I} \mF^1_{\Sigma_J,\gog}
			\]
			which are compatible with each other in a factorization structure fashion, so that $\mF^1_{\Sigma,\gog}$ is a complete topological pseudo factorization algebra (with respect to $\prod$).
 \end{enumerate}
\end{proposition}

\begin{proof}
In the case of the map $\Ran$ all claims follow from the property of the pullback, proved above. 

In the case of the map $\fact$ we can still apply the results proved about pullback, and then we have to analyze the case of families of Sections $\Sigma_{J_i}$ with $i\in I$ which do not intersect. The claim of point a) then follows from the fact $\pbarOpoli$ splits as a product in this case (c.f. \ref{ssec:descrizionelocale1}). Moreover the resulting morphism is a Lie algebra morphism by the factorization properties of $\Res$ (c.f. \ref{sec:residuo}). We construct the map of point b) using that of point a) and check it is an isomorphism using the PBW basis. 
Finally point c) follows from point b) and by \ref{ssec:fattorizzazionecampigenerale}.
\end{proof}


\subsection{Definition of \texorpdfstring{$\VkSg$}{V Sigma}}

We now define the chiral algebra associated to $\gog$ and $\kappa$. 
When considering spaces of fields related to $\gog$, to shorten notation, will write $\mF^{\bullet}_{\Sigma,\gog}$ in place of $\mF^\bullet_{\Sigma,\ugsig{\kappa}}$, leaving the level implicit since it won't play any significant role in this paper. 

Given an element $X \in \gog$ we may attach to it a canonical field
\[ X \in \mF^1_{\Sigma,\gog} \qquad f \mapsto X\otimes f \in \ugsig{\kappa}. \]
With a slight abuse of notation we will denote them both by $X$, specifying when necessary if we consider it as an element of $\gog$ or of $\mF^1_{\Sigma,\gog}$.
These fields are mutually local, indeed for all $X,Y\in \gog$ we have 
$$
\mu(X\otimes Y)(f\otimes g)=[X,Y](fg) +\kappa(X,Y)\Res_{\Sigma}(gdf)\mathbf{1};
$$
it is easily seen that the map $f\otimes g \mapsto \Res_{\Sigma}(gdf)$ vanishes on $\calJ^2_\Delta$. Hence, for all $X,Y \in \gog$ we have $\mu(X\otimes Y)\in \Delta_!^{\leq 2} \mF^1_{\Sigma,\gog}$. 

Let $\calG_{\gog}$ be the $\calO_S$ span of the fields $X$ for $X\in \gog$. Recall that in Section \ref{ssez:generatechiral}
we have associated to such a space four possible chiral algebras: $$\vbasic(\calG_{\gog}),\vbasicpiu(\calG_{\gog}),\vcom(\calG_{\gog}),\vcompiu(\calG_{\gog}).$$
\begin{definition}
	We denote by $\VkSg$ the\index{$\VkSg$} chiral algebra $\vbasicpiu(\calG_{\gog})$.
\end{definition}

The other three chiral algebras will be denoted by $\vbasic^{\gog}$, $\vbasicpiu^{\gog}$ and $\vcom^{\gog}$. In what follows we will investigate their differences, but our main focus will be describing $\VkSg$, this is why we keep a different notation for the latter. We denote the corresponding filtration by $\vbasic^{\gog}(n)$ and $\vcom^{\gog}(n)$. We also denote by $\mvbasic^{\gog}$, $\mvbasic^{\gog}(n)$, etc. the sections over  $S$ of the corresponding sheaves.

We want to give an explicit local description of these chiral algebras. We will prove also that the sheaves $\vbasic^{\gog}(n)$ are \QCCF and $w$-compact sheaves. In particular this implies that $\vbasic^{\gog}(n)=\vcom^{\gog}(n)$. 

\subsection{Vertex algebras vs chiral algebras: \texorpdfstring{$(\pbarO(S),t)$}{(O,t)}-vertex algebras}\label{ssez:Otalgebre}
To describe the local structure of the chiral algebra  $\VkSg$ we introduce the notion of vertex algebra over $(\pbarO(S),t)$ and we explain its relation with the notion of chiral algebra. Let $S=\Spec A$ and let $t$ be a local coordinate. 

\begin{definition}
A vertex algebra over $(\pbarO(S),t)$\index{vertex algebra over $(\pbarO(S),t)$} is the datum of a right $\calD_{\overline{\Sigma}}(S)$ \cgmod module $\mV$ together with a structure of an $A$-linear vertex algebra $(\mV,T,|0\rangle,Y : \mV \to \End_{A}(\mV)[[z^{\pm 1}]])$ for which the following hold
	\begin{enumerate}
		\item $T = - \partial_t$
		\item For $f \in \pbarO(S)$ and $X,Z \in \mV$ we have 
		\begin{align}
			(X)_{(n)}(Zf) &\label{eq:dmodulevoa1}= (X_{(n)}Z)f, \\
			(Xf)_{(n)}(Z) &\label{eq:dmodulevoa2}= \sum_{k \geq 0} \frac{1}{k!}X_{(n+k)}Z \cdot (\partial^k_tf).
		\end{align} 
	\end{enumerate}
 where as usual $Y(X,z)=\sum_n X_{(n)}z^{-n-1}$. In addition, we require that all operators $X_{(n)}$ are continuous. We say that $\mV$ is a filtered vertex algebra over $(\pbarO(S),t)$ if $\mV$ is equipped with a filtration $\mV(n)$, such that  for all $n$ there exists $k$ such that 
$X_{(i)}Z\in \mV(k)$ for all $X,Z\in \mV(n)$ and for all $i\geq -n$. 
In this case we further assume that for all $n,m$, the products $X\otimes Z\mapsto X_{(m)}Z$ restricted to $\mV(n)\otimes \mV(n)$ are continuous with respect to the $\otimes^*$-topology. 
\end{definition}

\subsubsection{From vertex algebras to \texorpdfstring{$(\pbarO(S),t)$}{(O,t)}-vertex algebras}\label{rmk:extvoaoslin} Let $V$ be an ordinary vertex algebra over $\mC$. Then $\mV=V \tensor\la _{\mC} \pbarO(S)$ has a natural structure of a vertex algebra over $(\pbarO(S),t)$, defining $(w \otimes f) \partial_t = -Tw \otimes f - w \otimes \partial_tf$ and  
	\[
	(X\otimes f)_{(n)}(Y\otimes g) = \sum_{k\geq 0} \frac{1}{k!}X_{(n+k)}Y \otimes g\partial_t^kf
	\]
Notice that all operators 
$v_{(n)}$ are continuos and that, without considering the topology, $\mV=V\otimes \pbarO(S)$ as a module. Hence, this is indeed a vertex algebra over $(\pbarO(S),t)$.

This establishes a functor from the category of $\mC$ vertex algebras to the category of vertex algebras over $(\pbarO(S),t)$ which may be checked to be left adjoint to the forgetful functor. 

\medskip

If $V$ is equipped with a filtration $V(n)$ such that for all $n$ there exists $k$ such that $X_{(i)}Z\in V(k)$ for all $X,Z\in V(n)$ and for all $i\geq -n$, we define $\mV(n)=V(n)\tensor \la \pbarO(S)$. If we assume in addition that all products $X\otimes Z\mapsto X_{(n)}Z$ restricted to $\mV(m)\otimes \mV(m)$ are continuous for the $\otimes^*$-topology, then 
$\mV$ is naturally a filtered vertex algebra over $(\pbarO(S),t)$.

Notice that, if the subspaces $V(n)$ is finite dimensional then the condition about the continuity of the product $X_{(n)}Z$ is automatically satisfied. 

\subsubsection{From chiral algebras to \texorpdfstring{$(\pbarO(S),t)$}{(O,t)}-vertex algebras}\label{sssez:chiraliOt}
Here we discuss how filtered vertex algebras over $(\pbarO(S),t)$ are a local description of chiral algebras. In particular we will see that, given a chiral algebra $\calV$, the module $\calV(S)$ (in the case $S$ affine and well covered) is a vertex algebra over $(\pbarO(S),t)$. Making some further assumptions on the chiral algebra we will see that this procedure establishes an equivalence between chiral algebras and vertex algebras over $(\pbarO(S),,t)$.

\smallskip

We start with a given chiral algebra $\calV$, with a filtration $\calV(n)$ and consider the topological modules $\mV=\calV(S),\mV(n)=\calV(n)(S)$. Let $|0\rangle \in \mV$ denote the image of $dt \in \Omega^1_{\overline{\Sigma}}(S)$ via the unit of $\calV$ and define $T=-\partial_t$. 

We may define $(n)$-products on $\mV$ by writing the chiral product on global sections as 
	\[
	\mu : \left(\calV (S) \otimesst_A \calV(S)\right)\left[\frac{1}{t\otimes 1 - 1 \otimes t}\right] \to \bigoplus_{k \geq 0} \calV(S)\partial^k_{t\otimes 1} = \Delta_!\calV(S).
	\]
For the action $\pbarDm{2}$ on the target module we refer to Section \ref{sez:Dmoduli} and in particular to \ref{ssez:tuvy}. Then, for every integer $n$ define $X_{(n)}Z$ by
	\begin{equation}\label{eq:nprodchiral}
	X_{(n)}Z \stackrel{\text{def}}{=} dR_r \mu((t\otimes1 - 1\otimes t)^n X\otimes Z),
	\end{equation}
	where $dR_r : \Delta_!\calV(S) \to \calV(S)$ corresponds to the projection along the factor $\Delta ^!\Delta_!\calV(S)=\calV(S)\partial^0_{t\otimes 1}$ (see Definition \ref{def:dR}). Notice that, using the fact that $\mu$ is $\pbarOq(S)$ linear, multiplying on both sides by positive powers of $(t\otimes 1 - 1\otimes t)$, and computing $dR_r$ we have
	\begin{equation}\label{eq:muvsope}
	\mu((t\otimes1 - 1\otimes t)^n X\otimes Z) = \sum_{k \geq 0} \frac{1}{k!}X_{(n+k)}Z\partial_{t\otimes 1}^k
	\end{equation}
	for every integer $n \in \mZ$. So that the products $X_{(n)}Z$ characterize $\mu$. 

Conversely, if we are given a filtered vertex algebra $\mV$ over $(\pbarO(S),t)$ we can consider $\calV(n)$, the QCC sheaf associated to $\mV(n)$ and define $\calV=\limind \calV(n)$. We can construct a morphism $\mu $ by formula \eqref{eq:muvsope} consider the unit map $\mathbf{1} : \Omega^1_{\overline{\Sigma}} \to \calV$ by 
$$
\unoU(fdt)=|0\rangle f.  
$$

With respect to the above discussion, we have the following Lemma. 

\begin{lemma}\label{lem:chiralivertice}
Given a chiral algebra $\calV$, the space $\mV = \calV(S)$ with the filtration, the operator $T$, the vacuum vector, and the $n$ products defined above is a filtered vertex algebra over $(\pbarO(S),t)$. 

Conversely, given $\mV$ a filtered vertex algebra over $(\pbarO(S),t)$ the sheaf $\calV$ constructed above with the filtration, the unit map and the chiral product defined above is a chiral algebra. 
\end{lemma}
	
\begin{proof}[Sketch of the proof] We just underline the main steps of the proof which is just a lengthy but trivial verification of all the properties. 

	\begin{itemize}
		\item A collection of $A$-linear products $X_{(n)}Z$ determines a morphism of right $\calD_{\overline{\Sigma}^2}(S)$-modules $\mu$ via formula \eqref{eq:muvsope} if and only if formulas \eqref{eq:dmodulevoa1} and \eqref{eq:dmodulevoa2} hold. 
		\item The collection of multiplications $X_{(n)}Z$ satisfy the Borcherds identities if and only if we $\mu$ is a Lie structure in BD pseudo-tensor structure $P^{chBD}$.
		\item The morphism determined by $|0\rangle$
		\[
		\Omega^1_{\overline{\Sigma}}(S) =\pbarO(S)dt \to \mV \qquad fdt \mapsto f|0\rangle
		\]
		satisfies the chiral unit condition if and only if $|0\rangle$ is a unit as a vertex algebra.\qedhere
	\end{itemize}
\end{proof}

The Lemma above does not give a perfect equivalence between chiral algebras and 
filtered vertex algebras. This is caused only by topological issues and in particular by our choice to require very mild assumptions on the topology of our chiral algebras. To fix this problem we give the following definitions. 

\begin{definition}\label{def:qccchirale}
We say that a chiral algebra $\calV(n)$ is a filtered \QCC \index{filtered \QCC} if all sheaves $\calV(n)$ are QCC sheaves and if $\calV=\limind \calV(n)$. We will use this terminology for topological sheaves in general.

We say that $\mV$, a filtered vertex algebra over $(\pbarO(S),t)$, is filtered complete if all modules $\mV(n)$ are complete and if $\mV=\limind \mV(n)$.
\end{definition}

Then the following Corollary is a direct consequence of this definition and the Lemma above. 

\begin{corollary}\label{cor:chiraliOt}
The functor $\calV \mapsto \calV(S)$ establishes an equivalence between filtered \QCC chiral algebras and filtered complete vertex algebras over $(\pbarO(S),t)$.
\end{corollary}

\subsubsection{Commutative chiral algebras and \texorpdfstring{$\pbarD$}{D}-commutative algebras}\label{ssec:commutativechiralalgebraDalgebra}

We take the opportunity to use the local description of chiral algebras as vertex algebras over $(\pbarO(S),t)$ to prove that commutative chiral algebras correspond to commutative $\pbarD$ algebras. This is a well known fact of the theory of chiral algebras, we prove it for completeness and to highlight the relationship between our chiral algebras and usual vertex algebras.

We first recall the definition of commutative chiral algebra.

\begin{definition}\label{def:commutativechiralalgebra}
	Let $\calV$ be a chiral algebra over $\pbarO$. We say that $\calV$ is \emph{commutative}\index{commutative chiral algebra} if the restriction of the chiral product
	\[
		\mu : \calV \boxtimes \calV \to \Delta_!\calV \quad \text{vanishes.}
	\]
\end{definition}

Assume that $S$ is affine, well covered and equipped with a coordinate $t$. If we assume that we have a commutative chiral algebra $\calV$ which is also filtered \QCC, under the correspondence of Corollary \ref{cor:chiraliOt}, the module $\mV = \calV(S)$ is a commutative vertex algebra over $(\pbarO(S),t)$, which means that the $n$-products vanish for $n\geq 0$. Let us recall the following well known fact on commutative vertex algebras, which we adapt to our notion of vertex algebra over $(\pbarO(S),t)$.

\begin{proposition}\label{prop:localcommutativechiralalgcorresp}
	Let $\mV$ be a commutative vertex algebra over $(\pbarO(S),t)$. Then the $(-1)$ product is commutative and associative, $T = - \partial_t$ is a derivation of it. So that $(V,\cdot_{(-1)}\cdot,T)$ is a commutative, unital algebra with a derivation. 
	
	In addition, the $(-1)$ product is $\pbarO(S)$ linear; as we have a coordinate $t$ we may consider $\mV$ as a \emph{left} $\pbarD(S)$ module as well, where now $\partial_t$ acts as $T$. This construction establishes a correspondence between commutative vertex algebras over $(\pbarO(S),t)$ and commutative $\pbarO(S)$ modules which are also left $\pbarD(S)$ modules for which $\partial_t$ acts as a derivation. 
\end{proposition}

\begin{proof}
	The first assertion follows by ordinary vertex algebra theory. The fact that the $(-1)$ product is $\pbarO(S)$ linear follows by formulas \eqref{eq:dmodulevoa1}, \eqref{eq:dmodulevoa2}. The fact that $\mV$ has a natural structure of a \emph{left} $\pbarD(S)$ module follows by the fact that, in the case where we have a coordinate $t$, left and right $\pbarD(S)$ modules are identified by changing sign to the action of $\partial_t$. 
	The fact that this establishes a correspondence follows by the analogous statement for ordinary vertex algebras.
\end{proof}

Given sections $a,b \in \mV$ we write $a \cdot_t b$ for the commutative product coming from the structure of vertex algebra over $(\pbarO(S),t)$.  In the next Lemma we show how this product behaves under change of coordinates.

\begin{lemma}\label{lem:changecoordcommutativeproduct}
	Let $\calV$ be a filtered \QCC commutative chiral algebra over $\pbarO$. Assume that $S$ is affine, well covered with coordinates $t,s$. Then
	\[
		a \cdot_s b = (\partial_s t)(a \cdot_t b).
	\]
\end{lemma}

\begin{proof}
	Recall that by definition we have
	\[
		a \cdot_s b = dR_r \mu\left( (s\otimes 1- 1 \otimes s)^{-1}a\otimes b \right).
	\]
	If we consider $s$ as a function of $t$, the Taylor expansion (see Remark \ref{ssez:Taylor}) of  $(s \otimes 1 - 1 \otimes s)$ is given by
	$$(s \otimes 1 - 1 \otimes s) = (t\otimes 1 - 1\otimes t)(1\otimes \partial_t s) + (t\otimes 1 - 1\otimes t)^2(1\otimes \frac{1}{2}\partial_t^2s) + ...,$$ using the fact that positive $n$ product vanish we get
	\begin{align*}
		a \cdot_s b &= dR_r \mu\left( (s\otimes 1- 1 \otimes s)^{-1}a\otimes b \right)\\ &= dR_r \mu\left( \left((t\otimes 1- 1 \otimes t)(1\otimes\partial_t s)\right)^{-1}a\otimes b \right) = (\partial_t s)^{-1} a\cdot_t b = (\partial_s t)(a \cdot_t b).
	\end{align*}
\end{proof}

\begin{corollary}\label{cor:canonicalporductcommutativechiral}
	Let $\calV$ be a filtered \QCC commutative chiral algebra over $\pbarO$ and assume that $S$ is affine and well covered with a coordinate $t$. Consider $\calV^\ell = \calV \otimes_{\pbarO} \Tan_{\overline{\Sigma}}$. Then the module of sections $\mV^\ell = \calV^\ell(S)$ has a canonical structure of commutative algebra. Given a coordinate $t$ on $S$ and $a,b \in \mV = \calV(S)$ this commutative product reads as $$(a\otimes\partial_t) \cdot (b\otimes\partial_t) = (a\cdot_t b)\otimes\partial_t.$$
\end{corollary}

\begin{proof}
	We check that the product in the statement of the Corollary is independent from the choice of $t$. Let $a,b \in \mV^\ell$ be two elements, $t,s$ two coordinates on $S$. We can write $a = a_t\otimes\partial_t = a_s\otimes\partial_s$ and $b = b_t\otimes\partial_t = b_s\otimes\partial_s$, so that $a_t = (\partial_s t)a_s$ and $b_t = (\partial_s t)b_s$. It follows that 
	\begin{align*}
		(a_t \cdot_t b_t ) \otimes \partial_t = (\partial_s t)^2 (a_s \cdot_t b_s) \otimes \partial_t = (\partial_s t) (a_s \cdot_t b_s) \otimes \partial_s = (a_s \cdot_s b_s) \otimes \partial_s;
	\end{align*}
	where the second equality follows from $(\partial_s t) \partial_t= \partial_s$, while the last equality is exactly Lemma \ref{lem:changecoordcommutativeproduct}.
\end{proof}

\begin{proposition}\label{prop:twisttanchiralcommutative}
	Let $\calV$ be a filtered \QCC commutative chiral algebra. Then the sheaf $\calV^\ell = \calV \otimes_{\pbarO} \Tan_{\overline{\Sigma}}$ has a natural structure of a commutative algebra. As $\calV$ is a right $\pbarD$ module, the sheaf $\calV^\ell$ is naturally a left $\pbarD$ module, for which $\Tan_{\overline{\Sigma}}$ acts as derivations of the commutative algebra structure. We call such an object a $\pbarD$ commutative algebra.
	
	The assignment $\calV \mapsto \calV^\ell$ establishes an equivalence between filtered \QCC commutative chiral algebras and filtered \QCC sheaves of $\pbarD$ commutative algebras.   
\end{proposition}

\begin{proof}
	The fact that $\calV^\ell$ has natural commutative algebra structure follows by Corollary \ref{cor:canonicalporductcommutativechiral}: locally, after the choice of a coordinate $t$, the module of section has a canonical structure of commutative algebra for which $\partial_t$ acts as a derivation and these structures, being canonical behave well under gluing. 
	
	The second assertion can be proved locally since the objects in question glue and for $S$ affine and well covered it follows by Proposition \ref{prop:localcommutativechiralalgcorresp}.
\end{proof}

\subsection{Construction of the map \texorpdfstring{$\calY_{\Sigma,t}$}{}}\label{ssez:Y}
We still assume that $S=\Spec A$ is well covered and that $t$ is a coordinate. 

Fix a Lie algebra $\gog$ and a level $\kappa$ as in the beginning of this Section. Let $V^{\kappa}(\gog)$ the associated vertex algebra. We can associated to $V^{\kappa}(\gog)$ a vertex algebra over $(\pbarO(S),t)$ 
$$\mVkSgt=V^\kappa(\gog)\otimesl_{\mC} \pbarO(S)$$
as in Section \ref{rmk:extvoaoslin}. The module 
$\mVkSgt$ \index{$\mVkSgt$}is complete and a \ccgmod module. In particular, considering the filtration we are going to define in the next paragraph, we get that $\mVkSgt$ is a filtered complete vertex algebra over $(\pbarO(S),t)$. As such we can associate to this algebra a filtered \QCC chiral algebra $\calVkSgt =\qcc{\mVkSgt}$. \index{$\calVkSgt$}

\subsubsection{The basic filtration}
The vertex algebra $V^{\kappa}(\gog)$ has a canonical set of generators $G_{V}=\{x_{(-1)}|0\rangle$ for $x\in \gog\}$, we denote by $\calG_V  = G_V \otimes \calO_S$ the corresponding subsheaf of $\calVkSgt$. We set
$\Vkg_{\leq 0}=\mC|0\rangle+G_V$ and $$
\Vkg_{\leq n+1}=T (\Vkg_{\leq n}) + \Vkg_{\leq n}+ \left\langle X_{(-1)}Z :X,Y\in \Vkg_{\leq n} \right\rangle 
$$
We notice that $\mVkSgt(n)=\Vkg_{\leq n}\otimesr \pbarO(S)$ is a free module of finite rank over $\pbarO(S)$ and in particular it is $w$-compact and complete. We denote by $\calVkSgt(n)$ the corresponding subsheaf of $\calVkSgt$\index{$\calVkSgt$}\index{$\calVkSgt$:$\calVkSgt(n)$}. The definition of the filtration $\Vkg_{\leq n}$ matches the one given in Section \ref{ssez:generatechiral}, hence the 
filtration $\calVkSgt(n)$ can be constructed from the generators $\calG$ exactly as in that Section. Following the recipes there explained we could have obtained four different chiral algebras attached to this filtration of $\calVkSgt$, however, in this case, the subsheaves $\calVkSgt(n)$ are already complete and moreover we have that 
$\mVkSgt=\limind \mVkSgt(n)$ and $\calVkSgt=\limind \calVkSgt(n)$ so the four possible chiral subalgebras agree:
$$
\calVkSgt=(\calVkSgt)_{\text{basic}}=(\calVkSgt)_{\text{basic}+}=(\calVkSgt)_{\text{com}}=(\calVkSgt)_{\text{com}+},
$$
so that there is no other interpretation for $\calVkSgt$ and $\mVkSgt$.

\subsubsection{The map \texorpdfstring{$\calY_{\Sigma,t}$}{Y}}
We now start to give a local description of the chiral algebra $\VkSg$.

\begin{proposition}\label{prp:mappaY}
	There exists a unique morphism of vertex algebras over $(\pbarO(S),t)$
	\[
	\calY_{\Sigma,t} : \mVkSgt \to \mV_{\text{basic}}^\gog=\vbasic^\gog(S)
	\]
	such\index{$\calY_{\Sigma,t}$} that $|0\rangle$ maps to $\unoU$ and for any element $X \in \gog$ the field $\calY_{\Sigma,t}(X_{(-1)}|0\rangle)$ is the canonical field associated to $X$:
	\[ 
	\pbarOpoli(S) \ni f \mapsto X \otimes f \in \ugsig{\kappa}(S)    
	\]
    Moreover this morphism is continuous and $\calY_{\Sigma,t} (\mVkSgt(n))=\mvbasic^\gog(n)$. In particular the image of $\calY_{\Sigma,t}$ is equal to $\mvbasic^\gog$ and the morphism $\calY_{\Sigma,t}$ factors through a continuous morphism to $\mvbasicpiu^\gog = \varinjlim \mvbasic^\gog(n)$. 
\end{proposition}

\begin{proof}
	By Remark \ref{rmk:extvoaoslin} it is enough to show that there is a unique morphism of $\mC$-vertex algebras 
	\[
	\calY_{\Sigma,t} : \Vkg   \to \mvbasic^{\gog}
	\]
	such that the condition of the statement holds. By the universal property of $\Vkg$ to construct such a map it is enough to give a map $\gog \to \mvbasic^\gog$ such that the associated endomorphism $X_{(n)}$ of $\mvbasic^\gog$, attached to the vertex algebra structure of $\mvbasic^\gog$, satisfy the usual commutation relations of the affine algebra $\hat{\gog}_\kappa$
	(this parallels the discussion in Section 3 of \cite{cas2023}). Recall that the combination of Lemma \ref{lem:campilocali} together with formula \eqref{eq:nprodchiral} implies that the $(n)$-products for this structure, viewing elements of $\mvbasic^\gog$ as fields, are given by the following formulas:
\begin{equation}\label{eq:vertalgprod}
X_{(n)}Y (f) = \mu\left((t\otimes 1 - 1\otimes t)^nX\boxtimes Y\right)(1\otimes f),
\end{equation}
	while the vacuum vector corresponds to the field $\vac = \Res_{\Sigma} (\_ dt)$.
	To show that the same commutation relations of $\hat{\gog}_{\kappa}$ hold we use the fact that $\mV^{\gog}_{\text{basic}}$ is a vertex algebra. Let $X,Y \in \gog$ and let $\hat{X},\hat{Y}$ be the attached fields. We have
	\[
		[\hat{X}_{(n)},\hat{Y}_{(m)}] = \sum_{k \geq 0} \binom{n}{k} (\hat{X}_{(k)}\hat{Y})_{(n+m-k)} 
	\]
	so that the claim reduces to showing $\hat{X}_{(0)}\hat{Y} = \widehat{[X,Y]}$, $\hat{X}_{(1)}\hat{Y} = \kappa(X,Y)\vac$ and that the other positive $n$ product vanish. The third assertion is clear since $\mu(X \boxtimes Y)$ vanishes on $\calJ_{\Delta}^2$. Viewing these as fields the first and second claims follow by the very definition of $\hat{\gog}_{\Sigma,\kappa}$:
\begin{align*}
	\left( \hat{X}_{(0)} \hat{Y} \right) (f) &= \mu(\hat{X}\boxtimes \hat{Y})(1 \otimes f) = [X\otimes 1, Y \otimes f] = [X,Y] \otimes f = \widehat{[X,Y]}(f); \\
	\left( \hat{X}_{(1)} \hat{Y} \right) (f) &= \mu \big((t\otimes 1 - 1 \otimes t) \hat{X} \boxtimes \hat{Y}\big) (1 \otimes f) \\
	&= \mu(\hat{X}\boxtimes \hat{Y})(t\otimes f) - \mu(\hat{X}\boxtimes \hat{Y})(1 \otimes tf) = \kappa(X,Y) \Res_{\Sigma}(fdt) = \kappa(X,Y)\vac (f).
\end{align*}
Since the space of fields is a topological $\pbarO(S)$-module and we are considering $\Vkg$ with the discrete topology, it follows that the map $\calY_{\Sigma,t}$ is continuous. Moreover, by construction $\calY_{\Sigma,t}(G_V\otimes \calO_S)=\calG_{\gog}$, the generating space of mutually local fields of $\vbasic^\gog$. By formulas
\eqref{eq:nprodchiral} the construction of the two basic filtrations is intertwined by $\calY_{\Sigma,t}$, so we have $\calY_{\Sigma,t}(\mVkSgt(n))=\mvbasic^\gog(n)$. 

Finally, we notice that $\mVkSgt=\limind \mVkSgt(n)$ and $\mvbasic^\gog(n)$ is a submodule of 
$\mvbasic$, hence the morphism factors trough $\calY_{\Sigma,t} : \mVkSgt \to \mvbasicpiu^\gog \to \mvbasic^\gog$.
\end{proof}

Let us emphasize how the above proof leads to an inductive construction of map $\calY_{\Sigma,t}$.

\begin{remark}\label{rmk:descrycan}
	It follows by the fact that $\calY_{\Sigma,t}$ is a morphism of vertex algebras over $(\pbarO(S),t)$, that it may be defined inductively as follows. Let $X \in \gog, A,B \in \Vkg$, then 
	\begin{align*}
	\big(	\calY_{\Sigma,t}(X_{(-1)}\vac)\big)(f) &= X\otimes f, \\
	\big(	\calY_{\Sigma,t}(A_{(n)}B)\big)(f) &= \mu\bigg( (t\otimes 1 - 1 \otimes t)^n \calY_{\Sigma,t}(A)\boxtimes \calY_{\Sigma,t}(B)\bigg)(1 \otimes f).
	\end{align*}
\end{remark}

\subsection{Factorization properties of the map \texorpdfstring{$\calY_{\Sigma,t}$}{Y}} 

We state the following theorem, which we take from Proposition 7.2.1  and Corollary 7.2.3 \cite{cas2023}, regarding the factorization properties of that the map $\calY_{\Sigma,t}$. This statement is not trivial and crucial to our purposes. Recall notation \ref{ntz:factorization}, Section \ref{ssec:fattorizzazionecampigenerale} and fix a surjection of finite sets $J \twoheadrightarrow I$.

\begin{theorem}\label{thm:factorizationofcaly}
    The following diagrams commute:
    \[\begin{tikzcd}
	{\hat{i}^*_{J/I}\left(V^\kappa(\gog)\otimes\calO_{\overline{\Sigma}_J}\right)} && {\hat{i}^*_{J/I}\mF^1_{\Sigma_J,\gog}} \\
	\\
	{V^\kappa(\gog)\otimes\calO_{\overline{\Sigma}_I}} && {\mF^1_{\Sigma_I,\gog}}
	\arrow["{\hat{i}^*_{J/I}\calY_{\Sigma_J,t}}", from=1-1, to=1-3]
	\arrow["{\Ran^{\pbarO}_{J/I}}"', from=1-1, to=3-1]
	\arrow["{\Ran^{\mF^1}_{J/I}}", from=1-3, to=3-3]
	\arrow["{\calY_{\Sigma_I,t_I}}"', from=3-1, to=3-3]
\end{tikzcd}\]

    \[\begin{tikzcd}
	{\hat{j}^*_{J/I}\left( \prod_{i \in I} V^\kappa(\gog)\otimes\calO_{\overline{\Sigma}_{J_i}}\right)} &&& {\hat{j}^*_{J/I}\left( \prod_{i \in I} \mF^1_{\Sigma_{J_i},\gog}\right)}
	\\
	\\
	{\hat{j}^*_{J/I}\left(V^\kappa(\gog)\otimes\calO_{\overline{\Sigma}_J}\right)} &&& {\hat{j}^*_{J/I}\mF^1_{\Sigma_J,\gog}} 
	\arrow["{\hat{j}^*_{J/I}\calY_{\Sigma_J,t}}", from=3-1, to=3-4]
	\arrow["{\fact^{\pbarO}_{J/I}}"', from=1-1, to=3-1]
	\arrow["{\hat{j}^*_{J/I}\prod_{i \in I}\calY_{\Sigma_{J_i},t_i}}", from=1-1, to=1-4]
	\arrow["{\fact^{\mF^1}_{J/I}}"', hook, from=1-4, to=3-4]
\end{tikzcd}\]

\end{theorem}

\begin{proof}
	By linearity, it is enough to check commutativity when restricting to $V^\kappa(\gog)$. In this way the statement of the Lemma becomes completely analogous to Proposition 7.2.1 and Corollary 7.2.3 of \cite{cas2023} whose proof works in this case as well.
\end{proof}

\subsection{Local description of \texorpdfstring{$\VkSg$}{}}\label{ssec:localdescriptionofvksg}
Now  we prove that the map $\calY_{\Sigma,t}$ determines an isomorphism between $\mVkSgt$ and $\VkSg(S)$.

\begin{theorem}\label{teo:descrizionelocaleVkSg}
Assume $A$ is an integral domain, then the map $\calY_{\Sigma,t}$ determines an isomorphism of topological modules between  $\mVkSgt(n)$ and $\mvbasic^\gog(n)$. 

In particular:
\begin{itemize}
    \item $\mvbasic^\gog(n)=\mvcom^\gog(n)$ is complete and \QCCF;
    \item $\calY_{\Sigma,t}$ determines an isomorphism of  filtered complete vertex algebras over $(\pbarO,t)$ between $\mVkSgt$ and $\mvbasicpiu^\gog = \VkSg(S)$ and between the chiral algebras $\calVkSgt$ and $\VkSg = \vbasicpiu^\gog=\vcompiu^\gog$
    \item the chiral algebra $\VkSg$ is \QCCF and $w$-compact. 
\end{itemize}
\end{theorem}

We notice that $\mvbasic^\gog=\mvcom^\gog$ is not isomorphic to $\mVkSgt$. Indeed since $\mvbasic^\gog$ is a topological subsheaf of the space of fields it has a countable \fsonoz, while $\mVkSgt$ has not. 

The proof of the Theorem will be based on the following Lemma. Recall the local description of $\pbarOpoli$ given in Section \ref{ssec:descrizionelocale1} and Remark \ref{oss:unsolof}. 
In particular $t$ is a local coordinate $S=\Spec A$, and for $i=1,\dots, n$ the elements $a_i\in A$ describe the sections $\grs_i$. 

\begin{lemma}\label{lem:sottospazidiF^1}
    Let $E$ be a finite dimensional $\mC$-linear subspace of $\mF^1_{\Sigma,\gog}$. Assume that 
    $E(\pbarO)\subset \calU_+ = \calUgS (\gog \otimesr \pbarO)$ and that the map
\begin{equation}\label{eq:ipotesiE}
	E \otimes_\mC \calO_S \lra \frac{\calUgS}{\calU_+} \quad\text{ defined by } \quad 
    v\otimes b \mapsto \left[v\left( \frac{b}{t-a_i}\right)\right]
    \end{equation}
    is injective for all $i$. Consider the $\pbarO$ submodule $\calE$ generated by $E$ with the topology induced by $\mF^1$. Then $\calE$ is isomorphic to $E\otimes_\mC \pbarO$ as a topological module. 
\end{lemma}
\begin{proof}

Let $\calE^*$ be the $\pbarOpoli$ submodule of $\mF^1_{\Sigma,\gog}$ generated by $E$. We prove a slightly stronger statement by proving that the map $\psi:E\otimes\pbarOpoli\lra \calE'\subset \mF^1$ given by $\psi(v\otimes g)=v\cdot g$ is an isomorphism of topological modules. We already know that it is continuous. Consider the \fsonoz for $\mF^1_{\Sigma,\gog}$ given by: 
	\[\mH_{k,0} = \{ X \in \mF^{1}_{\Sigma,\gog} \text { such that } X(\pbarO(k)) \subset \calU_+ \}.\]
 We show that for all $k \geq 0$ we have 
\begin{equation}\label{eq:psiH}
		\psi^{-1}(\mH_{k,0}) \subset E \otimes \pbarO(-k).
\end{equation}
This implies that the map $\psi$ is injective and that $\psi$ is an isomorphism of topological modules between $E\otimes \pbarOpoli$ and its image with the induced topology. 

By multiplying by an appropriate power of $\grf=\prod_{i=1}^n(t-a_i)$ it is enough to prove formula \eqref{eq:psiH} in the case $k=0$. Let $e_j$ be a basis of $E$ and assume that  $\psi(\sum_j  e_j\otimes \gra_j ) \in \mH_{0,0} $. We want to prove that $\gra_j\in \pbarO$ for all $j$. Consider the following refinement of the filtration $\pbarO(m)$ of $\pbarOpoli$. Recall the definition of the topological basis $\grf_{m,i}=\grf^m(t-a_1)(t-a_2)\dots(t-a_i)$ from Section \ref{sssec:basespbarO}, and, for $m\in \mZ$ and $i=0,\dots,n-1$, define 
$\calR(nm+i)=\pbarO\grf_{m,i}$ so that $\calR(k)$ is a \emph{decreasing} filtration of $\pbarOpoli$ such that $\calR(nm)=\pbarO(-m)$ and $\calR(k-1)/\calR(k)\simeq \calO_S$. 
Consider $m_0$ and $i_0$ such that for all $j$ we have that $\gra_j\in \calR(nm_0+i_0)$ and that there exists $j_0$
such that $\gra_{j_0}\notin \calR(nm_0+i_0+1)$. By contradiction assume that $nm_0+i_0<0$. Define 
$$
\grb =\frac{1}{\grf_{m_0,i_0+1}}\; \mif i_0\neq n-1 
\; \mand \;
\grb =\frac{1}{\grf_{m_0+1,0}}\; \mif i_0= n-1 .
$$
Notice that under the assumption $nm_0+i_0<0$ we have $\grb\in\pbarO$, hence $\psi(\sum_j e_j \otimes \gra_j \grb)\in \mH_{0,0}$. Moreover
$$
\gra_j\cdot \grb\in \frac{1}{t-a_{i_0+1}}\pbarO \mforall j \;\mand\;\gra_{j_0}\notin \pbarO.
$$
Write 
$$
\gra_j\cdot \grb = \frac{b_{j}}{t-a_{i_0+1}}+\grg_j 
$$
with $b_j\in \calO_S$ and $\grg_j\in\pbarO$. Our assumption that $\alpha_{j_0} \notin \pbarO$ implies that $b_{j_0}\neq 0$. Then, by the assumption that $\psi(\sum_j e_j\otimes \alpha_j) \in \mH_{0,0}$, and by the fact that $(e_j\cdot \grg_j)(\pbarO)\subset \calU_+$ (by assumption $E(\pbarO) \subset \calU_+$) for all $j$, it follows that 
$$
\psi\left(\sum e_j \otimes \frac{b_j}{t- a_{i_0+1}}\right)(\pbarO)\subset \calU_+
$$
in particular if we evaluate at $1 \in \pbarO$ we obtain 
$$
\sum_j e_j\left(\frac{b_j}{t- a_{i_0+1}}\right) \in \calU_+
$$
against the injectivity of the map \eqref{eq:ipotesiE}.
\end{proof}

\begin{remark}
Notice that the sheaf $\calU_+ = \calUgS (\gog \otimesr \pbarO)$, as well as the sheaf $\calUgS$, satisfies the hypothesis of Remark \ref{oss:sqccccomologia}. In particular their cohomology vanish on affine subsets, hence 
$\calUgS/\calU_+(S')=\calUgS(S')/\calU_+(S')$ for all open affine subsets $S'$. 
\end{remark}

\begin{proof}[Proof of Theorem \ref{teo:descrizionelocaleVkSg}]
	We notice that the claim about $\mVkSgt(n) \simeq \mvbasic^\gog(n)$ implies $\mvbasic^\gog(n)=\mvcom^\gog(n)$ and that the map $\calY_{\Sigma,t}: \mVkSgt \to \mvbasicpiu^\gog=\mvcompiu^\gog$ is an isomorphism. To deduce the analogous claim about the chiral algebra $\calVkSgt$ notice that the same claim holds for all open affine subset $S'$ of $S$ and in particular  $\calVkSgt(S')\simeq\VkSg(S')$. This implies that $\calY_{\Sigma,t} : \calVkSgt \to \VkSg$ is an isomorphism. 
	
	To prove that $\calY_{\Sigma,t}$ determines an isomorphism between the topological module $\mVkSgt(n)$ and $\mvbasic^\gog(n)$ it is enough to prove that it is injective and an immersion. Since as plain modules $\mVkSgt(n) = V^\kappa(\gog)_{\leq n} \otimes \pbarO(S)$, by Lemma \ref{lem:sottospazidiF^1} it is enough to prove that 
	\begin{enumerate}[\indent {Claim} 1:]
		\item for any element $v \in V^\kappa(\gog)$ the field $\calY_{\Sigma,t}(v)$ satisfies $\calY_{\Sigma,t}(v)(\pbarO) \subset \calU_+$
		\item for all $i=1,\dots,n$ the evaluation map
		\begin{align*}
		\ev_{(t-a_i)^{-1}} : &V^\kappa(\gog)\otimes\calO_S \to \calU_k(\hat{\gog}_{\Sigma})/\calU_+ \quad \text{ defined by } \\  &v\otimes g \mapsto \left[\calY_{\Sigma,t}(v)\left(\frac{g}{t-a_{i}}\right) \right] 
		\end{align*}
		are injective.
	\end{enumerate}	
	
	Claim 1 follows by induction on the PBW degree in $V^\kappa(\gog)$ with the formula for the normally ordered product. We prove Claim $2$ (hence the Theorem) in two special cases first.

\subsubsection*{The case of one section and $A$ a field} We assume that $A=L$ is a field of characteristic zero containing $\mC$ and that we have only one section $\grs$ which is equal to the zero section: $a=\grs^\sharp(t)=0$. In this case claim 2 reduces to a well known fact about vertex algebras over a field of characteristic zero. We give the details for completeness.

Let $\gog_L=\gog\otimes L$, then this is a simple Lie algebra over $L$. 
We denote by $\gog_L$, $\hgog_{\kappa,L}$, $U_{L}(\hgog)$, $U_{+,L}$ and $V_L^\kappa(\gog)$ the finite dimensional Lie algebra, the affine Lie algebra, the completion of the enveloping algebra of $\hgog_L$ at level $\kappa$, the left ideal generated by $\gog_L[[t]]$ in $U_{L}(\hgog)$ and the universal vertex algebra at level $\kappa$ using $L$ as a base field. Up to completion, these objects are the base change to $L$ of the analogous object constructed over $\mC$. In particular in the case of $V_L^\kappa(\gog)$ completion is not necessary and we have
$$
V_L^\kappa(\gog)=\Vkg\otimes _\mC L 
$$
We denote by $Y_L:V_L^\kappa(\gog)\lra \Homcont_L\big( L((t)),\End(V^\kappa_L(\gog))\big)$ 
the usual vertex operator map. 

These objects are related to our $X,S,\Sigma$ geometric setting as follows: $\pbarO(S)=L[[t]]$, $\pbarOpoli(S)=L((t))$ and:
$$
\calUgS(S) =U_{L}(\hgog) 
\qquad 
\calU_{+}(S)= U_{+,L}.
$$

Hence the map $\calY_{\Sigma,t}$ gives a map from $V^\kappa(\gog)\otimes L[[t]]$ to $\Homcont(L((t)),U_L(\hgog))$. So that to every element $v \in V^\kappa(\gog)_L$ and any $f \in L((t))$ there is an attached element $\calY_{\Sigma,t}(v)(f) \in U_L(\hat{\gog})$. We notice that for all $v,u\in  V_L^\kappa(\gog)\subset V^\kappa(\gog)\otimes L[[t]]$  and for all $f\in L((t))$ 
we have
$$
 \Big(\big( \calY_{\Sigma,t}(v) \big)(f)\Big)\cdot u=\Big(\big(Y_L(v)\big)(f)\Big)\cdot u,
$$
where the left hand side stands for the action of $\big(\calY_{\Sigma,t}(v)\big)(f) \in U_L(\hat{\gog})$ on $V^\kappa_L(\gog)$, so that the above claim means that the actions of $\big(\calY_{\Sigma,t}(v))(f)$ and $Y_L(v)(f)$ on $V^\kappa_L(\gog)$ agree.
This follows from the construction of the two operators are given by the same inductive formula \eqref{eq:vertalgprod} (in the case $S = \Spec L$).

\begin{lemma}\label{lem:injectivityvertexalg}
The map 
$$
\ev_{t^{-1}} : \Vkg \otimes_\mC L \lra \frac{U_L(\hgog)}{U_{+,L}} \quad \text{ given by }\quad 
v\otimes b\mapsto \left[\calY_{\Sigma,t}(v) \left(\frac{b}{t}\right)\right] 
$$
is injective. 
\end{lemma}

\begin{proof}
We noticed above  that the map appearing in the Lemma identifies with
$$
v \mapsto  \Big(\big(Y_L(v)\big)(t^{-1})\Big)\cdot |0\rangle 
$$
from $V_L^\kappa (\gog)$ to $V_L^\kappa (\gog) = U_L(\hgog)/U_{+,L}$. By the usual properties of the vertex operators we have
\[
\Big(\big(Y_L(v)\big)(t^{-1})\Big)\cdot |0\rangle=v,
\]
which proves our claim.
\end{proof}

\subsubsection*{The case of one section} We now prove the claim in the case of one section $\grs$. Let $a=\grs^\sharp(t)$ and recall that we assume that $A$ is a domain. As in the case of a field our constructions boil down to completed pullback along $A$; in particular $\calU(\hat{\gog}_\sigma)(S) = U_A(\hat{\gog})$, $\calU_+(S)= U_{+,A} \stackrel{\text{def}}{=} U_A(\hat{\gog})(\gog \otimes A[[t]])$ and so on. Here, where $U_{A}(\hat{\gog})$ is the completed enveloping algebra of the $A$-linear version of the affine algebra $\hat{\gog}$.

\begin{lemma}\label{lem:Yunasezione}
The map 
$$
\ev_{(t-a)^{-1}}\Vkg \otimes_\mC A\lra \frac{U_A(\hat{\gog})}{\calU_{+,A}} \quad \text{ given by }\quad 
v\otimes b\mapsto \left[\calY_{\Sigma,t} \left(v\otimes\frac{b}{t-a}\right)(1)\right] 
$$
is injective.
\end{lemma}

\begin{proof}
By the discussion of Section \ref{ssec:descrizionelocale1} we have
\(
	\pbarOpoli(S) = A((t)), \mF^1_{\Sigma,\gog}(S) = \Homcont_A\left(A((t)),U_{A}(\hat{\gog})\right).
\) Let $L$ be the fraction field of $A$ then 
$$
V^\kappa(\gog) \otimes A \hookrightarrow V^\kappa(\gog)\otimes L
\quad \mand \quad 
U_{A}(\hat{\gog}) \hookrightarrow U_{L}(\hat{\gog}). 
$$	
Moreover $U_{+,L}\cap U_{A}(\hat{\gog})\supset\calU_{+,A}$.
Since the map $\calY_{\Sigma,t}$ on $\Vkg\otimes_\mC A$ is the restriction of the analogous map on  $\Vkg\otimes_\mC L$ the Lemma follows from Lemma 
\ref{lem:injectivityvertexalg}. 
\end{proof}

\subsubsection*{The general case}
We now prove Claim 2 in the case of $n$ sections. We still assume that $A$ is a domain. Without loss of generality we can assume that there are no two equal sections in $\Sigma$ and consider the nonempty open subset $S_{\neq} \subset S$ on which all sections in $\Sigma$ do not intersect with each other. Under the assumption that $A$ is domain we have $$
\Vkg \otimes \calO_S(S) \subset \Vkg\otimes \calO_S(S_{\neq})
$$ 
and the evaluation maps $\ev_{(t-a_i)^{-1}}$ on $S'$ are the restriction of the evaluation maps on $S$. Hence it is enough to prove our claim in the case $S=S_{\neq}$. 
To study this case we use the factorization property of the map $\calY_{\Sigma,t}$, of the space of fields and of the enveloping algebra $\calUgS$. 
By Lemma \ref{thm:factorizationofcaly} we have 
$$
\calY_{\Sigma,t}(v\otimes 1)\quad \text{ is the image of }\quad   \left(\calY_{\Sigma_j,t} (v\otimes 1)\right)_j\in \prod \mF^1_{\Sigma_j,\gog} \quad\text{ in }\quad  \mF^1_{\Sigma,\gog}.
$$
Notice that the elements $t-a_i$ are invertible in $\calO_{\overline \Sigma_j}$ for $j\neq i$, hence $\calY_{\Sigma_j,t}(v\otimes 1)\Big(b(t-a_i)^{-1}\Big) \in \calU_+$ for all $j\neq i$ and $b \in A$. 
Therefore, by the factorization properties of fields (see Section \ref{ssec:pullbackcampi}), we have 
$$ \ev_{(t-a_i)^{-1}}(v\otimes b)=\bigg[ \sum_j \calY_{\Sigma_j,t}(v\otimes 1) 
\left(\frac {b}{t-a_i} \right) \bigg]=\bigg[ \calY_{\Sigma_i,t}(v\otimes 1) 
\left(\frac {b}{t-a_i} \right) \bigg] $$
Where, recall, the embedding $\calU(\hat{\gog}_{\Sigma_i}) \to \calU(\hat{\gog}_{\Sigma})$ is induced by the embedding of $\calO_{\overline \Sigma_j^*}$ in $\pbarOpoli$. By the description of the PBW basis, this map in injective and $\calU_+\cap\calU(\hat{\gog}_{\Sigma_i})=
\calU(\hat{\gog}_{\Sigma_i})(\gog\otimesl\calO_{\overline{\Sigma}_i})$. So that the injectivity of $\ev_{(t-a_i)^{-1}}$ follows from 
Lemma \ref{lem:Yunasezione}.
\end{proof}

\subsection{Factorization properties}

We conclude by stating some factorization properties of our constructions. We refer to Notation \ref{ntz:factorization} and consider a surjection of finite sets $J \twoheadrightarrow I$. Recall the factorization morphisms of Section \ref{ssec:fattorizzazionecampigenerale}.

\begin{proposition}\label{prop:factpropertieschiralalg}
The maps $\Ran^{\mF^1}_{J/I}$, $\fact^{\mF^1}_{J/I}$ above induce natural isomorphisms
	\begin{align*}
		\Ran^{V^\kappa}_{J/I} &: \hat{i}_{J/I}^*\calV^\kappa_{\Sigma_J}(\gog), \to \calV^\kappa_{\Sigma_I}(\gog) \\ \fact^{V^\kappa}_{J/I} &: \hat{j}_{J/I}^*\left( \prod_{i\in I} \calV^\kappa_{\Sigma_{J_i}}(\gog) \right) \to  \hat{j}_{J/I}^*\calV^\kappa_{\Sigma_J}(\gog) 
	\end{align*}
which make $\calV^\kappa_{\Sigma}(\gog)$ a complete topological factorization algebra with respect to $\prod$.
\end{proposition}

\begin{proof}
    The maps are constructed via by a combination of the factorization morphisms of $\mF^1_{\Sigma,\gog}$ (see Section \ref{ssec:fattorizzazionecampigenerale}) and Theorem \ref{thm:factorizationofcaly}, which together with Theorem \ref{teo:descrizionelocaleVkSg} show that they are well defined and isomorphism. The fact that they induce factorization structures follows from the properties of $\Ran^{\mF^1}$ and $\fact^{\mF^1}$.
\end{proof}

Let us restate Theorem \ref{thm:factorizationofcaly} with respect to the factorization structure of $\calV^\kappa_\Sigma(\gog)$ of Proposition \ref{prop:factpropertieschiralalg}.

\begin{corollary}[Of Theorem \ref{thm:factorizationofcaly}]\label{coro:factpropchiralalgcaly}
    The following diagrams commute
    \[\begin{tikzcd}
	{\hat{i}^*_{J/I}\left(V^\kappa(\gog)\otimes\calO_{\overline{\Sigma}_J}\right)} && {\hat{i}^*_{J/I}\calV^\kappa_{\Sigma_J}(\gog)} \\
	\\
	{V^\kappa(\gog)\otimes\calO_{\overline{\Sigma}_I}} && {\calV^\kappa_{\Sigma_I}(\gog)}
	\arrow["{\hat{i}^*_{J/I}\calY_{\Sigma_J,t}}", from=1-1, to=1-3]
	\arrow["{\Ran^{\pbarO}_{J/I}}"', from=1-1, to=3-1]
	\arrow["{\Ran^{V^\kappa}_{J/I}}", from=1-3, to=3-3]
	\arrow["{\calY_{\Sigma_I,t_I}}"', from=3-1, to=3-3]
\end{tikzcd}\]
    \[\begin{tikzcd}
		{\hat{j}^*_{J/I}\left( \prod_{i \in I} V^\kappa(\gog)\otimes\calO_{\overline{\Sigma}_{J_i}}\right)} && {\hat{j}^*_{J/I}\left( \prod_{i \in I} \calV^\kappa_{\Sigma_{J_i}}(\gog)\right)} \\ \\
	{\hat{j}^*_{J/I}\left(V^\kappa(\gog)\otimes\calO_{\overline{\Sigma}_J}\right)} && {\hat{j}^*_{J/I}V^\kappa_\Sigma(\gog)} 
	\arrow["{\hat{j}^*_{J/I}\calY_{\Sigma_J,t}}", from=3-1, to=3-3]
	\arrow["{\fact^{\pbarO}_{J/I}}"', from=1-1, to=3-1]
	\arrow["{\fact^{V^\kappa}_{J/I}}", from=1-3, to=3-3]
	\arrow["{\prod_{i \in I}\calY_{\Sigma_{J_i},t_i}}"', from=1-1, to=1-3]
\end{tikzcd}\]

\end{corollary}

\section{Coordinate independent description of \texorpdfstring{$\VkSg$}{the chiral algebra}}\label{sec:identificationopers1}

In this Section we are going to give a description of our chiral algebra $\VkSg$ which is independent from the choice of a coordinate; we will do this by exploiting the coordinate dependent isomorphism $\calY_{\Sigma,t}$ and study how it behaves under change of coordinates. 

We will see that our chiral algebra identifies with an analogue of the vertex algebra bundle constructed in \cite[Section 6]{frenkel2004vertex}. Let us recall briefly how the latter is constructed. On any smooth curve $C$ over the complex numbers there is a canonical $\Autpiu{} O$ (left) torsor $\Aut_C \to C$, then, as the vertex algebra $V^\kappa(\gog)$ has a natural action of the group $\Autpiu{} O$, one can construct the twist $\Aut_C \times_{\Autpiu{} O} V^\kappa(\gog)$ obtaining a canonically defined vector bundle of infinite dimension on $C$. The goal of this Section is to reproduce this construction when considering $\oSigma$ in place of $C$ and show that the vector bundle we obtain is canonically isomorphic to $\VkSg\otimes_{\pbarO}T_{\oSigma}$. In order to do this we will need to take into account the topological nature of $\oSigma$, so our construction will be a little less straightforward to that of Frenkel and Ben-Zvi. Let us start with some recollections on the group scheme $\Autpiu{} O$ that we will need in the sequel.

\subsection{Recollections on the group \texorpdfstring{$\Aut O$}{Aut O}}\label{ssec:recollectionsauto}

Let $O = \mC[[z]]$ and let $\gom \subset O$ be the maximal ideal $z\mC[[z]]$. For any commutative $\mC$-algebra $R$ consider $O_R = R[[z]]$ and the ideal $\gom_R = zR[[z]]$. We write $\Aut O$ for the group functor $\Aut O(R) = \Aut_R^{\cont}(O_R)$, where the topology of $R[[z]]$ is that generated by the ideal $\gom_R$. \index{$\Aut O$} Any continuous $R$-linear automorphism of $R[[z]]$ is determined by its value on $z$ and via this identification one can show that
\[ \Aut O (R) \simeq \left\{ \rho(z) = \sum_{k \geq 0} \rho_kz^k : \rho_0 \text{ is nilpotent, } \rho_1 \in R^*\right\}. \] 
This also shows that $\Aut O$ is an ind-scheme. Under this identification the group multiplication reads as
    \(     \tau_1(z) \cdot \tau_2(z) = \tau_2(\tau_1(z)).      \)
We move on and recall the subgroup $\Autpiu{} O \subset \Aut O$ and some of its properties.

\begin{definition}
    We consider the group schemes $\Autpiu{} O,\Autpiu{n} (O)$ for $n\geq2$ defined as the following functors on $\mC$-algebras:
    \begin{align*}
		\Autpiu{}O(R) &= \{ \rho \in \Aut_R^{\cont}(O_R) : \rho(\gom_R) \subset \gom_R \},\\
		\Autpiu{n} O(R) &= \{ \rho \in \Aut_R(O_R/\gom_R^n) : \rho(\gom_R) \subset \gom_R \}.
	\end{align*}
	There\index{$\Autpiu{} O, \Autpiu{n} O$} are natural surjective (also on $R$ points) group homomorphisms $$\Autpiu{} O \xrightarrow{\pi_n} \Autpiu{n} O, \quad \Autpiu{m} O \xrightarrow{\pi_{n,m}} \Autpiu{n} O$$ which present $\Autpiu{} O$ as the limit of the group schemes $\Autpiu{n} O$ along the maps $\pi_{n,m}$.
\end{definition}

As for $\Aut O$, any automorphism $\rho \in \Autpiu{} O$ is determined by its value on $z$: $\rho(z) = \sum_{k \geq 0} \rho_kz^k$. The condition for this series to determine an element in $\Autpiu{} O (R)$ is $\rho_0 = 0, \rho_1 \in R^*$. Thus, we have bijections
    \[
        \Autpiu{} O (R) \simeq \left\{ \sum_{k\geq 1} \rho_kz^k : \rho_{1} \in R^* \right\} \quad \Autpiu{n} O (R) \simeq \left\{ \sum^{n-1}_{k = 1} \rho_kz^k : \rho_{1} \in R^* \right\}.
    \]
We denote by $\Der O$ and $\Der^0 O$\index{$\Autpiu{} O$!$\Der^0 O$} the Lie algebras of $\Aut O$ and $\Autpiu{} O$ respectively. We have $\Der O = \mC[[z]]\partial_z$,\index{$\Aut O$!$\Der O$} while $\Der^0 O$ can be identified with the subalgebra of derivations $f(z)\partial_z \in \Der O$ for which $f(z) \in \gom$.

We consider also the subgroups $\Autzero{} O \subset \Autpiu{} O $ , $\Autzero{n}
O \subset \Autpiu{n} O$  defined as
\begin{align*}
	\Autzero{} O (R) &= \left\{ \rho \in \Autpiu{} O(R) \text{ s.t. } \overline{\rho} : O_R/\gom_R^2 \to O_R/\gom_R^2 \text{ is the identity}\right\}, \\
	\Autzero{n} O (R) &= \left\{ \rho \in \Autpiu{n} O(R) \text{ s.t. } \overline{\rho} : O_R/\gom_R^2 \to O_R/\gom_R^2 \text{ is the identity}\right\}.
\end{align*}
As for the groups $\Autpiu{}$ there are natural surjective morphisms $\Autzero{} O \xrightarrow{\pi_n} \Autzero{n} O$ and $\Autzero{m} O \xrightarrow{\pi_{n,m}} \Autzero{n} O$ which present $\Autzero{} O$ as the limit of the group schemes $\Autzero{n} O$. The kernel of $\pi_{n,n+1}$ is isomorphic to $\mG_a$ so that by induction on $n$ each group $\Autzero{n} O$ is unipotent and $\Autzero{} O$ is pro-unipotent.

Finally there is a natural embedding $\mG_m \hookrightarrow \Autpiu{} O$, given by $\lambda \mapsto (z \mapsto \lambda z)$. This copy of $\mG_m$ normalizes $\Autzero{} O$ and induces an isomorphism
	\[
		\Autpiu{} O = \mG_m \ltimes \Autzero{} O.
	\]
This decomposition will be fundamental to prove an equivariance Lemma in \ref{lem:repsofautodero}. 

\subsection{Spaces, Jet schemes and the canonical \texorpdfstring{$\Autpiu{} O$}{Automorphism group} torsor}\label{ssec:spacesjetscanonicalbundle}

Here we explain how to construct an analogue of the canonical $\Autpiu{} O$ torsor $\Aut_C \to C$ when $C$ is replaced by $\overline{\Sigma}$, which we think of as a formal curve over $S$. In order to do construct it we will need to introduce a suitable notion of space and the analogue of jet schemes. 

\subsubsection{Spaces}\label{sssec:spaces}
Our geometric setting will be the following. Let $\Sch_S$ be the category of schemes over $S$. We equip $\Sch_S$ with the Zariski topology. A \emph{space} over $S$ is a set valued sheaf over $\Sch_S$; we denote by $\Sp_S$ the category of spaces. We call them spaces because in our constructions they will play the role of the geometric objects, while we will use the word sheaf to denote sheaves for the Zariski topology on $S$ (such as $\pbarO,\VkSg$ etc.). Let us notice here that the restriction of a space $X \in \Sp_S$ to open subsets determines a set valued sheaf over $S$, that we will also denote by $X$ and a sheaf of algebras $\Fun(X)$ over $S$ that on an open subset $U \subset S$ is defined as $\Fun(X)(U) = \Hom_{\Sp_U}(X_{|U},\mA^1_U)$.  Every scheme over $S$ determines a space, that will be denoted by the same symbol.

Practically, we prefer to work with affine schemes instead of $\Sch_S$, since this makes the notation and various arguments a bit more transparent. We denote by $\Aff_S$\index{$\Aff_S$} the category of affine schemes equipped with a map $\Spec R \to S$. We endow $\Aff_S$ with the Zariski topology. Any space can be restricted to a functor $\Aff_S^{\op} \to \Set$, since any scheme admits a cover by affine schemes this procedure establishes an equivalence between spaces and sheaves on $\Aff_S$ for the Zariski topology. We will therefore never stress the difference between sheaves on $\Sch_S$ and sheaves on $\Aff_S$; we introduced spaces as sheaves on the former because we will need to evaluate our functors on $S$, from time to time. We call a functor $\Aff_S^{\op} \to \Set$ \emph{a prespace}. Every space determines a prespace. Given a prespace $X$ we define the associated space $X^{sp}$ as its sheafification.

\subsubsection{The spaces \texorpdfstring{$\oSigma$}{attached to the formal neighbourhood} and \texorpdfstring{$\oSigma_n$}{the n-th formal neighborhood}}

Given an object $\Spec R \in \Aff_S$ we will write $\calO_{\overline{\Sigma}_R}$ for the completed pullback of $\pbarO$ along $\Spec R \to S$, this is a complete topological sheaf of $\calO_{\Spec R}$-algebras that and just as $\pbarO$, is a \QCCF w-compact sheaf. This allows us to consider our $\overline{\Sigma}$ as the space which on affine schemes over $S$ is defined by
\[
	\overline{\Sigma} : \Aff_S^{\op} \to \text{Set} \qquad R \mapsto \Hom^{\cont}_{\mathrm{Alg}_R}\left(\calO_{\overline{\Sigma}_R},\calO_R\right) = \varinjlim_n \Hom_{\mathrm{Alg}_R}\left(\calO_{\overline{\Sigma}_R}/\calO_{\overline{\Sigma}_R}(-n),\calO_R\right).
\]

Similarly we define $\oSigma_n$ as the space given by  $R\mapsto \Hom^{\cont}_{\mathrm{Alg}_R}(\calO_{\overline{\Sigma}_R}/\calO_{\overline{\Sigma}_R}(-n),\calO_R)$. This functor is represented by the relative spectrum $\Spec_{S}(\pbarO/\pbarO(-n))$. It follows that $\overline{\Sigma}$ is an ind-affine scheme over $S$.

\begin{remark}
   Notice that defining a space $\oSigma'$ over all schemes $T$ over $S$ using the same formula does \emph{not} give us a sheaf. For example, if $S=\Spec \mC$, $X=\mA^1$ and $\Sigma$ is the zero section, then $\oSigma(R)$ is the set of nilpotent elements of $R$, while $\oSigma'(T)$ would give the nilpotent functions on $T$. This is not a Zariski sheaf over non noetherian spaces. In particular the extension as a sheaf of $\oSigma$ from affine schemes over $S$ to all schemes over $S$ is not equal to $\oSigma'$. However, since we assume $S$ to be topologically noetherian, then the restriction of $\oSigma$ and $\oSigma'$ to the open subsets of $S$, are equal. A similar remark applies to many constructions we are going to give in this Sections. 
\end{remark}

\subsubsection{Vector bundles on spaces}\label{sssec:vectorbundlesonspaces}
We can define vector bundles over $\oSigma$, or more in general over a space, as in the case of schemes.

\begin{definition}
	A vector bundle on $\overline{\Sigma}$\index{vector bundle on $\overline{\Sigma}$} is a space $V$ together with a map $V \to \overline{\Sigma}$ such that there is an addition map $+_V :V \times_{\overline{\Sigma}} V \to V$ and a multiplication map $\mA^1_{\overline{\Sigma}} \times_{\overline{\Sigma}} V \to V$ defined over $\overline{\Sigma}$ which satisfy the usual axioms of a vector bundle. We require that Zariski locally on $S$ there exist isomorphisms $V \simeq \mA^n_{\overline{\Sigma}}$ which are linear and commute with the action of the monoid $\mA^1_{\overline{\Sigma}}$. 
\end{definition}

Being $\oSigma$ inf-affine over $S$, we can describe vector bundles also as sheaves. Indeed if $\calV$ is a locally free $\pbarO$-module of finite rank then it has a natural topology which is given by $\calV(-n)=\calO(-n)\cdot \calV$. On the direct sum of finitely many copies of $\pbarO$ this coincides with the topology of the sum. Any morphism of locally free modules of finite rank is automatically continuous if we equip the modules with this topology. In particular the sheaf underlying a vector bundle is $\pbarO$ locally free sheaf of finite rank and has a natural topology. 

Viceversa any locally free  $\pbarO$-module $\calV$ of finite rank we can associate a space as in the case of ordinary schemes. Let $\calV^*=\calHom_{\pbarO}(\calV,\pbarO)$ be the dual of $\calV$ and define a space $V$ as follows. Given $f:\Spec R\lra S$ then $V(R)$ is the set of couples $(\eta,\theta)$ such that $\eta:\pbarO\lra f_*\calO_R$ is a continuous homomorphism of $\calO_S$-algebras and where $f_*\calO_R$ is equipped with the discrete topology (that is $\eta\in\oSigma(R)$) and  $\theta:\calV^*\lra f_*\calO_R$ is a morphism of 
$\pbarO$-modules. Notice that $\theta$ is automatically continuous since the image of $\pbarO(-n)$ in $f_* \calO_R$ is zero  for $n\gg 0$. 

The analogous construction over $\oSigma_n$ gives as a result the space associated to the relative spectrum over $S$ of the symmetric algebra of $\calV^*$ over the ring $\pbarO/\pbarO(-n)$. Moreover, let $\calV$ be a locally free $\pbarO$-module of finite rank and set $\calV_n$ be equal to its restriction to $\oSigma_n$. If $V$ is the space associated to $\calV$ and $V_n$ the space associated to $\calV_n$, then $V$ is the colimit of the spaces $V_n$. 

\subsubsection{Torsors on spaces}\label{ssec:torsorisuspazi}
Let $G$ be a group functor $G : \Aff_S^{\op} \to \mathrm{Grp}$, which is also a sheaf for the Zariski topology (that is, a group object in $\Sp_S$). Given a space $X$, a $G$-torsor over $X$ is a space $\goF\lra X$ such that for each $\Spec R\in \Aff_S$ there is a functorial right action of $G(R)$ over $\goF(R)$ such that the map $\goF(R)\lra X(R)$ is $G(R)$-invariant. We require $\goF \simeq X\times G$ in an equivariant way, locally on the Zariski topology on $S$. Of course one can define left torsors and right torsors, these are equivalent by considering the inverse action of $G$.

If $\goF$ is a $G$ torsor over $X$, $G$ is the pullback of an algebraic group over $\mC$, and $W$ is a finite dimensional representation of $G$ we define $W_\goF=\goF\times_G W$ as the space associate to the prespace $R\mapsto (\goF(R)\times W(R))/G(R)$.  This is a vector bundle over $X$. 


\subsubsection{Jet schemes}\label{sssec:jetsonspaces}

If $X$ is a space over $S$ then we define $JX$ and $J_n$ , the jet scheme and $n$-th jet scheme of $X$ as follows if $f:\Spec R \lra S$ then define
$$
J_nX(R)=X\big(R[[z]]/z^n\big), \qquad JX(R)= \limpro J_nX\big(R\big),
$$
where we consider  $\Spec (R[z]/z^n)$ over $S$ by precomposing $f$ with 
the natural map $\Spec(R[z]/z^n)\lra \Spec R$ (a similar convention will be used for $\Spec R[[z]]$). They are also spaces over $S$. Notice that the second space does not need to be equal to 
$JX(R[[z]])$. Indeed, in the case of $\oSigma$, we can rephrase the definitions of Jet spaces as follows:
\begin{align*}
	J\overline{\Sigma}(R) &= \Hom^{\cont}_{R\mathrm{-alg}}(\calO_{\overline{\Sigma}_R},\calO_R[[z]]) \\
	J_n\overline{\Sigma}(R) &= \overline{\Sigma}(R[z]/z^n) = \Hom^{\cont}_{R\mathrm{-alg}}(\calO_{\overline{\Sigma}_R},\calO_R[z]/z^n)
\end{align*}
where we consider $R[[z]]$ with the topology defined by the powers of $z$ (while in the definition of $\oSigma(R[[z]])$ we would consider this space with the discrete topology). 

\index{$J\overline{\Sigma},J_n\overline{\Sigma}$}We would like to study the behaviour of these spaces in the case when $S$ is affine and well covered.
Since $J\overline{\Sigma} = \varprojlim_n J_n\overline{\Sigma}$ and we will treat the former frequently as a pro-object. This will apply to other constructions related to $J\overline{\Sigma}$ as well.

\begin{lemma}\label{lem:descrjetpoints}
	Assume that $S = \Spec A$ is affine and well covered, so that there exists a coordinate $t \in \calO_{\overline{\Sigma}}(S)$. Fix a morphism $\Spec R \to S$. Then a continuous, $R$-linear morphism $\varphi : \calO_{\overline{\Sigma}_R} \to \calO_R[[z]]$ is uniquely determined by the image $\varphi(t) \in R[[z]]$. There exists a polynomial $p_{S,t}(x) \in A[x]$ such that a series $\varphi(t) \in R[[z]]$ determines a morphism $\varphi$ if and only if $p_{S,t}(\varphi(t)) \in R[[z]]$ is nilpotent modulo $z$. 
\end{lemma}

\begin{proof}
	This follows from the fact that by the local description of $\pbarO$ given in Section \ref{ssec:descrizionelocale1} we have $\calO_{\overline{\Sigma}} = \varprojlim \calO_S[t]/(\prod (t-a_i))^n$ is the completion of $\calO_S[t]$ along $\prod (t-a_i)$ and we may consider $p_{S,t}(x) = \prod (x-a_i)$.
\end{proof}

\begin{corollary}\label{coro:descrjn}
	The maps $J\overline{\Sigma} \to \overline{\Sigma}$, $J_n\overline{\Sigma} \to \overline{\Sigma}$ are representable and affine. It follows that the functors $J\overline{\Sigma}$,$J_n\overline{\Sigma}$ are representable by ind-affine schemes over $S$.
\end{corollary}

\begin{proof}
	We write the proof only for $J_n\overline{\Sigma}$ as the proof for $J\overline{\Sigma}$ is completely analogous. The assertion being Zariski local on $S$, so we may assume that $S$ is affine and well covered, write $S = \Spec A$, let $s \in \pbarO(S)$ be a coordinate for $\overline{\Sigma}$. By Lemma \ref{lem:descrjetpoints} the functor $J\overline{\Sigma}$ may be described as follows (recall that in this case $\Aff_S = \text{Alg}_A$ is the category of commutative $A$-algebras).
	\[
		\text{for } R \in \text{Alg}_A, \quad J_n\overline{\Sigma}(R) \simeq \left\{ \sum^{n-1}_{k\geq 0} r_kz^k : \prod (r_0 - a_i) \text{ is nilpotent} \right\} \simeq \overline{\Sigma}(R) \times \mA^{n-1}_S (R). \qedhere
	\]
\end{proof}

\subsubsection{The tangent bundle} 
If $X$ is a space we define $TX=J_2X$. For $J_2X$, as usual, we denote the variable $z\in R[[z]]/z^2$ by $\gre$. 
The evaluation to $\gre=0$ defines map $\pi:J_2X\lra X$ and the inclusion $R\subset R[\gre]$ gives  a section of this map. The map of $R$ algebras
$$
\{(a+\gre b,a+\gre c)\in R[\gre]\times R[\gre]\}\lra R[\gre]
$$
given by $(a+\gre b,a+\gre c)\mapsto a+\gre(b+c)$ defines a sum over the fibers of the map $\pi$ and the multiplication of $\gre$ by an element of $R$ defines an action of the monoid $\mA^1$ on these fibers. 

As shown above for $X=\oSigma$, the map $T\oSigma\lra\oSigma$ is locally (in the Zariski topology for $S$) isomorphic to the first projection $\overline{\Sigma}\times \mA^1 \to \overline{\Sigma}$; one can show that this identification respects sums over the fibers and the action of $\mA^1$ so that $T\oSigma$ is a vector bundle of rank $1$ on $\overline{\Sigma}$.
Its associated locally free sheaf of $\pbarO$-modules is $\Tan_{\overline{\Sigma}}$ and its dual is $\pbarOmega$.

\subsection{The canonical \texorpdfstring{$\Autpiu{S} O$}{Aut+ O}-torsor}\label{ssec:canonicaltorsor}

We construct canonical  $\Autpiu{S} O = S \times \Autpiu{} O$ and $\Autpiu{S,n} O = S \times \Autpiu{n} O$ torsors on $\overline{\Sigma}$, to be denoted by $\Aut_{\overline{\Sigma}}$ and $\Aut_{\overline{\Sigma},n}$ respectively. These are analogues of the torsor $\Aut_C$ on a smooth curve (see for instance \cite{casarin2025bundle}).
Recall that $\Autpiu{S} O$ and $\Autpiu{S,n} O$, as functors on $\Aff^{\text{op}}_S$, are given by the restriction of $\Autpiu{} O,\Autpiu{n} O$ to $\Aff_S \subset \Aff_{\mC}$.

\begin{definition}
	We define $\Aut_{\overline{\Sigma}}$\index{canonical $\Autpiu{}O$ torsor: $\Aut_{\overline{\Sigma}}$} as a subfunctor of $J\overline{\Sigma}$:
	\[
		\Aut_{\overline{\Sigma}}(R) = \left\{ \rho \in J\overline{\Sigma}(R) \text{ such that } d\rho : \rho^*\Omega^1_{\overline{\Sigma}} \to \Omega^1_{R[[z]]/R} \text{ is an isomorphism} \right\}
	\]
	where
	\[
		\rho: \calO_{\overline{\Sigma}_R} \to R[[z]], \qquad \rho^*\Omega^1_{\overline{\Sigma}} = 
		\Omega^1_{\overline{\Sigma}_R/R}\otimes_{\calO_{\overline{\Sigma}_R}} R[[z]].
	\]
	Similarly we define $\Aut_{\overline{\Sigma},n}$ as a subfunctor of $J_n\overline{\Sigma}$.
	\[
		\Aut_{\overline{\Sigma},n}(R) = \left\{ \rho \in J_n\overline{\Sigma}(R) \text{ such that } d\rho : \rho^*\Omega^1_{\overline{\Sigma}} \to \Omega^1_{(R[z]/z^n)/R} \text{ is an isomorphism} \right\}
	\]
\end{definition}

	The group $\Autpiu{S} O$ naturally acts by post-composition on $J\overline{\Sigma}$ and the map $J\overline{\Sigma} \to \overline{\Sigma}$ is invariant under this action. Since any element $\tau \in\Autpiu{S} O(T)$ induces an automorphism $d\tau : \Omega^1_{R[[z]]/R} \to \Omega^1_{R[[z]]/R}$ the $\Autpiu{S} O$ action preserves $\Aut_{\overline{\Sigma}}$. In the case where $S$ is affine, well covered, and equipped with a coordinate $s \in \pbarO (S)$, under the identification of Lemma \ref{lem:descrjetpoints}
	\[
		J\overline{\Sigma} = \left\{ \sum_{k \geq 0} r_k z^k : \prod (r_0 - a_i) \text{ is nilpotent} \right\}
	\]
	the action of an element $\rho \in \Autzero{} O$ is given by
	\(
		\sum_{k \geq 0} r_k z^k \mapsto \sum_{k \geq 0} r_k \rho(z)^k
	\) and $\Aut_{\oSigma}$ can be described as
	\[
		\Aut_{\overline{\Sigma}} = \left\{ \sum_{k \geq 0} r_k z^k : \prod (r_0 - a_i) \text{ is nilpotent and } r_1 \in R^* \right\}
	\]

\begin{proposition}\label{prop:trivsigmabarauto}
	The map $\Aut_{\overline{\Sigma}} \to \overline{\Sigma}$ with the above action of $\Autpiu{S} O$ on $\Aut_{\overline{\Sigma}}$ exhibits $\Aut_{\overline{\Sigma}}$ as a left $\Autpiu{S} O$-torsor over $\overline{\Sigma}$. Similarly, the map $\Aut_{\overline{\Sigma},n} \to \overline{\Sigma}$ exhibits $\Aut_{\overline{\Sigma},n}$ as an $\Autpiu{S,n}O$ torsor. The choice of a local coordinate $t$ for $\pbarO$ induces a section $\triv_t : \overline{\Sigma} \to \Aut_{\overline{\Sigma}}$ and hence a trivialization $\triv_t : \overline{\Sigma} \times \Autpiu{} O \to \Aut_{\overline{\Sigma}}$.
\end{proposition}

\begin{proof}
	We prove the statement of the Proposition for $\Aut_{\overline{\Sigma}}$, the corresponding statement for $\Aut_{\overline{\Sigma},n}$ is completely analogous. Recall the results of Lemma \ref{lem:descrjetpoints} and Corollary \ref{coro:descrjn}.

	Given a coordinate $t$ we construct the section $\triv_t : \overline{\Sigma} \to \Aut_{\overline{\Sigma}}$ as follows. Recall that the choice of $t$ and a hence an isomorphism $\pbarO \simeq \varprojlim \calO_S[t]/(\prod(t - a_i))$ induces isomorphisms
	\begin{align*}
		\overline{\Sigma}(R) &\simeq \left\{ r \in R \text{ such that } \prod(r - a_i) \text{is nilpotent} \right\} \\
		\Aut_{\overline{\Sigma}}(R) &\simeq \left\{ \sum_{k\geq 0} r_kz^k \text{ such that } \prod (r_0 - a_i) \text{ is nilpotent and } r_1 \in R^* \right\}
	\end{align*}

	\noindent so we define $\triv_t(r) = r + z$, or in other words, $\triv_t(r)$ is the unique continuous morphism $\calO_{\overline{\Sigma}_R} \to R[[z]]$ which maps $t \mapsto r + z$. The $\Autpiu{} O$ equivariant map $\triv_t : \overline{\Sigma}(R) \times \Autpiu{} O(R) \to \Aut_{\overline{\Sigma}}(R)$ that we get this way is clearly bijective by the description of $\Aut_{\overline{\Sigma}}(R)$
\end{proof}

We state the following Remark to make the statement of Proposition \ref{prop:changecoordinate} more transparent

\begin{remark}
	Let us remark that morphisms $\oSigma \to \Autpiu{S} O$ are easily described: 
	\[
		\Hom_{\Sp_S}\left( \oSigma, \Autpiu{S} O\right) = \left\{ \sum_{k\geq 1} f_kz^k \text{ with } f_k \in \pbarO(S) \text{ and } f_1 \in \pbarO(S)^*\right\} = \Aut_{S}^0(\pbarO(S)). 
	\]
	This follows by the fact that $\oSigma = \varinjlim \oSigma_n$ and by the fact that both $\oSigma_n$ and $\Autpiu{S} O = \Autpiu{} O \times S$.
\end{remark}

\begin{proposition}\label{prop:changecoordinate}
	Assume that $S$ is affine and well covered. Given coordinates $t,s$ for $\pbarO$, denote by $\triv_t,\triv_s : \overline{\Sigma} \times \Autpiu{S} O \to \Aut_{\overline{\Sigma}}$ the trivializations induced by the choice of $t$ and $s$ respectively. Then the map
	\[
		\triv_{st} = \triv_s^{-1}\triv_t : \overline{\Sigma} \times \Autpiu{S} O \to \overline{\Sigma} \times \Autpiu{S} O,
	\]
	which, being equivariant, is equivalent to give a morphism $\triv^{\univ}_{st}:\overline{\Sigma} \to \Autpiu{S} O$ such that 
	$\triv_t(r,\gra)  = \triv_s(r,\gra\cdot\triv^{\univ}_{st}(r))$. Then 
	\index{$\triv^{\univ}_{st}$}
	\[
		\triv^{\univ}_{st} = \sum_{k\geq 1} \frac{1}{k!}(\partial^k_ts)z^k \in \Hom_{\mathrm{Sp}_S}(\overline{\Sigma},\Autpiu{S} O)= \Aut_{S}^0(\pbarO(S)).
	\]
\end{proposition}

\begin{proof}
	We compute first $\triv_t$. Recall the for every $\Spec R \to S$, the trivialization $\triv_t$ is a map from $\oSigma(R)$ to $J\oSigma(R)$. Recall also that  $\oSigma(R)$ is 
	the space of continuous homomorphism $\calO_{\oSigma_R}\lra R$ and that $J\oSigma(R)$ is the space of continuous homomorphism from $\calO_{\oSigma_R}\lra R[[z]]$. 
	By Proposition  \ref{prop:trivsigmabarauto} for each $\gra\in \oSigma(R)$ we have 
	that $\triv_t(\gra):\calO_{\oSigma_R}\lra R[[z]]$ is the only continuous homomorphism such that
	$$
	\triv_t(\gra)(t)=\gra(t)+z.
	$$
    It follows that for all $g\in \calO_{\oSigma_R}$ we have
	$$
	\triv_t(\gra)(g)= \sum_{n\geq 0}\frac 1{n!}\, \gra(\partial^n_t g) \, z^n = \gra(g)+\sum_{n\geq 1}\frac 1{n!}\, \gra(\partial^n_t g) \, z^n.
	$$
	Indeed the formula on the right hand side defines a continuous homomorphism of $R$-algebras, with the same properties. Since, $\triv^{\univ}_{st}(\gra)$ is the unique element of $\Aut^0_S O$, such that
	$$
	\triv^{\univ}_{st}(\gra) \cdot \triv_s(\gra)(s) =  \triv_{t}(\gra)(s)
	$$
    from the previous formula applied to $g=s$ and $\triv_s(\gra)(s)=\gra(s)+z$ the claim follows.
\end{proof}

\subsection{The sheaf \texorpdfstring{$\twistAut{\oSigma}{V}$}{of functions on Opers over the formal neighborhood}}\label{ssec:gammasigmav}

Here we explain how to emulate the vertex algebra bundle construction of $V \mapsto \Aut_C \times_{\Autpiu{} O} V$ of \cite[Section 6]{frenkel2004vertex}. Having at hand the canonical torsor $\Aut_{\oSigma} \to \oSigma$ this will be fairly easy, but to deal with the topology of $\oSigma$ some care will be needed.

If $W$ is a finite dimensional representation of $\Aut^0$ recall from Section \ref{ssec:torsorisuspazi} the definition of the vector bundle $\Aut_\oSigma \times_{\Autpiu{} O} W$ over $\oSigma$. We denote the sheaf of sections of this vector bundle by $\twistAut{\oSigma}{W}$.


\begin{definition}[Construction of $\twistAut{\oSigma}{V}$]
	Let $V$ be an $\Autpiu{} O$ representation. Let $(V)^{\text{fin}}$ denote the directed poset of finite dimensional sub-representations of $V$. Recall that, as for any representation of an affine group scheme, $V = \varinjlim_{W \in (V)^{\text{fin}}} W$. 
    We define the topological sheaf
    \[
        \twistAut{\oSigma}{V} \stackrel{\text{def}}{=} \varinjlim_{W \in (V)^{\text{fin}}} \twistAut{\oSigma}{W},
    \]
    \index{$\twistAut{\oSigma}{V}$}where $\twistAut{\oSigma}{W}$ denotes the sheaf of $\pbarO$ modules on $S$ of sections of the finite dimensional vector bundle $\Aut_{\overline{\Sigma}} \times_{\Autpiu{} O} W$ on $\overline{\Sigma}$. Notice that in this case, locally, as a topological sheaf we have $\twistAut{\oSigma}{V} \simeq V \otimes \pbarO$ so that it is a direct sum of complete topological sheaves and hence complete. In the case where $V$ is an $\Autpiu{} O$-commutative algebra the sheaf $\twistAut{\oSigma}{V}$ is naturally a sheaf of commutative $\pbarO$-algebras.
\end{definition}

\begin{remark}\label{rmk:changecoordfunopd}
	Assume that $S$ is affine and well covered, so that we may pick a coordinate $t \in \pbarO(S)$, the isomorphism $\triv_t : \overline{\Sigma} \times \Autpiu{} O \to \Aut_{\overline{\Sigma}}$
	of Proposition \ref{prop:changecoordinate} induces, for any $\Autpiu{} O$ (finite dimensional) representation $V$, an isomorphism
	\[
		\triv^V_t : \oSigma \times V \to \Aut_{\oSigma}\times_{\Autpiu{S} O} V
		\qquad \text{ given by }\qquad 
		\triv^V_t(r\times v)=[\triv_t(r,1),v] .
	\]
	where $[\triv(r,1),v]$ is the class of $(\triv_t(r,1),v) \in \Aut_\oSigma \times V$ modulo the diagonal action of $\Autpiu{S} O$. This induces an isomorphism on the sheaf of sections
	\[
		\varphi_t : V \otimes \pbarO \to \twistAut{\oSigma}{V}
	\]
	\index{$\varphi_t$}
	Given any two coordinates $t,s$ a point $r \in \oSigma$, we have
	\begin{align*}
		(\triv^V_t)^{-1}\triv^V_s (r \times v) &= (\triv^V_t)^{-1}[\triv_s(r,1),v] = (\triv^V_t)^{-1} [(\triv^{\univ}_{st}(r))^{-1} \cdot \triv_t(r,1),v] \\
		&= (\triv^V_t)^{-1}[\triv_t(r,1),\triv^{\univ}_{st}(r) \cdot\, v] \\
		&= r \times (\triv^{\univ}_{st}(r)\cdot\, v)
	\end{align*}
	It follows that the automorphism $\varphi_{t,s} = \varphi^{-1}_{t}\varphi_s : V \otimes \pbarO \to  V\otimes \pbarO$ is then given by the action of $\triv^{\univ}_{st}$ on $V \otimes \pbarO$, so that $\grf_{t,s}$	corresponds to the element in $\Autpiu{} O(\pbarO(S)) = \Hom_{\Sp_S}(\oSigma,\Autpiu{S} O)$ 
	\[
		\varphi_{t,s}(z) = \sum_{k \geq 0} \frac{1}{k!} (\partial_t^k s) z^k \in \Autpiu{} O(\pbarO(S))
	\]
    \index{$\varphi_{t,s}$}acting on $V \otimes \pbarO$.
\end{remark}

\subsubsection{Factorization properties}

Recall notation \ref{ntz:factorization} and fix a surjection of finite sets $J \twoheadrightarrow I$. Recall the factorization structure of $\pbarO$ of Example \ref{ex:factstructurepbaro}:
\begin{align*}
	\Ran^{\pbarO}_{J/I} &: \hat{i}_{J/I}^*\calO_{\overline{\Sigma}_J} \rightarrow \calO_{\overline{\Sigma}_I}, \\
	\fact^{\pbarO}_{J/I} &: \hat{j}_{J/I}^*\left(\prod_{i \in I} \calO_{\overline{\Sigma}_{J_i}} \right) \to \hat{j}_{J/I}^*\calO_{\overline{\Sigma}_J} .
\end{align*}
and that a coordinate $t_J \in \calO_{\overline{\Sigma}_J}$ induces coordinates $t_i \in \calO_{\overline{\Sigma}_{J_i}}$ and $t_I \in \calO_{\overline{\Sigma}_J}$.

\begin{lemma}\label{lem:factorizationpsi}
	Let $V$ be any $\Autpiu{} O$ module. With the notations above, there are natural isomorphisms
	\[
    \Ran^{V,\Sigma}_{J/I} : \hat{i}_{J/I}^*\left( \twistAut{\oSigma_{J}}{V} \right) \rightarrow\left(\twistAut{\oSigma_{I}}{V} \right), \quad
    \fact^{V,\Sigma}_{J/I} : \hat{j}_{J/I}^*\left( \prod_{i\in I}\twistAut{\oSigma_{J_i}}{V} \right) \to  \hat{j}_{J/I}^*\left( \twistAut{\oSigma_{J}}{V} \right).
	\]
        Which make $\twistAut{\oSigma}{V}$ into a complete topological factorization algebra with respect to $\prod$. In addition on any well covered open subset with a coordinate $t$, the following diagrams commute
	\[\begin{tikzcd}
        {\hat{i}_{J/I}^*(V\otimes\calO_{\overline{\Sigma}_J})} && {\hat{i}_{J/I}^*(\twistAut{\oSigma_{J}}{V})} \\ \\
    {V\otimes\calO_{\overline{\Sigma}_I}} && {\twistAut{\oSigma_{I}}{V}}  \\  
	{\hat{j}_{J/I}^*\left(\prod_{i\in I} V\otimes\calO_{\overline{\Sigma}_{J_i}}\right)} && {\hat{j}_{J/I}^*\left(\prod_{i\in I} \twistAut{\oSigma_{J_i}}{V}\right)} \\ \\
	{\hat{j}_{J/I}^*(V\otimes\calO_{\overline{\Sigma}_J})} && {\hat{j}_{J/I}^*(\twistAut{\oSigma_{J}}{V})}
	\arrow["{\hat{j}_{J/I}^*\varphi_t}", from=4-1, to=4-3]
	\arrow["{\fact^{\pbarO}_{J/I}}"', from=4-1, to=6-1]
	\arrow["{\fact^{V,\Sigma}_{J/I}}", from=4-3, to=6-3]
	\arrow["{\hat{i}^*_{J/I}\varphi_t}", from=1-1, to=1-3]
	\arrow["{\Ran^{\pbarO}_{J/I}}"', from=1-1, to=3-1]
	\arrow["{\Ran^{V,\Sigma}_{J/I}}", from=1-3, to=3-3]
	\arrow["{\hat{j}_{J/I}^*(\prod_{i\in I} \varphi_{t_i}) }"', from=6-1, to=6-3]
	\arrow["{\varphi_{t_I}}"', from=3-1, to=3-3]
\end{tikzcd}\]
\end{lemma}

\begin{proof}
    The factorization structure on $\twistAut{\oSigma}{V}$ comes from the natural structure of a factorization space on $\Aut_{\overline{\Sigma}}$. Indeed, it follows directly from the definition that $i_{J/I}^*\Aut_{\overline{\Sigma}_J} \simeq \Aut_{\overline{\Sigma}_I}$ and that $j_{J/I}^*\Aut_{\overline{\Sigma}_J} = \coprod_{i \in I} j_{J/I}^*\Aut_{\overline{\Sigma}_{J_i}}$. The statement of commutativity of the diagram then easily follows.
\end{proof}

\subsubsection{The canonical connection}

Recall that the action of $\Autpiu{} O$ on $V^\kappa(\gog)$ actually comes from the structure of an $\Aut O$-module. We will prove that given any $\Aut O$-module $V$, the twisted sheaf $\twistAut{\oSigma}{V}$ has a natural connection, that's to say, its sheaf of sections is a left $\calD_{\overline{\Sigma}}$-module. In what follows we work with a fixed $\Aut O$-module $V$.

Let $T_{\varepsilon} : V[\varepsilon] \to V[\varepsilon]$ be the automorphism of $V[\varepsilon]$ corresponding to the element $T_{\varepsilon} \in \Aut O (\mC[\varepsilon])$ which satisfies $T_{\varepsilon}(z) = z - \varepsilon$. Notice that there exists a unique endomorphism $T$ of $V$ such that for any $v,w \in V$, $T_{\varepsilon}(v+\varepsilon w) = v + \varepsilon(Tv + w)$. 

Locally, on a well covered $S$ with a coordinate $t$, using the isomorphism $\grf_t$ of Remark \ref{rmk:changecoordfunopd} the action of $\partial_t$ will be given on $V\otimes \pbarO$ by $$\partial _t = T\otimes 1 + 1\otimes \partial _t.$$ We will prove that this definition is indepedent of the choice of the local coordinate, hence it defines a global structure of left $\pbarD$-module on $\twistAut{\oSigma}{V}$.

\begin{lemma}\label{lem:canonicalconnectionauto}
	Let $V$ be any $\Aut O$-module. Consider the associated locally free sheaf $\twistAut{\oSigma}{V}$ on $\overline{\Sigma}$. Let $s,t$ be two local coordinates for $\overline{\Sigma}$. Then the following diagram is commutative 
	\[\begin{tikzcd}
		{V\otimes\pbarO} &&& {V\otimes\pbarO} \\
		\\
		{V\otimes\pbarO} &&& {V\otimes\pbarO}
		\arrow["{T\otimes 1 + 1 \otimes \partial_t}", from=1-1, to=1-4]
		\arrow["{\varphi_{t,s}}"', from=3-1, to=1-1]
		\arrow["{\varphi_{t,s}}", from=3-4, to=1-4]
		\arrow["{  T\otimes (\partial_ts) + 1\otimes \partial_t}"', from=3-1, to=3-4]
	\end{tikzcd}\]
	where $\varphi_{t,s}$ is the automorphism constructed in Remark \ref{rmk:changecoordfunopd}.
\end{lemma}

We will prove the Lemma below. We first notice that this immediately implies that the action described above is well given.   

\begin{proposition}\label{prop:canonicalconnectionauto}
	Let $V$ be any $\Aut O$-module. Then the sheaf $\twistAut{\oSigma}{V}$ is canonically a left $\calD_{\overline{\Sigma}}$-module. After the choice of a local coordinate $t$ the following diagram commutes:
	\[\begin{tikzcd}
	{V\otimes\pbarO} && {\twistAut{\oSigma}{V}} \\
	\\
	{V\otimes\pbarO} && {\twistAut{\oSigma}{V}}
	\arrow["{\varphi_t}", from=1-1, to=1-3]
	\arrow["{T\otimes 1 + 1\otimes \partial_t}"', from=1-1, to=3-1]
	\arrow["{\partial_t}", from=1-3, to=3-3]
	\arrow["{\varphi_t}"', from=3-1, to=3-3]
	\end{tikzcd}\]
\end{proposition}

\begin{proof}
	After the choice of a coordinate $t$ define the action of $\calD_{\overline{\Sigma}} = \pbarO[\partial_t]$ by making $\partial_t$ act like the diagram above indicates. As all morphisms $\varphi$ are $\pbarO$-linear by writing $\partial_t = (\partial_t s)\partial_s$ and applying Lemma \ref{lem:canonicalconnectionauto} we see that the action constructed is independent from the choice of the coordinate $t$.
\end{proof}

To prove Lemma \ref{lem:canonicalconnectionauto} we need the following remark on group functor actions on vector spaces.

\begin{remark}\label{rmk:extensionautoaction}
	Let $R' \subset R$ be commutative algebras over $\mC$ and let $G : \Aff_{\mC} \to \text{Grp}$ be any group functor. By functoriality of $G$ there is a natural action of the group $\Aut_{R'} R$ on the group $G(R)$. Let $\widetilde{G(R)}_{R'} = G(R) \rtimes \Aut_{R'} R$ be the semidirect product coming from this action. 
	
	If $V \in \text{Vect}_{\mC}$ is a $G$ representation then there is a natural $R'$-linear action of $\widetilde{G(R)}_{R'}$ on $V\otimes R$. It is described as follows: the restricted action of $G(R)$ is given by the $G$-module structure of $V$ while the restricted action of $\Aut_{R'} R$ fixes $V$ and acts naturally on $R$. 
\end{remark}

\begin{proof}[Proof of Lemma \ref{lem:canonicalconnectionauto}]
	Notice that the statement of the Lemma is equivalent to the commutativity of the diagram 
	\[\begin{tikzcd}
	{V\otimes\pbarO(S)[\varepsilon]} &&&& {V\otimes\pbarO(S)[\varepsilon]} \\
	\\
	{V\otimes\pbarO(S)[\varepsilon]} &&&& {V\otimes\pbarO(S)[\varepsilon]}
	\arrow["{(T\otimes 1 + 1\otimes \partial_t)_\varepsilon}", from=1-1, to=1-5]
	\arrow["{\varphi_{t,s}}"', from=3-1, to=1-1]
	\arrow["{\varphi_{t,s}}", from=3-5, to=1-5]
	\arrow["{(T\otimes (\partial_ts) + 1\otimes \partial_t)_\varepsilon}"', from=3-1, to=3-5]
	\end{tikzcd}
	\]
	where if $\ell$ is invertible in $\pbarO$ we define $(T\otimes \ell + 1\otimes \partial_t)_{\varepsilon}$ as the automorphism of $V\otimes\pbarO[\varepsilon]$ given by
	\[
	(T\otimes \ell + 1\otimes \partial_t)_\varepsilon(v\otimes f + \varepsilon(w\otimes g)) = v \otimes f + \varepsilon(Tv\otimes \ell f + v\otimes\partial_t f + w\otimes g),
	\]
	while $\varphi_{t,s}$, by a slight abuse of notation, stands for the $\pbarO(S)[\varepsilon]$-linear extension of $\varphi_{t,s}$. All morphisms appearing in the latter diagram come from the action of $\widetilde{\Aut O (\pbarO(S)[\varepsilon])}_{\mC[\varepsilon]}$ introduced in Remark \ref{rmk:extensionautoaction} (in the case $G = \Aut O$, $R' = \mC[\varepsilon]$ and $R = \pbarO(S)[\varepsilon]$) so that we can make direct computations. Notice that in this case the action of $\eta \in \Aut_{\mC[\varepsilon]}(\pbarO(S)[\varepsilon])$ on $\rho \in \Aut O (\pbarO(S)[\varepsilon])$ is described as follows: the series $(\eta\cdot\rho)(z)$ is given by applying $\eta$ to the coefficients of $\rho(z)$. The automorphisms we are interested in correspond to the following elements of $\widetilde{\Aut O(\pbarO(S)[\varepsilon])}_{\mC[\varepsilon]}$
	\begin{align*}
	(T\otimes 1 + 1\otimes \partial_t)_{\varepsilon} &= \left( z \mapsto z - \varepsilon, f \mapsto f + \varepsilon\partial_tf \right), \\
	(T \otimes (\partial_ts)+ 1\otimes \partial_t)_{\varepsilon} &= \left( z \mapsto z - \varepsilon(\partial_ts), f \mapsto f + \varepsilon\partial_tf \right), \\
	\varphi_{t,s} &= \left( z \mapsto \sum_{k\geq 1}\frac{1}{k!}(\partial^k_ts)z^k, f \mapsto f \right).
	\end{align*}
	By direct computation we have
	\begin{align*}
	\varphi_{t,s}( T\otimes (\partial_ts) + 1\otimes \partial_t)_\varepsilon &= \left( z \mapsto \sum_{k\geq 1} (\partial_t^ks)z^{k} - \varepsilon(\partial_ts), f \mapsto f +\varepsilon\partial_tf \right), \\
	(T\otimes 1 + 1\otimes \partial_t)_\varepsilon\varphi_{t,s} &= \left( z \mapsto \sum_{k\geq 1}\frac{1}{k!}(\partial_t^ks)z^k + \varepsilon \sum_{k \geq 1}\frac{1}{k!}(\partial_t\partial_t^{k}s)z^k - \varepsilon\sum_{k\geq 0} \frac{1}{k!}(\partial_t^{k+1}s)z^k, f \mapsto f + \varepsilon\partial_tf \right)
	\end{align*}
	which are equal.
\end{proof}

As a particular case of the proposition above we have.

\begin{corollary}\label{coro:funopDalgebra}
	The topological sheaf of $\pbarO$-modules $\twistAut{\oSigma}{V^\kappa(\gog)}$ is naturally a left $\calD_{\overline{\Sigma}}$ module.
\end{corollary}

\subsection{The case of  \texorpdfstring{$\twistAut{\oSigma}{V^\kappa(\gog)}$}{of functions on Opers over the formal neighborhood}}

In this Section we give another description of the $\pbarD$ module $\twistAut{\oSigma}{V^\kappa(\gog)}$. As for a right $\calD$ modules on varieties, given a right $\pbarD$-module $\calM$ the tensor product $\calM\otimes_{\pbarO} T_{\oSigma}$ has a structure of  
left $\pbarD$ module where the action of $\theta \in T_{\oSigma}$ is given by 
$ \theta \cdot (m\otimes \eta)=-m\cdot \theta \otimes \eta + m \otimes [\theta,\eta]$.

\begin{theorem}\label{thm:globaldescrchiralalg}
	Assume that $S$ is integral and quasi separated. Let $\Sigma : S \to X^I$ be an arbitrary finite collection of sections. Then there is a canonical isomorphism of left $\calD_{\overline{\Sigma}}$-modules
	\[
	\beta : \VkSg\otimes_\pbarO \Tan_{\overline{\Sigma}} \to \twistAut{\overline{\Sigma}}{V^\kappa(\gog)}
	\]
	For any choice of a local coordinate $t$ we have a commutative diagram 
	\[\begin{tikzcd}
	& {V^\kappa(\gog)\otimes\pbarO} \\
	{\calV^\kappa_{\overline{\Sigma}}(\gog)\otimes\Tan_{\overline{\Sigma}}} && {\twistAut{\oSigma}{\Vkg}}
	\arrow["{\psi_t}"', from=1-2, to=2-1]
	\arrow["{\varphi_t}", from=1-2, to=2-3]
	\arrow["\beta"', from=2-1, to=2-3]
	\end{tikzcd}\]
		where $\varphi_t$ is the isomorphism\index{$\psi_t$} of Remark \ref{rmk:changecoordfunopd} and $\psi_t : V^\kappa(\gog)\otimes\pbarO \to \VkSg\otimes_\pbarO\Tan_{\overline{\Sigma}} $ is defined by 
	$$
	\psi_t(x\otimes f) = \calY_{\Sigma,t}(x \otimes f) \otimes \partial_t
	$$
	in which $\calY_{\Sigma,t}$ is the isomorphism constructed in \ref{ssez:Y}, while $\partial_t \in \Tan_{\overline{\Sigma}}(S)$ is the derivation associated to the coordinate $t$.
\end{theorem}

The proof of this Theorem will occupy the remaining pages. 

We notice first that $\grf_t$ and $\psi_t$ are isomorphism of left $\pbarD$-modules, where $\partial_t$ acts as $T\otimes 1 + 1\otimes \partial_t$ on $\Vkg\otimes\pbarO$.
Hence locally we can define $\beta_t = \varphi_t\psi_t^{-1}$. We would like to set $\beta = \beta_t$, but given another coordinate $s$ we have a priori different isomorphism $\beta_s = \varphi_s\psi_s^{-1}$. To show that $\beta_t = \beta_s$ we show that $\varphi_{ts}=\varphi_t^{-1}\varphi_s = \psi_t^{-1}\psi_s=\psi_{ts}$\index{$\psi_{ts}$} for every choice of $t,s$.

We will first focus on the case where $S=\Spec A$ is affine and $\Sigma = \{ \sigma \}$ consists of a single section such that $\calO_{\overline{\sigma}} \simeq A[[t]]$ and $\calO_{\overline{\sigma}^*} \simeq A((t))$ and then use this result, together with some factorization properties to deal with the case of multiple sections.

\subsubsection{The case of a single section: translation}

Assume that $\Sigma = \{ \sigma \}$ consists of a single section, that $S =\Spec A$ is affine and that $\calO_{\overline{\sigma}} \simeq \calO_S[[t]]$ as topological sheaves of $\calO_S$ algebras. We study this case first because, under these assumptions, coordinates on $\pbarO$ are an $\Autpiu{S} O(A)$-torsor, so that the compositions $\varphi_s^{-1}\varphi_t, \psi_s^{-1}\psi_t$ induce, after the choice of a specific coordinate $t$, $\Autpiu{} O(A)$ actions on $V^\kappa(\gog)\otimes A[[t]]$ that we must check to be equal. We will show that to prove this is enough compare the actions of the Lie algebra $\Der^0 O$ and show that they are equal, which we check by direct computation. In what follows we formalize this idea.

We will work over global sections for simplicity. Recall that our notion of a coordinate is that of an étale map $X \to \mA^1_S$. Considering a coordinate $t \in \pbarO(S)$ we get an isomorphism of topological sheaves $\pbarO \simeq \varprojlim_n \calO_S[t]/\prod (t-a_i)^n$ for some functions $a_i \in A$. In this section we extend our notion of a coordinate to the following one: a coordinate will be a function $t \in \pbarO(S)$ which induces an isomorphism $\pbarO \simeq \varprojlim_n \calO_S[t]/\prod (t-b_i)^n$ as above, for some possibly different functions $b_i \in A$. In the particular case of a single section $\Sigma = \{\sigma \}$, after the choice of a coordinate $t$, which induces an isomorphism $\calO_{\overline{\sigma}}(S) \simeq A[[t]]$ the set of coordinates identifies with the series $f(t) = f_0 + f_+(t)$ such that $f_+(t) \in t\calO_S[[t]]$ and $\partial_t f_+(t) \in (\calO_S[[t]])^{*}$. In particular there is a natural action of $\mG_a(A) \times \Autpiu{} O(A)$ on the set of coordinates which is simply transitive.

\begin{lemma}\label{lem:trivialactionga}
	Let $s,t$ be two coordinates and assume that $s = t + a$ for some function $a \in A$. Then $\varphi_{ts},\psi_{ts}$ are the identity morphism.
\end{lemma}

\begin{proof}
	The assertion for $\varphi$ immediately follows from Remark \ref{rmk:changecoordfunopd} and by the fact that under the  assumptions on $s$ we have $\partial_st = 1$. The claim for $\psi$ follows by $\partial_t = \partial_s$ and by construction of $\calY$, in particular it follows from the formulas in Remark \ref{rmk:descrycan} that the formation of $\calY$ depends only on the differences $t\otimes 1-1\otimes t, s\otimes 1 - 1\otimes s$ which coincide.
\end{proof}

It follows from this Lemma that it is enough to focus on coordinates $s \in A[[t]]$ which are obtained from $t$ via the action of $\Autpiu{} O(A)$.

\subsubsection{The case of a single section: the \texorpdfstring{$\Autpiu{} O$}{Aut+ O} actions}\label{ssec:autoactions}
After the choice of a specific coordinate $t$, we claim that we get two (a priori different) actions $\psi,\varphi$ of $\Autpiu{} O(A)$ on $V^\kappa(\gog)\otimes A[[z]]$, associated to the morphisms $\varphi_t,\psi_t$ as we let our coordinate vary. We introduce another formal coordinate $z$ via the isomorphism $\rho_t : A[[z]] \to \calO_{\overline{\sigma}}(S)$ which maps $z$ to $t$. Consider the isomorphisms
\begin{align*}
        \tilde{\calY_t} &: V^\kappa(\gog) \otimes A[[z]] \xrightarrow{\id \otimes \rho_t} V^\kappa(\gog) \otimes \calO_{\overline{\sigma}} \xrightarrow{\calY_{\sigma,t}} \calV^\kappa_{\sigma}(\gog) \\
	\tilde{\psi}_t &: V^\kappa(\gog)\otimes A[[z]] \xrightarrow{\id \otimes \rho_t} V^\kappa(\gog)\otimes\calO_{\overline{\sigma}} \xrightarrow{\psi_t} \calV^\kappa_\sigma(\gog)\otimes\Tan_{\overline{\sigma}}, \\
	\tilde{\varphi}_t &: V^\kappa(\gog)\otimes A[[z]] \xrightarrow{\id \otimes \rho_t} V^\kappa(\gog) \otimes\calO_{\overline{\sigma}} \xrightarrow{\varphi_t} \twistAut{\overline{\sigma}}{V^\kappa(\gog)}.
\end{align*}
 \index{$\tilde{\calY_t},\tilde{\psi}_t,\tilde{\varphi}_t$}We need to introduce the new formal coordinate $z$ to obtain an actual action of $\Autpiu{} O$, which will act on $A[[z]]$ in the natural way.

\begin{lemma}\label{lem:groupactions} The formulas 
\begin{align*}
    \calY &: \Autpiu{} O(A) \to \Aut_A(V^\kappa(\gog)\otimes A[[z]]) \quad \tau(z) \mapsto \tilde{\calY}_{t}^{-1}\tilde{\calY}_{\tau(t)} \\
	\psi &: \Autpiu{} O(A) \to \Aut_A(V^\kappa(\gog)\otimes A[[z]]) \quad \tau(z) \mapsto \tilde{\psi}_{t}^{-1}\tilde{\psi}_{\tau(t)} \\
	\varphi &: \Autpiu{} O(A) \to \Aut_A(V^\kappa(\gog)\otimes A[[z]]) \quad \tau(z) \mapsto \tilde{\varphi}_{t}^{-1}\tilde{\varphi}_{\tau(t)}
\end{align*}
	\index{$\calY,\psi,\varphi$}define actions of $\Autpiu{} O(A)$ on $V^\kappa(\gog)\otimes A[[z]]$. Both these actions are $\Autpiu{} O(A)$ sesquilinear with respect to the natural action of $\Autpiu{} O(A)$ on $A[[z]]$, meaning that for $v \in V^\kappa(\gog)\otimes A[[z]]$ and $f \in A[[z]]$ we have
	\begin{align*}
		\varphi(\tau) (v\cdot f(z)) &= \varphi(\tau)(v)\cdot f(\tau(z)) \\
		\psi(\tau) (v\cdot f(z)) &= \psi(\tau)(v)\cdot f(\tau(z))
	\end{align*}
\end{lemma}

\begin{proof}
	Sesquilinearity follows by the fact that $\psi_t,\varphi_t$ are $\calO_{\overline{\sigma}}(S)$ linear and that $\rho_t^{-1}\rho_{\tau(t)} = \tau$ as automorphisms of $A[[z]]$. The different cases of $\varphi$, $\calY$ and $\psi$ need different proofs.
	\begin{itemize}
		\item[($\varphi$)] Here we think $\tau$ as its associated formal series. Consider $$\widetilde{\Autpiu{} O(A[[z]])}_A = \Autpiu{} O (A[[z]])\rtimes \Autpiu{A}(A[[z]]),$$ the semidirect product of Remark \ref{rmk:extensionautoaction}, which naturally acts on $V^\kappa(\gog)\otimes A[[z]]$. To avoid possible confusion we introduce another formal variable $w$ and write $\Autpiu{} O (A[[z]]) = \Autpiu{A[[z]]}(A[[z]][[w]]) $, while we keep the notation $\Autpiu{} O (A) = \Autpiu{A}(A[[z]])$. Unravelling the definitions one may check that, using Remark \ref{rmk:changecoordfunopd}, when acting on $V^\kappa(\gog)\otimes A[[z]]$ the following equality holds:
		\[
			\varphi(\tau) = \left( w \mapsto \sum_{k \geq 0} \frac{1}{k!}(\partial^k_z \tau(z))w^k, f(z) \mapsto f(\tau(z)) \right) \in \widetilde{\Autpiu{} O(A[[z]])}_A.
		\]	
		To check that the corresponding map $\varphi : \Autpiu{} O(A) \to \widetilde{\Autpiu{} O(A[[z]])}_A$ is a morphism of groups we look at the embedding $\widetilde{\Autpiu{} O(A[[z]])}_A \subset \Autpiu{A}(A[[z,w]])$ where the latter group is the group of continuous automorphisms which preserve the ideal $(z,w)$. Let $u = w + z$, it is easy to check that $\varphi(\tau)$ corresponds the automorphism of $A[[z,u]] = A[[z,w]]$ which maps $z \mapsto \tau(z)$, $u \mapsto \tau(u)$, so that the fact that $\varphi$ is a morphism of groups is evident.

		\item[($\calY$)] Consider the space of fields with a fixed coordinate $z$, $\mF^1_{z,\gog} = \Homcont_A\left(A((z)),U_\kappa(\hat{\gog}_A) \right)$, where $\hat{\gog}_{A,\kappa}$ is the $A$ linear version of the affine algebra and $U_\kappa(\hat{\gog}_A)$ is the corresponding completed enveloping algebra. This is equipped with a natural $\Autpiu{} O(A)$ action by conjugation that we will denote by $\calY^{\mF}_{\tau}$, for $\tau \in \Autpiu{} O (A)$. 
		
		For any coordinate $t$ the isomorphism $\rho_t : A[[z]] \to \calO_{\overline{\sigma}}$ induces, by conjugating by $\rho_t$, an isomorphism $\calY^{\mF}_t : \mF^1_{z,\gog} \to \mF^1_{\sigma,\gog}$, it's easy to check that the composite $(\calY^{\mF}_t)^{-1}\calY^{\mF}_{\tau(t)}$ agrees with the action $\calY^{\mF}_{\tau}$ on $\mF^1_{z,\gog}$. 

		In addition, we have that the following diagram commutes:

		\[\begin{tikzcd}
	{V^\kappa(\gog)\otimes A[[z]]} && {\mF^1_{z,\gog}} \\
	\\
	{V^\kappa(\gog)\otimes\calO_{\overline{\sigma}}} && {\mF^1_{\sigma,\gog}}
	\arrow["{\calY_z}", from=1-1, to=1-3]
	\arrow["{\id\otimes \rho_t}"', from=1-1, to=3-1]
	\arrow["{\calY^{\mF^1}_{t}}", from=1-3, to=3-3]
	\arrow["{\calY_{\sigma,t}}"', from=3-1, to=3-3]
\end{tikzcd}\]
		This diagram is indeed equivalent to $\calY_{\sigma,t}\left(v\otimes \rho_t(f)\right)(\rho_t(g)) = \rho_t\left(\calY_{z}(v\otimes f)(g)\right)$, which immediately follows by induction from the construction of $\calY$.
		It follows that the subspace determined by the image of $\calY_z$ is invariant under the $\Autpiu{} O(A)$ action and that for any $\tau \in \Autpiu{} O(A)$ the restriction of $\calY^{\mF^1}_\tau$ to $V^\kappa(\gog) \otimes A[[z]]$ agrees with $\calY_\tau$. Since the former is a group action, so is the latter.
	\item[($\psi$)] To prove the Lemma for $\psi$ it is enough to notice that $\psi(\tau) = \mathrm{mult}((\partial_z\tau(z))^{-1}) \circ \calY(\tau)$ so that $\psi$ being a morphism of groups is automatic by sesquilinearity.\qedhere
	\end{itemize}
\end{proof}

By the definition of the above actions, it looks like that they depend on the choice of the coordinate $t$. It turns out that they do not actually depend from it, as the following Lemma shows. That's why we dropped $t$ from the notation.

\begin{lemma}
    The actions $\varphi,\psi,\calY$ are independent from the choice of the coordinate $t$.
\end{lemma}

\begin{proof}
    The same proof works for all actions, so let us pick $\varphi$ as an example. Write, for the purposes of this lemma $\varphi^t(\tau)=\varphi_t^{-1}\varphi_{\tau(t)}$. Here we think $\tau$ as a power series, so that it make sense to apply it to $t \in \pbarO$. We want to show that, given another coordinate $s = \tau_0(t)$, we have $\varphi^s(\tau) =\varphi^t(\tau)$. This follows by direct computation:
    \[
        \varphi^t(\tau_0)\varphi^s(\tau) =\varphi_t^{-1}\varphi_{\tau_0(t)}\varphi_s^{-1}\varphi_{\tau(s)} =\varphi_t^{-1}\varphi_{\tau_0(t)}\varphi_{\tau_0(t)}^{-1}\varphi_{\tau(\tau_0(t))} =\varphi^t(\tau_0\cdot\tau),
    \]
    here $\tau_0\cdot \tau$ denotes the product in $\Autpiu{} O$, so that $(\tau_0\cdot\tau)(t)= \tau(\tau_0(t))$. Using that $\varphi^t$ is a morphism of groups it follows that $\varphi^s(\tau)=\varphi^t(\tau)$, as claimed.
\end{proof}

In order to prove Theorem \ref{thm:globaldescrchiralalg} in the case of one section we reduce ourselves to prove that these actions are actually the same.

\begin{lemma}\label{lem:actionsequal}
	Assume that the actions constructed above are equal, so that $\varphi = \psi$. Then for any two coordinates $s,t$ we have $\varphi_s\psi_{s}^{-1} = \varphi_t\psi_{t}^{-1}$.
\end{lemma}

\begin{proof}
	By Lemma \ref{lem:trivialactionga}, the case $s = t + a$ with $a \in \calO(S)$ is trivial, so we may reduce to the case when there is an element $\tau \in \Autpiu{} O(A)$ such that $s = \tau(t)$. Recall that $\tilde{\varphi}_t = \varphi_t\rho_t$ where $\rho_t : A[[z]] \to \pbarO(S)$ is the automorphism $z \mapsto t$ so that $\varphi_t\psi_t^{-1} = \tilde{\varphi}_t\tilde{\psi}_t^{-1}$. By assumption, we have $\tilde{\varphi}_{t}^{-1}\tilde{\varphi}_{\tau(t)} = \tilde{\psi}_{t}^{-1}\tilde{\psi}_{\tau(t)}$ so the claim easily follows.
\end{proof}

We now describe the actions of $\Autpiu{} O(A)$ in better terms. Recall that since both actions of $\Autpiu{} O(A)$ are $A[[z]]$ sesquilinear, they preserve the subspaces $V^\kappa(\gog)\otimes z^nA[[z]]$. Since in addition $V^\kappa(\gog)\otimes A[[z]] \subset V^\kappa(\gog)[[z]]$ we reduce ourselves to check that the two actions are equal on the quotients $V^\kappa(\gog)\otimes A[[z]]/(z^n)$.

We start with a technical Lemma on representations of the group $\Autpiu{} O$. We refer to Section \ref{ssec:recollectionsauto} for the notation on the group $\Autzero{} O$ and the fact that $\Autpiu{} O = \mG_m \ltimes \Autzero{} O$.

\begin{lemma}\label{lem:repsofautodero}
    Let $V$ and $W$ be two representations of $\Autpiu{} O$ (by that we mean a morphism of group functors $\Autpiu{} O(R)\to GL_R(V\otimes R)$) and let $\varphi : V \to W$ be a $\mC$ linear map. If $\varphi$ is $\Der^0 O$-equivariant then it is also $\Autpiu{}O$ equivariant.
\end{lemma}

\begin{proof}
    Since $\Autpiu{} O$ is a group scheme any representation of it is the union of its finite dimensional sub-representations; it is then enough to prove the claim for $V,W$ finite dimensional. Since $\Autpiu{} O = \mG_m \ltimes \Autzero{} O$, it is enough to show equivariance separately for $\mG_m$ and for $\Autzero{} O$. Equivariance for $\mG_m$ follows by the fact that $\mG_m$ is connected. To check that $\varphi$ is $\Autzero{} O$ equivariant notice that the action of $\Autzero{} O$ factors via an action of $\Autzero{n} O$ on both $V$ and $W$ for some $n$ so we are reduced to prove the analogue of the statement of the Lemma for an unipotent group of finite type, which follows since any such group is connected.
\end{proof}

\begin{lemma}\label{lem:autoaaction}
	The $\Autpiu{} O(A)$ actions on $\psi,\varphi$ on $V^\kappa(\gog)\otimes A[[z]]/(z^n)$ are compatible with base change of $A$. In particular, letting $A$ vary, $\psi,\varphi$ determine an action of the group scheme $\Autpiu{} O$ on $V^\kappa(\gog)\otimes \mC[[z]]/(z^n)$. 
\end{lemma}

\begin{proof}
	The assertion for $\varphi$ is evident since $\Aut_{\overline{\Sigma}}\left(V^\kappa(\gog)\right)$ is obtained from the $\Autpiu{A} O$ torsor $\Aut_{\overline{\Sigma}}$ and by definition the action of $\varphi$ is exactly the one induced by a trivialization as in Remark \ref{rmk:changecoordfunopd}. To prove the statement for $\psi$ it is enough to notice that the morphisms $\psi_t$ behaves well under (completed) base change and are defined for any choice of the base ring $A$.
\end{proof}

\begin{corollary}\label{cor:sameactiondero}
	To check that the actions $\psi,\varphi$ are equal it is enough to prove that the associated actions of $\Der^0 O$ are the same when $S = \Spec \mC$. In addition by $A[[z]]$ sesquilinearity we just need to check that the actions $\varphi,\psi$ of $\Der^0 O$ are equal when restricted to $V^\kappa(\gog)$. 
\end{corollary}

\begin{proof}
	Use Lemma \ref{lem:autoaaction} and Lemma \ref{lem:repsofautodero}.
\end{proof}

\subsubsection{The case of a single section: comparison of the \texorpdfstring{$\Der^0 O$}{Der+ O} actions}\label{ssec:comparisonderoactionsphipsi}

Thanks to Corollary \ref{cor:sameactiondero} and Lemma \ref{lem:actionsequal}, to prove Theorem \ref{thm:globaldescrchiralalg} we just need to compare the infinitesimal actions of $\varphi,\psi$. By Lemma \ref{lem:autoaaction}, we only need to treat the $S = \Spec \mC$ case.

\begin{remark}\label{rmk:nontildeversion}
	Given any automorphism $\tau \in \Autpiu{} O$, to check that $\varphi(\tau) = \psi(\tau)$ it is enough to check that $\varphi_{t,\tau(t)} = \psi_{t,\tau(t)}$ (the non-\textquote{tilde} versions). This follows from $\varphi(\tau) = \rho_t^{-1}\varphi_{t,\tau(t)}\rho_{\tau(t)}$ and $\psi(\tau) = \rho_t^{-1}\psi_{t,\tau(t)}\rho_{\tau(t)}$. So, in order to prove \ref{thm:globaldescrchiralalg} we reduce ourselves to prove that $\varphi_{t,\tau(t)} = \psi_{t,\tau(t)}$, instead of showing $\tilde{\varphi}_{t,\tau(t)} = \tilde{\psi}_{t,\tau(t)}$. 
 
    We needed to introduce $\tilde{\varphi},\tilde{\psi}$ to obtain actual group actions: this wouldn't have been the case working with $\varphi_{t,s},\psi_{t,s}$ alone. We believe that the computations we will perform from now on are better understood by considering $\varphi_{t,s},\psi_{t,s}$, since they are $\calO_{\overline{\sigma}}$-linear, so in what follows the \textquote{tilde} versions $\tilde{\varphi}_{t,s},\tilde{\psi}_{t,s}$ won't appear again. 
\end{remark}

To compare the actions of $L_f = f(z)\partial_z \in \Der^0 O$ , we have to work with $S = \Spec \mC[\varepsilon]$, consider the new coordinate $s = t + \varepsilon f(t)$ and compute $\psi_{t,s},\varphi_{t,s}$. In order to compute $\psi_{t,s}$ it will come in handy to consider $\calY_{t,s}$ as well. We have the following formulas

\begin{align}\label{eq:derivativesvsautmorph}
        \calY_{t,s}(x) &= x + \varepsilon L^{\calY}_f x \\
	\psi_{t,s}(x) &= x + \varepsilon L^\psi_f x \\
	\varphi_{t,s}(x) &= x + \varepsilon L^\varphi_f x
\end{align}

Since these act $\calO_{\overline{\sigma}}$-linearly we just need to compute the action on elements $v \in V^\kappa(\gog) \subset V^\kappa(\gog)\otimes \calO_{\overline{\sigma}}$, so that we reduce to show that $\psi_{t,s} = \varphi_{t,s}$ on $V^\kappa(\gog)\subset V^\kappa(\gog)\otimes\calO_{\overline{\sigma}}$.  As the formulas above indicate, we denote by $L^{\varphi}_f,L^\psi_f$ the actions of $L_f$ with respect to $\varphi,\psi$ respectively. 

In what follows we consider the operators $X_{(n)}$ for $X \in \gog$ which act $\calO_{\overline{\sigma}}[\varepsilon]$ linearly on $V^\kappa(\gog)\otimes \calO_{\overline{\sigma}}[\varepsilon]$. 

\begin{remark}\label{rmk:deroactioncommutationvacuum}
	The action of the operators $X_{(m)}$ on the vacuum vector $\vac$ generates the whole space $V^\kappa(\gog)$ so that to prove $L_f^{\varphi} = L_f^\psi$ it is enough to prove that $L_f^\varphi\vac = L_f^\psi \vac$ and that $[L_f^\varphi,X_{(m)}] = [L_f^\psi,X_{(m)}]$.
\end{remark}

\begin{proposition}\label{prop:actionvacuum}
    Let $\vac \in V^\kappa(\gog)$ be the vacuum vector. Then for any coordinates $t,s \in \calO_{\overline{\Sigma}}$ we have
    \[
        \varphi_{t,s}\vac = \psi_{t,s}\vac = \vac,
    \]
    so $\vac$ is an invariant vector for both actions $\varphi,\psi$.
\end{proposition}

\begin{proof}
    Notice that since the $\Autpiu{} O$ action on $\vac \in V^\kappa(\gog)$ is trivial the statement about $\varphi$ is immediate.
    To prove the statement about $\psi$, we claim that $\psi_t = \psi_s$ when restricted to $\vac\otimes\calO_{\overline{\sigma}}$.
    We notice that $\psi_t : \vac\otimes\calO_{\overline{\sigma}} \to \calV^{\kappa}_\sigma(\gog)\otimes\Tan_{\overline{\sigma}}$ factors through $\calO_{\overline{\sigma}} \xrightarrow{\text{unit}} \Omega^1_{\overline{\sigma}}\otimes\Tan_{\overline{\sigma}} \xrightarrow{u\otimes\id} \calV_{\sigma}^\kappa(\gog)\otimes\Tan_{\overline{\sigma}}$, where the morphism $\text{unit}$ is the unit of the $\calO_{\overline{\sigma}}$-linear duality pairing between $\Omega^1_{\overline{\sigma}}$ and $\Tan_{\overline{\sigma}}$. This shows that $\psi_t$, restricted to $\vac\otimes\calO_{\overline{\sigma}}$, does not depend on the choice of $t$.
\end{proof}

\begin{remark}\label{rmk:psivscaly}
	We start with some remarks about $\psi$. Recall that $\psi_t = \calY_t \partial_t$, so that for any $x \in V^\kappa(\gog)\otimes\calO_{\overline{\Sigma}}$ we have 
	
	\[
		\psi_{ts} x = \psi_t^{-1}(\calY_s(x)\partial_s) = \psi_t^{-1}(\calY_s(x)(\partial_st)\partial_t) = \calY_t^{-1}\calY_s(x)\partial_st = \calY_{ts}(x)\partial_s t,
	\]
	where we set $\calY_{ts} \stackrel{\text{def}}{=} \calY_t^{-1}\calY_s$, which is a $\calO_{\overline{\Sigma}}$ linear automorphism of $V^\kappa(\gog)\otimes\calO_{\overline{\sigma}}$. In particular, recalling the notation of formula \eqref{eq:derivativesvsautmorph} in the $\calY$ case, we have that for any $x \in V\otimes\calO_{\overline{\sigma}}$,
	\[
		x + \varepsilon L^\psi_f x = (x + \varepsilon L^{\calY}_f x)(1 - \varepsilon\partial_t f(t)) 
	\]
	so that 
	\[
	L^\psi_f = L^{\calY}_f - \partial_tf(t)\id.
	\]
\end{remark}

\begin{proposition}\label{prop:deroactionpsi}
	The endomorphisms $L_f^{\calY},L_f^\psi$ on $V^\kappa(\gog) \subset V^\kappa(\gog)\otimes \calO_{\overline{\sigma}}$ satisfy the following:
	\[
		[L^{\calY}_f]=[L^\psi_f,X_{(m)}] = m\sum_{k \geq 0} \frac{1}{(k+1)!}X_{(k+m)} \otimes \partial_t^{k+1}f(t).
	\]
    Notice that the above infinite sum is a well defined endomorphism of $V^\kappa(\gog) \otimes \calO_{\overline{\sigma}}$. 
\end{proposition}

\begin{proof}
	As before let $s = t + \varepsilon f(t)$. Since $X_{(m)}$ is $\calO_{\overline{\sigma}}$-linear, by Remark \ref{rmk:psivscaly} we have $[L^\psi_f,X_{(m)}] = [L^\calY_f,X_{(m)}]$ so that it is enough to prove the analogous statement for $\calY$. Recall that $\calY_t$ is constructed inductively on $V^\kappa(\gog)$ by formula \eqref{eq:vertalgprod}. As before consider $\calY_{t,s} = \calY_{t}^{-1}\calY_s$ and write, for $v \in V^\kappa(\gog)$ 
	\[
		\calY_{t,s}v = v + \varepsilon L^\calY_fv \in V^\kappa(\gog) \otimes \calO_{\overline{\sigma}}[\varepsilon].
	\]
	In particular, with our notation, we have
	\[
		\calY_s v = \calY_{t}(v + \varepsilon L^\calY_f v)
	\]
	Recall that by construction for $X \in \gog$ we have $\calY_{s}(X_{(-1)}\vac) = \calY_{t}(X_{(-1)}\vac)$. We compute $\calY_{t,s}$ on $X_{(m)}Y$ for $X\in \gog$. As a field we have
	\[
		\calY_{s} (X_{(m)}Y) = (\calY_s X)_{(m,s)}(\calY_s Y)
	\]
	where, given any two fields $A,B \in \mF^1_{\Sigma,\calU}$, we set
	\[
		A_{(m,s)}B(g) = (AB)((s_1-s_2)^m(1\otimes g)) - BA((s_2-s_1)^m(g\otimes 1)).
	\]
	Writing $(s_1 - s_2)^m$ as a sum of powers of $(t_1 - t_2)$,
	\begin{equation*}
		(s_1 - s_2)^m = (t_1 - t_2)^m + m\varepsilon\sum_{k \geq 0} \frac{1}{(k+1)!}\partial^{k+1}_{t_2}f(t_2)(t_1 - t_2)^{k+m},
	\end{equation*}
	we deduce that given two fields $A,B \in \mF^1$ we have
	\[
		A_{(m,s)}B = A_{(m,t)}B + m\varepsilon\sum_{k\geq 0} \frac{1}{(k+1)!} A_{(k+m,t)}B \partial^{k+1}_tf(t).
	\]
	Finally, for $X \in \gog$ and $Y \in V^\kappa(\gog)$ we compute
	\begin{align*}
		&\calY_{t,s}(X_{(m)}Y) = \calY^{-1}_t\left( \calY_s(X)_{(m,s)}\calY_s(Y) \right) \\
		&= \calY_t^{-1} \left( \calY_s(X)_{(m,t)}\calY_s(Y) + m\varepsilon\sum_{k\geq 0} \frac{1}{(k+1)!} \calY_s(X)_{(k+m,t)}\calY_s(Y) \partial^{k+1}_tf(t) \right) \\
		&= \calY_t^{-1} \left( \calY_t(X)_{(m,t)}\calY_t\calY_{t,s}(Y) + m\varepsilon\sum_{k\geq 0} \frac{1}{(k+1)!} \calY_t(X)_{(k+m,t)}\calY_t\calY_{t,s}(Y) \partial^{k+1}_tf(t) \right) \\
		&= X_{(m)}(\calY_{t,s}Y)  + m\varepsilon\sum_{k\geq 0} \frac{1}{(k+1)!} X_{(k+m)}\calY_{t,s}(Y) \partial^{k+1}_t f(t)
	\end{align*}
	and writing $\calY_{t,s}(X) = X + \varepsilon L^\calY_fX$ we get
	\[
		[L^\calY_f,X_{(m)}] = m\sum_{k \geq 0} \frac{1}{(k+1)!}X_{(k+m)}Y \otimes \partial_t^{k+1}f(t)
	\]
	and the Proposition follows.
\end{proof}

We now show the same formula for $L^\varphi_f$.

\begin{proposition}\label{prop:deroactionphi}
	The endomorphisms $L_f^\varphi$ on $V^\kappa(\gog) \subset V^\kappa(\gog)\otimes \calO_{\overline{\sigma}}$ satisfy the following:
	\[
		[L^\varphi_f,X_{(m)}] = m\sum_{k \geq 0} \frac{1}{(k+1)!}X_{(k+m)} \otimes \partial_t^{k+1}f(t).
	\]
    Notice that the above infinite sum is a well defined endomorphism of $V^\kappa(\gog) \otimes \calO_{\overline{\sigma}}$. 
\end{proposition}

\begin{proof}
	Let $s = t + \varepsilon f(t)$ and as before write $\varphi_{t,s}v = v + \varepsilon L_f^{\varphi}v$. Recall that in this case, by Remark \ref{rmk:changecoordfunopd}, the action of $\Autpiu{} O$ is induced by the the action of $\Autpiu{} O(\calO_{\overline{\sigma}}[\varepsilon])$ on $V^\kappa(\gog)\otimes\calO_{\overline{\sigma}}[\varepsilon]$, from the element
	\[
		\left( z \mapsto z + \varepsilon\sum^n_{k  =  0} \frac{1}{(k+1)!}\partial_t^{k+1}f(t) z^{k+1} \right) \in \Autpiu{} O (\calO_{\overline{\sigma}}[\varepsilon]).
	\]
	This action comes from the action of $\Autpiu{} O$ on the Lie algebra $\hat{\gog}_\kappa$ so that 
	\begin{equation}\label{eq:changecoordliealg}
		\varphi_{t,s}(X_{(m)}Y) = \varphi_{t,s}((Xz^m) \cdot Y) = (\varphi_{t,s}(Xz^m))\cdot (\varphi_{t,s}Y),
	\end{equation}
        Where $Xz^m \in \hat{\gog}_\kappa = \gog((z)) \oplus \mC\mathbf{1}$ acts on $Y \in V^\kappa(\gog)$. By direct calculation we have
	\begin{align*}
		\varphi_{t,s}(Xz^m) &= X\left(z + \varepsilon\sum_{k  \geq 0} \frac{1}{(k+1)!}\partial_t^{k+1}f(t)z^{k+1}\right)^m \\ &= X\left( z^m + m\varepsilon\sum^n_{k=0} \frac{1}{(k+1)!}\partial_t^{k+1}f(t)z^{k+m}\right) \\
		&= X_{(m)} + m\varepsilon\sum_{k \geq 0} \frac{1}{(k+1)!}X_{(k+m)}\otimes \partial_t^{k+1}f(t).
	\end{align*}
	Putting this information into equation \eqref{eq:changecoordliealg} and writing $\varphi_{t,s}X = X + \varepsilon L^\varphi_f X$ we get
	\begin{align*}
		X_{(m)}Y &+ \varepsilon L^\varphi_f(X_{(m)}Y) \\ &= \left( X_{(m)} + m\varepsilon\sum_{k\geq 0} \frac{1}{(k+1)!}X_{(k+m)}\otimes \partial_t^{k+1}f(t) \right) \cdot (Y + \varepsilon L^{\varphi}_fY) \\
		&=X_{(m)}Y + m\varepsilon\sum_{k\geq 0} \frac{1}{(k+1)!} X_{(k+m)}Y\otimes \partial_t^{k+1}f(t) + \varepsilon X_{(m)}(L^\varphi_f Y)
	\end{align*}
	so that the Proposition follows.
\end{proof}

\begin{proposition}\label{prop:canisozetaopdisc}
	Assume that $\Sigma = \{\sigma\}$ consists of a single section. Then the isomorphism $\beta = \varphi_t\psi_t^{-1}$ is independent from the choice of the coordinate $t$. It therefore induces a canonical isomorphism of left $\calD_{\overline{\sigma}}$-modules
	\[
		\beta : \calV^\kappa_{\sigma}(\gog)\otimes_{\calO_{\overline{\sigma}}} \Tan_{\overline{\sigma}} \to \twistAut{\overline{\sigma}}{V^\kappa(\gog)}.
	\]
\end{proposition}

\begin{proof}
	The combination of Corollary \ref{cor:sameactiondero} with Remark \ref{rmk:deroactioncommutationvacuum}, Proposition \ref{prop:actionvacuum} Proposition \ref{prop:deroactionphi} and Proposition \ref{prop:deroactionpsi} proves that for any $U \subset S$ which is affine and well covered there is a natural identification 
	\[
		\beta : \calV^\kappa_\sigma(\gog)\otimes_{\calO_{\overline{\sigma}}} \Tan_{\overline{\sigma}} \to \twistAut{\overline{\sigma}}{V^\kappa(\gog)},
	\]
	constructed via the choice of a coordinate, but independent from this choice. These can therefore be glued together to get a morphism $\beta$ which provides a canonical identification on the whole $S$. 
\end{proof}

\subsubsection{Proof of Theorem \ref{thm:globaldescrchiralalg} in the general case}
We now prove Theorem \ref{thm:globaldescrchiralalg} in the case of more than one section. 
Recall first that under our assumptions on $S$ (i.e. to be integral, topologically noetherian) for any open inclusion $j : U \subset S$ we have $\pbarO \subset j_*j^* \pbarO$.

	We prove that $\varphi_t\psi^{-1}_t$ is independent of the choice of a coordinate, by showing that $\varphi_t^{-1}\varphi_s = \psi_t^{-1}\psi_s$ for any choice of coordinates $t,s$. Recall that by Proposition \ref{prop:canisozetaopdisc} we already know this statement when $\Sigma$ consists of a single section, so we may assume that in $\Sigma$ there are at least two different sections. Without loss of generality we can assume that there are no sections which are equal so that we can consider $S_{\neq}$, the (non-empty) open subset of $S$ where all sections are different. Let $j : S_{\neq} \to S$ be the natural inclusion. By Lemma \ref{lem:factorizationpsi} and Lemma \ref{coro:factpropchiralalgcaly} we have a commutative diagram 
	\[\begin{tikzcd}
	{V^\kappa(\gog)\otimes\pbarO} && {\prod_{i\in I} V^\kappa(\gog)\otimes j_*\calO_{\overline{\Sigma}_i}} && {V^\kappa(\gog)\otimes\pbarO} \\
	\\
	{V^\kappa(\gog)\otimes\pbarO} && {\prod_{i\in I} V^\kappa(\gog)\otimes j_*\calO_{\overline{\Sigma}_i}} && {V^\kappa(\gog)\otimes\pbarO}
	\arrow["{(\fact^{\pbarO})^{-1}}", hook, from=1-1, to=1-3]
	\arrow["{\psi_{t,s}}", from=1-1, to=3-1]
	\arrow["{\prod_{i\in I}\varphi_{t_is_i}}", shift left=5, from=1-3, to=3-3]
	\arrow["{\prod_{i\in I}\psi_{t_is_i}}"', shift right=5, from=1-3, to=3-3]
	\arrow["{(\fact^{\pbarO})^{-1}}"', hook', from=1-5, to=1-3]
	\arrow["{\varphi_{t,s}}"', from=1-5, to=3-5]
	\arrow["{(\fact^{\pbarO})^{-1}}", hook, from=3-1, to=3-3]
	\arrow["{(\fact^{\pbarO})^{-1}}"', hook', from=3-5, to=3-3]
\end{tikzcd}\]
	so that the result follows from the case of a single section (i.e. Proposition \ref{prop:canisozetaopdisc}). \hfill \qedsymbol

\section*{Acknowledgements}

Both authors want to thank Alberto De Sole for his support and the encouragement during the preparation of this work. The first author also wants to thank the University of Pisa for the hospitality during his visits in which part of this work was carried out. 

\textbf{Funding:} The first author was funded by national
PRIN Grants 2022S8SSW and 2022HMBTTL and by INFN - CSN4 (Commissione Scientifica Nazionale 4 - Fisica Teorica), MMNLP project.

\printindex

\bibliography{biblio.bib}
\bibliographystyle{alpha}

\end{document}